\documentclass[reqno]{amsart}

\usepackage{amsmath, amssymb, amsthm}
\usepackage{aliascnt}
\usepackage[colorlinks, linkcolor=purple, citecolor=teal]{hyperref}
\usepackage{cleveref}
\usepackage{colonequals}
\usepackage{enumitem}
\usepackage{fullpage}
\usepackage[mathscr]{eucal}
\usepackage{mathtools}
\usepackage{stmaryrd}
\usepackage{tikz}
\usepackage{tikz-cd}
\usepackage[foot]{amsaddr}

\let\oldtocsection=\tocsection
\let\oldtocsubsection=\tocsubsection
\let\oldtocsubsubsection=\tocsubsubsection

\renewcommand{\tocsection}[2]{\hspace{0em}\oldtocsection{#1}{#2}}
\renewcommand{\tocsubsection}[2]{\hspace{1em}\oldtocsubsection{#1}{#2}}
\renewcommand{\tocsubsubsection}[2]{\hspace{2em}\oldtocsubsubsection{#1}{#2}}

\theoremstyle{definition}	

\newtheorem{thm}{Theorem}[subsection]

\newaliascnt{prp}{thm}
\newtheorem{prp}[prp]{Proposition}
\aliascntresetthe{prp}

\newaliascnt{lem}{thm}
\newtheorem{lem}[lem]{Lemma}
\aliascntresetthe{lem}

\newaliascnt{cor}{thm}
\newtheorem{cor}[cor]{Corollary}
\aliascntresetthe{cor}

\newaliascnt{dfn}{thm}
\newtheorem{dfn}[dfn]{Definition}
\aliascntresetthe{dfn}

\newaliascnt{exm}{thm}
\newtheorem{exm}[exm]{Example}
\aliascntresetthe{exm}

\newaliascnt{rmk}{thm}
\newtheorem{rmk}[rmk]{Remark}
\aliascntresetthe{rmk}

\newaliascnt{fct}{thm}
\newtheorem{fct}[fct]{Fact}
\aliascntresetthe{fct}

\numberwithin{equation}{section}
\numberwithin{figure}{section}
\numberwithin{table}{section}

\crefname{thm}{Theorem}{Theorems}
\crefname{prp}{Proposition}{Propositions}
\crefname{lem}{Lemma}{Lemmas}
\crefname{cor}{Corollary}{Corollaries}
\crefname{dfn}{Definition}{Definitions}
\crefname{exm}{Example}{Examples}
\crefname{rmk}{Remark}{Remarks}
\crefname{fct}{Fact}{Facts}
\crefname{equation}{}{}
\crefname{section}{\S\!}{\S\!}
\crefname{subsection}{\S\!}{\S\!}

\setlist[itemize]{nosep, leftmargin=2.3em}
\setlist[enumerate]{label=(\arabic*), nosep, leftmargin=2.3em}
\newlist{clist}{enumerate}{1}
\setlist[clist]{label=(\roman*), nosep, leftmargin=2.3em}

\newcommand{\bbC}{\mathbb{C}}
\newcommand{\bbK}{\mathbb{K}}
\newcommand{\bbN}{\mathbb{N}}

\newcommand{\bbR}{\mathbb{R}}
\newcommand{\bbZ}{\mathbb{Z}}

\newcommand{\clA}{\mathcal{A}}

\newcommand{\clF}{\mathcal{F}}
\newcommand{\clG}{\mathcal{G}}
\newcommand{\clL}{\mathcal{L}}

\newcommand{\frB}{\mathfrak{B}}

\newcommand{\frS}{\mathfrak{S}}
\newcommand{\frU}{\mathfrak{U}}

\newcommand{\frg}{\mathfrak{g}}
\newcommand{\frh}{\mathfrak{h}}

\newcommand{\sfA}{\mathsf{A}}

\newcommand{\sfC}{\mathsf{C}}
\newcommand{\sfM}{\mathsf{M}}

\newcommand{\bs}{\backslash}
\newcommand{\ceq}{\coloneqq}

\newcommand{\ol}{\overline}
\newcommand{\wh}{\widehat}
\newcommand{\wt}{\widetilde}

\newcommand{\inj}{\hookrightarrow}
\newcommand{\srj}{\twoheadrightarrow}
\newcommand{\sto}{\xrightarrow{\,\sim\,}}
\newcommand{\rarr}{\ratio\Longleftrightarrow}

\newcommand{\dbr}[1]{\llbracket #1 \rrbracket} 
\newcommand{\dpr}[1]{(\!( #1 )\!)}

\newcommand{\pdd}{\partial}

\newcommand{\pd}[2]{\frac{\partial #1}{\partial #2}}

\DeclareMathOperator{\Aut}{Aut}
\DeclareMathOperator{\End}{End}
\DeclareMathOperator{\Fld}{Field}

\DeclareMathOperator{\Hom}{Hom}
\DeclareMathOperator{\Img}{Im}
\DeclareMathOperator{\Ker}{Ker}

\DeclareMathOperator{\Span}{Span}

\DeclareMathOperator{\id}{id}
\DeclareMathOperator{\op}{op}

\DeclareMathOperator{\supp}{supp}

\usepackage{braket}

\newtheorem{theorem}{Theorem}

\crefname{theorem}{Theorem}{Theorems}

\newcommand{\ve}{\varepsilon}
\newcommand{\BVS}{\mathsf{BVS}}
\newcommand{\rms}{\mathrm{s}}
\newcommand{\sep}{\mathrm{sep}}
\newcommand{\ModC}{\mathsf{Mod}\,\bbC}
\newcommand{\CVS}{\mathsf{CVS}}
\newcommand{\bfV}{\mathbf{V}}
\newcommand{\bfVR}{\bfV_{\!R}}
\newcommand{\bfVi}{\mathbf{V}_{\!\infty}}
\newcommand{\vac}{\ket{0}}
\newcommand{\clD}{\mathcal{D}}
\newcommand{\rmc}{\mathrm{c}}
\newcommand{\sCVS}{\mathsf{CVS}^{\mathsf{super}}}
\newcommand{\super}{\mathrm{super}}
\newcommand{\clS}{\mathcal{S}}
\newcommand{\dgLie}[1]{\mathsf{dgLie}(#1)}
\newcommand{\LXD}{L^{\bullet}_{X, D}}
\newcommand{\CE}{\mathrm{CE}}
\newcommand{\LCD}{L^\bullet_{\bbC, \pdd_z}}
\newcommand{\LCDs}{L^{\bullet+1}_{\bbC, \pdd_z}}
\newcommand{\clU}{\mathcal{U}}
\newcommand{\Born}{\mathsf{Born}}
\newcommand{\olb}[1]{\ol{#1}^{\mathrm{b}}}
\newcommand{\CBVS}{\mathsf{CBVS}}
\newcommand{\ModK}{\mathsf{Mod}\,\bbK}
\newcommand{\frC}{\mathfrak{C}}
\newcommand{\frA}{\mathfrak{A}}

\DeclareMathOperator{\conv}{conv}
\DeclareMathOperator{\disk}{disk}
\DeclareMathOperator*{\blim}{blim}
\DeclareMathOperator{\totimes}{\wt{\otimes}}
\DeclareMathOperator{\Disks}{Disks}
\DeclareMathOperator{\Conf}{Conf}
\DeclareMathOperator*{\tbotimes}{\wt{\bigotimes}}
\DeclareMathOperator{\Sym}{Sym}
\DeclareMathOperator{\Lie}{Lie}
\DeclareMathOperator{\Cur}{Cur}
\DeclareMathOperator{\Vir}{Vir}
\DeclareMathOperator{\ev}{ev}
\DeclareMathOperator{\CA}{CA}
\DeclareMathOperator{\Spin}{Spin}

\begin{document}

\title{Factorization envelopes and enveloping vertex algebras}
\author{Yusuke Nishinaka}
\address{Department of Mathematics, Osaka Metropolitan University. Sugimoto, Sumiyoshi-ku, Osaka, 558-8585, Japan.}
\email{b26716g@omu.ac.jp}
\date{December 8, 2025; revised February 25 and September 21, 2026.} 

\begin{abstract}
We develop a bornological version of Costello and Gwilliam's procedure for extracting vertex algebras from suitable prefactorization algebras on the complex plane. Using bornological complex analysis, we remove the discreteness condition imposed in their extraction theorem. We then construct, from a suitable Lie conformal algebra, a prefactorization algebra to which this extraction procedure applies, and prove that the resulting vertex algebra is isomorphic to the enveloping vertex algebra of the original Lie conformal algebra. Our construction uses a factorization envelope and extends the construction of Costello--Gwilliam in the affine vertex algebra case and that of Williams in the Virasoro vertex algebra case. Moreover, a super analogue yields new prefactorization algebras corresponding to vertex superalgebras, such as the Neveu--Schwarz vertex superalgebra, the $N=2$ vertex superalgebra, and the $N=4$ vertex superalgebra. 
\end{abstract}

\maketitle

\setcounter{tocdepth}{2}
\tableofcontents

\section*{Introduction}

Factorization algebras, as developed by Costello and Gwilliam \cite{CG, CG2}, provide a mathematical framework for encoding observables in quantum field theory. They assign a space of observables to each spacetime region and describe how observables supported in disjoint regions combine. A particularly important source of factorization algebras is the factorization-envelope construction, which associates a factorization algebra to a local dg Lie algebra and may be regarded as a field-theoretic analogue of the universal enveloping algebra construction. 
This construction is fundamental to perturbative quantum field theory because the factorization algebras of observables of free field theories can be realized as factorization envelopes. Building on the renormalization techniques developed in \cite{C11}, Costello and Gwilliam constructed the factorization algebras of quantum observables of interacting theories as perturbative deformations of those associated with the corresponding free theories.

Vertex algebras, on the other hand, offer an algebraic formulation of two-dimensional chiral conformal field theory. It is therefore natural to ask how factorization algebras on the complex plane $\mathbb{C}$ relate to vertex algebras. In \cite[Chapter 5]{CG}, Costello and Gwilliam established a general procedure for extracting a vertex algebra from a suitable factorization algebra on the complex plane. The idea is roughly as follows. 

A prefactorization algebra $\clF$ on a manifold $X$, viewed physically as space-time, assigns the following data: 
\begin{itemize}
\item 
A linear space $\clF(U)$ for each open subset $U\subset X$. It represents the space of observables supported in $U$. 

\item 
A linear map $\clF^U_V\colon \clF(U)\to \clF(V)$, called the extension morphism, for each inclusion $U\subset V$ of open subsets of $X$. It expresses the fact that an observable supported in $U$ may also be regarded as an observable on the larger region $V$. 

\item 
A linear map $\clF^{U, V}_W\colon \clF(U)\otimes \clF(V)\to \clF(W)$, called the factorization product, for disjoint open subsets $U, V\subset X$ and an open subset $W\subset X$ such that $U\sqcup V\subset W$. It combines observables supported in the disjoint regions $U$ and $V$ into an observable on $W$. Products of observables localized at the same point may be singular, which is why factorization products are initially defined for disjoint regions. 
\end{itemize}
Here, we have described prefactorization algebras as taking values in linear spaces. More generally, they may take values in any symmetric monoidal category. 

For prefactorization algebras on $\bbC$, one can impose translation equivariance with respect to the additive group $\bbC$ acting by
translations. More precisely, a translation-equivariant
prefactorization algebra is a prefactorization algebra $\clF$ on
$\bbC$ equipped with morphisms
\[
\sigma_{z, U}\colon \clF(U)\to \clF(U+z)
\]
for each $z\in \bbC$ and each open $U\subset \bbC$, which are compatible with the prefactorization algebra structure of $\clF$. 

Let $\clF$ be a translation-equivariant prefactorization algebra on $\bbC$. The underlying linear space of the vertex algebra is obtained as a suitable subspace of the value of $\clF$ on an open disk, while its vertex operators arise from the factorization products.
Consider the open disk $D_R$ centered at the origin $0$ with radius $R>0$. For $z\in D_R\bs\{0\}$, choose $r>0$ sufficiently small so that $D_r(z)\sqcup D_r\subset D_R$. We then have a factorization product 
\[
\clF^{D_r(z), D_r}_{D_R}\colon \clF(D_r(z))\otimes \clF(D_r)\to \clF(D_R). 
\]
Composing it with the translation-equivariant structure $\sigma_z$ of $\clF$, we obtain 
\[
\clF^{r, r; R}_{z, 0}\colon 
\begin{tikzcd}[column sep=huge]
\clF(D_r)\otimes \clF(D_r) \arrow[r, "\sigma_{z, D_r}\otimes\, \id"] & \clF(D_r(z))\otimes \clF(D_r) \arrow[r, "\clF^{D_r(z), D_r}_{D_R}"] & \clF(D_R). 
\end{tikzcd}
\]
Since vertex operators are expressed as Laurent series $Y(a, z)b=\sum_{n\in\bbZ}z^{-n-1}a_{(n)}b$, 
one would like to regard $\clF^{r, r; R}_{z, 0}(a\otimes b)$ as a holomorphic function of $z$ on the relevant configuration space and define the products $a_{(n)}b$ as the coefficients of its Laurent expansion at the origin. To formulate this holomorphic dependence, however, the linear spaces $\clF(U)$ must be equipped with additional analytic structure. One is therefore led to consider prefactorization
algebras taking values in a suitable category of topological vector spaces, in which vector-valued holomorphic functions and their Laurent expansions can be defined.

At first sight, it may seem sufficient to consider prefactorization algebras taking values in Banach spaces, since Banach spaces provide a natural setting for vector-valued holomorphic functions. However, this category is too restrictive to accommodate the known examples. For instance, in \cite[Chapter~5]{CG}, Costello and Gwilliam constructed, by means of the factorization envelope, a factorization algebra on $\bbC$ to which the above extraction procedure applies, and proved that the resulting vertex algebra is isomorphic to the universal affine vertex algebra $V^\kappa(\frg)$. Their construction involves compactly supported Dolbeault complexes, which are naturally LF-spaces rather than Banach spaces. Consequently, the resulting factorization algebra does not take values in Banach spaces.

One may therefore enlarge the target category to that of complete Hausdorff locally convex vector spaces. However, this introduces some technical difficulties. There are several natural completed tensor products of locally convex vector spaces, such as the projective, injective, and inductive tensor products. If one wishes to work with all complete Hausdorff locally convex vector spaces, it is natural to adopt the completed projective tensor product, since it equips their category with a symmetric monoidal structure. 
The completed projective tensor product, however, does not in general commute with colimits, which is particularly relevant to constructions involving compactly supported Dolbeault complexes and factorization envelopes. 
Moreover, homological algebra in the category of complete Hausdorff locally convex vector spaces also requires additional care. This category is additive but not abelian, so homology cannot always be treated as in the category of ordinary linear spaces. Consequently, the choice of the target category for prefactorization algebras is a subtle issue, and different analytic frameworks have been adopted in the literature.

Costello and Gwilliam work in the category of differentiable vector spaces; see \cite[Appendix~B]{CG} for their definition and basic properties. 
Every locally convex vector space naturally determines a differentiable vector space, while the category of differentiable vector spaces is a Grothendieck abelian category. 
This flexible framework is well suited to their general treatment of perturbative quantum field theories based on the Batalin--Vilkovisky formalism, in which homological-algebraic constructions play a central role. On the other hand, the notion of a differentiable vector space is rather abstract, and this level of generality does not seem to be necessary for studying the relation between prefactorization algebras and vertex algebras.

Williams \cite{W} later constructed, using the factorization-envelope construction in the setting of nuclear locally convex vector spaces, a factorization algebra whose associated vertex algebra is the Virasoro vertex algebra $V^c$. Moreover, Bruegmann \cite{B20a, B20b} constructed a factorization algebra from every $\bbZ$-graded vertex algebra in the setting of complete separated convex bornological vector spaces. 
His construction uses bornologies, which axiomatize the notion of boundedness, rather than topologies. Other work on the relationship between factorization algebras and vertex algebras includes contributions by Gorbounov--Gwilliam--Williams \cite{GGW}, Saberi--Williams \cite{SWW}, Vicedo \cite{V}, and Yanagida \cite{Y}. 

Despite these remarkable developments, several fundamental questions remain open. First, although Bruegmann’s construction associates a factorization algebra with every $\bbZ$-graded vertex algebra, no characterization is known of a class of factorization algebras for which his construction yields an equivalence of categories, with the Costello–Gwilliam extraction procedure as its quasi-inverse.
Second, for the universal affine and Virasoro vertex algebras, it remains unknown whether the factorization algebras arising from Bruegmann’s construction agree with those obtained by the factorization-envelope constructions of Costello–Gwilliam and Williams, respectively.

One of the main obstacles to resolving these questions is that the factorization algebras involved take values in different target categories. 
The bornological setting is particularly well suited to factorization-envelope constructions, since the completed bornological tensor product preserves colimits in each variable. 
In this paper, we first establish a bornological version of Costello and Gwilliam’s procedure for extracting vertex algebras from suitable prefactorization algebras, without requiring the discreteness condition imposed in their extraction theorem. Using a factorization envelope, we then construct, from a suitable Lie conformal algebra, a prefactorization algebra to which this extraction procedure applies, and identify the resulting vertex algebra with the enveloping vertex algebra of the original Lie conformal algebra. 

\addtocontents{toc}{\protect\setcounter{tocdepth}{0}}
\subsection*{Main results}

We now describe these results in more detail. Our proofs rely on complex analysis for functions taking values in complete separated convex bornological vector spaces.
We call a complete separated convex bornological vector space a convenient vector space. For spaces whose bornology is the von Neumann bornology of a locally convex vector space, this terminology agrees with \cite{KM}. 
Roughly speaking, a convenient vector space $X$ is a linear space that can be written as a union $X=\bigcup_B X_B$ of Banach spaces $X_B$; see \cref{ss:sepcomp} for the precise definition. Following Bruegmann \cite[Definition 1.1.8]{B20b}, we define a function $f\colon U\to X$, where $U\subset \bbC^n$ is an open subset, to be bornologically holomorphic if every $a\in U$ has an open neighborhood $V\subset U$ and a Banach space $X_B$ such that $f(V)\subset X_B$ and $f|_V\colon V\to X_B$ is holomorphic in the usual sense. We show that this definition of holomorphy, together with complex integrals defined as bornological limits of Riemann sums, yields bornological analogues of many classical results, including Cauchy's integral theorem, Cauchy's integral formula, and Laurent expansion, thereby retaining the central features of classical complex analysis. 

\smallskip

Using the results of bornological complex analysis, we establish a bornological version of Costello and Gwilliam's procedure for extracting vertex algebras from prefactorization algebras. 

We start with an $S^1\ltimes \bbC$-equivariant prefactorization algebra $\clF$ on $\bbC$ taking values in convenient vector spaces. Here the group $S^1\ltimes \bbC$ acts on $\bbC$ as affine isometries. We denote the equivariant structure by 
\[
\sigma_{(q, z), U}\colon \clF(U)\to \clF((q, z)U)
\]
for each $(q, z)\in S^1\ltimes \bbC$ and each open $U\subset \bbC$. Note that the $\bbC$-action makes $\clF$ a translation-equivariant prefactorization algebra. We say that $\clF$ is holomorphically translation-equivariant if the function $(z_1, \ldots, z_n)\mapsto \clF^{r_1, \ldots, r_n; R}_{z_1, \ldots, z_n}(a_1\otimes \cdots \otimes a_n)$, defined by the factorization products and the translation-equivariant structure, is bornologically holomorphic on the appropriate configuration space of disks; see \cref{dfn:holoPFA} for the precise definition. Turning to the $S^1$-action, since the open disk $D_R$ is invariant under rotations, for each $\Delta\in \bbZ$, one can define the $\Delta$-weight space by 
\[
\clF(D_R)_\Delta \ceq \{a\in \clF(D_R)\mid \forall q\in S^1,\, \sigma_{(q, 0), D_R}(a)=q^\Delta a\}. 
\]
The compatibility of the equivariant structure with extension morphisms implies that, for $0<r<R\le \infty$, the image of $\clF(D_r)_\Delta$ under the linear map $\clF^{D_r}_{D_R}$ lies in $\clF(D_R)_\Delta$. 

We call $\clF$ an \emph{amenably holomorphic prefactorization algebra} if it satisfies the following conditions: 

\begin{clist}
\item 
The $\bbC$-action makes $\clF$ a holomorphically translation-equivariant prefactorization algebra. 

\item 
For $0<R\le \infty$, we have $\clF(D_R)_\Delta=0$ ($\Delta\ll0$). 

\item
For $0<r<R\le\infty$ and $\Delta\in \bbZ$, the map $\clF^{D_r}_{D_R}\colon \clF(D_r)_\Delta\to \clF(D_R)_\Delta$ is a linear isomorphism. 
\end{clist}
The following theorem gives a bornological version of Costello and Gwilliam's vertex-algebra extraction procedure.

\begin{theorem}[{\cref{thm:FAtoVA}}]\label{theorem:1A}
Let $\clF$ be an $S^1\ltimes \bbC$-equivariant prefactorization algebra on $\bbC$ taking values in convenient vector spaces. 
If $\clF$ is amenably holomorphic, then the linear space
\[
\bfVR(\clF)=\bigoplus_{\Delta\in \bbZ}\clF(D_R)_\Delta
\]
admits the structure of a $\bbZ$-graded vertex algebra. 
\end{theorem}

In \cite{B20b}, Bruegmann constructed geometric vertex algebras from suitable prefactorization algebras on $\bbC$ taking values in convenient vector spaces. He also proved in \cite{B20a} that geometric vertex algebras are equivalent to $\bbZ$-graded vertex algebras. Thus, via geometric vertex algebras, he obtained a construction of vertex algebras from prefactorization algebras in the setting of bornological vector spaces. \cref{theorem:1A} instead gives a direct construction of vertex algebras from prefactorization algebras. Moreover, his construction assumed that the weight spaces of $\clF$ are discrete, so the relevant holomorphy arguments reduce to locally finite-dimensional-valued functions in each weight component. \cref{theorem:1A} shows that we can dispense with this discreteness condition using bornological complex analysis. 

Costello and Gwilliam started with prefactorization algebras taking values in complexes of differentiable vector spaces and constructed vertex algebra structures on cohomology. We compare \cref{theorem:1A} with their construction in \cref{prp:CGconst}. This comparison also shows that our bornological formulation does not require the discreteness condition imposed in their extraction theorem. 
The term amenably holomorphic was used in \cite{GGW} for factorization algebras to which the Costello–Gwilliam extraction procedure applies. \cref{prp:CGconst} confirms that our use of this term is consistent with theirs.

\smallskip

We introduce the notion of convenient vector superspaces and extend \cref{theorem:1A} to the super setting, thereby obtaining a construction of vertex superalgebras from prefactorization algebras. 

In the super setting, we replace the rotation group $S^1$ with its double cover $\Spin(2)$. We identify $\Spin(2)$ with $S^1$ so that the covering map is $q\mapsto q^2$. The group $\Spin(2)$ acts on $\bbC$ by $q\cdot z=q^2z$. Let $\clF$ be a $\Spin(2)\ltimes \bbC$-equivariant prefactorization algebra on $\bbC$ taking values in convenient vector superspaces. For $\Delta\in \frac{1}{2}\bbZ$, we define the $\Delta$-weight space by 
\[
\clF(D_R)_\Delta \ceq \{a\in \clF(D_R)\mid \forall q\in \Spin(2),\, \sigma_{(q, 0), D_R}(a)=q^{2\Delta} a\}. 
\]
Then $\clF(D_R)_\Delta$ is a convenient vector superspace with parity decomposition given by 
\[
\clF(D_R)_{\Delta, \ol{i}}\ceq \clF(D_R)_\Delta\cap \clF(D_R)_{\ol{i}} \quad (i\in\{0, 1\}). 
\]
The notion of amenably holomorphic prefactorization algebra remains valid in the super case, with only minor modifications. 

\begin{theorem}[{\cref{thm:superFAtoVA}}]\label{theorem:1B}
Let $\clF$ be a $\Spin(2)\ltimes \bbC$-equivariant prefactorization algebra on $\bbC$ taking values in convenient vector superspaces. If $\clF$ is amenably holomorphic, then the linear superspace
\[
\bfVR^{\super}(\clF)\ceq \bigoplus_{\Delta\in \frac{1}{2}\bbZ}\clF(D_R)_\Delta, \quad \text{ where} \quad   \bfVR^{\super}(\clF)_{\ol{i}}\ceq \bigoplus_{\Delta\in \frac{1}{2}\bbZ}\clF(D_R)_{\Delta, \ol{i}} \quad (i\in\{0, 1\})
\]
admits the structure of a $\frac{1}{2}\bbZ$-graded vertex superalgebra. 
\end{theorem} 

We now turn to the results on the construction of prefactorization algebras from Lie conformal algebras via the factorization envelope. 

A Lie conformal algebra consists of a linear space $L$, a linear map $T\colon L\to L$ called the translation operator, and a linear map $[\cdot_\lambda\cdot]\colon L\otimes L\to L[\lambda]$ called the $\lambda$-bracket. Here $L[\lambda]$ denotes the linear space of polynomials in the variable $\lambda$ with coefficients in $L$. 
The notion of Lie conformal algebras captures the singular part of the operator product expansion of vertex algebras. 
Thus, every vertex algebra naturally gives rise to a Lie conformal algebra. This assignment admits a left adjoint
\[
\{\text{Lie conformal algebras}\}\to\{\text{vertex algebras}\}, \quad L\mapsto V(L). 
\]
The vertex algebra $V(L)$ obtained in this way is called the enveloping vertex algebra. 

Let $X$ be a complex manifold. The Dolbeault complexes $\clA^{0, \bullet}_X(U)$ form a sheaf as $U$ varies over the open subsets of $X$. Given a Lie conformal algebra $L$ and a morphism $D\colon \clA^{0, \bullet}_X\to \clA^{0, \bullet}_X$ of sheaves satisfying 
\[
D_U(\xi\wedge\eta)=D_U\xi\wedge\eta+\xi\wedge D_U\eta \quad (\xi\in \clA^{0, k}_X(U), \eta\in \clA^{0, l}_X(U)), 
\]
we construct a prefactorization algebra $\LXD$ on $X$ taking values in dg Lie algebras internal to the category of convenient vector spaces. 
The factorization envelope of $\LXD$ is obtained by applying the Chevalley--Eilenberg chain functor $C^\CE_\bullet$. It gives a prefactorization algebra $C^\CE_\bullet\LXD$ on $X$ taking values in chain complexes of convenient vector spaces. Passing to homology then yields a prefactorization algebra $H^\CE_\bullet\LXD$ on $X$ with values in convenient vector spaces. 

We focus on the case where $X=\bbC$ and $D=\pdd_z$. In this case, for an $\bbN$-graded Lie conformal algebra $L$, we can equip $\LCD$ with an $S^1\ltimes \bbC$-equivariant structure. This in turn induces an $S^1\ltimes \bbC$-equivariant structure on $C^\CE_\bullet\LCD$ and $H^\CE_\bullet \LCD$. Moreover: 

\begin{theorem}[{\cref{prp:CEdiff}, \cref{prp:CEholo}}]
\label{theorem:2}
For every $\bbN$-graded Lie conformal algebra $L$, the prefactorization algebra $C^\CE_\bullet\LCD$ is differentiably translation-equivariant (\cref{dfn:smPFA}) and the prefactorization algebra $H^\CE_\bullet \LCD$ is holomorphically translation-equivariant. 
\end{theorem}

To verify the remaining conditions for $H^\CE_\bullet\LCD$ to be amenably holomorphic, we impose the following technical condition:
\begin{itemize}
\item [$(*)$]
There exists a graded linear subspace $E\subset L$ such that 
\[
\bbC[T]\otimes E\to L,\quad p(T)\otimes x\mapsto p(T)x
\]
is a linear isomorphism. Here $T\colon L\to L$ denotes the translation operator of $L$. 
\end{itemize}

\smallskip

Many interesting Lie conformal algebras, such as the Virasoro conformal algebra (\cref{exm:Virconf}) and the current Lie conformal algebra (\cref{exm:curconf}), are one-dimensional central extensions $L_{\omega_\lambda}=L\oplus \bbC C$ of a Lie conformal algebra $L$ satisfying the condition $(*)$, associated with a $2$-cocycle $\omega_\lambda\colon L\otimes L\to \bbC[\lambda]$. 

For a 2-cocycle $\omega_\lambda$ of a Lie conformal algebra $L$, one defines a $(-1)$-shifted $2$-cocycle $\omega$ of the prefactorization algebra $\LCD$. Thus, the twisted version of the factorization envelope yields $S^1\ltimes \bbC$-equivariant prefactorization algebras $C^\CE_{\bullet, \omega}\LCD$ and $H^\CE_{\bullet, \omega}\LCD$. The analogue of \cref{theorem:2} also holds for these twisted factorization envelopes.

\begin{theorem}[{\cref{thm:main}, \cref{thm:main2}}]\label{theorem:main3A}
Let $L=\bigoplus_{\Delta\in \bbN}L_\Delta$ be an $\bbN$-graded Lie conformal algebra satisfying the condition $(*)$ above. 
\begin{enumerate}
\item 
The $S^1\ltimes \bbC$-equivariant prefactorization algebra $H^\CE_\bullet\LCD$ is amenably holomorphic, and its associated vertex algebra $\bfVR(H^\CE_\bullet\LCD)$ is isomorphic to the enveloping vertex algebra $V(L)$ as an $\bbN$-graded vertex algebra. 

\item 
For a graded $2$-cocycle $\omega_\lambda\colon L\otimes L\to \bbC[\lambda]$ of $L$, the $S^1\ltimes \bbC$-equivariant prefactorization algebra $H^\CE_{\bullet, \omega}\LCD$ is amenably holomorphic, and its associated vertex algebra $\bfVR(H^\CE_{\bullet, \omega}\LCD)$ is isomorphic to the enveloping vertex algebra $V(L_{\omega_\lambda})$ as an $\bbN$-graded vertex algebra. 
\end{enumerate}
\end{theorem}

\cref{theorem:main3A} (2) generalizes the factorization-envelope constructions of Costello and Gwilliam \cite[\S5.5]{CG} and Williams \cite{W}, which recover the affine vertex algebras and the Virasoro vertex algebra, respectively.
To prove \cref{theorem:main3A}, we examine the value of $H^\CE_\bullet \LCD$ on annuli. We find that it contains the universal enveloping algebra $U(\Lie(L))$ of the mode Lie algebra $\Lie(L)$ (\cref{rmk:ULieL}). This motivates viewing $\LXD$, for a general complex manifold $X$, as a higher-dimensional analogue of $\Lie(L)$. 

Moriwaki \cite{Mor1,Mor2} introduced the notion of full vertex algebras, non-chiral analogues of vertex algebras that incorporate both holomorphic and antiholomorphic sectors. Vicedo \cite{V} constructed, within a unified framework, linear-space-valued prefactorization algebras associated with the full counterparts of the affine, Virasoro, and $\beta\gamma$ vertex algebras. He verified the full vertex algebra axioms in the affine and Virasoro cases, while noting that his proof of one of the axioms does not apply to the $\beta\gamma$ case. It would be interesting to establish a full vertex algebra analogue of the Costello–Gwilliam extraction procedure, develop a bornological refinement of Vicedo’s construction, and prove a full analogue of \cref{theorem:main3A}. 

\smallskip

The construction of $\LXD$ can be extended to the super setting with minor modifications. That is, for a Lie conformal superalgebra $L$, one can construct a prefactorization algebra $\LXD$ on $X$ taking values in dg Lie algebras internal to the category of convenient vector superspaces. In the case where $X=\bbC$, $D=\pdd_z$, and $L$ is $\frac{1}{2}\bbN$-graded, one can equip $\LCD$ with a $\Spin(2)\ltimes \bbC$-equivariant structure. 

\begin{theorem}[{\cref{thm:main3}}]
\label{theorem:main3B}
Let $L=\bigoplus_{\Delta\in \frac{1}{2}\bbN}L_\Delta$ be a $\frac{1}{2}\bbN$-graded Lie conformal superalgebra satisfying the following condition: There exists a graded linear subsuperspace $E\subset L$ such that $\bbC[T]\otimes E\to L$, $p(T)\otimes x\mapsto p(T)x$ is a linear isomorphism. 
\begin{enumerate}
\item 
The $\Spin(2)\ltimes \bbC$-equivariant prefactorization algebra $H^\CE_\bullet\LCD$ is amenably holomorphic, and its associated vertex superalgebra $\bfVR^\super(H^\CE_\bullet\LCD)$ is isomorphic to the enveloping vertex superalgebra $V(L)$ as a $\frac{1}{2}\bbN$-graded vertex superalgebra. 

\item 
For a graded $2$-cocycle $\omega_\lambda\colon L\otimes L\to \bbC[\lambda]$ of $L$, the $\Spin(2)\ltimes \bbC$-equivariant prefactorization algebra $H^\CE_{\bullet,\omega}\LCD$ is amenably holomorphic, and its associated vertex superalgebra $\bfVR^\super(H^\CE_{\bullet,\omega}\LCD)$ is isomorphic to the enveloping vertex superalgebra $V(L_{\omega_\lambda})$ as a $\frac{1}{2}\bbN$-graded vertex superalgebra. 
\end{enumerate}
\end{theorem}

\cref{theorem:main3B} (2) applies, for example, to the Neveu--Schwarz Lie conformal superalgebra (\cref{exm:N=1LCA}), the $N=2$ Lie conformal superalgebra (\cref{exm:N=2LCA}), and the $N=4$ Lie conformal superalgebra (\cref{exm:N=4LCA}). Thus, we obtain new prefactorization algebras corresponding to vertex superalgebras. 
To the best of our knowledge, \cref{theorem:main3B} gives the first systematic construction of prefactorization algebras associated with vertex superalgebras. 

\smallskip

In summary, the main results consist of constructions that make the following diagram commute, as well as its super analogue: 
\smallskip
\[
\begin{tikzcd}[row sep=2cm, column sep=2.2cm]
\{\text{Lie conformal algebras}\} \arrow[r, "\text{\cref{theorem:main3A}}"] \arrow[d, "\text{enveloping vertex algebras\ }"'] & 
\{\text{amenably holomorphic prefactorization algebras}\} \arrow[ld, "\text{\cref{theorem:1A}}"] \\
\{\text{vertex algebras}\} &
\end{tikzcd}
\]
\smallskip

These results lay the groundwork for addressing the two questions raised above: characterizing the factorization algebras for which Bruegmann's construction and the Costello--Gwilliam extraction procedure are mutually inverse, and comparing Bruegmann's construction with the factorization-envelope constructions. We leave a detailed investigation of these questions for future work. 

\subsection*{Organization}

\cref{s:preFAVA} provides recollections of factorization algebras and vertex algebras. In \cref{ss:baFA}, we give the precise definition of prefactorization algebras and a detailed description of the factorization envelope. In \cref{ss:VA}, we briefly review the notion of vertex algebras and the construction of the enveloping vertex algebra of a Lie conformal algebra. We also present basic examples of vertex algebras. 

\cref{s:BVS} is devoted to preliminaries on bornological vector spaces. In \cref{ss:defBVS}, we give their definition and fundamental properties. In \cref{ss:bornconv}, we introduce bornological convergence and explain its relationship to ordinary convergence in normed spaces. In \cref{ss:sepcomp}, we discuss two important properties of bornological vector spaces: separatedness and completeness. Moreover, in \cref{ss:constBVS}, we provide basic constructions of bornological vector spaces. Many of the results in \cref{s:BVS} can be found in \cite{H77}. Nevertheless, we include proofs throughout, since it is useful to have a self-contained account of the facts about bornological vector spaces needed for our study of factorization algebras and vertex algebras.

\cref{s:bornholo} develops complex analysis in a bornological setting. In \cref{ss:borndiffint}, we first define and study bornological differentiation and bornological Riemann integration. In \cref{ss:bornholo}, we then introduce bornological holomorphy, and prove bornological analogues of the Cauchy--Riemann equations and the identity theorem. In \cref{ss:bornser}, we study bornological convergence of power series and show that such series are bornologically holomorphic on their domains of convergence. In \cref{ss:borncompint}, we define bornological complex integration using bornological Riemann integration, and establish bornological analogues of classical results, including Cauchy’s integral theorem, Cauchy’s integral formula, and Laurent expansion.

\cref{s:FAtoVA} contains a detailed description of the general procedure for extracting vertex algebras from amenably holomorphic prefactorization algebras. In \cref{ss:holotrans}, we briefly review the notion of holomorphically translation-equivariant prefactorization algebras in a form suitable for our purposes. In \cref{ss:FAtoVA}, we prove \cref{theorem:1A} and compare it with the construction of Costello and Gwilliam. In \cref{ss:FAtoVSA}, we explain how to extend \cref{theorem:1A} to the super setting. 

\cref{s:FAEVA} describes the construction of prefactorization algebras starting from Lie conformal algebras via the factorization envelope. In \cref{ss:compDol}, we review the compactly supported Dolbeault complex. We then construct, in \cref{ss:HDCLA}, the prefactorization algebra $\LXD$ which can be viewed as a higher-dimensional analogue of the mode Lie algebra. Finally, in \cref{ss:FEEV}, we prove \cref{theorem:main3A} and extend it to the super setting. 

\subsection*{Acknowledgements}

The author would like to thank Shintaro Yanagida for valuable comments on the manuscript and helpful discussions. 
The author is grateful to Hiroaki Kanno, Hidetoshi Awata, and Masashi Hamanaka for their valuable feedback. 
The author is also grateful to Benoit Vicedo for helpful discussions on factorization envelopes and for sharing his work \cite{V}. 
During the revision of this manuscript, the author used ChatGPT (OpenAI) to assist with English-language editing and to provide feedback on the exposition, possible gaps in the proofs, and potential revisions. The author critically assessed its suggestions and takes full responsibility for the mathematical arguments and the final manuscript. 
This work was supported by JSPS KAKENHI Grant Numbers 23KJ1120 and 26K24494.

\subsection*{Notations}

Here we list the notations used throughout this paper. Some of them are also introduced in each section. 

\smallskip
\begin{itemize}
\item 
We denote by $\varnothing$ the empty set. 

\item
We denote by $\bbN=\{0, 1, 2, \ldots\}$ the set of non-negative integers, by $\bbZ$ the ring of integers, by $\bbR$ the field of real numbers, and by $\bbC$ the field of complex numbers. 

\item
For a linear space $V$, we denote by $V[z]$ the linear space of polynomials in the variable $z$ with coefficients in $V$. We denote by $V\dbr{z}\ceq\{\sum_{n\in \bbN}z^nv_n\mid v_n\in V\}$ the linear space of formal power series and by $V\dpr{z}\ceq \{\sum_{n\in \bbZ}z^nv_n\mid v_n\in V,\, v_n=0\text{ for }n\ll0\}$ the linear space of formal Laurent series. Moreover, we denote by $V\dbr{z^{\pm1}}\ceq \{\sum_{n\in \bbZ}z^nv_n\mid v_n\in V\}$ the linear space of doubly-infinite formal series. 

\item 
For a Lie algebra $\frg$, we denote by $U(\frg)$ the universal enveloping algebra of $\frg$. 

\item 
For a topological space $X$, we denote by $\frU_X$ the set of all open subsets of $X$. 

\item 
For a topological vector space $X$, we denote by $X'$ the continuous dual space of $X$. 

\item 
For open intervals $I=(a, b), J=(c, d)\subset \bbR$, we denote $I<J$ whenever $a<b<c<d$. 

\item 
For $z\in \bbC$ and $R>0$, we denote by $D_R(z)\subset \bbC$ the open disk centered at $z$ with radius $R$. If $z=0$, we simply write $D_R\ceq D_R(0)$. We also set $D_\infty\ceq \bbC$. 

\item 
For $z\in \bbC$ and $0\le r<R$, we denote by $A_{r, R}(z)\subset \bbC$ the open annulus centered at $z$ with inner radius $r$ and outer radius $R$. If $z=0$, we simply write $A_{r, R}\ceq A_{r, R}(0)$. We also set $A_{r, \infty}\ceq \{z\in \bbC\mid |z|>r\}$. 

\item 
For a pre-abelian category $\sfA$, we denote by $\sfC(\sfA)$ the pre-abelian category of complexes in $\sfA$. 

\item 
For a symmetric monoidal linear category $\sfM$, we denote by $\dgLie{\sfM}$ the category of dg Lie algebras internal to $\sfM$. 

\item 
We denote by $\CBVS$ the category of convex bornological vector spaces with bounded linear maps. 

\item
A complete separated convex bornological vector space is called a convenient vector space. 
We denote by $\CVS$ the full subcategory of $\CBVS$ consisting of convenient vector spaces. 

\item 
For a complex $X$ of convenient vector spaces, we denote by $H^n(X)$ (resp. $H_n(X)$) the cohomology (resp. homology) taken in $\sfC(\CBVS)$, and by $H^n_\sep(X)$ (resp. $H_n^\sep(X)$) the cohomology (resp. homology) taken in $\sfC(\CVS)$. 
\end{itemize}

\addtocontents{toc}{\protect\setcounter{tocdepth}{2}}
\section{Preliminaries: factorization algebras and vertex algebras}
\label{s:preFAVA}

This section provides a minimal introduction to factorization algebras and vertex algebras needed for the purposes of this paper.

\subsection{Factorization algebras}
\label{ss:baFA}

In this subsection, we review the basic definitions of prefactorization algebras and describe the construction, known as the factorization envelope, which produces a prefactorization algebra from a precosheaf taking values in dg Lie algebras. Our main reference is \cite[Chapter 3]{CG}. 

\subsubsection{Definition of factorization algebras}

Recall that a precosheaf $\clF$ on a topological space $X$ taking values in a category $\sfA$ is simply a functor $\clF\colon \frU_X\to \sfA$, where $\frU_X$ denotes the poset of all open subsets of $X$ ordered by inclusion. Moreover, a morphism $\varphi\colon \clF\to \clG$ of precosheaves is just a natural transformation from the functor $\clF$ to $\clG$. 

\smallskip

Fix a topological space $X$ and a symmetric monoidal category $\sfM$. We denote by $\otimes$ the tensor product of $\sfM$, and by $1_{\sfM}$ its unit object.  

\begin{dfn}[{\cite[\S3.1.1]{CG}}]\label{dfn:PFA}
A \emph{prefactorization algebra} on $X$ taking values in $\sfM$ consists of 
\begin{itemize}
\item 
a precosheaf $\clF\colon \frU_X\to \sfM$, 

\item 
a family of morphisms of $\sfM$
\[
\clF^{U, V}_W\colon \clF(U)\otimes \clF(V)\to \clF(W) \quad (U, V, W\in \frU_X, U\sqcup V\subset W), 
\]
called the factorization product, 

\item 
a morphism $1_\clF\colon 1_\sfM\to \clF(\varnothing)$ of $\sfM$, called the unit,
\end{itemize}
and satisfies the following conditions: 
\begin{clist}
\item 
For open subsets $U_1, U_2, V, W\subset X$ such that $U_1\sqcup U_2\subset V\subset W$, the following diagram commutes: 
\[
\begin{tikzcd}[row sep=huge, column sep=huge]
\clF(U_1)\otimes \clF(U_2) \arrow[r, "\clF^{U_1, U_2}_V"] \arrow[rd, "\clF^{U_1, U_2}_W"'] &
\clF(V) \arrow[d, "\clF^V_W"] \\
&
\clF(W)
\end{tikzcd}
\]
Moreover, for open subsets $U_1, U_2, V_1, V_2, W\subset X$ such that $U_1\subset V_1$, $U_2\subset V_2$, $V_1\sqcup V_2\subset W$, the following diagram commutes: 
\[
\begin{tikzcd}[row sep=huge, column sep=huge]
\clF(U_1)\otimes \clF(U_2) \arrow[r, "\clF^{U_1}_{V_1}\otimes \clF^{U_2}_{V_2}"] \arrow[rd, "\clF^{U_1, U_2}_W"'] &
\clF(V_1)\otimes \clF(V_2) \arrow[d, "\clF^{V_1, V_2}_W"] \\
&
\clF(W)
\end{tikzcd}
\]

\item 
For open subsets $U_1, U_2, V\subset X$ such that $U_1\sqcup U_2\subset V$, the following diagram commutes: 
\[
\begin{tikzcd}[row sep=huge, column sep=huge]
\clF(U_1)\otimes \clF(U_2) \arrow[r, "\sim"] \arrow[rd, "\clF^{U_1, U_2}_V"'] &
\clF(U_2)\otimes \clF(U_1) \arrow[d, "\clF^{U_2, U_1}_V"] \\
&
\clF(V)
\end{tikzcd}
\]
Here $\clF(U_1)\otimes \clF(U_2)\sto \clF(U_2)\otimes \clF(U_1)$ denotes the braiding of $\sfM$. 

\item 
For open subsets $U_1, U_2, U_3, V_1, V_2, W\subset X$ such that 
\[
U_1\sqcup U_2\subset V_1, \quad U_2\sqcup U_3\subset V_2, \quad 
U_1\sqcup V_2\subset W, \quad V_1\sqcup U_3\subset W, 
\]
the following diagram commutes: 
\[
\begin{tikzcd}[row sep=huge]
(\clF(U_1)\otimes\clF(U_2))\otimes \clF(U_3) \arrow[rr, "\sim"] \arrow[d, "\clF^{U_1, U_2}_{V_1}\otimes \id_{\clF(U_3)}"'] &
&
\clF(U_1)\otimes (\clF(U_2)\otimes \clF(U_3)) \arrow[d, "\id_{\clF(U_1)}\otimes \clF^{U_2, U_3}_{V_2}"] \\
\clF(V_1)\otimes \clF(U_3) \arrow[rd, "\clF^{V_1, U_3}_W"'] &
&
\clF(U_1)\otimes \clF(V_2) \arrow[ld, "\clF^{U_1, V_2}_W"] \\
&
\clF(W)&
\end{tikzcd}
\]
Here $(\clF(U_1)\otimes \clF(U_2))\otimes \clF(U_3)\sto \clF(U_1)\otimes (\clF(U_2)\otimes \clF(U_3))$ denotes the associator of $\sfM$. 

\item 
For an open subset $U\subset X$, the following diagram commutes: 
\[
\begin{tikzcd}[row sep=huge, column sep=huge]
1_\sfM\otimes \clF(U) \arrow[r, "1_\clF\otimes \id_{\clF(U)}"] \arrow[rd, "\sim"', sloped] &
\clF(\varnothing)\otimes \clF(U) \arrow[d, "\clF^{\varnothing, U}_U"] \\
&
\clF(U)
\end{tikzcd}
\]
Here $1_\sfM\otimes \clF(U)\sto \clF(U)$ denotes the left unitor of $\sfM$. 
\end{clist} 
\end{dfn}

Let $\clF\colon \frU_X\to \sfM$ be a prefactorization algebra. For open subsets $U_1, \ldots, U_n, V\subset X$ such that $U_1\sqcup \cdots \sqcup U_n\subset V$, define a morphism
\[
\clF^{U_1, \ldots, U_n}_V\colon \bigotimes_{i=1}^n\clF(U_i)\to \clF(V)
\]
inductively as follows: 
\[
\clF^{U_1, \ldots, U_n}_V\ceq \clF^{U_1\sqcup \cdots \sqcup U_{n-1}, U_n}_{V}\circ (\clF^{U_1, \ldots, U_{n-1}}_{U_1\sqcup \cdots \sqcup U_{n-1}}\otimes \id_{\clF(U_n)}). 
\]

A factorization algebra is a prefactorization algebra that satisfies a local-to-global condition analogous to that of sheaves. As this condition will not play a role in what follows, we omit the precise definition (see \cite[Chapter 6]{CG} for details).

\smallskip

Let $G$ be a group acting continuously on $X$. Given a precosheaf $\clF\colon \frU_X\to \sfM$ and $a\in G$, the assignment
\[
a\clF\colon \frU_X\to \sfM, \quad U\mapsto \clF(aU)
\]
is a precosheaf by letting 
\[
(a\clF)^U_V\ceq \clF^{aU}_{aV}\colon (a\clF)(U)\to (a\clF)(V)
\]
for each inclusion of open subsets $U\subset V\subset X$. 

\begin{dfn}[{\cite[Definition 3.7.1]{CG}}]\label{dfn:equivFA}
A \emph{$G$-equivariant prefactorization algebra} is a prefactorization algebra $\clF\colon \frU_X\to \sfM$ equipped with morphisms of precosheaves
\[
\sigma^\clF_a=\{\sigma^\clF_{a, U}\}_{U\in\frU_X}\colon \clF\to a\clF \quad (a\in G)
\]
such that the following conditions hold: 
\begin{clist}
\item 
For $a, b\in G$ and an open subset $U\subset X$, we have $\sigma^\clF_{a, bU}\circ \sigma^\clF_{b, U}=\sigma^\clF_{ab, U}$. 

\item 
For an open subset $U\subset X$, we have $\sigma^\clF_{1_G, U}=\id_{\clF(U)}$, where $1_G$ denotes the unit of $G$. 

\item 
For $a\in G$ and open subsets $U, V, W\subset X$ such that $U\sqcup V\subset W$, we have
\[
\sigma^\clF_{a, W}\circ \clF^{U, V}_W=\clF^{aU, aV}_{aW}\circ (\sigma^\clF_{a, U}\otimes \sigma^\clF_{a, V}). 
\]

\item 
For $a\in G$, we have $\sigma^\clF_{a, \varnothing}\circ 1_\clF=1_\clF$. 
\end{clist}
\end{dfn}

In \cref{ss:FAtoVA}, we will consider prefactorization algebras on the complex plane $\bbC$ that are equivariant under the natural action of the group $S^1\ltimes \bbC$ of isometric affine transformations. 

\subsubsection{Factorization envelopes}

We recall a key construction of prefactorization algebras called the factorization envelope. For simplicity, we present the construction in the category $\ModC$ of linear spaces over $\bbC$, but the same construction works in any symmetric monoidal linear category admitting all colimits and such that the tensor product preserves colimits in each variable. 

For a complex $X$ of linear spaces, we denote its differential by $d_X=\{d_X^n\colon X^n\to X^{n+1}\}_{n\in \bbZ}$, and write $|x|$ for the degree of homogeneous element $x\in X$. 
The category $\sfC(\ModC)$ of complexes of linear spaces naturally inherits a symmetric monoidal structure: 
\begin{itemize}
\item 
For $X, Y\in \sfC(\ModC)$, the tensor product $X\otimes Y$ is defined as follows: For each $n\in \bbZ$, the object in degree $n$ is 
\[
(X\otimes Y)^n\ceq \bigoplus_{l+m=n} X^l\otimes Y^m, 
\]
and the differential in degree $n$ is 
\[
d_{X\otimes Y}^n\colon (X\otimes Y)^n\to (X\otimes Y)^{n+1}, \quad 
x\otimes y\mapsto d_X(x)\otimes y+(-1)^{|x|}x\otimes d_Y(y).
\]

\item
The unit object is $\bbC$, regarded as a complex concentrated in degree $0$. 

\item 
The braiding is given by
\[
(X\otimes Y)^n\to (Y\otimes X)^n, \quad x\otimes y\mapsto (-1)^{|x||y|}y\otimes x
\]
for each $X, Y\in \sfC(\ModC)$. 
\end{itemize} 
Moreover, the category $\sfC(\ModC)$ has all colimits, and its tensor product preserves colimits separately in each variable. 

Let $X\in \sfC(\ModC)$. For $p\in \bbN$, the symmetric group $\frS_p$ acts on $X^{\otimes p}$ by signed permutations according to the Koszul sign rule. The symmetric power $\Sym^p(X)$ is defined as the quotient complex
\[
\Sym^p(X)\ceq X^{\otimes p}/\Span\{\sigma x-x\mid \sigma\in \frS_p, x\in X^{\otimes p}\}.
\]
Moreover, the symmetric algebra $\Sym(X)$ is defined as the direct sum
\[
\Sym(X)\ceq\bigoplus_{p\in \bbN}\Sym^p(X). 
\]

Let $k\in \bbZ$. For $X\in \sfC(\ModC)$, the $k$-shifted complex $X[k]$ is defined by 
\[
X[k]^n\ceq X^{n+k}, \quad d_{X[k]}^n\ceq(-1)^kd_X^{n+k}. 
\]
For a morphism $f\colon X\to Y$ in $\sfC(\ModC)$, the $k$-shifted morphism $f[k]\colon X[k]\to Y[k]$ is defined by 
\[
f[k]^n\ceq f^{n+k}. 
\]

\smallskip

A \emph{dg Lie algebra} in $\ModC$ is just a Lie algebra object in $\sfC(\ModC)$. More explicitly, a dg Lie algebra is an object $\frg\in \sfC(\ModC)$ endowed with a morphism $[\cdot, \cdot]\colon \frg\otimes \frg\to \frg$ in $\sfC(\ModC)$, called the Lie bracket, satisfying the following conditions: 

\begin{clist}
\item (skew-symmetry) 
$[y, x]=-(-1)^{|x||y|}[x, y]$ for $x, y\in \frg$. 

\item (Jacobi identity) 
$[x, [y, z]]=[[x, y], z]+(-1)^{|x||y|}[y, [x, z]]$ for $x, y, z\in \frg$. 
\end{clist}
A morphism of dg Lie algebras is a morphism in $\sfC(\ModC)$ that respects the Lie brackets. We denote by $\dgLie{\ModC}$ the category of dg Lie algebras in $\ModC$. 

Let $\frg$ be a dg Lie algebra in $\ModC$. For $p\in \bbZ_{>0}$, the morphism of graded objects 
\begin{align*}
\frg[1]^{\otimes p}
&\to \frg[1]^{\otimes (p-1)}, \\
\otimes_{i=1}^px_i
&\mapsto \sum_{1\le i<j\le p}(-1)^{(|x_i|-1)(|x_{i+1}|+\cdots +|x_{j-1}|+i-j+1)+|x_i|}x_1\otimes \cdots \otimes \check{x_i}\otimes \cdots \otimes [x_i, x_j]\otimes \cdots \otimes x_p, 
\end{align*}
has degree 1. Here $|x|$ denotes the degree of a homogeneous element $x\in \frg$, and $\check{x_i}$ indicates that the factor $x_i$ is omitted. This morphism induces a morphism of graded objects of degree $1$ 
\[
d_{[\cdot, \cdot]}\colon \Sym^p(\frg[1])\to \Sym^{p-1}(\frg[1]). 
\]
Since $d_{[\cdot, \cdot]}$ satisfies
\[
d_{[\cdot, \cdot]}^{n+1}\circ d^n_{[\cdot, \cdot]}=0, \quad 
d^{n+1}_{[\cdot, \cdot]}\circ d^n_{\Sym(\frg[1])}=-d^{n+1}_{\Sym(\frg[1])}\circ d_{[\cdot, \cdot]}^n,
\]
the $\bbZ$-graded object $\{\Sym(\frg[1])^n\}_{n\in \bbZ}$ becomes a complex with differential $\{d^n_{\Sym(\frg[1])}+d^n_{[\cdot, \cdot]}\}_{n\in \bbZ}$. This complex is called the \emph{Chevalley--Eilenberg chain complex} of $\frg$, and we denote it by $C^\CE_\bullet(\frg)$: 
\[
C^\CE_n(\frg)\ceq \Sym(\frg[1])^{-n}, \quad d_n^\CE\ceq d_{\Sym(\frg[1])}^{-n}+d_{[\cdot, \cdot]}^{-n}. 
\]

For a morphism $f\colon \frg\to \frh$ in $\dgLie{\ModC}$, 
\[
C^\CE_\bullet(f)\ceq \Sym(f[1])\colon C_\bullet^\CE(\frg)\to C_\bullet^\CE(\frh)
\]
is a morphism of complexes. The assignment 
\[
C^\CE_\bullet\colon \dgLie{\ModC}\to \sfC(\ModC), \quad \frg\mapsto C_\bullet^\CE(\frg)
\]
gives rise to a functor by $f\mapsto C^\CE_\bullet(f)$. Note that the natural morphism $\mu_{\frg, \frh}\colon \Sym(\frg[1])\otimes \Sym(\frh[1])\to \Sym((\frg\oplus \frh)[1])$ is compatible with the Chevalley--Eilenberg differential $d^\CE$. Hence 
\[
\mu_{\frg, \frh}\colon C^\CE_\bullet(\frg)\otimes C^\CE_\bullet(\frh)\to C^\CE_\bullet(\frg\oplus \frh)
\]
is a morphism of complexes. We find that $C^\CE_\bullet\colon \dgLie{\ModC}\to \sfC(\ModC)$ becomes a symmetric monoidal functor via $\mu_{\frg, \frh}$ and $\bbC\to C^\CE_\bullet(\{0\})$. 

\smallskip

Let $X$ be a topological space, and let $\clL\colon \frU_X\to \dgLie{\ModC}$ be a precosheaf. For open subsets $U, V, W\subset X$ with $U\sqcup V\subset W$, the universality of coproducts in $\sfC(\ModC)$ induces a unique morphism $\clL^{U, V}_W\colon \clL(U)\oplus \clL(V)\to \clL(W)$ such that the following diagram commutes: 
\[
\begin{tikzcd}[row sep=huge, column sep=huge]
\clL(U) \arrow[r, hook] \arrow[rd, "\clL^U_W"']& 
\clL(U)\oplus \clL(V) \arrow[d, "\clL^{U, V}_W"] &
\clL(V) \arrow[l, hook'] \arrow[ld, "\clL^V_W"] \\
&
\clL(W) &
\end{tikzcd}
\]
In general, $\clL^{U, V}_W$ is not a morphism in $\dgLie{\ModC}$. However, if 
\begin{equation}\label{eq:disj=0}
[\cdot, \cdot]\circ (\clL^U_W\otimes \clL^V_W)=0, 
\end{equation}
then $\clL^{U, V}_W$ is a morphism in $\dgLie{\ModC}$. 

A precosheaf $\clL\colon \frU_X\to \dgLie{\ModC}$ satisfying \eqref{eq:disj=0} for all open subsets $U, V, W\subset X$ with $U\sqcup V\subset W$ naturally becomes a prefactorization algebra on $X$ with values in $\dgLie{\ModC}$, where the factorization product is given by $\clL^{U, V}_W$ and the unit by $\{0\}\to \clL(\varnothing)$. 
Hence, composing $\clL$ with the symmetric monoidal functor $C^\CE_\bullet$ yields a prefactorization algebra
\[
C^\CE_\bullet \clL\colon \frU_X\to \sfC(\ModC). 
\]
We call $C^\CE_\bullet\clL$ the \emph{factorization envelope} of $\clL$. Moreover, since the homology functor $H_\bullet\colon \sfC(\ModC)\to \ModC$ is naturally symmetric monoidal, we obtain a prefactorization algebra
\[
H^\CE_\bullet\clL\ceq H_\bullet C^\CE_\bullet\clL\colon \frU_X\to \ModC. 
\]
We also call $H^\CE_\bullet\clL$ the factorization envelope of $\clL$. 

\smallskip

There is a twisted version of the factorization envelope. To describe it, recall the notion of 2-cocycles for a dg Lie algebra. Let $\frg$ be a dg Lie algebra in $\ModC$. A $k$-shifted $2$-cocycle of $\frg$ is a morphism $\omega\colon \frg\otimes \frg\to \bbC[k]$ in $\sfC(\ModC)$ satisfying: 
\begin{clist}
\item 
$\omega(y\otimes x)=-(-1)^{|x||y|}\omega(x\otimes y)$, 

\item 
$\omega(x\otimes [y, z])=\omega([x, y]\otimes z)+(-1)^{|x||y|}\omega(y\otimes [x, z])$. 
\end{clist}
Given a $k$-shifted $2$-cocycle $\omega$ of $\frg$, the object $\frg_\omega\ceq \frg\oplus \bbC[k]$ in $\sfC(\ModC)$ carries the structure of a dg Lie algebra whose bracket is defined by
\[
[x, y]^\sim\ceq [x, y]+\omega(x, y) \quad (x, y\in \frg), \quad 
[\frg_\omega, \bbC]\ceq 0. 
\]

Let $\clL\colon \frU_X\to \dgLie{\ModC}$ be a precosheaf satisfying \eqref{eq:disj=0} for all open subsets $U, V, W\subset X$ with $U\sqcup V\subset W$. A $k$-shifted $2$-cocycle of $\clL$ is a family of $k$-shifted 2-cocycles 
\[
\omega_U\colon \clL(U)\otimes \clL(U)\to \bbC[k]
\] 
of dg Lie algebras satisfying the following conditions: 

\begin{clist}
\item
For open subsets $U, V\subset X$ with $U\subset V$, we have $\omega_U=\omega_V\circ (\clL^U_V\otimes \clL^U_V)$. 

\item 
For open subsets $U, V, W\subset X$ with $U\sqcup V\subset W$, we have $\omega_W\circ (\clL^U_W\otimes \clL^V_W)=0$. 
\end{clist}
Given a $k$-shifted 2-cocycle $\omega=\{\omega_U\}_{U\in \frU_X}$ of $\clL$, the assignment
\[
\clL_\omega\colon \frU_X\to \dgLie{\ModC}, \quad U\mapsto \clL(U)_{\omega_U}=\clL(U)\oplus \bbC[k]
\]
is a precosheaf by letting 
\[
(\clL_\omega)^U_V\ceq \clL^U_V\oplus\id_{\bbC[k]}\colon \clL_\omega(U)\to \clL_\omega(V)
\]
for each inclusion $U\subset V$ of open subsets of $X$. Moreover, this precosheaf satisfies \eqref{eq:disj=0} for all open subsets $U, V, W\subset X$ with $U\sqcup V\subset W$. Hence, applying the factorization envelope to $\clL_\omega$, we obtain prefactorization algebras
\[
C^\CE_{\bullet, \omega}\clL\ceq C^\CE_\bullet\clL_\omega\colon \frU_X\to \sfC(\ModC), \qquad
H^\CE_{\bullet, \omega}\clL\ceq H^\CE_\bullet\clL_\omega\colon \frU_X\to \ModC. 
\]
We call these the \emph{twisted factorization envelopes} of $\clL$ by $\omega$. 

\subsection{Vertex algebras}
\label{ss:VA}

This subsection is an exposition of vertex algebras, containing only the material necessary for later use. We refer the reader to \cite{FBZ} and \cite{K98} for more details on the theory of vertex algebras. 

\subsubsection{Definition of vertex algebras}

We briefly recall the definitions and basic examples of vertex algebras. 

An (algebraic) field on a linear space $V$ is a formal series $a(z)\in (\End V)\dbr{z^{\pm1}}$ such that $a(z)v\in V\dpr{z}$ for all $v\in V$. 
Here, writing $a(z)=\sum_{n\in \bbZ}z^na_n$ with $a_n\in \End V$, we set $a(z)v\ceq \sum_{n\in \bbZ}z^na_n(v)$. We denote by $\Fld V$ the linear subspace of $(\End V)\dbr{z^{\pm1}}$ consisting of all fields on $V$. 

\begin{dfn}\label{dfn:VA}
A \emph{vertex algebra} is a tuple $(V, \vac, T, Y)$ consisting of 
\begin{itemize}
\item 
a linear space $V$, called the state space, 

\item 
an element $\vac\in V$, called the vacuum, 

\item 
a linear map $T\colon V\to V$, called the translation operator, 

\item 
a linear map 
\[
Y\colon V\to \Fld V, \quad a\mapsto Y(a, z)=\sum_{n\in \bbZ}z^{-n-1}a_{(n)}, 
\]
called the state-field correspondence, or the vertex operators, 
\end{itemize}
which satisfies the following conditions: 
\begin{clist}
\item (vacuum axiom)
We have $Y(\vac, z)=\id_V$. Furthermore, for $a\in V$, we have $Y(a, z)\vac\in V\dbr{z}$, and $Y(a, z)\vac|_{z=0}=a$. 

\item (translation invariance) 
We have $T\vac=0$ and $[T, Y(a, z)]=\pdd_zY(a, z)$ for $a\in V$. 

\item (locality) 
For $a, b\in V$, there exists $N\in\bbN$ such that 
\[
(z-w)^N[Y(a, z), Y(b, w)]=0
\]
in $(\End V)\dbr{z^{\pm1}, w^{\pm1}}$. 
\end{clist}
\end{dfn}

Let $V$ and $W$ be vertex algebras. A \emph{morphism of vertex algebras} is a linear map $f\colon V\to W$ which satisfies
\[
f(\vac)=\vac, \quad f\circ T=T\circ f, \quad f(Y(a, z)b)=Y(f(a), z)f(b) \quad (a, b\in V). 
\]

We present two basic examples of vertex algebras: the Virasoro vertex algebra and the affine vertex algebra. The constructions of these vertex algebras go back to the work of Frenkel and Zhu \cite{FZ}. 

\begin{exm}[{\cite[\S2.5]{FBZ}}]
Consider the Virasoro algebra
\[
\Vir=\bigoplus_{n\in \bbZ}\bbC L_n\oplus \bbC C,
\]
whose Lie bracket is given by 
\[
[L_m, L_n]=(m-n)L_{m+n}+\frac{1}{12}\delta_{m+n, 0}(m^3-m)C, \quad 
[L_m, C]=0. 
\]
For $c\in \bbC$, define a one-dimensional representation of the Lie subalgebra $\Vir_+\ceq \bigoplus_{n\ge -1}\bbC L_n\oplus \bbC C$ on $\bbC$ by 
\[
L_n\cdot 1\ceq 0 \quad (n\ge -1), \quad C\cdot 1\ceq c. 
\]
Denote this representation by $\bbC_c$. Then the linear space 
\[
V^c\ceq U(\Vir)\otimes_{U(\Vir_+)}\bbC_c
\]
has the structure of a vertex algebra such that: 
\begin{itemize}
\item 
The vacuum is $v_c\ceq 1_{U(\Vir)}\otimes 1$. 

\item
The translation operator is $L_{-1}\colon V^c\to V^c$. 

\item 
The state-field correspondence $Y\colon V^c\to \Fld(V^c)$ is determined by 
\[
Y(L_{-2}v_c, z)\ceq \sum_{n\in \bbZ}z^{-n-2}L_n. 
\]
\end{itemize}
The vertex algebra $V^c$ is called the \emph{Virasoro vertex algebra} of central charge $c$. 
\end{exm}

\begin{exm}[{\cite[\S2.4]{FBZ}}]
Let $\frg$ be a finite-dimensional simple Lie algebra, and let $\kappa\colon \frg\times \frg\to \bbC$ be a symmetric invariant form. Consider the affine Lie algebra 
\[
\wh{\frg}_\kappa\ceq \bbC[t^{\pm1}]\otimes \frg\oplus \bbC K,
\]
whose Lie bracket is given by
\[
[t^m\otimes a, t^n\otimes b]=t^{m+n}\otimes [a, b]+m\delta_{m+n, 0}\kappa(a, b)K, \quad 
[t^m\otimes a, K]=0. 
\]
Define a one-dimensional representation of the Lie subalgebra $\wh{\frg}_{\kappa, +}\ceq \bbC[t]\otimes \frg\oplus \bbC K$ on $\bbC$ by
\[
(t^n\otimes a)\cdot 1\ceq 0, \quad (n\in \bbN, a\in \frg) \quad K\cdot 1\ceq 1. 
\]
Then the linear space
\[
V^\kappa(\frg)\ceq U(\wh{\frg})\otimes_{U(\wh{\frg}_{\kappa, +})}\bbC
\]
has the structure of a vertex algebra such that:
\begin{itemize}
\item 
The vacuum is $\vac\ceq 1_{U(\wh{\frg}_\kappa)}\otimes 1$. 

\item 
The translation operator $T\colon V^\kappa(\frg)\to V^\kappa(\frg)$ is determined by
\[
T\vac\ceq 0, \quad [T, t^n\otimes a]=-nt^{n-1}\otimes a\quad (a\in \frg, n\in \bbZ). 
\]

\item 
The state-field correspondence $Y\colon V^\kappa(\frg)\to \Fld(V^\kappa(\frg))$ is determined by
\[
Y((t^{-1}\otimes a)\cdot \vac, z)\ceq \sum_{n\in \bbZ}z^{-n-1}(t^n\otimes a).
\]
\end{itemize}
The vertex algebra $V^\kappa(\frg)$ is called the \emph{affine vertex algebra} associated to $\frg$ and $\kappa$. 
\end{exm}

Many vertex algebras are naturally equipped with a $\bbZ$-grading, which we now define precisely:  

\begin{dfn}\label{dfn:grVA}
A \emph{$\bbZ$-graded vertex algebra} is a vertex algebra $V=\bigoplus_{\Delta\in \bbZ}V_\Delta$ with a direct sum decomposition as a linear space, satisfying the following conditions: 
\begin{clist}
\item 
$\vac\in V_0$. 

\item 
For $a\in V_{\Delta}$, $b\in V_{\Delta'}$ and $n\in \bbZ$, we have $a_{(n)}b\in V_{\Delta+\Delta'-n-1}$. 
\end{clist}
Moreover, a $\bbZ$-graded vertex algebra $V=\bigoplus_{\Delta\in \bbZ}V_\Delta$ is said to be $\bbN$-graded if $V_\Delta=\{0\}$ for all $\Delta<0$. 
\end{dfn}

For an element $a$ of a $\bbZ$-graded vertex algebra $V$, if $a\in V_\Delta\bs\{0\}$, we say that $a$ is a homogeneous element of $V$, and denote $\Delta(a)\ceq \Delta$. 

Let $V$ and $W$ be $\bbZ$-graded vertex algebras. A \emph{morphism of $\bbZ$-graded vertex algebras} is a morphism of vertex algebras $f\colon V\to W$ satisfying $f(V_\Delta)\subset W_\Delta$ for all $\Delta\in \bbZ$. 

\begin{exm}
The Virasoro vertex algebra $V^c$ carries an $\bbN$-grading defined as follows: 
For each $\Delta\in \bbN$, let $V^c_\Delta$ be the subspace of $V^c$ spanned by elements of the form
\[
L_{m_1}\cdots L_{m_p}\vac, \quad (p\in \bbN,\, m_i\le -2,\, \sum_{i=1}^p(-m_i)=\Delta).
\]
Then $V^c=\bigoplus_{\Delta\in \bbN}V_\Delta^c$ is an $\bbN$-graded vertex algebra. 
\end{exm}

\begin{exm}
The affine vertex algebra $V^\kappa(\frg)$ carries an $\bbN$-grading defined as follows: For each $\Delta\in \bbN$, let $V^\kappa(\frg)_\Delta$ be the subspace of $V^\kappa(\frg)$ spanned by elements of the form
\[
(t^{m_1}\otimes a^1)\cdots (t^{m_p}\otimes a^p)\vac \quad (p\in \bbN,\,  a^i\in \frg,\,m_i\le -1,\, \sum_{i=1}^p(-m_i)=\Delta).
\]
Then $V^\kappa(\frg)=\bigoplus_{\Delta\in \bbN}V^\kappa(\frg)_\Delta$ is an $\bbN$-graded vertex algebra. 
\end{exm}

\subsubsection{Enveloping vertex algebras of Lie conformal algebras}
\label{ss:EVALCA}

We recall the notion of Lie conformal algebras and the construction of enveloping vertex algebras. 

\begin{dfn}
A \emph{Lie conformal algebra} (or \emph{vertex Lie algebra}) is a tuple $(L, T, [\cdot_\lambda\cdot])$ consisting of 
\begin{itemize}
\item 
a linear space $L$, 

\item 
a linear map $T\colon L\to L$, 

\item 
a linear map 
\[
[\cdot_\lambda\cdot]\colon L\otimes L\to L[\lambda], \quad a\otimes b\mapsto [a_\lambda b]=\sum_{n\in \bbN}\frac{1}{n!}\lambda^na_{(n)}b,
\]
called the $\lambda$-bracket, 
\end{itemize}
which satisfies the following conditions for $a, b, c\in L$: 
\begin{clist}
\item (sesquilinearity) 
\hspace{0.185cm}$[Ta_\lambda b]=-\lambda[a_\lambda b]$. 

\item (skew-symmetry) 
$[b_\lambda a]=-[a_{-\lambda-T}b]$. 

\item (Jacobi identity) 
\hspace{0.085cm}$[a_\lambda [b_\mu c]]=
[[a_\lambda b]_{\lambda+\mu}c]+[b_\mu[a_\lambda c]]$. 
\end{clist}
\end{dfn}

Let $L$ be a Lie conformal algebra. A \emph{$2$-cocycle of $L$} is a linear map 
\[
\omega_\lambda\colon L\otimes L\to \bbC[\lambda], \quad 
a\otimes b\mapsto \omega_\lambda(a, b)=\sum_{n\in \bbN}\frac{1}{n!}\lambda^n\omega_{(n)}(a, b),
\]
which satisfies the following conditions for $a, b, c\in L$: 
\begin{clist}
\item (sesquilinearity) 
\hspace{0.185cm}$\omega_\lambda(Ta, b)=-\lambda\omega_\lambda(a, b)$. 

\item (skew-symmetry) 
$\omega_\lambda(b, a)=-\omega_{-\lambda}(a, b)$. 

\item (Jacobi identity) 
\hspace{0.085cm}$\omega_\lambda(a, [b_\mu c])=
\omega_{\lambda+\mu}([a_\lambda b], c)+\omega_\mu(b, [a_\lambda c])$. 
\end{clist}

For a Lie conformal algebra $L$ and a $2$-cocycle $\omega_\lambda$ of $L$, one can endow $L_{\omega_\lambda}\ceq L\oplus\bbC C$ with the structure of a Lie conformal algebra as follows: 
\begin{itemize}
\item 
The translation operator $\wt{T}\colon L_{\omega_\lambda}\to L_{\omega_\lambda}$ is defined by
\[
\wt{T}(a)\ceq T(a) \quad (a\in L), \quad \wt{T}(C)\ceq 0. 
\]

\item 
The $\lambda$-bracket $[\cdot_\lambda\cdot]^\sim$ is defined by
\[
[a_\lambda b]^\sim\ceq [a_\lambda b]+\omega_\lambda(a, b)C, \quad 
[a_\lambda C]^\sim=[C_\lambda a]^\sim=[C_\lambda C]^\sim\ceq 0 \quad (a, b\in L). 
\]
\end{itemize}
We call $L_{\omega_\lambda}$ the \emph{one-dimensional central extension of $L$ associated with $\omega_\lambda$}. 

\begin{exm}\label{exm:Virconf}
The linear space $\bbC[T]L$ has the structure of a Lie conformal algebra defined by
\[
[L_\lambda L]\ceq (T+2\lambda)L. 
\]
This Lie conformal algebra admits a one-dimensional central extension $\CA\ceq\bbC[T]L\oplus \bbC C$ associated with the $2$-cocycle 
\begin{equation}\label{eq:Virconf2cocy}
\omega_\lambda(L, L)\ceq \frac{1}{12}\lambda^3. 
\end{equation}
We call $\CA$ the \emph{Virasoro conformal algebra}. 
\end{exm}

\begin{exm}\label{exm:curconf}
For a Lie algebra $\frg$, the linear space $\bbC[T]\otimes \frg$ has the structure of a Lie conformal algebra defined by
\[
[a_\lambda b]\ceq [a, b] \quad (a, b\in \frg). 
\]
If a symmetric invariant form $\kappa\colon \frg\otimes \frg\to \bbC$ is given, this Lie conformal algebra admits a one-dimensional central extension $\Cur_\kappa(\frg)\ceq \bbC[T]\otimes \frg\oplus \bbC K$ associated with the $2$-cocycle
\begin{equation}\label{eq:curconf2cocy}
\kappa_\lambda(a, b)\ceq \lambda \kappa(a, b). 
\end{equation}
We call $\Cur_\kappa(\frg)$ the current Lie conformal algebra.  
\end{exm}

For a vertex algebra $V$, define the $\lambda$-bracket by 
\[
[a_\lambda b]\ceq \sum_{n\ge 0}\frac{1}{n!}\lambda^na_{(n)}b \quad (a, b\in V). 
\]
Then $V$, equipped with the translation operator $T\colon V\to V$ and this $\lambda$-bracket, is a Lie conformal algebra. Conversely, for every Lie conformal algebra $L$, one can associate a vertex algebra $V(L)$, called the \emph{enveloping vertex algebra} of $L$. We now describe the construction of $V(L)$. 

For a Lie conformal algebra $L$, define 
\[
\Lie(L)\ceq (\bbC[t^{\pm1}]\otimes L)/\Img(\pdd_t\otimes \id+\id\otimes T).
\]
It becomes a Lie algebra with the bracket
\[
[\ol{f\otimes a}, \ol{g\otimes b}]\ceq \sum_{n\ge 0}\frac{1}{n!}\ol{(\pdd^n_tf)g\otimes a_{(n)}b} \quad (f, g\in \bbC[t^{\pm1}], a, b\in L). 
\]
The Lie algebra $\Lie(L)$ is called the \emph{mode Lie algebra} (or \emph{Borcherds Lie algebra}) of $L$. For  $a\in L$ and $n\in \bbZ$, write $a_{[n]}\ceq \ol{t^n\otimes a}$. Then we have 
\[
\Lie(L)=\Span_\bbC\{a_{[n]}\mid a\in L, n\in \bbZ\}, \quad 
[a_{[m]}, b_{[n]}]=\sum_{j\ge 0}\binom{m}{j}(a_{(j)}b)_{[m+n-j]}. 
\]
Note that $\Lie_+(L)\ceq \Span_\bbC\{a_{[n]}\mid a\in L, n\ge0\}\subset \Lie(L)$ is a Lie subalgebra. The linear space
\[
V(L)\ceq U(\Lie(L))\otimes_{U(\Lie_+(L))}\bbC, 
\]
where $\bbC$ is regarded as the one-dimensional trivial representation of $\Lie_+(L)$, carries the structure of a vertex algebra such that:
\begin{itemize}
\item 
The vacuum is given by $\vac\ceq 1_{U(\Lie(L))}\otimes 1$. 

\item 
The translation operator $T\colon V(L)\to V(L)$ is determined by
\[
T\vac\ceq 0, \quad [T, a_{[n]}]=-na_{[n-1]} \quad (a\in L, n\in \bbZ). 
\]

\item 
The state-field correspondence $Y\colon V(L)\to \Fld (V(L))$ is determined by
\[
Y(a_{[-1]}\vac, z)\ceq \sum_{n\in \bbZ}z^{-n-1}a_{[n]}. 
\]
\end{itemize}

Before moving on to examples of enveloping vertex algebras, we recall the notions of vertex algebra ideals and quotient vertex algebras. Let $V$ be a vertex algebra. A \emph{vertex algebra ideal} of $V$ is a linear subspace $I\subset V$ satisfying the following conditions: 
\begin{clist}
\item
$I$ is invariant under the translation operator: $TI\subset I$. 

\item 
For $a\in V$ and $b\in I$, we have $Y(a, z)b\in I\dpr{z}$. 
\end{clist}
For a vertex algebra ideal $I\subset V$, the quotient linear space $V/I$ naturally inherits the structure of a vertex algebra. 

\begin{exm}
Consider the enveloping vertex algebra $V(\CA)$ of the Virasoro conformal algebra. The linear map 
\[
\Vir\to \Lie(\CA), \quad L_n\mapsto L_{[n+1]}, \quad C\mapsto C_{[-1]}
\]
gives an isomorphism of Lie algebras. Moreover, for $c\in \bbC$, the linear subspace $(C-c)V(\CA)$ is a vertex algebra ideal of $V(\CA)$, and the quotient vertex algebra $V(\CA)/(C-c)V(\CA)$ is isomorphic to the Virasoro vertex algebra $V^c$: 
\[
V(\CA)/(C-c)V(\CA)\cong V^c. 
\]
\end{exm}

\begin{exm}
For a Lie algebra $\frg$ and a symmetric invariant form $\kappa\colon \frg\otimes \frg\to \bbC$, consider the enveloping vertex algebra $V(\Cur_\kappa(\frg))$ of the current Lie conformal algebra. The linear map
\[
\wh{\frg}_\kappa\to \Lie(\Cur_\kappa(\frg)), \quad t^n\otimes a\mapsto a_{[n]}, \quad K\mapsto K_{[-1]}
\]
gives an isomorphism of Lie algebras. Moreover, the linear subspace $(K-1)V(\Cur_\kappa(\frg))$ is a vertex algebra ideal of $V(\Cur_\kappa(\frg))$, and the quotient vertex algebra $V(\Cur_\kappa(\frg))/(K-1)V(\Cur_\kappa(\frg))$ is isomorphic to the affine vertex algebra $V^\kappa(\frg)$: 
\[
V(\Cur_\kappa(\frg))/(K-1)V(\Cur_\kappa(\frg))\cong V^\kappa(\frg). 
\]
\end{exm}

\begin{dfn}\label{dfn:grLCA}
A \emph{$\bbZ$-graded Lie conformal algebra} is a Lie conformal algebra $L=\bigoplus_{\Delta\in \bbZ}L_\Delta$ with a direct sum decomposition as a linear space, satisfying the following conditions: 
\begin{clist}
\item 
For $\Delta\in \bbZ$, we have $T(L_\Delta)\subset L_{\Delta+1}$. 

\item
For $a\in L_\Delta$, $b\in L_{\Delta'}$ and $n\in \bbN$, we have $a_{(n)}b\in L_{\Delta+\Delta'-n-1}$. 
\end{clist}
Moreover, a $\bbZ$-graded Lie conformal algebra $L=\bigoplus_{\Delta\in \bbZ}L_\Delta$ is said to be $\bbN$-graded if $L_\Delta=\{0\}$ for all $\Delta<0$. 
\end{dfn}

For an element $a$ of a $\bbZ$-graded Lie conformal algebra $L$, if $a\in L_\Delta\bs\{0\}$, we say that $a$ is a homogeneous element of $L$, and denote $\Delta(a)\ceq \Delta$. 

\begin{exm}\label{exm:Virconfgr}
The Virasoro conformal algebra $\CA$ is $\bbN$-graded by declaring 
\[
\Delta(L)\ceq 2, \quad \Delta(C)\ceq 0. 
\]
\end{exm}

\begin{exm}
For a Lie algebra $\frg$ and a symmetric invariant form $\kappa\colon \frg\otimes \frg\to \bbC$, the current Lie conformal algebra $\Cur_\kappa(\frg)$ is $\bbN$-graded by declaring 
\[
\Delta(a)\ceq 1 \quad (a\in \frg), \quad \Delta(K)\ceq 0. 
\]
\end{exm}

Clearly, every $\bbZ$-graded vertex algebra naturally becomes a $\bbZ$-graded Lie conformal algebra. Conversely, let $L=\bigoplus_{\Delta\in \bbZ}L_\Delta$ be a $\bbZ$-graded Lie conformal algebra. Define $V(L)_\Delta$ to be the subspace of $V(L)$ spanned by elements of the form
\[
a^1_{[m_1]}\cdots a^p_{[m_p]}\vac \quad (p\in \bbN,\,a_i\in L,\, m_i\le -1, \, \sum_{i=1}^p(\Delta(a^i)-m_i-1)=\Delta) .
\]
Then $V(L)=\bigoplus_{\Delta\in \bbZ}V(L)_\Delta$ is a $\bbZ$-graded vertex algebra. Moreover, the Lie algebra $\Lie(L)$ is also $\bbZ$-graded by declaring 
\[
\Lie(L)_\Delta\ceq \Span_\bbC\{a_{[m]}\mid a\in L,\, m\in \bbZ,\, \Delta(a)-m-1=\Delta\}, 
\]
and its enveloping algebra $U(\Lie(L))$ carries the induced $\bbZ$-grading. 

\subsubsection{Vertex superalgebras and Lie conformal superalgebras}
\label{ss:VSALCSA}

There are super analogues of the notions of vertex algebras and Lie conformal algebras, which arise in two-dimensional supersymmetric conformal field theory. We call them \emph{vertex superalgebras} and \emph{Lie conformal superalgebras}, respectively. In this section, we recall several examples of these structures. For precise definitions of the super analogues, we refer the reader to \cite{K98}. 

Many vertex superalgebras and Lie conformal superalgebras naturally come equipped with an $\frac{1}{2}\bbZ$-grading, defined in a manner similar to \cref{dfn:grVA} and \cref{dfn:grLCA}. 

\begin{exm}[Neveu--Schwarz vertex superalgebra]\label{exm:N=1LCA}
The linear superspace $\bbC[T]L\oplus \bbC[T]G$, where $L$ is even and $G$ is odd, has the structure of a Lie conformal superalgebra defined by
\[
[L_\lambda L]\ceq (T+2\lambda)L, \quad 
[L_\lambda G]\ceq \bigl(T+\frac{3}{2}\lambda\bigr)G, \quad 
[G_\lambda G]\ceq 2L. 
\]
This Lie conformal superalgebra admits a one-dimensional central extension 
\[
L_{N=1}\ceq \bbC[T]L\oplus \bbC[T]G\oplus \bbC C, 
\]
where $C$ is even, associated with the $2$-cocycle
\begin{equation}\label{eq:N1cocy}
\omega_\lambda(L, L)\ceq \frac{1}{12}\lambda^3, \quad 
\omega_\lambda(G, G)\ceq \frac{1}{3}\lambda^2. 
\end{equation}
The Lie conformal superalgebra $L_{N=1}$ is $\bbN/2$-graded by declaring 
\[
\Delta(L)\ceq 2, \quad \Delta(G)\ceq \frac{3}{2}, \quad \Delta(C)\ceq 0. 
\]
We call $L_{N=1}$ the \emph{Neveu--Schwarz Lie conformal superalgebra} or the \emph{$N=1$ Lie conformal superalgebra}. 
Consider the enveloping vertex superalgebra $V(L_{N=1})$. For $c\in \bbC$, the linear subspace $(C-c)V(L_{N=1})$ is an ideal of $V(L_{N=1})$. The quotient vertex superalgebra $V_{N=1}^c\ceq V(L_{N=1})/(C-c)V(L_{N=1})$ is called the \emph{Neveu--Schwarz vertex superalgebra} or the \emph{$N=1$ vertex superalgebra}. 
\end{exm}

\begin{exm}[$N=2$ vertex superalgebra]\label{exm:N=2LCA}
The linear superspace 
\[
M\ceq\bbC[T]L\oplus \bbC[T]J\oplus \bbC[T]G^+\oplus \bbC[T]G^-, 
\]
where $L$, $J$ are even and $G^{\pm}$ are odd, has the structure of a Lie conformal superalgebra defined by
\begin{align*}
&[L_\lambda L]\ceq (T+2\lambda)L, \quad 
[L_\lambda J]\ceq (T+\lambda)J, \quad 
[L_\lambda G^{\pm}]\ceq \bigl(T+\frac{3}{2}\lambda\bigr)G^{\pm}, \\
&[J_\lambda J]\ceq 0, \quad
[J_\lambda G^\pm]=\pm G^{\pm}, \quad 
[G^\pm{}_\lambda G^\pm]\ceq 0, \quad 
[G^+{}_\lambda G^-]\ceq L+\bigl(\frac{1}{2}T+\lambda\bigr)J.
\end{align*}
This Lie conformal superalgebra admits a one-dimensional central extension 
\[
L_{N=2}\ceq M\oplus \bbC C,
\]
where $C$ is even, associated with the $2$-cocycle 
\begin{equation}\label{eq:N2cocy}
\omega_\lambda(L, L)\ceq \frac{1}{12}\lambda^3, \quad 
\omega_\lambda(L, J)\ceq 0, \quad 
\omega_\lambda(J, J)\ceq \frac{1}{3}\lambda, \quad 
\omega_\lambda(G^\pm, G^\pm)\ceq0, \quad 
\omega_\lambda(G^+, G^-)\ceq \frac{1}{6}\lambda^2. 
\end{equation}
The Lie conformal superalgebra $L_{N=2}$ is $\bbN/2$-graded by declaring 
\[
\Delta(L)\ceq 2, \quad \Delta(J)\ceq 1, \quad \Delta(G^\pm)\ceq \frac{3}{2}, \quad \Delta(C)\ceq 0. 
\]
We call $L_{N=2}$ the \emph{$N=2$ Lie conformal superalgebra}. Consider the enveloping vertex superalgebra $V(L_{N=2})$. For $c\in \bbC$, the linear subspace $(C-c)V(L_{N=2})$ is an ideal of $V(L_{N=2})$. The quotient vertex superalgebra $V_{N=2}^c\ceq V(L_{N=2})/(C-c)V(L_{N=2})$ is called the \emph{$N=2$ vertex superalgebra}. 
\end{exm}

\begin{exm}[$N=4$ vertex superalgebra]\label{exm:N=4LCA}
The linear superspace 
\[
M\ceq\bbC[T]L\oplus \bbC[T]J^0\oplus\bbC[T]J^+\oplus \bbC[T]J^-\oplus \bbC[T]G^+\oplus \bbC[T]G^-\oplus \bbC[T]\ol{G}^+\oplus \bbC[T]\ol{G}^-,
\]
where $L$, $J^0$, $J^{\pm}$ are even and $G^\pm$, $\ol{G}^\pm$ are odd, has the structure of a Lie conformal superalgebra defined by
\begin{align*}
&[L_\lambda L]\ceq (T+2\lambda)L, \quad 
[L_\lambda J^0]\ceq(T+\lambda)J^0, \quad 
[L_\lambda J^{\pm}]\ceq (T+\lambda)J^\pm, \\
&[L_\lambda G^{\pm}]\ceq \bigl(T+\frac{3}{2}\lambda\bigr)G^\pm, \quad 
[L_\lambda \ol{G}^\pm]\ceq\bigl(T+\frac{3}{2}\lambda\bigr)\ol{G}^\pm, \\
&[J^0{}_\lambda J^0]\ceq 0, \quad 
[J^0{}_\lambda J^\pm]\ceq \pm2J^\pm, \quad
[J^0{}_\lambda G^\pm]\ceq \pm G^\pm, \quad 
[J^0{}_\lambda \ol{G}^\pm]\ceq \pm\ol{G}^\pm\\
&[J^+{}_\lambda J^-]\ceq J^0, \quad 
[J^\pm{}_\lambda G^\pm]\ceq 0, \quad 
[J^\pm{}_\lambda G^{\mp}]\ceq G^\pm, \quad 
[J^\pm{}_\lambda \ol{G}^\pm]\ceq 0, \quad
[J^\pm{}_\lambda \ol{G}^{\mp}]\ceq -\ol{G}^\pm, \\
&[G^\pm{}_\lambda G^\pm]\ceq0, \quad 
[G^\pm{}_\lambda G^{\mp}]\ceq 0, \quad 
[G^\pm{}_\lambda \ol{G}^\pm]\ceq (T+2\lambda)J^\pm, \quad
[G^\pm{}_\lambda \ol{G}^{\mp}]\ceq L\pm \frac{1}{2}(T+2\lambda)J^0, \\
&[\ol{G}^\pm{}_\lambda\ol{G}^\pm]\ceq 0, \quad
[\ol{G}^\pm{}_\lambda \ol{G}^\mp]\ceq 0. 
\end{align*}
This Lie conformal superalgebra admits a one-dimensional central extension 
\[
L_{N=4}\ceq M\oplus \bbC C, 
\]
where $C$ is even, associated with the 2-cocycle
\begin{align}\label{eq:N4cocy}
\begin{split}
&\omega_\lambda(L, L)\ceq \frac{1}{12}\lambda^3, \quad
\omega_\lambda(L, J^0)\ceq 0, \quad 
\omega_\lambda(L, J^\pm)\ceq 0, \\ 
&\omega_\lambda(J^0, J^0)\ceq \frac{1}{3}\lambda, \quad 
\omega_\lambda(J^0, J^\pm)\ceq 0, \quad
\omega_\lambda(J^+, J^-)\ceq \frac{1}{6}\lambda \\
&\omega_\lambda(G^\pm,G^\pm)\ceq 0, \quad
\omega_\lambda(G^\pm, G^\mp)\ceq 0, \quad 
\omega_\lambda(G^\pm, \ol{G}^\pm)\ceq 0, \quad 
\omega_\lambda(G^\pm, \ol{G}^\mp)\ceq \frac{1}{6}\lambda^2. 
\end{split}
\end{align}
The Lie conformal superalgebra $L_{N=4}$ is $\bbN/2$-graded by declaring 
\[
\Delta(L)\ceq 2, \quad \Delta(J^0)\ceq 1, \quad \Delta(J^\pm)\ceq 1, \quad \Delta(G^\pm)\ceq \frac{3}{2}, \quad \Delta(\ol{G}^\pm)\ceq \frac{3}{2}, \quad \Delta(C)\ceq 0. 
\]
We call $L_{N=4}$ the \emph{$N=4$ Lie conformal superalgebra}. 
Consider the enveloping vertex superalgebra $V(L_{N=4})$. For $c\in \bbC$, the linear subspace $(C-c)V(L_{N=4})$ is an ideal of $V(L_{N=4})$. The quotient vertex superalgebra $V_{N=4}^c\ceq V(L_{N=4})/(C-c)V(L_{N=4})$ is called the \emph{$N=4$ vertex superalgebra}. 
\end{exm}

\section{Preliminaries: bornological vector spaces}
\label{s:BVS}

This section provides a preliminary introduction to bornological vector spaces. The main reference is \cite{H77}. We refer the reader to \cite{J81} for terminology on topological vector spaces.  

Let $\bbK$ denote either the field $\bbR$ of real numbers or the field $\bbC$ of complex numbers. All linear spaces and linear maps are considered over $\bbK$, unless otherwise stated.

\subsection{Definition of bornological vector spaces}
\label{ss:defBVS}

\begin{dfn}
A \emph{bornological space} is a set $X$ equipped with a collection $\frB_X$ of subsets of $X$, called a \emph{bornology} on $X$, satisfying the following conditions: 
\begin{clist}
\item 
For every $x\in X$, we have $\{x\}\in \frB_X$. 

\item 
If $A, B\in \frB_X$, then $A\cup B\in \frB_X$. 

\item 
If $A\subset B\subset X$ and $B\in \frB_X$, then $A\in \frB_X$. 
\end{clist}
\end{dfn}

Let $X$ be a bornological space. Each element $B\in \frB_X$ is called a \emph{bounded set} of $X$. A \emph{basis} of $X$ is a subset $\frB\subset \frB_X$ such that for every $A\in \frB_X$, there exists $B\in \frB$ such that $A\subset B$. 

We say that a map $f\colon X\to Y$ between bornological spaces is \emph{bounded} if $f(B)\subset Y$ is bounded for every bounded set $B\subset X$. Clearly, for a given basis $\frB$ of $X$, a map $f\colon X\to Y$ is bounded if and only if $f(B)\subset Y$ is bounded for all $B\in \frB$. 

We denote by $\Born$ the category of bornological spaces with bounded maps. 

\begin{exm}
Any topological vector space $X$ over $\bbK$ has a natural bornology, called the \emph{von Neumann bornology}, whose bounded sets are those subsets $B\subset X$ for which the following holds: For every neighborhood $N\subset X$ of $0$, there exists $a\in \bbK^{\times}$ such that $B\subset a N$. 
\end{exm}

\begin{exm}
Consider the von Neumann bornology on $\bbK^n$. Since $\{B_\ve(0)\mid \ve\in \bbR_{>0}\}$ is a neighborhood basis at $0$, a subset $B\subset \bbK^n$ is bounded if and only if there exists $\ve\in \bbR_{>0}$ such that $B\subset B_\ve(0)$. In particular, the collection $\{B_\ve(0)\mid \ve\in \bbR_{>0}\}$ is a basis for the von Neumann bornology on $\bbK^n$. 
\end{exm}

Suppose that a collection $\frB$ of subsets of $X$ satisfies the following conditions: 
\begin{clist}
\item
For any $x\in X$, there exists $B\in \frB$ such that $x\in B$. 

\item
For any $A_1, A_2\in \frB$, there exists $B\in \frB$ such that $A_1\cup A_2\subset B$. 
\end{clist}
Then $X$ admits a unique bornology for which $\frB$ is a basis. 
This bornology is called the \emph{bornology generated by $\frB$}. 

\smallskip

Now, we collect some basic constructions of bornological spaces: 

\emph{Subspace bornology}. For a bornological space $X$ and a subset $E\subset X$, the collection
\[
\{B\cap E\mid B\subset X\text{ is bounded}\}
\]
is a bornology on $E$. This bornology is called the \emph{subspace bornology} on $E$. Note that if $\frB$ is a basis of $X$, then $\{B\cap E\mid B\in \frB\}$ forms a basis of $E$. Clearly, the inclusion $\iota\colon E\inj X$ is bounded. Moreover, a map $f\colon Y\to E$, where $Y$ is another bornological space, is bounded if and only if $\iota\circ f\colon Y\to X$ is bounded. 

\emph{Quotient bornology}. Let $X$ be a bornological space, and let $R$ be an equivalent relation. Consider the quotient set $X/R$, and denote by $\pi\colon X\srj X/R$ the projection. The collection 
\[
\{\pi(B)\mid B\subset X \text{ is bounded}\}
\]
generates the bornology on $X/R$, called the \emph{quotient bornology}. Note that if $\frB$ is a basis of $X$, then $\{\pi(B)\mid B\in \frB\}$ forms a basis of $X/R$. Clearly, the projection $\pi\colon X\srj X/R$ is bounded. Moreover, a map $f\colon X/R\to Y$, where $Y$ is another bornological space, is bounded if and only if $f\circ \pi\colon X\to Y$ is bounded. 

\emph{Product bornology}. Let $\{X_i\}_{i\in I}$ be a family of bornological spaces. Consider the product set $\prod_{i\in I}X_i$, and denote by $\pi_i\colon \prod_{i\in I}X_i\to X_i$ the projections. The collection
\[
\Bigl\{B\subset \prod_{i\in I}X_i \Bigm| \pi_i(B)\subset X_i \text{ is bounded for all }i\in I\Bigr\}
\]
is a bornology on $\prod_{i\in I}X_i$, called the \emph{product bornology}. Note that if $\frB_i$ is a basis of $X_i$ for each $i\in I$, then $\{\prod_{i\in I}B_i\mid B_i\in \frB_i\}$ forms a basis of $\prod_{i\in I}X_i$. Clearly, the projections $\pi_i\colon \prod_{i\in I}X_i\to X_i$ are all bounded. Moreover, a map $f\colon Y\to \prod_{i\in I}X_i$, where $Y$ is another bornological space, is bounded if and only if $\pi_i\circ f\colon Y\to X_i$ is bounded for all $i\in I$. The bornological space $\prod_{i\in I}X_i$ is actually a product in the category $\Born$ of bornological spaces. 

\emph{Projective limit bornology}. Let $(\{X_i\}_{i\in I}, \{f_{i, j}\colon X_j\to X_i\}_{i\le j})$ be a projective system in $\Born$. Consider the projective limit
\[
\varprojlim_{i\in I}X_i=
\Bigl\{(x_i)_{i\in I}\in \prod_{i\in I}X_i \Bigm| \forall i, j\in I,\, i\le j\, \Longrightarrow\, f_{i, j}(x_j)=x_i\Bigr\}
\]
in the category of sets, and denote by $\pi_i\colon \varprojlim_{i\in I}X_i\to X_i$ the canonical morphism. The collection 
\[
\Bigl\{B\subset \varprojlim_{i\in I}X_i \Bigm| \pi_i(B)\subset X_i \text{ is bounded for all }i\in I\Bigr\}
\]
is a bornology on $\varprojlim_{i\in I}X_i$, called the \emph{projective limit bornology}. It is clear that the canonical morphisms $\pi_i\colon \varprojlim_{i\in I}X_i\to X_i$ are all bounded. Moreover, a map $f\colon Y\to \varprojlim_{i\in I}X_i$, where $Y$ is another bornological space, is bounded if and only if $\pi_i\circ f\colon Y\to X_i$ is bounded for all $i\in I$. The bornological space $\varprojlim_{i\in I}X_i$ is actually a projective limit in $\Born$. Note that the projective limit bornology on $\varprojlim_{i\in I}X_i$ coincides with the subspace bornology induced by the product bornology on $\prod_{i\in I}X_i$. 

\begin{dfn}
A bornological vector space is a linear space $X$ over $\bbK$ equipped with a bornology such that the maps
\[
X\times X\to X, \quad (x, y)\mapsto x+y, \qquad
\bbK\times X\to X, \quad (a, x)\mapsto ax
\]
are both bounded. Here, $\bbK$ is equipped with the von Neumann bornology, and $X\times X$ and $\bbK\times X$ are both equipped with the product bornology. 
\end{dfn}

A bornology on a linear space $X$ is called a \emph{vector bornology} if it makes $X$ a bornological vector space. 

We denote by $\BVS$ the category of bornological vector spaces with bounded linear maps. 

\begin{exm}
The von Neumann bornology on a topological vector space is a vector bornology. In fact, the von Neumann bornology associated with the product topology coincides with the product bornology of the von Neumann bornologies.
\end{exm}

Recall that a subset $A$ of a linear space is \emph{balanced} if $a A\subset A$ for all $a\in \bbK$ with $|a|\le 1$. 

\begin{lem}\label{lem:balboubasis}
Let $X$ be a bornological vector space. For any bounded set $A\subset X$, there exists a bounded balanced set $B\subset X$ such that $A\subset B$. Hence the set of all bounded balanced sets forms a basis of $X$. 
\end{lem}

\begin{proof}
Fix $\ve>1$. Then $B\ceq B_\ve(0)\cdot A$ is a bounded set of $X$ satisfying $A\subset B$. Moreover, for any $a\in \bbK$ with $|a|\le 1$ and $x=tx'\in B$ ($t\in B_\ve(0)$, $x'\in A$), we have $a x=atx'\in B$, since $at\in B_\ve(0)$. Hence $B$ is balanced. 
\end{proof}

The above constructions of bornological spaces are compatible with the linear structure: 

\emph{Subspace bornology}: For a bornological vector space $X$ and a linear subspace $E\subset X$, the subspace bornology on $E$ is a vector bornology. 

\emph{Quotient bornology}: For a bornological vector space $X$ and a linear subspace $E\subset X$, the quotient bornology on $X/E$ is a vector bornology. 

\emph{Product bornology}: For a family $\{X_i\}_{i\in I}$ of bornological vector spaces, the product bornology on $\prod_{i\in I}X_i$ is a vector bornology. The bornological vector space $\prod_{i\in I}X_i$ is a product in the category $\BVS$ of bornological vector spaces. 

\emph{Projective limit bornology}: For a projective system $(\{X_i\}_{i\in I}, \{f_{i, j}\colon X_j\to X_i\}_{i\le j})$ in $\BVS$, the projective limit bornology on $\varprojlim_{i\in I}X_i$ is a vector bornology. The bornological vector space $\varprojlim_{i\in I}X_i$ is a projective limit in $\BVS$. 

\smallskip

Recall that a subset $A$ of a linear space is \emph{convex} if $ax+(1-a)y\in A$ for all $x, y\in A$ and $0\le a\le 1$. For a subset $A$ of a linear space, we denote by $\conv A$ the convex hull of $A$, which is the smallest convex set containing $A$. Moreover, a balanced convex set is called a \emph{disked set}. For a subset $A$ of a linear space, we denote by $\disk A$ the disked hull of $A$, which is the smallest disked set containing $A$. 

\begin{dfn}
A \emph{convex bornological vector space} is a bornological vector space $X$ whose bornology admits a basis consisting of convex sets of $X$. 
\end{dfn}

A bornology on a linear space $X$ is called a \emph{convex vector bornology} if it makes $X$ a convex bornological vector space. 

\begin{exm}
The von Neumann bornology on a locally convex vector space is a convex vector bornology. Indeed, for a bounded set $B\subset X$, the convex hull $\conv B$ is also bounded. 
\end{exm}

\begin{lem}\label{lem:diskhullbound}
Let $X$ be a convex bornological vector space. If $A\subset X$ is bounded, then both $\conv A$ and $\disk A$ are bounded in $X$. Hence the set of all bounded disked sets forms a basis of $X$. 
\end{lem}

\begin{proof}
By the definition of convex bornological vector spaces, there exists a bounded convex set $B\subset X$ such that $A\subset B$. Since $\conv A\subset \conv B=B$, we see that $\conv A$ is bounded in $X$. 

Moreover, by \cref{lem:balboubasis}, there exists a bounded balanced set $B\subset X$ such that $A\subset B$. Since $\disk A\subset \disk B=\conv B$, we see that $\disk A$ is bounded in $X$. 
\end{proof}

\begin{lem}\label{lem:constcbvs}
Let $X$ be a linear space and $\frB$ a collection of subsets of $X$ satisfying the following conditions: 
\begin{clist}
\item 
Each element of $\frB$ is a nonempty disked set of $X$. 

\item 
For any $x\in X$, there exists $B\in \frB$ such that $x\in B$. 

\item 
For any $A_1, A_2\in \frB$, there exists $B\in \frB$ such that $A_1+A_2\subset B$. 

\item 
For any $a\in \bbR_{>0}$ and $A\in \frB$, there exists $B\in \frB$ such that $aA\subset B$. 
\end{clist}
Then $X$ admits a unique convex vector bornology for which $\frB$ is a basis. 
\end{lem}

\begin{proof}
For $A_1, A_2\in \frB$, since every nonempty balanced set contains $0$, we have $A_1\cup A_2\subset A_1+A_2$. Hence by (iii), there exists $B\in \frB$ such that $A_1\cup A_2\subset B$. This, together with (ii), implies that $\frB$ generates a bornology. Hence it suffices to show that the resulting bornology is a vector bornology. The boundedness of addition follows directly from (iii). Let $\ve\in \bbR_{>0}$ and $A\in \frB$. By (i), we have $B_\ve(0)\cdot A\subset \ve A$. Since $\ve A$ is bounded by (iv), we see that $B_\ve(0)\cdot A$ is bounded in $X$. Thus the boundedness of scalar multiplication follows. 
\end{proof}

\subsection{Bornological convergence}
\label{ss:bornconv}

\begin{dfn}
Let $X$ be a convex bornological vector space. A net $\{x_\lambda\}_{\lambda\in \Lambda}\subset X$ is said to \emph{converge bornologically} to $x\in X$, or equivalently, $x$ is called a \emph{bornological limit} of $\{x_\lambda\}_{\lambda\in \Lambda}$, if there exist a bounded set $B\subset X$, a net $\{\ve_\lambda\}_{\lambda\in \Lambda}\subset \bbR$ with $\lim_{\lambda\in \Lambda}\ve_\lambda=0$, and $\lambda_0\in \Lambda$ such that
\[
\lambda\in \Lambda,\, \lambda\ge \lambda_0\, \Longrightarrow\, x_\lambda-x\in \ve_\lambda B. 
\] 
\end{dfn}

\begin{rmk}\label{exm:LCVSborncov}
Let $X$ be a locally convex vector space. If a net $\{x_\lambda\}_{\lambda\in \Lambda}\subset X$ converges bornologically to $x\in X$ with respect to the von Neumann bornology, then $\{x_\lambda\}_{\lambda\in \Lambda}$ converges to $x$ in the usual sense. 
\end{rmk}

\begin{lem}\label{lem:boulinborncov}
Let $f\colon X\to Y$ be a bounded linear map between convex bornological vector spaces. If a net $\{x_\lambda\}_{\lambda\in \Lambda}\subset X$ converges bornologically to $x\in X$, then $\{f(x_\lambda)\}_{\lambda\in \Lambda}$ converges bornologically to $f(x)$ in $Y$. 
\end{lem}

\begin{proof}
Choose a bounded set $B\subset X$, a net $\{\ve_\lambda\}_{\lambda\in \Lambda}\subset \bbR$ with $\lim_{\lambda\in \Lambda}\ve_\lambda=0$, and $\lambda_0\in\Lambda$ such that $x_\lambda-x\in \ve_\lambda B$ for all $\lambda\ge\lambda_0$. Then $f(B)$ is bounded in $Y$, and we have $f(x_\lambda)-f(x)\in \ve_\lambda f(B)$ for all $\lambda\ge\lambda_0$. 
\end{proof}

We establish a relationship between bornological convergence and the usual convergence in topological space. Let $X$ be a linear space. For a nonempty disked set $A\subset X$, since $A$ is an absorbing set of $\Span_\bbK A$, one can define 
\[
p_A\colon \Span_\bbK A\to \bbR, \quad 
x\mapsto \inf\{\ve\in \bbR_{>0}\mid x\in \ve A\}. 
\]
Then $p_A$ is a seminorm. We denote $\Span_\bbK A$ by $X_A$ regarded as a seminormed space via $p_A$. Note that $\{\ve A\mid \ve\in \bbR_{>0}\}$ forms a basis for the von Neumann bornology of $X_A$. Moreover, for nonempty disked sets $A, B\subset X$ such that $A\subset B$, we have $p_A(x)\ge p_B(x)$ for all $x\in X_A$, so the inclusion $X_A\inj X_B$ is continuous. 

\begin{lem}\label{lem:XBincbound}
Let $X$ be a convex bornological vector space. For a nonempty bounded disked set $B\subset X$, the inclusion $X_B\inj X$ is bounded. 
\end{lem}

\begin{proof}
This is immediate since $\{\ve B\mid \ve\in \bbR_{>0}\}$ forms a basis of $X_B$, and each $\ve B$ is bounded in $X$. 
\end{proof}

For a directed set $\Lambda$ and $\lambda_0\in \Lambda$, we denote
\[
\Lambda(\lambda_0)\ceq\{\lambda\in \Lambda\mid \lambda\ge\lambda_0\}. 
\]

\begin{lem}\label{lem:borcovXB}
Let $X$ be a convex bornological vector space. If a net $\{x_\lambda\}_{\lambda\in \Lambda}\subset X$ converses bornologically to $x\in X$, then there exists a nonempty bounded disked set $B\subset X$ and $\lambda_0\in \Lambda$ such that $\{x_\lambda\}_{\lambda\in \Lambda(\lambda_0)}\subset X_B$, $x\in X_B$, and $\{x_\lambda\}_{\lambda\in \Lambda(\lambda_0)}$ converges to $x$ in $X_B$. 
\end{lem}

\begin{proof}
Choose a bounded set $A\subset X$, a net $\{\ve_\lambda\}_{\lambda\in \Lambda}\subset \bbR$ with $\lim_{\lambda\in \Lambda}\ve_\lambda=0$, and $\lambda_0\in \Lambda$ such that $x_\lambda-x\in \ve_\lambda A$ for all $\lambda\ge\lambda_0$. By \cref{lem:diskhullbound}, the subset $B\ceq \disk(A\cup\{x\})$ is bounded. For $\lambda\in \Lambda(\lambda_0)$, we have
\[
x_\lambda=(x_\lambda-x)+x\in \ve_\lambda A+B\subset X_B. 
\]
Moreover, for $\lambda\in \Lambda(\lambda_0)$, since $B$ is balanced, we have
\[
x_\lambda-x\in \ve_\lambda A\subset \ve_\lambda B=|\ve_\lambda|B. 
\]
Hence we obtain $p_B(x_\lambda-x)\le |\ve_\lambda|$ for all $\lambda\in \Lambda(\lambda_0)$, which shows that $\{x_\lambda\}_{\lambda\in \Lambda(\lambda_0)}$ converges to $x$ in $X_B$. 
\end{proof}

\begin{lem}\label{lem:XBborcov}
Let $X$ be a convex bornological vector space and $B\subset X$ a nonempty bounded disked set. If a sequence $\{x_n\}_{n\in \bbN}\subset X_B$ converges to $x\in X_B$, then there exists a sequence $\{\ve_n\}_{n\in \bbN}\subset \bbR_{>0}$ with $\lim_{n\to \infty}\ve_n=0$ such that $x_n-x\in \ve_nB$ for all $n\in \bbN$. In particular, $\{x_n\}_{n\in \bbN}$ converges bornologically to $x$. 
\end{lem}

\begin{proof}
Set $\ve_n\ceq p_B(x_n-x)+1/2^n$ for each $n\in \bbN$. Since $\{x_n\}_{n\in \bbN}$ converges to $x$ in $X_B$, we have $\lim_{n\to \infty}\ve_n=0$. Moreover, for every $n\in \bbN$, since $\ve_n>p_B(x_n-x)$, we have $x_n-x\in \ve_n B$. 
\end{proof}

\begin{prp}
Let $X$ be a first-countable locally convex vector space. If a sequence $\{x_n\}_{n\in \bbN}\subset X$ converges to $x\in X$ in the usual sense, then $\{x_n\}_{n\in \bbN}$ converges bornologically to $x$ with respect to the von Neumann bornology. 
\end{prp}

\begin{proof}
We may assume that $x=0$. Fix a countable neighborhood basis $\{N_j\}_{j\in \bbN}$ at $0$ such that each $N_j$ is a disked set. Since $\{x_n\}_{n\in \bbN}$ is bounded in $X$, for each $j\in \bbN$, one can choose $\lambda_j\in \bbR_{>0}$ so that $\{x_n\}_{n\in \bbN}\subset \lambda_jN_j$. Set $B\ceq \bigcap_{j\in \bbN} j\lambda_j N_j$. Then $B$ is a nonempty bounded disked set of $X$ satisfying $\{x_n\}_{n\in \bbN}\subset B$. Hence, by \cref{lem:XBborcov}, it suffices to show that $\{x_n\}_{n\in \bbN}$ converges to $0$ in $X_B$. 
Let $\ve\in \bbR_{>0}$ be arbitrary, and fix an integer $m\ge1$ so that $m\ge1/\ve$. If $j\ge m$, then $\lambda_j\le \ve j\lambda_j$, and hence $\{x_n\}_{n\in \bbN}\subset \lambda_j N_j\subset \ve j\lambda_j N_j$. Moreover, since $\bigcap_{j=0}^m\ve j\lambda_j N_j$ is a neighborhood of $0$, there exists $k\in \bbN$ such that $N_k\subset \bigcap_{j=0}^m\ve j\lambda_j N_j$. Thus, for any $j\in \bbN$, we obtain $\{x_n\}_{n\in \bbN}\cap N_k\subset \ve j\lambda_jN_j$, which implies $\{x_n\}_{n\in \bbN}\cap N_k\subset \ve B$. Since $\{x_n\}_{n\in \bbN}$ converges to $0$ in $X$, there exists $n_0\in \bbN$ such that $x_n\in N_k$ for all $n\ge n_0$. Therefore, for $n\ge n_0$, we have $x_n\in \ve B$. This proves that $\{x_n\}_{n\in \bbN}$ converges to $0$ in $X_B$. 
\end{proof}

\begin{dfn}
Let $X$ be a convex bornological vector space. A net $\{x_\lambda\}_{\lambda\in \Lambda}\subset X$ is called a \emph{bornological Cauchy net} if there exist a bounded set $B\subset X$, a net $\{\ve_{\lambda, \mu}\}_{(\lambda, \mu)\in \Lambda\times \Lambda}\subset \bbR$ with $\lim_{(\lambda, \mu)\in \Lambda\times \Lambda}\ve_{\lambda, \mu}=0$, and $\lambda_0\in \Lambda$ such that 
\[
\lambda, \mu\in \Lambda,\, \lambda, \mu\ge\lambda_0\, \Longrightarrow\, x_\lambda-x_\mu\in \ve_{\lambda, \mu}B. 
\]
\end{dfn}

\begin{rmk}
Let $X$ be a locally convex vector space. If $\{x_\lambda\}_{\lambda\in \Lambda}\subset X$ is a bornological Cauchy net with respect to the von Neumann bornology, then $\{x_\lambda\}_{\lambda\in \Lambda}$ is Cauchy in the usual sense. 
\end{rmk}

\begin{lem}
Let $f\colon X\to Y$ be a bounded linear map between convex bornological vector spaces. If $\{x_\lambda\}_{\lambda\in \Lambda}\subset X$ is a bornological Cauchy net, then $\{f(x_\lambda)\}_{\lambda\in \Lambda}$ is a bornological Cauchy net in $Y$. 
\end{lem}

\begin{proof}
Choose a bounded set $B\subset X$, a net $\{\ve_{\lambda, \mu}\}_{(\lambda, \mu)\in \Lambda\times \Lambda}\subset \bbR$ with $\lim_{(\lambda, \mu)\in \Lambda\times \Lambda}\ve_{\lambda, \mu}=0$, and $\lambda_0\in \Lambda$ such that $x_\lambda-x_\mu\in \ve_{\lambda, \mu}B$ for all $\lambda, \mu \ge\lambda_0$. Then $f(B)$ is bounded in $Y$, and we have $f(x_\lambda)-f(x_\mu)\in \ve_{\lambda, \mu}f(B)$ for all $\lambda, \mu\ge\lambda_0$. 
\end{proof}

\begin{lem}\label{lem:bornCauXB}
Let $X$ be a convex bornological vector space. If $\{x_\lambda\}_{\lambda\in \Lambda}\subset X$ is a bornological Cauchy net, then there exists a nonempty bounded disked set $B\subset X$ and $\lambda_0\in \Lambda$ such that $\{x_\lambda\}_{\lambda\in \Lambda(\lambda_0)}\subset X_B$ and $\{x_\lambda\}_{\lambda\in \Lambda(\lambda_0)}$ is a Cauchy net in $X_B$. 
\end{lem}

\begin{proof}
Choose a bounded set $A\subset X$, a net $\{\ve_{\lambda, \mu}\}_{(\lambda, \mu)\in \Lambda\times \Lambda}\subset \bbR$ with $\lim_{(\lambda, \mu)\in \Lambda\times \Lambda}\ve_{\lambda, \mu}=0$, and $\lambda_0\in \Lambda$ such that $x_\lambda-x_\mu\in \ve_{\lambda, \mu}A$ for all $\lambda, \mu\ge\lambda_0$. By \cref{lem:diskhullbound}, the subset $B\ceq \disk(A\cup\{x_{\lambda_0}\})$ is bounded. For $\lambda\in \Lambda(\lambda_0)$, we have
\[
x_\lambda=(x_\lambda-x_{\lambda_0})+x_{\lambda_0}\in \ve_{\lambda, \lambda_0}A+B\subset X_B. 
\]
Moreover, for $\lambda, \mu\in \Lambda(\lambda_0)$, since $B$ is balanced, we have
\[
x_\lambda-x_\mu\in \ve_{\lambda, \mu}A\subset \ve_{\lambda, \mu}B=|\ve_{\lambda, \mu}|B. 
\]
Hence we obtain $p_B(x_\lambda-x_\mu)\le |\ve_{\lambda, \mu}|$ for all $\lambda, \mu\in \Lambda(\lambda_0)$, which shows that $\{x_\lambda\}_{\lambda\in \Lambda(\lambda_0)}$ is Cauchy in $X_B$. 
\end{proof}

\begin{lem}\label{lem:XBbornCau}
Let $X$ be a convex bornological vector space and $B\subset X$ a nonempty bounded disked set. If a sequence $\{x_n\}_{n\in \bbN}\subset X_B$ is Cauchy, then there exists a sequence $\{\ve_{m, n}\}_{m, n\in \bbN}\subset \bbR_{>0}$ with $\lim_{m, n\to \infty}\ve_{m, n}=0$ such that $x_m-x_n\in \ve_{m, n}B$ for all $m, n\in \bbN$. In particular, $\{x_n\}_{n\in \bbN}$ is a bornological Cauchy net. 
\end{lem}

\begin{proof}
Set $\ve_{m, n}\ceq p_B(x_m-x_n)+1/2^{mn}$ for each $m, n\in \bbN$. Since $\{x_n\}_{n\in \bbN}$ is Cauchy in $X_B$, we have $\lim_{m, n\to \infty}\ve_{m, n}=0$. Moreover, for every $m, n\in \bbN$, since $\ve_{m, n}>p_B(x_m-x_n)$, we have $x_m-x_n\in \ve_{m, n}B$. 
\end{proof}

Based on the notion of bornological convergence, we introduce the concept of closed subsets in a convex bornological vector space. 

\begin{dfn}
Let $X$ be a convex bornological vector space. A subset $F\subset X$ is said to be \emph{bornologically closed} if, whenever a net $\{x_\lambda\}_{\lambda\in \Lambda}\subset F$ converges bornologically to some $x\in X$, we have $x\in F$. 
\end{dfn}

\begin{lem}
Let $f\colon X\to Y$ be a bounded linear map between convex bornological vector spaces. For any bornological closed set $F\subset Y$, the set $f^{-1}(F)$ is bornologically closed in $X$. 
\end{lem}

\begin{proof}
Let $\{x_\lambda\}_{\lambda\in \Lambda}\subset f^{-1}(F)$ be a net converging bornologically to some $x\in X$. Then $\{f(x_\lambda)\}_{\lambda\in \Lambda}$ is a net in $F$ converging to $f(x)$. Since $F$ is bornologically closed, we have $f(x)\in F$, and hence $x\in f^{-1}(F)$. 
\end{proof}

\begin{lem}\label{lem:borncl}
Let $X$ be a convex bornological vector space. 
For a subset $F\subset X$, the following conditions are equivalent: 
\begin{clist}
\item 
$F$ is bornologically closed in $X$. 

\item 
Whenever a sequence $\{x_n\}_{n\in \bbN}\subset F$ converges bornologically to some $x\in X$, we have $x\in F$. 

\item 
For every nonempty bounded disked set $B\subset X$, the set $F\cap X_B$ is closed in $X_B$. 
\end{clist}
\end{lem}

\begin{proof}
(i) $\Longrightarrow$ (ii): This is trivial. 

(ii) $\Longrightarrow$ (iii): 
If a sequence $\{x_n\}_{n\in \bbN}\subset F\cap X_B$ converges to $x\in X_B$, then by \cref{lem:XBborcov}, the sequence $\{x_n\}_{n\in \bbN}$ converges bornologically to $x$. Hence $x\in F\cap X_B$. 

(iii) $\Longrightarrow$ (i): 
If a net $\{x_\lambda\}_{\lambda\in \Lambda}\subset F$ converges bornologically to $x\in X$, then by \cref{lem:borcovXB}, there exists a nonempty bounded disked set $B\subset X$ and $\lambda_0\in \Lambda$ such that $\{x_\lambda\}_{\lambda\in \Lambda(\lambda_0)}$ converges to $x$ in $X_B$. Since $F\cap X_B$ is closed in $X_B$, we have $x\in F$. 
\end{proof}

Let $X$ be a convex bornological vector space. Note that the collection of all bornologically closed sets of $X$ satisfies the axioms of closed sets, and hence defines a topology on $X$. For a subset $S\subset X$, its closure with respect to this topology is called the \emph{bornological closure} of $S$, and is denoted by $\olb{S}$. 

\begin{lem}
Let $X$ be a convex bornological vector space. For a linear subspace $E\subset X$, the bornological closure $\olb{E}$ is a linear subspace of $X$. 
\end{lem}

\begin{proof}
For $x\in X$, the set $F_x\ceq \{y\in X\mid x+y\in \olb{E}\}$ is the preimage of $\olb{E}$ under the bounded linear map $X\to X$, $y\mapsto x+y$, and hence is bornologically closed. If $x\in E$, then  $E\subset F_x$, and hence $\olb{E}\subset F_x$. Thus we obtain $E+\olb{E}\subset \olb{E}$. Therefore, for $x\in \olb{E}$, we have $E\subset F_x$, and hence $\olb{E}\subset F_x$. This proves $\olb{E}+\olb{E}\subset \olb{E}$. 

For $a\in\bbK$, the set $F\ceq \{x\in X\mid ax\in \olb{E}\}$ is the preimage of $\olb{E}$ under the bounded linear map $X\to X$, $x\mapsto ax$, and hence is bornologically closed. Since $E\subset F$, we have $\olb{E}\subset F$. This proves $a\cdot \olb{E}\subset \olb{E}$. 
\end{proof}

\subsection{Separatedness and completeness}
\label{ss:sepcomp}

\begin{dfn}
A convex bornological vector space $X$ is said to be \emph{separated} (or, more precisely, \emph{bornologically separated}) if $\{0\}$ is the only bounded linear subspace of $X$. 
\end{dfn}

\begin{prp}\label{prp:sepcbvs}
For a convex bornological vector space $X$, the following conditions are equivalent: 
\begin{clist}
\item 
$X$ is bornologically separated. 

\item 
$\{0\}$ is bornologically closed in $X$. 

\item 
For any nonempty bounded disked set $B\subset X$, the seminormed space $X_B$ is Hausdorff, i.e., its seminorm $p_B$ is a norm. 

\item 
Every bornological convergent net $\{x_\lambda\}_{\lambda\in \Lambda}\subset X$ has a unique bornological limit. 

\item
Every bornological convergent sequence $\{x_n\}_{n\in \bbN}\subset X$ has a unique bornological limit. 
\end{clist}
\end{prp}

\begin{proof}
(i) $\Longrightarrow$ (iii): If $x\in X_B$ satisfies $p_B(x)=0$, then $\Span_\bbK\{x\}$ is bounded in $X_B$. Hence, by \cref{lem:XBincbound}, the set $\Span_\bbK\{x\}$ is also bounded in $X$. This implies $x=0$. 

(iii) $\Longrightarrow$ (i): If $E\subset X$ is a bounded linear subspace, then $E$ is a nonempty bounded disked set of $X$. Hence $p_E$ is a norm. For any $x\in E$, we have $p_E(x)=0$, which shows $E=\{0\}$. 

(ii) $\Longleftrightarrow$ (iii): This follows from \cref{lem:borncl}. 

(iii) $\Longrightarrow$ (iv): If $x, y\in X$ are both bornological limits of $\{x_\lambda\}_{\lambda\in \Lambda}$, then by \cref{lem:borcovXB}, there exists a nonempty bounded disked set $A\subset X$ (resp. $B\subset X$) and $\lambda_0\in \Lambda$ such that $\{x_\lambda\}_{\lambda\in \Lambda(\lambda_0)}$ converges to $x$ (resp. $y$) in $X_A$ (resp. $X_B$). Let $C\ceq \disk(A\cup B)$, then $C$ is a nonempty bounded disked set of $X$. Since $\{x_\lambda\}_{\lambda\in \Lambda(\lambda_0)}$ converges to $x$ and $y$ in $X_C$, we obtain $x=y$. 

(iv) $\Longrightarrow$ (v): This is trivial. 

(v) $\Longrightarrow$ (iii): It suffices to show that every convergent sequence $\{x_n\}_{n\in \bbN}\subset X_B$ has a unique limit, but this follows \cref{lem:XBborcov}. 
\end{proof}

If $X$ is a separated convex bornological vector space, then for a bornologically convergent net $\{x_\lambda\}_{\lambda\in \Lambda}$, we denote by $\blim_{\lambda\in \Lambda}x_\lambda$ its unique bornological limit. 

\begin{exm}\label{exm:Hausbornsep}
Every Hausdorff locally convex vector space $X$ is bornologically separated. Indeed, for every nonempty bounded disked set $B\subset X$, the inclusion $X_B\inj X$ is continuous, and hence $X_B$ is Hausdorff. 
\end{exm}

\begin{lem}\label{lem:sepXbcovborn}
Let $X$ be a separated convex bornological vector space and $B\subset X$ a nonempty bounded disked set. If a net $\{x_\lambda\}_{\lambda\in \Lambda}\subset X_B$ converges to $x\in X_B$, then there exists a net $\{\ve_\lambda\}_{\lambda\in \Lambda}\subset \bbR_{\ge0}$ with $\lim_{\lambda\in \Lambda}\ve_\lambda=0$ such that $x_\lambda-x\in \ve_\lambda B$ for all $\lambda\in \Lambda$. In particular, $\{x_\lambda\}_{\lambda\in \Lambda}$ converges bornologically to $x$. 
\end{lem}

\begin{proof}
Fix $\delta>1$, and set $\ve_\lambda\ceq \delta p_B(x_\lambda-x)$ for each $\lambda\in \Lambda$. Since $\{x_\lambda\}_{\lambda\in \Lambda}$ converges to $x$ in $X_B$, we have $\lim_{\lambda\in \Lambda}\ve_\lambda=0$. Let $\lambda\in \Lambda$. If $p_B(x_\lambda-x)=0$, then by \cref{prp:sepcbvs}, we have $x_\lambda-x=0$, and hence $x_\lambda-x\in \ve_\lambda B$. If $p_B(x_\lambda-x)\neq 0$, then we have $\ve_\lambda>p_B(x_\lambda-x)$, and hence $x_\lambda-x\in \ve_\lambda B$. 
\end{proof}

By \cref{lem:borcovXB} and \cref{lem:sepXbcovborn}, a net $\{x_\lambda\}_{\lambda\in \Lambda}$ in a separated convex bornological vector space $X$ converges bornologically to $x\in X$ if and only if there exists a nonempty bounded disked set $B\subset X$ and $\lambda_0\in \Lambda$ such that $\{x_\lambda\}_{\lambda\in \Lambda(\lambda_0)}$ converges to $x$ in the normed space $X_B$. 

\begin{lem}\label{lem:sepXBCau}
Let $X$ be a separated convex bornological vector space and $B\subset X$ a nonempty bounded disked set. If a net $\{x_\lambda\}_{\lambda\in \Lambda}\subset X_B$ is Cauchy, then there exists a net $\{\ve_{\lambda, \mu}\}_{(\lambda, \mu)\in \Lambda\times \Lambda}\subset \bbR_{\ge0}$ with $\lim_{(\lambda, \mu)\in \Lambda\times \Lambda}\ve_{\lambda, \mu}=0$ such that $x_\lambda-x_\mu\in \ve_{\lambda, \mu}B$ for all $\lambda, \mu\in \Lambda$. In particular, $\{x_\lambda\}_{\lambda\in \Lambda}$ is a bornological Cauchy net. 
\end{lem}

\begin{proof}
Fix $\delta>1$, and set $\ve_{\lambda, \mu}\ceq \delta p_B(x_\lambda-x_\mu)$ for each $\lambda, \mu\in \Lambda$. Since $\{x_\lambda\}_{\lambda\in \Lambda}$ is Cauchy in $X_B$, we have $\lim_{(\lambda, \mu)\in \Lambda\times \Lambda}\ve_{\lambda, \mu}=0$. Let $\lambda, \mu\in \Lambda$. If $p_B(x_\lambda-x_\mu)=0$, then by \cref{prp:sepcbvs}, we have $x_\lambda-x_\mu=0$, and hence $x_\lambda-x_\mu\in \ve_{\lambda, \mu} B$. If $p_B(x_\lambda-x_\mu)\neq 0$, then we have $\ve_{\lambda, \mu}>p_B(x_\lambda-x_\mu)$, and hence $x_\lambda-x_\mu\in \ve_{\lambda, \mu}B$. 
\end{proof}

By \cref{lem:bornCauXB} and \cref{lem:sepXBCau}, a net $\{x_\lambda\}_{\lambda\in \Lambda}$ in a separated convex bornological vector space $X$ is bornologically Cauchy if and only if there exists a nonempty bounded disked set $B\subset X$ and $\lambda_0\in \Lambda$ such that $\{x_\lambda\}_{\lambda\in \Lambda(\lambda_0)}$ is Cauchy in $X_B$. 

\smallskip

A nonempty disked set $A$ of a linear space $X$ is called a \emph{completant disked set} (resp. a \emph{Banach disked set}) if the seminormed space $X_A$ is complete (resp. a Banach space). 

\begin{lem}\label{lem:isometric}
Let $X$, $Y$ be linear spaces and $f\colon X\to Y$ a linear map. For a nonempty disked set $A\subset X$, the subset $f(A)$ is a nonempty disked set of $Y$ such that $Y_{f(A)}$ is isometric to the quotient seminormed space $X_A/(\Ker f\cap X_A)$. 
\end{lem}

\begin{proof}
It is clear that $f(A)$ is a nonempty disked set of $Y$. Since $f|_{X_A}\colon X_A\to Y_{f(A)}$ is a surjective linear map with $N\ceq \Ker(f|_{X_A})=\Ker f\cap X_A$, it induces the linear isomorphism $X_A/N\sto Y_{f(A)}$. Recall that the seminorm on $X_A/N$ is defined by $\ol{p}_A(\ol{x})\ceq \inf_{y\in N}p_A(x-y)$. For $x\in X_A$, it is then easy to verify that $p_{f(A)}(f(x))=\ol{p}_A(\ol{x})$. 
\end{proof}

\begin{prp}\label{prp:mapcompdisk}
Let $X$, $Y$ be convex bornological vector spaces and $f\colon X\to Y$ a bounded linear map. Suppose that $Y$ is separated. Then for a bounded Banach disked set $B\subset X$, the subset $f(B)$ is a bounded Banach disked set of $Y$. 
\end{prp}

\begin{proof}
Since $p_{f(B)}(f(x))\le p_B(x)$ for all $x\in X_B$, the linear map $f|_{X_B}\colon X_B\to Y_{f(B)}$ is continuous. Moreover, by \cref{prp:sepcbvs}, the seminormed space $Y_{f(B)}$ is Hausdorff. Thus the subset $\Ker f\cap X_B=\Ker(f|_{X_B})$ is closed in $X_B$. Hence the quotient seminormed space $X_B/(\Ker f\cap X_B)$ is a Banach space. Therefore, the proposition follows from \cref{lem:isometric}. 
\end{proof}

\begin{dfn}
A convex bornological vector space $X$ is said to be \emph{complete} (or, more precisely, \emph{bornologically complete}) if there exists a basis $\frB$ of $X$ such that every $B\in \frB$ is completant. 
\end{dfn}

By \cref{prp:sepcbvs}, every bounded completant disked set of a separated convex bornological vector space is a Banach disked set. Hence a complete and separated convex bornological vector space admits a basis consisting of Banach disked sets. 

\begin{lem}\label{lem:borcompcov}
Let $X$ be a complete separated convex bornological vector space. 
Every bornological Cauchy net $\{x_\lambda\}_{\lambda\in \Lambda}\subset X$ converges bornologically.
\end{lem}

\begin{proof}
By \cref{lem:bornCauXB}, there exists a nonempty bounded disked set $A\subset X$ and $\lambda_0\in \Lambda_0$ such that $\{x_\lambda\}_{\lambda\in \Lambda(\lambda_0)}$ is a Cauchy net in $X_A$. Choose a bounded Banach disked set $B\subset X$ such that $A\subset B$. Since the inclusion $X_A\inj X_B$ is continuous, the net $\{x_\lambda\}_{\lambda\in \Lambda(\lambda_0)}$ is also Cauchy in $X_B$. Hence $\{x_\lambda\}_{\lambda\in \Lambda(\lambda_0)}$ converges to some $x\in X_B$. Thus the net $\{x_\lambda\}_{\lambda\in \Lambda}$ converges bornologically to $x$. 
\end{proof}

\begin{lem}\label{lem:tvsolbound}
Let $X$ be a topological vector space. For a bounded set $B\subset X$, the (topological) closure $\ol{B}$ is bounded in $X$. 
\end{lem}

\begin{proof}
Recall that every topological vector space admits a neighborhood basis at $0$ consisting of closed subsets. Hence for any neighborhood $N\subset X$ of $0$, there exists a closed neighborhood $M\subset X$ of $0$ such that $M\subset N$. Choose $a\in \bbK^\times$ with $B\subset aM$. Since $aM$ is closed in $X$, we have $\ol{B}\subset aM\subset aN$. 
\end{proof}

\begin{prp}
For a Hausdorff locally convex vector space $X$, the following conditions are equivalent: 
\begin{clist}
\item 
$X$ is bornologically complete. 

\item 
Every nonempty closed bounded disked set $B\subset X$ is a Banach disked set. 

\item 
Every bornological Cauchy net $\{x_\lambda\}_{\lambda\in \Lambda}\subset X$ converges bornologically. 

\item 
Every bornological Cauchy sequence $\{x_n\}_{n\in \bbN}\subset X$ converges bornologically. 

\item 
Every bornological Cauchy net $\{x_\lambda\}_{\lambda\in \Lambda}\subset X$ converges topologically. 

\item 
Every bornological Cauchy sequence $\{x_n\}_{n\in \bbN}\subset X$ converges topologically.  
\end{clist}
\end{prp}

\begin{proof}
(ii) $\Longrightarrow$ (i) follows from \cref{lem:tvsolbound}. (i) $\Longrightarrow$ (iii) follows from \cref{lem:borcompcov}. (iii) $\Longrightarrow$ (iv) and (v) $\Longrightarrow$ (vi) are trivial. (iii) $\Longrightarrow$ (v) and (iv) $\Longrightarrow$ (vi) follows from  \cref{exm:LCVSborncov}.  

(vi) $\Longrightarrow$ (ii): Let $\{x_n\}_{n\in \bbN}\subset X_B$ be a Cauchy sequence. Then, by \cref{lem:XBbornCau}, the sequence $\{x_n\}_{n\in \bbN}$ is bornological Cauchy in $X$, and converges topologically to some $x\in X$. Since $\{x_n\}_{n\in \bbN}$ is Cauchy in $X_B$, for arbitrary $\ve\in \bbR_{>0}$, there exists $N\in \bbN$ such that $m, n\ge N\, \Longrightarrow\, x_m-x_n\in\ve B$. Taking the limit for $n\to \infty$ and using the closedness of $B$, we obtain $m\ge N\, \Longrightarrow\, x_m-x\in \ve B$. This shows that $x\in X_B$ and $\{x_n\}_{n\in \bbN}$ converges to $x$ in $X_B$. 
\end{proof}

\begin{cor}\label{cor:seqcomp}
Every sequentially complete Hausdorff locally convex vector space is bornologically complete. 
\end{cor}

\subsection{Basic constructions of bornological vector spaces}
\label{ss:constBVS}

We denote by $\CBVS$ the category of convex bornological vector spaces with bounded linear maps. 

\subsubsection{Subspaces and quotients}

Let $X$ be a convex bornological vector space and $E\subset X$ a linear subspace. 
The subspace bornology on $E$ is a convex vector bornology, since $\{B\cap E\mid B\in \frB\}$ forms a basis of $E$ for any basis $\frB$ of $X$. Clearly, if $X$ is separated, then so is $E$. 

\begin{lem}\label{lem:subAcapE}
Let $X$ be a linear space, let $E\subset X$ be a linear subspace, and let $A\subset X$ be a nonempty disked set. 
\begin{enumerate}
\item 
$E_{A\cap E}=X_A\cap E$. 

\item 
For $x\in E_{A\cap E}$, we have $p_{A\cap E}(x)=p_A(x)$. 
\end{enumerate}
\end{lem}

\begin{proof}
(1) It is clear that $E_{A\cap E}\subset X_A\cap E$. Let $x\in X_A\cap E$. Since $X_A=\bigcup_{\ve>0}\ve A$, there exist $\ve\in \bbR_{>0}$ and $y\in A$ such that $x=\ve y$. We have $y\in A\cap E$, and hence $x\in E_{A\cap E}$. 

(2) This is immediate, since
\[
\{\ve\in \bbR_{>0}\mid x\in \ve(A\cap E)\}=
\{\ve\in \bbR_{>0}\mid x\in \ve A\}. 
\qedhere
\]
\end{proof}

\begin{lem}\label{lem:subcomp}
Let $X$ be a convex bornological vector space, and let $E\subset X$ be a bornologically closed linear subspace. For a completant disked set $B\subset X$, the set $B\cap E$ is a completant disked set of $E$. 
\end{lem}

\begin{proof}
By \cref{lem:borncl} and \cref{lem:subAcapE} (1), the linear subspace $E_{B\cap E}$ is complete with respect to the induced topology. By \cref{lem:subAcapE} (2), the seminormed space $E_{B\cap E}$ is complete. 
\end{proof}

\begin{prp}\label{prp:subcomp}
Let $X$ be a convex bornological vector space.
\begin{enumerate}
\item
If $X$ is complete, then every bornologically closed linear subspace $E\subset X$ is complete. 

\item 
If $X$ is separated, then every complete linear subspace $E\subset X$ is bornologically closed. 
\end{enumerate}
\end{prp}

\begin{proof}
(1) Choose a basis $\frB$ of $X$ consisting of completant disked sets. Then, by \cref{lem:subcomp}, the collection $\{B\cap E\mid B\in \frB\}$ is a basis of $E$ consisting of completant disked sets. 

(2) Let $\{x_\lambda\}_{\lambda\in \Lambda}\subset E$ be a net converging bornologically to $x\in X$. Then $\{x_\lambda\}_{\lambda\in \Lambda}$ is bornologically Cauchy in $X$, and hence it is also bornologically Cauchy in $E$. Since $E$ is complete and separated, \cref{lem:borcompcov} implies that the net $\{x_\lambda\}_{\lambda\in \Lambda}$ converges bornologically to some $y\in E$. By \cref{lem:boulinborncov}, it also converges bornologically to $y$ in $X$. Since $X$ is separated, we have $x=y\in E$. 
\end{proof}

Let $X$ be a convex bornological vector space and $E\subset X$ a linear subspace. The quotient bornology on $X/E$ is a convex vector bornology, since $\{\pi(B)\mid B\in \frB\}$ forms a basis of $X/E$ for any basis $\frB$ of $X$. 

\begin{prp}\label{prp:quotsep}
For a convex bornological vector space $X$ and a linear subspace $E\subset X$, the convex bornological vector space $X/E$ is separated if and only if $E$ is bornologically closed in $X$. 
\end{prp}

\begin{proof}
Suppose that $X/E$ is separated. If a net $\{x_\lambda\}_{\lambda\in \Lambda}\subset E$ converges bornologically to $x\in X$, then by \cref{lem:boulinborncov}, the net $\{\pi(x_\lambda)\}_{\lambda\in \Lambda}$ converges bornologically to $\pi(x)$ in $X/E$. Since $\pi(x_\lambda)=0$ for all $\lambda\in\Lambda$, the net $\{\pi(x_\lambda)\}_{\lambda\in \Lambda}$ also converges bornologically to $0$ in $X/E$. Thus, by \cref{prp:sepcbvs}, we have $x\in E$. 

Conversely, suppose that $E$ is bornologically closed in $X$. To prove that $X/E$ is separated, it suffices to show that $(X/E)_{\pi(B)}$ is a normed space for each nonempty bounded disked set $B\subset X$. Note that every element of $(X/E)_{\pi(B)}$ can be written in the form $\pi(x)$ for some $x\in X_B$. For $x\in X_B$, if $p_{\pi(B)}(\pi(x))=0$, then we have $x\in \bigcap_{\ve\in \bbR_{>0}}(\ve B+E)$. Hence, for each $n\in \bbN$, there exists $x_n\in E$ such that $x_n-x\in (1/2^n)B$. The sequence $\{x_n\}_{n\in \bbN}$ converges bornologically to $x$. Thus the bornologically closedness of $E$ implies that $\pi(x)=0$. 
\end{proof}

\begin{prp}
For a complete separated convex bornological vector space $X$ and a bornologically closed linear subspace $E\subset X$, the convex bornological vector space $X/E$ is complete and separated. 
\end{prp}

\begin{proof}
Choose a basis $\frB$ of $X$ consisting of Banach disked sets. Then by \cref{prp:mapcompdisk}, the collection $\{\pi(B)\mid B\in \frB\}$ is a basis of $X/E$ consisting of Banach disked sets. 
\end{proof}

For a convex bornological vector space $X$, by \cref{prp:quotsep}, the quotient bornology on $X/\olb{\{0\}}$ is separated. We call $X/\olb{\{0\}}$ the \emph{separation} of $X$ and denote it by $X^{\rms}$. The separation has the following universal property: for any separated convex bornological vector space $Y$ and any bounded linear map $f\colon X\to Y$, there exists a unique bounded linear map $\wt{f}\colon X^{\rms}\to Y$ such that the diagram
\[
\begin{tikzcd}[row sep=huge, column sep=huge]
X \arrow[r, ->>] \arrow[rd, "f"'] & X^{\rms} \arrow[d, "\wt{f}"] \\
& Y
\end{tikzcd}
\]
commutes. Note that if $X$ is already separated, then $X^\rms$ canonically isomorphic to $X$ as a bornological vector space. Moreover, for a convex bornological vector space $X$ and a linear subspace $E\subset X$, we have a canonical isomorphism $X/\olb{E}\cong (X/E)^{\rms}$ as bornological vector spaces. 

\subsubsection{Products and coproducts}

Let $\{X_i\}_{i\in I}$ be a family of convex bornological vector spaces. 
The product bornology on $\prod_{i\in I}X_i$ is a convex vector bornology, since $\{\prod_{i\in I}B_i\mid B_i\in \frB_i\}$ forms a basis of $\prod_{i\in I}X_i$ for any basis $\frB_i$ of $X_i$. Thus, the convex bornological vector space $\prod_{i\in I}X_i$ is a product in $\CBVS$. 

\begin{prp}\label{prp:prodsep}
Let $\{X_i\}_{i\in I}$ be a family of convex bornological vector spaces. The convex bornological vector space $\prod_{i\in I}X_i$ is separated if and only if $X_i$ is separated for all $i\in I$. 
\end{prp}

\begin{proof}
Suppose that $\prod_{i\in I}X_i$ is separated. For each $i\in I$, the linear map $\iota_i\colon X_i\to \prod_{j\in I}X_j$, $x\mapsto \{\delta_{i, j}x\}_{j\in I}$ is bounded. If $E\subset X_i$ is a bounded linear subspace, then $\iota_i(E)$ is a bounded linear subspace of $\prod_{i\in I}X_i$. Hence $\iota_i(E)=\{0\}$. Since $\iota_i$ is injective, we obtain $E=\{0\}$. 

Conversely, suppose that $X_i$ is separated for all $i\in I$. We denote by $\pi_i\colon \prod_{i\in I}X_i\to X_i$ the canonical projection. If $E\subset \prod_{i\in I}X_i$ is a bounded linear space, then $\pi_i(E)$ is a bounded linear subspace of $X_i$ for all $i\in I$. Thus $\pi_i(E)=\{0\}$ for all $i\in I$, and hence $E=\{0\}$. 
\end{proof}

\begin{lem}\label{lem:proddisk}
Let $\{X_i\}_{i\in I}$ be a family of linear spaces, and let $A_i\subset X_i$ be a disked set for each $i\in I$. Set $X\ceq \prod_{i\in I}X_i$ and $A\ceq \prod_{i\in I}A_i$. Then: 
\begin{enumerate}
\item 
$X_A=\{(x_i)_{i\in I}\in X\mid \sup_{i\in I}p_{A_i}(x_i)<\infty\}$. 

\item 
For $x=\{x_i\}_{i\in I}\in X_A$, we have $p_A(x)=\sup_{i\in I}p_{A_i}(x_i)$. 
\end{enumerate}
\end{lem}

\begin{proof}
(1) For $x=\{x_i\}_{i\in I}\in X_A$, choose $\ve\in \bbR_{>0}$ such that $x\in \ve A$. Then $x_i\in \ve A_i$ for all $i\in I$, and hence $p_{A_i}(x_i)\le \ve$ for all $i\in I$. Thus $\sup_{i\in I}p_{A_i}(x_i)\le \ve$. Conversely, let $\{x_i\}_{i\in I}\in X$ satisfy $\sup_{i\in I}p_{A_i}(x_i)<\infty$. Choose $\ve\in \bbR_{>0}$ such that $\sup_{i\in I}p_{A_i}(x_i)<\ve$. Then $x_i\in \ve A_i$ for all $i\in I$, and hence $x\in \ve A$. Thus $x\in X_A$. 

(2) Set $s\ceq \sup_{i\in I}p_{A_i}(x_i)$. If $\lambda\in \bbR_{>0}$ satisfies $x\in \lambda A$, then $x_i\in \lambda A_i$ for all $i\in I$, and hence $p_{A_i}(x_i)\le \lambda$ for all $i\in I$. Thus, for every $\lambda\in \bbR_{>0}$ with $x\in \lambda A$, we have $s\le \lambda$. Therefore $s\le p_A(x)$. To prove the equality, let $\ve\in \bbR_{>0}$ be arbitrary. For each $i\in I$, since $p_{A_i}(x_i)<s+\ve$, we have $x_i\in (s+\ve)A_i$. Thus $x\in (s+\ve)A$, and hence $p_A(x)\le s+\ve$. Therefore $0\le p_A(x)-s\le\ve$ for every $\ve\in \bbR_{>0}$. This proves $p_A(x)=s$. 
\end{proof}

\begin{lem}\label{lem:prodcompl}
Let $\{X_i\}_{i\in I}$ be a family of linear spaces. Given a completant disked set $A_i\subset X_i$ for each $i\in I$, the set $\prod_{i\in I}A_i$ is a completant disked set of $\prod_{i\in I}X_i$. 
\end{lem}

\begin{proof}
Set $X\ceq \prod_{i\in I}X_i$ and $A\ceq \prod_{i\in I}A_i$. Let $\{x_\lambda\}_{\lambda\in \Lambda}\subset X_A$ be a Cauchy net, and write $x_\lambda=\{x_{\lambda, i}\}_{i\in I}$ for each $\lambda\in \Lambda$. Note that $\{x_{\lambda, i}\}_{\lambda\in \Lambda}$ is a net in $X_{A_i}$. Moreover, by \cref{lem:proddisk} (2), the net $\{x_{\lambda, i}\}_{\lambda\in \Lambda}$ is Cauchy in $X_{A_i}$. Hence $\{x_{\lambda, i}\}_{\lambda\in \Lambda}$ converges to some $x_i\in X_{A_i}$. We claim that $x\ceq \{x_i\}_{i\in I}\in X_A$, and that $\{x_\lambda\}_{\lambda\in \Lambda}$ converges to $x$ in $X_A$. Since $\{x_\lambda\}_{\lambda\in \Lambda}$ is Cauchy in $X_A$, for every $\ve\in \bbR_{>0}$, there exists $\lambda_0\in \Lambda$ such that 
\[
\lambda, \mu\in \Lambda,\, \lambda, \mu\ge\lambda_0\, \Longrightarrow\, 
p_{A_i}(x_{\lambda, i}-x_{\mu, i})\le p_A(x_\lambda-x_{\mu})<\ve/2 \quad (i\in I). 
\] 
Taking the limit for $\mu$, we obtain 
\[
\lambda\in \Lambda,\, \lambda\ge\lambda_0\, \Longrightarrow\, 
p_{A_i}(x_{\lambda, i}-x_i)\le \ve/2 \quad (i\in I). 
\]
In particular, we have 
\[
p_{A_i}(x_i)\le p_{A_i}(x_{\lambda_0, i}-x_i)+p_{A_i}(x_{\lambda_0, i})\le \ve/2+p_A(x_{\lambda_0}) \quad (i\in I). 
\]
Thus $\sup_{i\in I}p_{A_i}(x_i)<\infty$, and hence $x\in X_A$ by \cref{lem:proddisk} (1). Moreover, for $\lambda\in \Lambda$ with $\lambda\ge\lambda_0$, 
\[
p_A(x_\lambda-x)=\sup_{i\in I}p_{A_i}(x_{\lambda, i}-x_i)\le \ve/2<\ve, 
\]
which proves the claim. 
\end{proof}

\begin{prp}\label{prp:prodcomp}
Let $\{X_i\}_{i\in I}$ be a family of convex bornological vector space. If $X_i$ is complete for all $i\in I$, then $\prod_{i\in I}X_i$ is complete. 
\end{prp}

\begin{proof}
For each $i\in I$, choose a basis $\frB_i$ of $X_i$ consisting of completant disked sets. Then by \cref{lem:prodcompl}, the collection $\{\prod_{i\in I}B_i\mid B_i\in \frB_i\}$ is a basis of $\prod_{i\in I}X_i$ consisting of completant disked sets. 
\end{proof}

Let $\{X_i\}_{i\in I}$ be a family of convex bornological vector spaces. Consider the direct sum $\bigoplus_{i\in I}X_i$, and denote by $\iota_i\colon X_i\inj \bigoplus_{i\in I}X_i$ the canonical inclusion. 
Then the collection
\[
\frB\ceq\Bigl\{\sum_{j\in J}\iota_j(B_j)\Bigm| J\subset I \text{ is finite},\, B_j\subset X_j \text{ are nonempty bounded disked sets}\Bigr\}
\]
satisfies the conditions in \cref{lem:constcbvs}. Hence $\bigoplus_{i\in I}X_i$ admits the convex vector bornology for which $\frB$ is a basis. This bornology is called the \emph{direct sum bornology}. Note that if $\frB_i$ is a basis of $X_i$ for each $i\in I$, then 
\[
\Bigl\{\sum_{j\in J}\iota_j(B_j)\Bigm| J\subset I \text{ is finite},\, B_j\in \frB_j\Bigr\}
\]
is a basis of $\bigoplus_{i\in I}X_i$. It is clear that $\iota_i\colon X_i\inj \bigoplus_{i\in I}X_i$ are all bounded. Moreover, a linear map $f\colon \bigoplus_{i\in I}X_i\to Y$, where $Y$ is another convex bornological vector space, is bounded if and only if $f\circ \iota_i\colon X_i\to Y$ is bounded for all $i\in I$. Thus, the convex bornological vector space $\bigoplus_{i\in I}X_i$ is a coproduct in $\CBVS$. 

\begin{prp}\label{prp:sepoplus}
Let $\{X_i\}_{i\in I}$ be a family of convex bornological vector spaces. The convex bornological vector space $\bigoplus_{i\in I}X_i$ is separated if and only if $X_i$ is separated for all $i\in I$. 
\end{prp}

\begin{proof}
Suppose that $\bigoplus_{i\in I}X_i$ is separated. For each $i\in I$, if $E\subset X_i$ is a bounded linear subspace, then $\iota_i(E)$ is a bounded linear subspace of $\bigoplus_{i\in I}X_i$. Hence $\iota_i(E)=\{0\}$. Since $\iota_i$ is injective, we obtain $E=\{0\}$. 

Conversely, suppose that $X_i$ is separated for all $i\in I$. For each $i\in I$, the linear map $\pi_i\colon \bigoplus_{i\in I}X_i\to X_i$, $\{x_i\}_{i\in I}\mapsto x_i$ is bounded. If $E\subset \bigoplus_{i\in I}X_i$ is a bounded linear subspace, then $\pi_i(E)$ is a bounded linear subspace of $X_i$ for all $i\in I$. Thus $\pi_i(E)=\{0\}$ for all $i\in I$, and hence $E=\{0\}$. 
\end{proof}

\begin{lem}\label{lem:XA+B}
Let $X$ be a linear space, and let $A, B\subset X$ be nonempty disked sets. 
\begin{enumerate}
\item 
$X_{A+B}=X_A+X_B$. 

\item 
For $z\in X_{A+B}$, we have
\[
p_{A+B}(z)=\inf\{\max\{p_A(x), p_B(y)\}\mid x\in X_A, y\in X_B, z=x+y\}. 
\]
\end{enumerate}
\end{lem}

\begin{proof}
(1) Since $A+B\subset X_A+X_B$, it follows that $X_{A+B}\subset X_A+X_B$. Moreover, since $A$ and $B$ are nonempty balanced sets, we have $0\in A, B$. Thus $A, B\subset A+B$, and hence $X_A, X_B\subset X_{A+B}$. It follows that $X_A+X_B\subset X_{A+B}$. 

(2) For $\lambda\in \bbR_{>0}$ with $z\in \lambda(A+B)$, choose $x\in \lambda A$ and $y\in \lambda B$ such that $z=x+y$. Then 
\[
\text{(RHS)}\le \max\{p_A(x), p_B(y)\}\le \lambda. 
\]
Hence $p_{A+B}(z)\ge \text{(RHS)}$. Conversely, for $x\in X_A$ and $y\in X_B$ with $z=x+y$, set $\lambda_{x, y}\ceq \max\{p_A(x), p_B(y)\}$. For arbitrary $\ve\in\bbR_{>0}$, we have $x\in (\lambda_{x, y}+\ve)A$ and $y\in (\lambda_{x, y}+\ve)B$. Thus $z\in (\lambda_{x, y}+\ve)(A+B)$, and hence $p_{A+B}(z)\le \lambda_{x, y}+\ve$. Since $\ve\in \bbR_{>0}$ is arbitrary, we obtain $p_{A+B}(z)\le \lambda_{x, y}$. Therefore $p_{A+B}(z)\le \text{(RHS)}$. 
\end{proof}

\begin{lem}\label{lem:sumBanadisk}
Let $X$ be a separated convex bornological vector space. For bounded Banach disked sets $A, B\subset X$, the subset $A+B$ is a Banach disked set of $X$. 
\end{lem}

\begin{proof}
Define a linear map $f$ by
\[
f\colon X_A\times X_B\to X_{A+B}, \quad (x, y)\mapsto x+y. 
\]
By \cref{lem:XA+B} (1), the map $f$ is surjective, and hence induces a linear isomorphism 
\[
X_{A+B}\cong (X_A\times X_B)/\Ker f. 
\]
Equip $X_A\times X_B$ with the product norm 
\[
p\colon X_A\times X_B\to \bbR, \quad (x, y)\mapsto \max\{p_A(x), p_B(y)\}. 
\]
Then, by \cref{lem:XA+B} (2), we have $p_{A+B}(f(x, y))\le p(x, y)$ for all $x\in X_A$ and $y\in X_B$. Hence $f$ is continuous. Moreover, by \cref{prp:sepcbvs}, the seminormed space $X_{A+B}$ is Hausdorff. Thus, the subspace $\Ker f$ is closed in $X_A\times X_B$. It follows that the quotient space $(X_A\times X_B)/\Ker f$, equipped with the quotient norm
\[
\ol{p}\colon (X_A\times X_B)/\Ker f\to \bbR, \quad \ol{(x, y)}\mapsto \inf\{p((x, y)-(u, v))\mid(u, v)\in \Ker f\},
\]
is a Banach space. By \cref{lem:XA+B} (2), we have $p_{A+B}(f(x, y))=\ol{p}(\ol{(x, y)})$ for all $x\in X_A$ and $y\in X_B$. Therefore, the normed space $X_{A+B}$ is isometric to $(X_A\times X_B)/\Ker f$, which proves the lemma. 
\end{proof}

\begin{prp}\label{prp:dsumcomp}
Let $\{X_i\}_{i\in I}$ be a family of convex bornological vector spaces. If $X_i$ is complete and separated for all $i\in I$, then $\bigoplus_{i\in I}X_i$ is complete and separated. 
\end{prp}

\begin{proof}
For each $i\in I$, choose a basis $\frB_i$ of $X_i$ consisting of Banach disked sets. Then by \cref{prp:mapcompdisk} and \cref{lem:sumBanadisk}, 
the collection
$\{\sum_{j\in J}\iota_j(B_j)\Bigm| J\subset I \text{ is finite},\, B_j\in \frB_j\}$
is a basis of $\bigoplus_{i\in I}X_i$ consisting of Banach disked sets. 
\end{proof}

\subsubsection{Projective limits and inductive limits}

For a projective system $(\{X_i\}_{i\in I}, \{f_{i, j}\}_{i\le j})$ in $\CBVS$, the projective limit bornology on $\varprojlim_{i\in I}X_i$ is a convex vector bornology, since it coincides with the subspace bornology induced by the projective limit bornology on $\prod_{i\in I}X_i$. Thus, the convex bornological vector space $\varprojlim_{i\in I}X_i$ is a projective limit in $\CBVS$. By \cref{prp:prodsep}, if $X_i$ is separated for all $i\in I$, then $\varprojlim_{i\in I}X_i$ is separated. 

\begin{prp}
Let $(\{X_i\}_{i\in I}, \{f_{i, j}\}_{i\le j})$ be a projective system in $\CBVS$. If $X_i$ is complete and separated for all $i\in I$, then $\varprojlim_{i\in I}X_i$ is complete and separated. 
\end{prp}

\begin{proof}
By \cref{prp:subcomp} (1) and \cref{prp:prodcomp}, it suffices to show that $\varprojlim_{i\in I}X_i$ is bornologically closed in $\prod_{i\in I}X_i$. Let $\{x_\lambda\}_{\lambda\in \Lambda}\subset \varprojlim_{i\in I}X_i$ be a net converging bornologically to $x\in \prod_{i\in I}X_i$. Write $x_\lambda=\{x_{\lambda, i}\}_{i\in I}$ and $x=\{x_i\}_{i\in I}$. Then for each $i\in I$, the net $\{x_{\lambda, i}\}_{\lambda\in \Lambda}$ converges bornologically to $x_i$ in $X_i$. Hence, for $i, j\in I$ with $i\le j$, 
\[
f_{i, j}(x_j)=\blim_{\lambda\in \Lambda}f_{i, j}(x_{\lambda, j})=
\blim_{\lambda\in \Lambda}x_{\lambda, i}=x_i. 
\]
Thus $x\in \varprojlim_{i\in I}X_i$, which proves the claim. 
\end{proof}

Let $(\{X_i\}_{i\in I}, \{f_{i, j}\}_{i\le i})$ be an inductive system in $\CBVS$. Consider the inductive limit $\varinjlim_{i\in I}X_i$ in $\ModK$, and denote by $\iota_i\colon X_i\to \varinjlim_{i\in I}X_i$ the canonical morphisms. Then the collection
\[
\frB\ceq\{\iota_i(B)\mid i\in I,\, \text{$B\subset X_i$ is a nonempty bounded disked set}\}
\]
satisfies the conditions in \cref{lem:constcbvs}. Hence $\varinjlim_{i\in I}X_i$ admits the convex vector bornology for which $\frB$ is a basis. This bornology is called the \emph{inductive limit bornology}. Note that if $\frB_i$ is a basis of $X_i$ for each $i\in I$, then $\{\iota_i(B)\mid i\in I, B\in \frB_i\}$ forms a basis of $\varinjlim_{i\in I}X_i$. It is clear that $\iota_i\colon X_i\to \varinjlim_{i\in I}X_i$ are all bounded. Moreover, a linear map $f\colon \varinjlim_{i\in I}X_i\to Y$, where $Y$ is another convex bornological vector space, is bounded if and only if $f\circ \iota_i\colon X_i\to Y$ is bounded for all $i\in I$. Thus, the convex bornological vector space $\varinjlim_{i\in I}X_i$ is an inductive limit in $\CBVS$. 

\begin{exm}\label{exm:Xindlim}
Let $X$ be a convex bornological vector space, and let $\frB$ be a basis of $X$ consisting of nonempty disked sets. Recall that for $A, B\in \frB$ with $A\subset B$, the inclusion $\iota_{A, B}\colon X_A\inj X_B$ is bounded. Hence $\{X_B\}_{B\in \frB}$, together with the bounded linear maps $\iota_{A, B}\colon X_A\inj X_B$ ($A, B\in \frB$, $A\subset B$), forms an inductive system in $\CBVS$. For each $B\in \frB$, denote by $\iota_B\colon X_B\inj X$ the inclusion. Then $(X, \{\iota_B\}_{B\in \frB})$ is an inductive limit of this inductive system. 
\end{exm}

\begin{prp}
Let $(\{X_i\}_{i\in I}, \{f_{i, j}\}_{i\le j})$ be an inductive system in $\CBVS$ such that $f_{i, j}\colon X_i\to X_j$ is injective for all $i\le j$. 
\begin{enumerate}
\item 
If $X_i$ is separated for all $i\in I$, then $\varinjlim_{i\in I}X_i$ is separated. 

\item 
If $X_i$ is complete and separated for all $i\in I$, then $\varinjlim_{i\in I}X_i$ is complete and separated. 
\end{enumerate}
\end{prp}

\begin{proof}
(1) 
Let $E\subset \varinjlim_{i\in I}X_i$ be a bounded linear subspace. Choose $i\in I$ and a bounded set $B\subset X_i$ such that $E\subset \iota_i(B)$. Since $\iota_i\colon X_i\to \varinjlim_{i\in I}X_i$ is injective, we have $\iota_i^{-1}(E)\subset B$. By the separatedness of $X_i$, it follows that $\iota_i^{-1}(E)=\{0\}$. For $x\in E$, there exists $y\in B$ such that $x=\iota_i(y)$. Since $y\in \iota_i^{-1}(E)$, we obtain $y=0$, and hence $x=0$. Thus $E=\{0\}$. 

(2) For each $i\in I$, choose a basis $\frB_i$ of $X_i$ consisting of Banach disked sets. Then by \cref{prp:mapcompdisk}, the collection $\{\iota_i(B)\mid i\in I, B\in\frB_i\}$ is a basis of $\varinjlim_{i\in I}X_i$ consisting of Banach disked sets. 
\end{proof}

In general, the inductive limit $\varinjlim_{i\in I}X_i$ in $\CBVS$ need not be separated even if all $X_i$ are separated. For an inductive system $(\{X_i\}_{i\in I}, \{f_{i, j}\}_{i\le j})$ in $\CBVS$, we denote by $\varinjlim_{i\in I}^\rms X_i$ the separation of $\varinjlim_{i\in I}X_i$, and call it the \emph{separated inductive limit}. We refer to the composition 
\[
\begin{tikzcd}
\iota_i^{\rms}: X_i \arrow[r, "\iota_i"] & \varinjlim_{i\in I}X_i \arrow[r, ->>] & \varinjlim_{i\in I}^\rms X_i
\end{tikzcd}
\]
of bounded linear maps 
as the canonical morphism of the separated inductive limit. 
The separated inductive limit has the following universal property: for any separated convex bornological vector space $Y$ and any family $\{f_i\colon X_i\to Y\}_{i\in I}$ of bounded linear maps such that $f_j\circ f_{i, j}=f_i$ ($i, j\in I$, $i\le j$), there exists a unique bounded linear map $f\colon \varinjlim_{i\in I}^{\rms}X_i\to Y$ satisfying $f\circ \iota_i^\rms=f_i$ ($i\in I$). 

\begin{prp}\label{prp:sepindlimcomp}
Let $(\{X_i\}_{i\in I}, \{f_{i, j}\}_{i\le j})$ be an inductive system in $\CBVS$. If $X_i$ is complete and separated for all $i\in I$, then $\varinjlim_{i\in I}^\rms X_i$ is complete and separated. 
\end{prp}

\begin{proof}
For each $i\in I$, choose a basis $\frB_i$ of $X_i$ consisting of Banach disked sets. Then by \cref{prp:mapcompdisk}, the collection $\{\iota_i^\rms(B)\mid i\in I, B\in\frB_i\}$ is a basis of $\varinjlim^{\rms}_{i\in I}X_i$ consisting of Banach disked sets. 
\end{proof}

\subsubsection{Completion}

Let $X$ be a convex bornological vector space, and write $\frB$ for the collection of all nonempty bounded disked sets of $X$. For each $B\in \frB$, denote by $\wh{X}_B$ the (topological) completion of the seminormed space $X_B$. If $A, B\in \frB$ satisfy $A\subset B$, then the inclusion $X_A\inj X_B$ induces a continuous linear map $\wh{X}_A\to \wh{X}_B$. Hence $\{\wh{X}_B\}_{B\in \frB}$, together with the bounded linear maps $\wh{X}_A\to \wh{X}_B$ ($A, B\in \frB$, $A\subset B$), forms an inductive system in $\CBVS$. By \cref{prp:sepindlimcomp}, the separated inductive limit
\[
\wt{X}\ceq \textstyle\varinjlim_{B\in \frB}^{\rms}\wh{X}_B
\]
is complete and separated. We call $\wt{X}$ the \emph{completion} of $X$ (or, more precisely, the \emph{bornological completion} of $X$). For each $B\in \frB$, denote by $\iota_B$ the composition of canonical morphisms 
\[
\begin{tikzcd}
\iota_B\colon X_B \arrow[r] &
\wh{X}_B \arrow[r] &
\wt{X}.
\end{tikzcd}
\]
Then by \cref{exm:Xindlim}, the maps $\{\iota_B\}_{B\in \frB}$ induces a bounded linear map $\iota\colon X\to \wt{X}$. We refer to $\iota$ as the canonical morphism of the completion. 

\begin{prp}
For a convex bornological vector space $X$, the completion $\wt{X}$ has the following universal property: for any complete and separated convex bornological vector space $Y$ and any bounded linear map $f\colon X\to Y$, there exists a unique bounded linear map $\wt{f}\colon \wt{X}\to Y$ satisfying $\wt{f}\circ \iota=f$. Here $\iota\colon X\to \wt{X}$ denotes the canonical morphism of the completion. 
\end{prp}

\begin{proof}
Write $\frB$ for the collection of all nonempty bounded disked sets of $X$. For each $A\in \frB$, define a bounded linear map $f_A\colon \wh{X}_A\to Y$ as follows: Choose a Banach disked set $B\subset Y$ such that $f(A)\subset B$. Then $f|_{X_A}\colon X_A\to Y_B$ is bounded, and hence is continuous. Thus $f|_{X_A}$ induces a continuous linear map $\wh{X}_A\to Y_B$. Define $f_A$ to be the composition of this map with the inclusion $Y_B\inj Y$. The definition of $f_A$ is independent of the choice of $B$. By the universal property of the separated inductive limit, the maps $\{f_A\}_{A\in \frB}$ induces a unique bounded linear map $\wt{f}\colon \wt{X}\to Y$ such that the diagram 
\[
\begin{tikzcd}[row sep=huge, column sep=huge]
\wh{X}_A \arrow[r] \arrow[rd, "f_A"'] & \wt{X} \arrow[d, "\wt{f}"] \\
& Y
\end{tikzcd}
\]
commutes for all $A\in \frB$, where $\wh{X}_A\to \wt{X}$ denotes the canonical morphism of the separated inductive limit. One can verify that $\wt{f}$ is a unique bounded linear map satisfying $\wt{f}\circ \iota=f$. 
\end{proof}

\subsubsection{Tensor products and internal homs}

Let $X$ and $Y$ be convex bornological vector spaces. Consider the (algebraic) tensor product $X\otimes Y$, and denote by $\otimes \colon X\times Y\to X\otimes Y$ the canonical bilinear map. 
Then the collection 
\[
\frC\ceq \{\disk(\otimes(A\times B))\mid \text{ $A\subset X$ and $B\subset Y$ are nonempty bounded sets}\Bigr\}
\]
satisfies the conditions in \cref{lem:constcbvs}. Hence $X\otimes Y$ admits the convex vector bornology for which $\frC$ is a basis. This bornology is called the \emph{tensor product bornology}. Note that if $\frA$ is a basis of $X$ and $\frB$ is a basis of $Y$, then $\{\disk(\otimes (A\times B))\mid A\in \frA, B\in \frB\}$ is a basis of $X\otimes Y$. It is clear that $\otimes \colon X\times Y\to X\otimes Y$ is bounded. Moreover, a linear map $f\colon X\otimes Y\to Z$, where $Z$ is another convex bornological vector space, is bounded if and only if $f\circ \otimes \colon X\times Y\to Z$ is bounded. The convex bornological vector space $X\otimes Y$ has the following universal property: for any convex bornological vector space $Z$ and any bounded bilinear map $\varphi\colon X\times Y\to Z$, there exists a unique bounded linear map $\wt{\varphi}\colon X\otimes Y\to Z$ satisfying $\wt{\varphi}\circ \otimes =\varphi$. 

\begin{lem}\label{lem:tenssep}
Let $X$ and $Y$ be linear spaces, and let $A\subset X$ and $B\subset Y$ be nonempty disked sets. Set $C\ceq \disk(\otimes (A\times B))$.
\begin{enumerate}
\item 
$(X\otimes Y)_C=X_A\otimes Y_B$. 

\item 
The seminorm $p_C$ coincides with the projective tensor product of $p_A$ and $p_B$, that is, for $z\in (X\otimes Y)_C$, 
\[
p_C(z)=\inf\Bigl\{\sum_{i=1}^np_A(x_i)p_B(y_i)\Bigm| x_i\in X_A,\, y_i\in Y_B,\, z=\sum_{i=1}^nx_i\otimes y_i\Bigr\}. 
\]
\end{enumerate}
\end{lem}

\begin{proof}
(1) We have 
\[
(X\otimes Y)_C=\Span_{\bbK}\bigl(\disk(\otimes(A\times B))\bigr)=\Span_\bbK(\otimes (A\times B))=X_A\otimes Y_B. 
\]

(2) Denote by $q$ the projective tensor product of $p_A$ and $p_B$. We first prove that $p_C(z)\le q(z)$. Let $x\in X_A$ and $y\in Y_B$. For arbitrary $\ve\in \bbR_{>0}$, since $x\in (p_A(x)+\ve)A$ and $y\in (p_B(y)+\ve)B$, we have $x\otimes y\in (p_A(x)+\ve)(p_B(y)+\ve)C$. Hence $p_C(x\otimes y)\le (p_A(x)+\ve)(p_B(y)+\ve)$. By letting $\ve\to 0$,  we obtain $p_C(x\otimes y)\le p_A(x)p_B(y)$. Thus, for $z=\sum_{i=1}^nx_i\otimes y_i$ with $x_i\in X_A$ and $y_i\in Y_B$, 
\[
p_C(z)\le \sum_{i=1}^np_C(x_i\otimes y_i)\le \sum_{i=1}^np_A(x_i)p_B(y_i). 
\]
This proves $p_C(z)\le q(z)$. We next prove that $q(z)\le p_C(z)$. For arbitrary $\lambda>p_C(z)$, we have $z\in \lambda C$. Hence, we can write $z=\lambda\sum_{i=1}^na_ix_i\otimes y_i$ with $x_i\in A$, $y_i\in B$, and $a_i\in \bbK$ such that $\sum_{i=1}^n|a_i|\le 1$. Thus
\[
q(z)\le \lambda\sum_{i=1}^n|a_i|p_A(x_i)p_B(y_i)\le \lambda\sum_{i=1}^n|a_i|\le \lambda. 
\]
Since $\lambda>p_C(z)$ is arbitrary, we obtain $q(z)\le p_C(z)$. 
\end{proof}

\begin{prp}\label{prp:tenssep}
For separated convex bornological vector spaces $X$ and $Y$, the convex bornological vector space $X\otimes Y$ is separated. 
\end{prp}

\begin{proof}
It suffices to show that for $C=\disk(\otimes(A\times B))$, where $A\subset X$ and $B\subset Y$ are nonempty bounded disked sets, the seminorm $p_C$ is a norm. By \cref{lem:tenssep} (2), the seminorm $p_C$ coincides with the projective tensor product of $p_A$ and $p_B$. Since $X$ and $Y$ are separated, the seminorms $p_A$ and $p_B$ are norms. It is well known that the projective tensor product of norms is a norm. Hence $p_C$ is a norm, as desired. 
\end{proof}

For convex bornological vector spaces $X$ and $Y$, we denote by $X\totimes Y$ the completion of $X\otimes Y$, and call it the \emph{completed tensor product}. We also denote by $\totimes\colon X\times Y\to X\totimes Y$ the composition of $\otimes \colon X\times Y\to X\otimes Y$ with the canonical morphism $X\otimes Y\to X\totimes Y$ of the completion, and call it the canonical bilinear map. The completed tensor product $X\totimes Y$ has the following universal property: for any complete and separated convex bornological vector space $Z$ and any bounded bilinear map $\varphi\colon X\times Y\to Z$, there exists a unique bounded linear map $\wt{\varphi}\colon X\totimes Y\to Z$ satisfying $\wt{\varphi}\circ \totimes=\varphi$. 

\smallskip

Let $X$ and $Y$ be convex bornological vector space. Consider the linear space $\Hom_{\CBVS}(X, Y)$ of bounded linear maps. Write $\frB$ for the collection of nonempty disked set $F\subset \Hom_{\CBVS}(X, Y)$ such that $\bigcup_{f\in F}f(A)$ is bounded in $Y$ for all bounded set $A\subset X$. Then $\frB$ satisfies the conditions in \cref{lem:constcbvs}. Hence $\Hom_{\CBVS}(X, Y)$ admits the convex vector bornology for which $\frB$ is a basis. For each $x\in X$, the evaluation map 
\[
\ev_x\colon \Hom_{\CBVS}(X, Y)\to Y, \quad f\mapsto f(x)
\]
is bounded. 

\begin{prp}
Let $X$ and $Y$ be convex bornological vector spaces. 
\begin{enumerate}
\item
If $Y$ is separated, then $\Hom_{\CBVS}(X, Y)$ is separated. 

\item 
If $Y$ is complete and separated, then $\Hom_{\CBVS}(X, Y)$ is complete and separated. 
\end{enumerate}
\end{prp}

\begin{proof}
Write $H\ceq \Hom_{\CBVS}(X, Y)$. 

(1) Let $E\subset H$ be a bounded linear subspace. It suffices to show that $f(x)=0$ for every $x\in X$ and $f\in E$. Let $x\in X$. Then $\ev_x(E)$ is a bounded linear subspace of $Y$. Thus $\ev_x(E)=\{0\}$, and hence $f(x)=0$ for every $f\in E$. 

(2) Let $F\subset H$ be a nonempty bounded disked set. For each bounded set $A\subset X$, there exists a Banach disked set $B_A\subset Y$ such that $\bigcup_{f\in F}f(A)\subset B_A$ and $B_A$ is closed in $Y_{B_A}$. Define 
\[
G\ceq\{f\in H\mid f(A)\subset B_A \text{ for every bounded set }A\subset X\}. 
\]
Then $G$ is a bounded disked set of $H$ such that $F\subset G$. Hence it suffices to show that $G$ is completant. 

Let $\{f_n\}_{n\in \bbN}\subset H_G$ be a Cauchy sequence. For each $x\in X$, the sequence $\{f_n(x)\}_{n\in \bbN}$ is Cauchy in $ Y_{B_{\{x\}}}$. Hence, by \cref{lem:XBbornCau}, the sequence $\{f_n(x)\}_{n\in \bbN}$ is bornologically Cauchy in $Y$. Thus, by \cref{lem:borcompcov}, we can define a map $f\colon X\to Y$ by $f(x)\ceq \blim_{n\to \infty}f_n(x)$. 
Since each $f_n$ is linear, so is $f$. 

We claim that for every bounded set $A\subset X$ and $x\in A$, we have $f(x)=\lim_{n\to \infty}f_n(x)$ in $Y_{B_A}$. Since $\{f_n(x)\}_{n\in \bbN}$ is Cauchy in $Y_{B_A}$, the limit $y\ceq\lim_{n\to \infty}f_n(x)$ exists. By \cref{lem:XBborcov}, the sequence $\{f_n(x)\}_{n\in \bbN}$ converges bornologically to $y$ in $Y$. Since $Y$ is separated, we have $f(x)=y$. 

We next claim that $f\in H_G$. Since $\{f_n\}_{n\in \bbN}$ is Cauchy in $H_G$, there exists $N\in \bbN$ such that $n\ge N$ implies $p_G(f_n-f_N)<1$, and hence $f_n-f_N\in G$. For a bounded set $A\subset X$ and $x\in A$, if $n\ge N$, then $f_n(x)-f_N(x)\in B_A$. Since $B_A$ is closed in $Y_{B_A}$, letting $n\to \infty$ gives $f(x)-f_N(x)\in B_A$. Thus $(f-f_N)(A)\subset B_A$ for all bounded set $A\subset X$. This shows that $f-f_N\in G$, and hence $f=(f-f_N)+f_N\in H_G$. 

We finally prove that $\{f_n\}_{n\in \bbN}$ converges to $f$ in $H_G$. Let $\ve\in \bbR_{>0}$. Choose $N\in \bbN$ such that $m, n\ge N$ implies $p_G(f_m-f_n)<\ve/2$, and hence $f_m-f_n\in (\ve/2)G$. For a bounded set $A\subset X$ and $x\in A$, if $m, n\ge N$, then $f_m(x)-f_n(x)\in (\ve/2)B_A$. Since $B_A$ is closed in $Y_{B_A}$, letting $n\to \infty$ gives $f_m(x)-f(x)\in (\ve/2)B_A$ for all $m\ge N$. Thus if $m\ge N$, then $(f_m-f)(A)\subset (\ve/2)B_A$ for all bounded set $A\subset X$, and hence $f_m-f\in (\ve/2)G$. Therefore, for $m\ge N$, we have $p_G(f_m-f)\le \ve/2<\ve$. This proves that $H_G$ is complete. 
\end{proof}

We obtain functors 
\[
(-)\otimes (-)\colon \CBVS\times \CBVS\to \CBVS, \quad 
\Hom_{\CBVS}(-, -)\colon \CBVS^{\op}\times \CBVS\to \CBVS. 
\]
For a convex bornological vector space $X$, it is straightforward to verify that $X\otimes(-)\colon \CBVS\to \CBVS$ is left adjoint to $\Hom_{\CBVS}(X, -)\colon \CBVS\to \CBVS$. 

\subsubsection{Fine bornology}

We first prove that a separated convex bornology on a finite-dimensional linear space is unique, and then recall fine bornology on arbitrary every linear space. 

\begin{lem}\label{lem:oplusborn}
Let $X$ be a convex bornological vector space, and let $V, W\subset X$ be linear subspaces. If 
\[
V\to X/W, \quad v\mapsto \ol{v}
\]
is an isomorphism of bornological vector spaces, then $X=V\oplus W$ as bornological vector spaces. 
\end{lem}

\begin{proof}
By definitions of subspace bornologies and direct sum bornologies, the inclusion $V\oplus W\inj X$ is bounded. It suffices to show that every bounded set $B\subset X$ is bounded in $V\oplus W$. We denote by $\pi\colon X\to X/W$ the canonical projection. Then, by assumption, the map $\pi|_V\colon V\to X/W$ is an isomorphism of bornological vector spaces. Hence $p\ceq (\pi|_V)^{-1}\circ \pi\colon X\to V$ is bounded. Moreover, $q\ceq \id_X-p\colon X\to X$ is bounded and $q(X)\subset W$. Hence $q\colon X\to W$ is bounded. Thus $p(B)+q(B)$ is bounded in $V\oplus W$. Since $B\subset p(B)+q(B)$, the subset $B$ is bounded in $V\oplus W$. 
\end{proof}

We equip $\bbK^n$ with the von Neumann bornology. Then $\bbK^n$ is a complete separated convex bornological vector space. 

\begin{lem}\label{lem:1dimborn}
If $S\subset \bbK$ is an unbounded balanced set, then $S=\bbK$. 
\end{lem}

\begin{proof}
Let $a\in \bbK$. Choose $R\in \bbR_{>0}$ such that $|a|<R$. Since $S$ is unbounded, we have $S\not\subset D_R$. Hence there exists $x\in S$ such that $|x|\ge R$. Thus $a=(a/x)x\in S$. 
\end{proof}

\begin{lem}\label{lem:1dimborniso}
If $E$ is a separated convex bornological vector space of dimension $1$, then $E\cong \bbK$ as bornological vector spaces. 
\end{lem}

\begin{proof}
Fix $v\in E\bs\{0\}$. Then $\bbK\to E$, $a\mapsto av$ is a bounded linear isomorphism. It suffices to show that its inverse map $f\colon E\to \bbK$ is bounded. Suppose that $f$ is unbounded. Then there exists a bounded balanced set $B\subset E$ such that $f(B)$ is unbounded in $\bbK$. By \cref{lem:1dimborn}, we have $f(B)=\bbK$, and hence $B=X$. Thus the separatedness of $E$ implies $E=\{0\}$. This contradicts $\dim_{\bbK}E=1$. 
\end{proof}

\begin{lem}\label{lem:findimborn}
Let $X$ be a separated convex bornological vector space, let $E\subset X$ be a one-dimensional linear subspace, and let $H\subset X$ be a bornologically closed hyperplane. If $E\cap H=\{0\}$, then $X=E\oplus H$ as bornological vector spaces. 
\end{lem}

\begin{proof}
By \cref{prp:quotsep}, the quotient bornological vector space $X/H$ is separated. By \cref{lem:1dimborniso}, the map $E\to X/H$, $v\mapsto \ol{v}$ is an isomorphism of bornological vector spaces. Thus the lemma follows from \cref{lem:oplusborn}. 
\end{proof}

\begin{prp}\label{prp:findimborn}
If $X$ is a separated convex bornological vector space of finite dimension $n$, then $X\cong\bbK^n$ as bornological vector spaces. 
\end{prp}

\begin{proof}
We proceed by induction on $n$. The case $n=1$ has already been proved in \cref{lem:1dimborniso}. Consider the case $n>1$. Choose a one-dimensional linear subspace $E\subset X$ and hyperplane $H\subset X$ such that $E\cap H=\{0\}$. Then, by induction hypothesis, we have $H\cong \bbC^{n-1}$ as bornological vector spaces. In particular,  $H$ is complete, and hence \cref{prp:subcomp} (2) implies that $H$ is bornologically closed. Thus, by \cref{lem:findimborn}, we have $X=E\oplus H\cong \bbC\oplus \bbC^{n-1}\cong \bbC^n$ as bornological vector spaces. 
\end{proof}

By \cref{prp:findimborn}, there is a unique separated convex vector bornology on every finite-dimensional linear space, and it is complete.

\smallskip

Every linear space $V$ carries a natural bornology, called the \emph{fine bornology}, whose bounded sets are those subsets $B\subset V$ satisfying the following condition: there exists a finite-dimensional linear subspace $E\subset V$ such that $B\subset E$ and $B$ is bounded in $E$, where $E$ is equipped with its unique separated convex vector bornology. One can verify that the fine bornology is a separated convex vector bornology.

\begin{prp}\label{prp:finechara}
Let $V$ be a linear space. 
\begin{enumerate}
\item 
The fine bornology is the smallest separated convex vector bornology on $V$. 

\item 
For a basis $\{v_i\}_{i\in I}\subset V$, the linear isomorphism 
\[
\bbC^{\oplus I}\to V, \quad \{a_i\}_{i\in I}\mapsto \sum_{i\in I}a_iv_i
\]
is an isomorphism of bornological vector spaces, where $\bbC^{\oplus I}$ is equipped with the direct sum bornology and $V$ with the fine bornology. 

\item 
The fine bornology on $V$ is complete. 
\end{enumerate}
\end{prp}

\begin{proof}
(1) Since there is a unique separated convex vector bornology on every finite-dimensional linear space, for every finite-dimensional linear subspace $E\subset V$, the inclusion $E\inj V$ is bounded with respect to any separated convex vector bornology on $V$. This proves (1). 

(2) Denote by $f\colon \bbC^{\oplus I}\to V$ the linear isomorphism in the statement. By construction of direct sum bornologies, the map $f$ is bounded. Moreover, by (1), the inverse map $f^{-1}$ is also bounded. 

(3) This follows from (2) and \cref{prp:dsumcomp}. 
\end{proof}

\begin{cor}\label{cor:finetens} 
For a complete separated convex bornological vector space $X$ and a linear space $V$ equipped with the fine bornology, the tensor product bornology on $X\otimes V$ is complete and separated. 
\end{cor}

\begin{proof}
Since $X\otimes(-)\colon \CBVS\to \CBVS$ is a left adjoint functor, it preserves colimits. Hence the corollary follows from \cref{prp:finechara} (2) and \cref{prp:dsumcomp}. 
\end{proof}

\section{Bornological complex analysis}
\label{s:bornholo}

A complete separated convex bornological vector space $X$ can be viewed as a union $X=\bigcup_{B\in \frB}X_B$ of Banach spaces $X_B$, where $\frB$ is a basis of $X$ consisting of Banach disked sets. From this perspective, we develop complex analysis of functions with values in a complete separated convex bornological vector space. The results here will be used in the next two sections to prove the main theorems. 

\subsection{Bornological differentiation and bornological Riemann integration}
\label{ss:borndiffint}

Let $X$ be a complete separated convex bornological vector space over $\bbR$. 

\begin{dfn}
Let $S$ be a topological space, and let $f\colon S\to X$ be a map. We say that \emph{$f$ is bornologically continuous on $S$} if for every $x\in S$, there exist an open neighborhood $U\subset S$ of $x$ and a bounded Banach disked set $B\subset X$ such that $f(U)\subset X_B$ and $f\colon U\to X_B$ is continuous on $U$. 
\end{dfn}

\begin{lem}\label{lem:borncont}
Let $S$ be a topological space, and let $f\colon S\to X$ be a bornologically continuous map. Then, for every compact set $K\subset S$, there exists a bounded Banach disked set $B\subset X$ such that $f(K)\subset X_B$ and $f\colon K\to X_B$ is continuous on $K$. 
\end{lem}

\begin{proof}
For each $x\in K$, there exist an open neighborhood $U_x\subset S$ of $x$ and a bounded Banach disked set $B_x\subset X$ such that $f\colon U_x\to X_{B_x}$ is continuous. Choose finitely many $x_1, \ldots, x_m\in K$ such that $K\subset U_{x_1}\cup \cdots \cup U_{x_m}$ and a bounded Banach disked set $B\subset X$ such that $B_{x_1}\cup \cdots \cup B_{x_m}\subset B$. Then $f(K)\subset X_B$ and $f\colon K\to X_B$ is continuous on $K$. 
\end{proof}

For an open subset $U\subset \bbR^n$ and $a\in U$, we direct the set $U\bs\{a\}$ by
\[
x\preceq y\rarr |x-a|\ge|y-a|. 
\]
For a map $f\colon U\bs\{a\}\to X$, if the net $\{f(x)\}_{x\in U\bs\{a\}}$ converges bornologically in $X$, we denote its bornological limit by $\blim_{x\to a}f(x)$. 

\begin{lem}
Let $U\subset \bbR^n$ be an open subset, and let $f\colon U\to X$ be a bornologically continuous function. For every $a\in U$, we have $\blim_{x\to a}f(x)=f(a)$. 
\end{lem}

\begin{proof}
For $a\in U$, choose an open neighborhood $V\subset U$ of $a$ and a bounded Banach disked set $B\subset X$ such that $f(V)\subset X_B$ and $f\colon V\to X_B$ is continuous. Then $\lim_{x\to a}f(x)=f(a)$ in $X_B$. Hence the lemma follows from \cref{lem:sepXbcovborn}. 
\end{proof}

\begin{dfn}
Let $U\subset \bbR^n$ be an open subset, and let $f\colon U\to X$ be a function. We say that \emph{$f$ is bornologically differentiable on $U$} if for every $a\in U$, there exist an open neighborhood $V\subset U$ of $a$ and a bounded Banach disked set $B\subset X$ such that $f(V)\subset X_B$ and $f\colon V\to X_B$ is differentiable on $V$. 
\end{dfn}

Let $U\subset \bbR^n$ be an open subset, and let $f\colon U\to X$ be bornologically differentiable. For each $a\in U$, choose an open neighborhood $V\subset U$ of $a$ and a bounded Banach disked set $B\subset X$ such that $f\colon V\to X_B$ is differentiable. Then the derivative $f'(a)\in X_B^n$ is independent of the choice of $V$ and $B$. 

\begin{lem}
Let $U\subset \bbR^n$ be an open subset, and let $f\colon U\to X$ be a function. The following conditions are equivalent: 
\begin{clist}
\item 
$f$ is bornologically differentiable on $U$. 

\item 
For every compact set $K\subset U$, there exists a bounded Banach disked set $B\subset X$ such that $f(K)\subset X_B$ and $f\colon K\to X_B$ is differentiable on the interior $K^\circ$.  
\end{clist}
\end{lem}

\begin{proof}
(i) $\Longrightarrow$ (ii): This can be proved in the same way as \cref{lem:borncont}. 

(ii) $\Longrightarrow$ (i): For $x\in U$, choose a compact neighborhood $K\subset U$ of $x$ and a bounded Banach disked set $B\subset X$ such that $f(K)\subset X_B$ and $f\colon K\to X_B$ is differentiable on $K^\circ$. Then $K^\circ$ is an open neighborhood of $x$ such that $f\colon K^\circ\to X_B$ is differentiable. 
\end{proof}

\begin{dfn}
Let $U\subset \bbR^n$ be an open subset, and let $f\colon U\to X$ be a function. For each $i=1, \ldots, n$, we say that \emph{$f$ is partially bornologically differentiable on $U$ with respect to the $i$-th variable} if for every $a\in U$, there exists an open neighborhood $V\subset U$ of $a$ and a bounded Banach disked set $B\subset X$ such that $f(V)\subset X_B$ and $f\colon V\to X_B$ is partially differentiable on $V$ with respect to the $i$-th variable. 
\end{dfn}

Let $U\subset \bbR^n$ be an open subset, and let $f\colon U\to X$ be partially bornologically differentiable with respect to the $i$-th variable. For each $a\in U$, choose an open neighborhood $V\subset U$ of $a$ and a bounded Banach disked set $B\subset X$ such that $f\colon V\to X_B$ is partially differentiable with respect to the $i$-th variable. Then the partial derivative $\pdd_if(a)\in X_B$ is independent of the choice of $V$ and $B$. 

\begin{lem}\label{lem:pddblim}
Let $U\subset \bbR^n$ be an open subset, and let $f\colon U\to X$ be a partially bornologically differentiable function with respect to the $i$-th variable. For every $a\in U$, 
\[
\pdd_if(a)=\blim_{x\to a_i}\frac{f(a_1, \ldots, a_{i-1}, x, a_{i+1}, \ldots, a_n)-f(a)}{x-a_i}. 
\]
\end{lem}

\begin{proof}
For $a\in U$, choose an open neighborhood $V\subset U$ of $a$ and a bounded Banach disked set $B\subset X$ such that $f(V)\subset X_B$ and $f\colon V\to X_B$ is partially differentiable with respect to the $i$-th variable. Then
\[
\pdd_if(a)=\lim_{x\to a_i}\frac{f(a_1, \ldots, a_{i-1}, x, a_{i+1}, \ldots, a_n)-f(a)}{x-a_i} \quad \text{in $X_B$}. 
\]
Hence the lemma follows from \cref{lem:sepXbcovborn}. 
\end{proof}

\begin{lem}
Let $U\subset \bbR^n$ be an open subset. If a map $f\colon U\to X$ is bornologically differentiable on $U$, then $f$ is partially bornologically differentiable on $U$ with respect to each variables. Moreover, for $a\in U$, 
\[
f'(a)=(\pdd_1f(a), \ldots, \pdd_nf(a)). 
\]
\end{lem}

\begin{proof}
This follows from the corresponding results for functions with values in a Banach space. 
\end{proof}

\begin{lem}
Let $U\subset \bbR^n$ and $V\subset \bbR^m$ be open subsets. If a function $f\colon U\to V$ is differentiable on $U$ and a function $g\colon V\to X$ is bornologically differentiable on $V$, then $g\circ f\colon U\to X$ is bornologically differentiable on $U$. Moreover, for every $a\in U$, 
\[
(g\circ f)'(a)=g'(f(a))f'(a). 
\]
\end{lem}

\begin{proof}
Let $a\in U$. Choose an open neighborhood $V_0\subset V$ of $f(a)$ and a bounded Banach disked set $B\subset X$ such that $g\colon V_0\to X_B$ is differentiable. Since $f$ is continuous, there exists an open neighborhood $U_0\subset U$ of $a$ such that $f(U_0)\subset V_0$. We have $(g\circ f)(U_0)\subset X_B$ and $g\circ f\colon U_0\to X_B$ is differentiable. Hence the lemma follows. 
\end{proof}

\begin{lem}\label{lem:borndiffE}
Let $U\subset \bbR^n$ be an open subset, and let $E\subset X$ be a bornologically closed linear subspace. If a function $f\colon U\to X$ is bornologically differentiable on $U$ and $f(U)\subset E$, then $f\colon U\to E$ is bornologically differentiable. 
\end{lem}

\begin{proof}
Let $a\in U$. Choose an open neighborhood $V\subset U$ of $a$ and a bounded Banach disked set $B\subset X$ such that $f\colon V\to X_B$ is differentiable. By \cref{lem:subcomp}, the subset $B\cap E$ is a Banach disked set of $E$. By \cref{lem:subAcapE}, we have $f(V)\subset X_B\cap E=E_{B\cap E}$ and $f\colon V\to E_{B\cap E}$ is differentiable. 
\end{proof}

Let $X$ and $Y$ be complete separated convex bornological vector spaces. 

\begin{lem}\label{lem:borndiffboudlin}
Let $U\subset \bbR^n$ be an open subset, and let $\varphi\colon X\to Y$ be a bounded linear map. If a function $f\colon U\to X$ is bornologically differentiable on $U$, then $\varphi \circ f\colon U\to Y$ is bornologically differentiable on $U$. Moreover, for every $a\in U$, 
\[
(\varphi\circ f)'(a)=\bigl(\varphi(\pdd_1f(a)), \ldots, \varphi(\pdd_nf(a))\bigr). 
\]
\end{lem}

\begin{proof}
Let $a\in U$. Choose an open neighborhood $V\subset U$ and a bounded Banach disked set $B\subset X$ such that $f\colon U\to X_B$ is differentiable. By \cref{prp:mapcompdisk}, the subset $\varphi(B)$ is a bounded Banach disked set of $Y$. We have $(\varphi\circ f)(U)\subset Y_{\varphi(B)}$. Moreover, since $\varphi\colon X_B\to Y_{\varphi(B)}$ is continuous, the function $\varphi\circ f\colon U\to Y_{\varphi(B)}$ is differentiable. Hence the lemma follows. 
\end{proof} 

\begin{rmk}
Let $U\subset \bbR^n$ be an open subset, and let $X$ be a complete Hausdorff locally convex vector space. Every bornologically differentiable function $f\colon U\to X$ is differentiable in the usual sense, since for any nonempty bounded disked set $B\subset X$, the inclusion $X_B\inj X$ is continuous. 
\end{rmk}

\begin{exm}
Let $U\subset \bbR^n$ be an open subset, and let $X$ be a Banach space. Every differentiable function $f\colon U\to X$ in the usual sense is bornologically differentiable. To prove this, let $B=\{x\in X\mid \|x\|<1\}$ be the unit open ball, which is a nonempty bounded disked set. Then we have $X_B=X$ and $p_B(x)=\|x\|$ for every $x\in X_B$. Hence $B$ is a Banach disked set of $X$, and $f\colon U\to X_B$ is differentiable. Thus $f$ is bornologically differentiable. 
\end{exm}

\begin{prp}\label{lem:Frebordiff}
Let $U\subset \bbR^n$ be an open subset, and let $X$ be a Fr\'{e}chet space. Every function $f\colon U\to X$ of class $C^1$ is bornologically differentiable. 
\end{prp}

\begin{proof}
Fix a countable basis $\{p_j\}_{j=1}^\infty$ of continuous seminorms of $X$. Let $a\in U$. Choose a bounded open disk $D$ around $a$ such that $\ol{D}\subset U$. We claim that there exists a monotonically increasing function $\omega\colon \bbR_{>0}\to \bbR_{>0}$ such that $\lim_{t\to 0}\omega(t)=0$ and the set
\[
S\ceq\Bigl\{\frac{\pdd_if(x)-\pdd_if(y)}{\omega(|x-y|)}\Bigm| x, y\in \ol{D}, x\neq y, i=1, \ldots, n\Bigr\}
\]
is bounded in $X$. For each $j\in \bbZ_{>0}$, define a function $\rho_j\colon \bbR_{>0}\to \bbR$ by 
\[
\rho_j(t)\ceq \max_{i=1, \ldots, n}\sup_{\substack{x, y\in \ol{D}\\ |x-y|\le t}}p_j(\pdd_if(x)-\pdd_if(y)). 
\]
Since each $\pdd_if$ is uniformly continuous on $\ol{D}$, we have $\lim_{t\to 0}\rho_j(t)=0$. There exists a strictly decreasing sequence $\{t_k\}_{k=1}^\infty\subset \bbR_{>0}$ such that $\lim_{k\to \infty}t_k=0$ and $\rho_j(t_k)\le 2^{-2k}$ for all $1\le j\le k$. Define a monotonically increasing function $\omega\colon \bbR_{>0}\to \bbR_{>0}$ by 
\[
\omega(t)\ceq
\begin{cases}
1 & (t>t_1) \\
2^{-k} & (t_{k+1}<t\le t_k). 
\end{cases}
\]
Clearly, $\lim_{t\to 0}\omega(t)=0$. To prove that $S$ is bounded, it suffices to show that 
\[
\sup_{\substack{x, y\in \ol{D}\\x\neq y}}p_j\Bigl(\frac{\pdd_if(x)-\pdd_if(y)}{\omega(|x-y|)}\Bigr)<\infty
\]
for each $i=1, \ldots, n$ and $j\in \bbZ_{>0}$. Let $x, y\in \ol{D}$ with $x\neq y$. If $|x-y|\le t_j$, then take $k\ge j$ such that $t_{k+1}<|x-y|\le t_k$. We have 
\[
p_j(\pdd_if(x)-\pdd_if(y))\le \rho_j(t_k)\le 2^{-2k}
\]
and $\omega(|x-y|)=2^{-k}$. Hence 
\[
p_j\Bigl(\frac{\pdd_if(x)-\pdd_if(y)}{\omega(|x-y|)}\Bigr)\le 2^{-k}\le 1. 
\]
If $|x-y|>t_j$, then $\omega(|x-y|)\ge\omega(t_j)=2^{-j}$. Hence 
\[
p_j\Bigl(\frac{\pdd_if(x)-\pdd_if(y)}{\omega(|x-y|)}\Bigr)\le 2^j\sup_{x, y\in \ol{D}}p_j(\pdd_if(x)-\pdd_if(y))<\infty. 
\]
Thus the claim follows. 

Since $\ol{D}$ is compact, the sets $f(\ol{D})$ and $\pdd_if(\ol{D})$ are compact. Thus 
\[
A\ceq f(\ol{D})\cup \bigcup_{i=1}^n(\pdd_if)(\ol{D})\cup  S
\]
is bounded in $X$. Define $B\ceq \ol{\disk A}$. Then $B$ is a closed bounded disked set, and hence is a bounded Banach disked set. We have $f(D)\subset X_B$. It suffices to show that $f\colon D\to X_B$ is differentiable. Let $b\in D$ be arbitrary. The fundamental theorem of calculus for functions with values in $X$ gives 
\[
f(x)-f(b)=\sum_{i=1}^n(x_i-b_i)\int_0^1\pdd_if(b+t(x-b))\, dt \quad (x\in D). 
\]
For each $i=1, \ldots, n$, we have $\pdd_if(\ol{D})\subset X_B$ and 
\[
p_B(\pdd_if(x)-\pdd_if(y))\le \omega(|x-y|) \quad (x, y\in \ol{D}, x\neq y). 
\]
In particular, the function $\pdd_if\colon D\to X_B$ is continuous. Hence the function $[0, 1]\to X_B$, $t\mapsto \pdd_if(b+t(x-b))$ is Riemann integrable in $X_B$. 
Since the inclusion $X_B\inj X$ is continuous, the image in $X$ of the $X_B$-valued integral agrees with corresponding $X$-valued integral. 
Thus we obtain the following identity in $X_B$: 
\[
f(x)-f(b)-\sum_{i=1}^n(x_i-b_i)\pdd_if(b)=
\sum_{i=1}^n(x_i-b_i)\int_0^1\bigl(\pdd_if(b+t(x-b))-\pdd_if(b)\bigr)\,dt \quad (x\in D). 
\]
Therefore, for $x\in D$ with $x\neq b$, 
\begin{align*}
p_B\Bigl(\frac{f(x)-f(b)-\sum_{i=1}^n(x_i-b_i)\pdd_if(b)}{|x-b|}\Bigr)
&\le 
|x-b|^{-1}\sum_{i=1}^n|x_i-b_i|\sup_{0\le t\le 1}p_B\bigl(\pdd_if(b+t(x-b))-\pdd_if(b)\bigr)\\
&\le 
n\omega(|x-b|). 
\end{align*}
This proves that $f\colon D\to X_B$ is differentiable at $b$. 
\end{proof}

Let $I=[a_1, b_1]\times \cdots \times [a_n, b_n]$ ($a_i, b_i\in \bbR$, $a_i<b_i$) be a bounded closed interval in $\bbR^n$. Write
\[
P(I)\ceq \prod_{i=1}^nP([a_i, b_i]), \quad 
P([a_i, b_i])\ceq \bigl\{(x_\alpha)_{\alpha=0}^l\mid l\in \bbZ_{>0}, a_i=x_0<x_1<\cdots<x_l=b_i\bigr\},
\]
and call each element $\Delta\in P(I)$ a partition of $I$. We also call 
\[
v(I)\ceq \prod_{i=1}^n(b_i-a_i), \quad d(I)\ceq |b-a|=\Bigl(\sum_{i=1}^n(b_i-a_i)^2\Bigr)^{1/2}
\]
the volume and the diameter of $I$, respectively. 

For a partition $\Delta=\{\Delta^i\}_{i=1}^n\in P(I)$, where $\Delta^i=\{x^i_{\alpha_i}\}_{\alpha_i=0}^{l_i}$, we use the following notation and terminology: 

\begin{itemize}
\item 
Write $A(\Delta)\ceq \prod_{i=1}^n\{1, \ldots, l_i\}$. 

\item 
For $\alpha=\{\alpha_i\}_{i=1}^n\in A(\Delta)$, define $I_\alpha(\Delta)\ceq \prod_{i=1}^n[x^i_{\alpha_i-1}, x^i_{\alpha_i}]$, which is a bounded closed interval in $\bbR^n$. 

\item 
Define $|\Delta|\ceq \max_{\alpha\in A(\Delta)}d(I_\alpha(\Delta))$, and call it the mesh of $\Delta$. 

\item
We call each element $\xi=\{\xi_\alpha\}_{\alpha\in A(\Delta)}\in \prod_{\alpha\in A(\Delta)}I_\alpha(\Delta)$ a choice of representative points of $\Delta$. 
\end{itemize}

For a bounded closed interval $I\subset \bbR^n$, a pair $(\Delta, \xi)$ consisting of a partition $\Delta\in P(I)$ and a choice $\xi$ of representative points of $\Delta$ is called a \emph{tagged partition} of $I$. We denote by $\wt{P}(I)$ the set of all tagged partitions of $I$. 

\begin{dfn}
Let $I\subset \bbR^n$ be a bounded closed interval, and let $f\colon I\to X$ be a function. For a tagged partition $(\Delta, \xi)\in \wt{P}(I)$, define 
\[
s_f(\Delta, \xi)\ceq\sum_{\alpha\in A(\Delta)}v(I_\alpha(\Delta))f(\xi_\alpha),
\]
which is called a Riemann sum of $f$. 
\end{dfn}

Define a directed order $\preceq$ on $\wt{P}(I)$ by
\[
(\Delta_1, \xi_1)\preceq (\Delta_2, \xi_2) \rarr |\Delta_1|\ge|\Delta_2|. 
\]

\begin{lem}\label{lem:bornint}
Let $I\subset \bbR^n$ be a bounded closed interval. For a function $f\colon I\to X$ and $s\in X$, the following conditions are equivalent: 
\begin{clist}
\item 
The net $(\wt{P}(I), \preceq)\to X$, $(\Delta, \xi)\mapsto s_f(\Delta, \xi)$ converges bornologically to $s$. 

\item 
There exists a bounded Banach disked set $B\subset X$ such that $f(I)\subset X_B$, $s\in X_B$, and the net $(\wt{P}(I), \preceq)\to X_B$, $(\Delta, \xi)\mapsto s_f(\Delta, \xi)$ converges to $s$ in $X_B$. 
\end{clist}
\end{lem}

\begin{proof}
(i) $\Longrightarrow$ (ii): Set $\Lambda\ceq (\wt{P}, \preceq)$. By \cref{lem:borcovXB}, there exists a nonempty bounded disked set $A\subset X$ and $\lambda_0=(\Delta, \xi)\in \Lambda$ such that $\{s_f(\lambda)\}_{\lambda\in \Lambda(\lambda_0)}\subset X_A$, $s\in X_A$, and $\{s_f(\lambda)\}_{\lambda\in \Lambda(\lambda_0)}$ converges to $s$ in $X_A$. 
Choose a bounded Banach disked set $B$ such that $A\cup \{f(\xi_\alpha)\mid \alpha\in A(\Delta)\}\subset B$. It suffices to show that $f(I)\subset X_B$. Let $x\in I$ be arbitrary. Then there exists $\alpha\in A(\Delta)$ such that $x\in I_\alpha(\Delta)$. For each $\beta\in A(\Delta)$, define 
\[
\eta_\beta\ceq
\begin{cases}
x & (\beta=\alpha) \\
\xi_\beta & (\beta\neq \alpha). 
\end{cases}
\]
Then $\eta\ceq \{\eta_\beta\}_{\beta\in A(\Delta)}$ is a choice of representative points of $\Delta$. Thus $s_f(\Delta, \eta)\in X_A\subset X_B$, and hence
\[
v(I_\alpha(\Delta))f(x)=s_f(\Delta, \eta)-\sum_{\substack{\beta\in A(\Delta)\\ \beta\neq \alpha}}v(I_\beta(\Delta))f(\xi_\beta)\in X_B. 
\]
Since $v(I_\alpha(\Delta))\neq 0$, we have $f(x)\in X_B$. 

(ii) $\Longrightarrow$ (i): This follows from \cref{lem:sepXbcovborn}. 
\end{proof}

\begin{dfn}
Let $I\subset \bbR^n$ be a bounded closed interval, and let $f\colon I\to X$ be a function. We say that \emph{$f$ is bornologically Riemann integrable on $I$} if the net 
\[
(\wt{P}(I), \preceq)\to X, \quad (\Delta, \xi)\mapsto s_f(\Delta, \xi)
\]
converges bornologically. In this case, we denote the bornological limit by 
\[
\int_If \quad \text{or} \quad \int_If(x)\,dx
\]
\end{dfn}

By \cref{lem:bornint}, a function $f\colon I\to X$ is bornologically Riemann integrable on $I$ if and only if there exists a Banach disked set $B\subset X$ such that $f\colon I\to X_B$ is Riemann integrable on $I$ in the usual sense. 

\begin{prp}\label{prp:borncontint}
Let $I\subset \bbR^n$ be a bounded closed interval. Every bornologically continuous function $f\colon I\to X$ is bornologically Riemann integrable. 
\end{prp}

\begin{proof}
Since $I$ is compact, \cref{lem:borncont} implies that there exists a bounded Banach disked set $B\subset X$ such that $f(I)\subset X_B$ and $f\colon I\to X_B$ is continuous. Thus $f\colon I\to X_B$ is Riemann integrable, and hence $f\colon I\to X$ is bornologically Riemann integrable. 
\end{proof}

\subsection{Bornologically holomorphic functions}
\label{ss:bornholo}

Fix a complete separated convex bornological vector space $X$ over $\bbC$. 

\begin{dfn}
Let $U\subset \bbC^n$ be an open subset, and let $f\colon U\to X$ be a function. We say that \emph{$f$ is bornologically holomorphic on $U$} if for every $a\in U$, there exists an open neighborhood $V\subset U$ of $a$ and a bounded Banach disked set $B\subset X$ such that $f(V)\subset X_B$ and $f\colon V\to X_B$ is holomorphic on $V$. 
\end{dfn}

Let $U\subset \bbC^n$ be an open subset, and let $f\colon U\to X$ be bornologically holomorphic. For each $a\in U$, choose an open neighborhood $V\subset U$ of $a$ and a bounded Banach disked set $B\subset X$ such that $f\colon V\to X_B$ is holomorphic. Then the complex derivative $f'(a)\in X_B^n$ is independent of the choice of $V$ and $B$. 

\begin{lem}\label{lem:bornholo}
Let $U\subset \bbC^n$ be an open subset, and let $f\colon U\to X$ be a function. The following conditions are equivalent: 
\begin{clist}
\item 
$f$ is bornologically holomorphic on $U$. 

\item 
For every compact set $K\subset U$, there exists a bounded Banach disked set $B\subset X$ such that $f(K)\subset X_B$ and $f\colon K\to X_B$ is holomorphic on the interior $K^\circ$. 
\end{clist}
\end{lem}

\begin{proof}
(i) $\Longrightarrow$ (ii): For each $x\in K$, there exist an open neighborhood $V_x\subset U$ of $x$ and a bounded Banach disked set $B\subset X$ such that $f\colon V_x\to X_{B_x}$ is holomorphic. Choose finitely many $x_1, \ldots, x_m\in K$ such that $K\subset V_{x_1}\cup \cdots \cup V_{x_m}$ and a bounded Banach disked set $B\subset X$ such that $B_{x_1}\cup \cdots \cup B_{x_m}\subset B$. Then $f(K)\subset X_B$ and $f\colon K\to X_B$ is holomorphic on $K^\circ$. 

(ii) $\Longrightarrow$ (i): For $x\in U$, choose a compact neighborhood $K\subset U$ of $x$ and a bounded Banach disked set $B\subset X$ such that $f\colon K\to X_B$ is holomorphic on $K^\circ$. Then $K^\circ$ is an open neighborhood of $x$ such that $f\colon K^\circ \to X_B$ is holomorphic. 
\end{proof}

We identify $\bbR^{2n}$ with $\bbC^n$ by
\[
\bbR^{2n}\sto \bbC^n, \quad (x_1, \ldots, x_{2n})\mapsto (x_1+\sqrt{-1}x_2, \ldots, x_{2n-1}+\sqrt{-1}x_{2n}). 
\]

\begin{lem}\label{lem:bornCR}
Let $U\subset \bbC^n$ be an open subset. For a function $f\colon U\to X$, the following conditions are equivalent: 
\begin{clist}
\item 
$f$ is bornologically holomorphic on $U$. 

\item 
$f$ is bornologically differentiable on $U$ as a function of $2n$ variables, and $\pdd_{2j-1}f+\sqrt{-1}\pdd_{2j}f=0$ for all $j=1, \ldots, n$. 
\end{clist}
Moreover, in this case, for every $a\in U$, 
\[
f'(a)=\frac{1}{2}\bigl(\pdd_1f(a)-\sqrt{-1}\pdd_2f(a), \ldots, \pdd_{2n-1}f(a)-\sqrt{-1}\pdd_{2n}f(a)\bigr). 
\]
\end{lem}

\begin{proof}
This follows from the corresponding result for functions with values in a Banach space. 
\end{proof}

Let $U\subset \bbC^n$ be an open subset, and let $f\colon U\to X$ be a bornologically differentiable function. For $j=1, \ldots, n$ and $a\in U$, define
\begin{align*}
&\pd{f}{z_i}(a)=\pd{}{z_i}f(z)\Bigm|_{z=a}\ceq \frac{1}{2}\bigl(\pdd_{2j-1}f(a)-\sqrt{-1}\pdd_{2j}f(a)\bigr), \\
&\pd{f}{\ol{z}_i}(a)=\pd{}{\ol{z}_i}f(z)\Bigm|_{z=a}\ceq \frac{1}{2}\bigl(\pdd_{2j-1}f(a)+\sqrt{-1}\pdd_{2j}f(a)\bigr). 
\end{align*}
By \cref{lem:bornCR}, the function $f$ is bornologically holomorphic on $U$ if and only if 
\[
\pd{f}{\ol{z}_i}(a)=0. 
\]
for every $j=1, \ldots, n$ and every $a\in U$. Moreover, if $f$ is bornologically holomorphic on $U$, then for $a\in U$, 
\[
f'(a)=\Bigl(\pd{f}{z_1}(a), \ldots, \pd{f}{z_n}(a)\Bigr). 
\]

\begin{lem}\label{lem:compbornholo}
Let $U\subset \bbC^n$ and $V\subset \bbC^m$ be open subsets. If $\varphi\colon U\to V$ is a holomorphic function and $f\colon V\to X$ is a bornologically holomorphic function, then $f\circ \varphi\colon U\to X$ is bornologically holomorphic on $U$. Moreover, for every $a\in U$, 
\[
(f\circ \varphi)'(a)=f'(\varphi(a))\varphi'(a). 
\]
\end{lem}

\begin{proof}
Let $a\in U$. Choose an open neighborhood $V_0\subset V$ of $\varphi(a)$ and a bounded Banach disked set $B\subset X$ such that $f\colon V_0\to X_B$ is holomorphic. Since $\varphi$ is continuous, there exists an open neighborhood $U_0\subset U$ of $a$ such that $\varphi(U_0)\subset V_0$. We have $(f\circ \varphi)(U_0)\subset X_B$ and $f\circ \varphi\colon U_0\to X_B$ is holomorphic. Hence the lemma follows. 
\end{proof}

\begin{lem}\label{lem:idensev}
Let $U\subset \bbC^n$ be a domain, and let $f, g\colon U\to X$ be bornologically holomorphic functions. If there exist $a\in U$ and a neighborhood $N\subset U$ of $a$ such that $f|_N=g|_N$, then $f=g$. 
\end{lem}

\begin{proof}
Set $h\ceq f-g$ and 
\[
A\ceq\{a\in U\mid \text{there exists a neighborhood $N\subset U$ of $a$ such that $h|_N=0$}\}. 
\]
Then $A$ is a nonempty open subset of $U$. We claim that $A$ is closed in $U$. Let $a\in \ol{A}\cap U$. Since $h$ is bornologically holomorphic, there exist an open ball $D\subset U$ around $a$ and a bounded Banach disked set $B\subset X$ such that $h(D)\subset X_B$ and $h\colon D\to X_B$ is holomorphic. Choose $b\in A\cap D$. Then there exists a neighborhood $N\subset U$ of $b$ such that $h|_N=0$. Since $N\cap D$ is a neighborhood of $b$ and $h|_{N\cap D}=0$, the identity theorem for functions with values in a Banach space implies that $h|_D=0$. Hence $a\in A$. Therefore, the connectivity of $U$ implies $U=A$. 
\end{proof}

\subsection{Bornological convergence of power series}
\label{ss:bornser}

Fix a complete separated convex bornological vector space $X$ over $\bbC$. For a sequence $\{x_n\}_{n\in \bbN}\subset X$, the sequence
\begin{equation}\label{eq:oneseries}
\Bigl\{\sum_{n=0}^Nx_n\Bigr\}_{N\in \bbN}
\end{equation}
is called the \emph{series} associated with $\{x_n\}_{n\in \bbN}$. We say that the series \eqref{eq:oneseries} converges bornologically if it converges bornologically as a sequence. In this case, we denote its bornological limit by
\[
\sum_{n=0}^\infty x_n. 
\]
By abuse of notation, we also use the symbol $\sum_{n=0}^\infty x_n$ to denote the series \eqref{eq:oneseries} itself. Note that the series $\sum_{n=0}^\infty x_n$ converges bornologically if and only if there exists a bounded Banach disked set $B\subset X$ such that $x_n\in X_B$ for every $n\in \bbN$ and the series $\sum_{n=0}^\infty x_n$ converges in $X_B$. 

\begin{lem}
Let $\{x_n\}_{n\in \bbN}\subset X$ be a sequence. If the series $\sum_{n=0}^\infty x_n$ converges bornologically, then $\{x_n\mid n\in \bbN\}$ is bounded in $X$. 
\end{lem}

\begin{proof}
Choose a bounded Banach disked set $B\subset X$ such that the series $\sum_{n=0}^\infty x_n$ converges in $X_B$. Then $\{x_n\mid n\in \bbN\}$ is bounded in $X_B$. Hence, by \cref{lem:XBincbound}, the set $\{x_n\mid n\in \bbN\}$ is bounded in $X$. 
\end{proof}

\begin{dfn}
Let $\{x_n\}_{n\in \bbN}\subset X$ be a sequence. We say that the series $\sum_{n=0}^\infty x_n$ \emph{converges bornologically absolutely} if there exists a bounded Banach disked set $B\subset X$ such that $x_n\in X_B$ for every $n\in \bbN$ and the series $\sum_{n=0}^\infty x_n$ converges absolutely in $X_B$. 
\end{dfn}

By definition, if the series $\sum_{n=0}^\infty x_n$ converges bornologically absolutely, then it converges bornologically. 

\begin{dfn}
Let $S$ be a set, let $\{f_n\colon S\to X\}_{n\in \bbN}$ be a sequence of maps, and let $f\colon S\to X$ be a map. 
\begin{enumerate}
\item 
We say that $\{f_n\}_{n\in \bbN}$ \emph{converges bornologically pointwise} to $f$ if $\blim_{n\to \infty}f_n(x)=f(x)$ for every $x\in S$. 

\item 
We say that $\{f_n\colon S\to X\}_{n\in \bbN}$ \emph{converges bornologically uniformly} to $f$ if there exists a bounded Banach disked set $B\subset X$ such that $f_n(S), f(S)\subset X_B$ ($n\in \bbN$) and $\{f_n\colon S\to X_B\}_{n\in \bbN}$ converges uniformly to $f\colon S\to X_B$. 
\end{enumerate}
\end{dfn}

If a sequence $\{f_n\colon S\to X\}_{n\in \bbN}$ of maps converges bornologically pointwise to a map $f\colon S\to X$, we denote $\blim_{n\to \infty}f_n\ceq f$. By definition, if $\{f_n\}_{n\in \bbN}$ converges bornologically uniformly, then it converges bornologically pointwise. 

Let $S$ be a set, and let $\{f_n\colon S\to X\}_{n\in \bbN}$ be a sequence of maps. We call the sequence
\begin{equation}\label{eq:sermap}
\Bigl\{\sum_{n=0}^N f_n\Bigr\}_{N\in \bbN}
\end{equation}
the series associated with $\{f_n\}_{n\in \bbN}$. We say that the series \eqref{eq:sermap} converges bornologically pointwise (resp.\ bornologically uniformly) if it converges bornologically pointwise (resp.\ bornologically uniformly) as a sequence of maps. If the series \eqref{eq:sermap} converges pointwise, we denote its bornological limit by
\[
\sum_{n=0}^\infty f_n\ceq
\blim_{N\to \infty}\sum_{n=0}^N f_n. 
\]
By abuse of notation, we also use the symbol $\sum_{n=0}^\infty f_n$ to denote the series \eqref{eq:sermap} itself.

\begin{prp}\label{prp:onepowserabs}
Let $\{x_n\}_{n\in \bbN}\subset X$ be a sequence, and let $R\in \bbR_{>0}$. If $\{R^nx_n\mid n\in \bbN\}$ is bounded in $X$, then there exists a bounded Banach disked set $B\subset X$ such that $x_n\in X_B$ ($n\in \bbN$), and for every $a\in \bbC$, the series with values in $X_B$
\[
\sum_{n=0}^\infty (z-a)^nx_n \quad (z\in D_R(a))
\]
converges absolutely and uniformly on compact sets. 
\end{prp}

\begin{proof}
Choose a bounded Banach disked set $B\subset X$ such that $\{R^nx_n\mid n\in \bbN\}\subset B$. Then $x_n\in X_B$ for every $n\in \bbN$. Since $\{R^nx_n\mid n\in \bbN\}$ is also bounded in $X_B$, the series $\sum_{n=0}^\infty (z-a)^nx_n$ converges absolutely and uniformly on compact sets. 
\end{proof}

\begin{prp}\label{prp:onepowserabs2}
Let $\{x_n\}_{n\in \bbN}\subset X$ be a sequence, let $0<R\le \infty$, and let $a\in \bbC$. If the series 
\[
\sum_{n=0}^\infty (z-a)^nx_n \quad (z\in D_R(a))
\]
converges bornologically, then it converges bornologically absolutely and bornologically uniformly on compact sets. 
\end{prp}

\begin{proof}
Let $z\in D_R(a)$. Choose $r\in \bbR_{>0}$ such that $|z-a|<r<R$. Since the series converges bornologically at $r+a$, the set $\{r^nx_n\mid n\in \bbN\}$ is bounded in $X$. Hence, by \cref{prp:onepowserabs}, the series $\sum_{n=0}^\infty (z-a)^nx_n$ converges bornologically absolutely. 

Let $K\subset D_R(a)$ be a compact set. Choose $r\in \bbR_{>0}$ such that $\max_{z\in K}|z-a|<r<R$. Then $K\subset \ol{D}_r(a)$ and $r<R$. Since the series converges bornologically at $r+a$, the set $\{r^n x_n\mid n\in \bbN\}$ is bounded in $X$. Hence, by \cref{prp:onepowserabs}, the series $\sum_{n=0}^\infty (z-a)^n x_n$ converges bornologically uniformly on $K$. 
\end{proof}

\begin{prp}\label{prp:radconv}
Let $\{x_n\}_{n\in \bbN}\subset X$ be a sequence, and set 
\[
R\ceq \sup\Bigl\{|z|\Bigm| z\in \bbC,\, \sum_{n=0}^\infty z^n x_n \text{ converges bornologically absolutely}\Bigr\}\in[0, \infty]. 
\]
For $a\in \bbC$, the following hold: 
\begin{enumerate}
\item
For every $z\in D_R(a)$, the series $\sum_{n=0}^\infty (z-a)^nx_n$ converges bornologically absolutely. 

\item 
For every $z\in \bbC$ such that $|z-a|>R$, the series $\sum_{n=0}^\infty (z-a)^nx_n$ does not converge bornologically. 
\end{enumerate}
\end{prp}

\begin{proof}
(1) For $z\in D_R(a)$, there exists $w\in \bbC$ such that $\sum_{n=0}^\infty w^nx_n$ converges bornologically absolutely and $|z-a|<|w|$. Choose a bounded Banach disked set $B\subset X$ such that $x_n\in X_B$ for all $n\in \bbN$ and the series $\sum_{n=0}^\infty w^nx_n$ converges absolutely in $X_B$. Since $p_B((z-a)^nx_n)\le p_B(w^nx_n)$ for all $n\in \bbN$ and $\sum_{n=0}^\infty p_B(w^nx_n)$ converges, the series $\sum_{n=0}^\infty (z-a)^nx_n$ converges absolutely in $X_B$. Hence $\sum_{n=0}^\infty (z-a)^nx_n$ converges bornologically absolutely. 

(2) Suppose that there exists $z\in \bbC$ such that $|z-a|>R$ and the series $\sum_{n=0}^\infty (z-a)^nx_n$ converges bornologically. Then $\{|z-a|^nx_n\mid n\in \bbN\}$ is bounded in $X$. Hence, by \cref{prp:onepowserabs}, the series 
\[
\sum_{n=0}^\infty w^nx_n \quad (w\in D_{|z-a|})
\]
converges bornologically absolutely. Since there exists $w\in \bbC$ such that $|z-a|>|w|>R$, this is a contradiction. 
\end{proof}

For a sequence $\{x_n\}_{n\in \bbN}\subset X$, the element $R\in [0, \infty]$ defined in \cref{prp:radconv} is called the \emph{radius of convergence} of the series $\sum_{n=0}^\infty z^nx_n$	. 

\begin{prp}\label{prp:Taydiff}
Let $\{x_n\}_{n\in \bbN}\subset X$ be a sequence, and let $R$ be the radius of convergence of $\sum_{n=0}^\infty z^nx_n$. For $a\in \bbC$, the series $\sum_{n=0}^\infty (z-a)^nx_n$ is bornologically holomorphic on $D_R(a)$ and 
\[
\pd{}{z}\sum_{n=0}^\infty (z-a)^nx_n\Bigm|_{z=w}=
\sum_{n=1}^\infty n(w-a)^{n-1}x_n \quad (w\in D_R(a)). 
\]
\end{prp}

\begin{proof}
Let $w\in D_R(a)$. Choose $r\in \bbR_{>0}$ such that $|w-a|<r<R$. By \cref{prp:onepowserabs2}, there exists a bounded Banach disked set $B\subset X$ such that $x_n\in X_B$ for all $n\in \bbN$ and $\sum_{n=0}^\infty (z-a)^nx_n$ converges uniformly on $\ol{D}_r(a)$ as a series in $X_B$. Hence
\[
D_r(a)\to X_B, \quad z\mapsto \sum_{n=0}^\infty (z-a)^nx_n
\]
is holomorphic and 
\[
\pd{}{z}\sum_{n=0}^\infty (z-a)^nx_n\Bigm|_{z=w}=
\sum_{n=1}^\infty n(w-a)^{n-1}x_n \quad \text{ in $X_B$}. 
\]
This proves the proposition. 
\end{proof}

\begin{prp}\label{prp:Lauuni}
Let $\{x_n\}_{n=1}^\infty\subset X$, let $a\in \bbC$, and let $0<R\le \infty$. If the series 
\[
\sum_{n=1}^\infty (z-a)^{-n}x_n \quad (z\in D_R^\times(a))
\]
converges bornologically, then it converges bornologically absolutely and bornologically uniformly on compact sets. 
\end{prp}

\begin{proof}
For $w\in \bbC$ with $|w|>R^{-1}$, since $w^{-1}+a\in D_R^\times(a)$, the series $\sum_{n=1}^\infty w^nx_n$ converges bornologically. Hence the radius of convergence of $\sum_{n=1}^\infty w^nx_n$ is infinite. Thus, by \cref{prp:onepowserabs2}, the series
\[
\sum_{n=1}^\infty w^nx_n \quad (w\in \bbC)
\]
converges bornologically absolutely and bornologically uniformly on compact sets. In particular, for every $z\in D_R^\times (a)$, the series $\sum_{n=1}^\infty (z-a)^{-n}x_n$ converges bornologically absolutely. 

Let $K\subset D_R^\times(a)$ be a compact set. Set $r\ceq \min\{|z-a|\mid z\in K\}>0$. Then, for every $z\in K$, we have $(z-a)^{-1}\in \ol{D}_{1/r}$. Since the series $\sum_{n=1}^\infty w^nx_n$ converges bornologically uniformly on $\ol{D}_{1/r}$, there exists a bounded Banach disked set $B\subset X$ such that $x_n\in X_B$ for every $n\in \bbN$ and the series $\sum_{n=1}^\infty w^nx_n$ converges uniformly on $\ol{D}_{1/r}$ as a series in $X_B$. Hence $\sum_{n=1}^\infty (z-a)^{-n}x_n$ converges uniformly on $K$ as a series in $X_B$. Thus it converges bornologically uniformly on $K$. 
\end{proof}

\begin{prp}\label{prp:Laudiff}
Let $\{x_n\}_{n=1}^\infty\subset X$, let $a\in \bbC$, and let $0<R\le \infty$. Suppose that the series 
\[
\sum_{n=1}^\infty (z-a)^{-n}x_n \quad (z\in D_R^\times(a))
\]
converges bornologically. Then the series $\sum_{n=1}^\infty (z-a)^{-n}x_n$ is bornologically holomorphic on $D_R^\times(a)$ and 
\[
\pd{}{z}\sum_{n=1}^\infty (z-a)^{-n}x_n\Bigm|_{z=w}=
\sum_{n=1}^\infty -n(w-a)^{-n-1}x_n \quad (w\in D^\times_R(a)). 
\]
\end{prp}

\begin{proof}
Let $w\in D_R^\times(a)$. Choose $r, s\in \bbR_{>0}$ such that $r<|w-a|<s<R$. By \cref{prp:Lauuni}, there exists a bounded Banach disked set $B\subset X$ such that $x_n\in X_B$ for all $n\in \bbN$ and $\sum_{n=1}^\infty (z-a)^{-n}x_n$ converges uniformly on $\ol{A}_{r, s}(a)$ as a series in $X_B$. Hence
\[
A_{r, s}(a)\to X_B, \quad z\mapsto \sum_{n=1}^\infty (z-a)^{-n}x_n
\]
is holomorphic and 
\[
\pd{}{z}\sum_{n=1}^\infty (z-a)^{-n}x_n\Bigm|_{z=w}=
\sum_{n=1}^\infty -n(w-a)^{-n-1}x_n \quad \text{ in }X_B. 
\]
This proves the proposition. 
\end{proof}

\subsection{Bornological complex integration}
\label{ss:borncompint}

Fix a complete separated convex bornological vector space $X$ over $\bbC$. 
We use the following notation and terminology for curves in the complex plane. 

\begin{itemize}
\item 
A \emph{parametrized piecewise $C^1$-curve} is a continuous function $\gamma\colon [a, b]\to \bbC$, where $[a, b]\subset \bbR$ is a bounded closed interval, satisfying the following condition: there exists a partition $a=t_0<t_1<\cdots <t_n=b$ such that the restriction $\gamma\colon [t_{i-1}, t_i]\to \bbC$ is a function of class $C^1$ for all $i=1, \ldots, n$. 

\item 
For parametrized piecewise $C^1$-curves $\gamma_1\colon[a_1, b_1]\to \bbC$ and $\gamma_2\colon[a_2, b_2]\to \bbC$, a \emph{reparametrization from $\gamma_1$ to $\gamma_2$} is a continuous function $\varphi\colon [a_2, b_2]\to [a_1, b_1]$ satisfying (i) $\varphi(a_2)=a_1$ and $\varphi(b_2)=b_1$, (ii) $\gamma_2=\gamma_1\circ \varphi$, (iii) there exists a partition $a_2=t_0<t_1<\cdots <t_n=b_2$ such that $\varphi\colon[t_{i-1}, t_i]\to \bbR$ is a function of class $C^1$ and $\varphi'(t)>0$ ($t\in [t_{i-1}, t_i]$) for all $i=1, \ldots, n$. Note that if $\varphi$ is a reparametrization from $\gamma_1$ to $\gamma_2$, then $\varphi$ is a bijection, and $\varphi^{-1}$ is a reparametrization from $\gamma_2$ to $\gamma_1$. We denote $\gamma_1\sim \gamma_2$ if there exists a reparametrization from $\gamma_1$ to $\gamma_2$. This relation $\sim$ is an equivalence relation on the set of all parametrized $C^1$-curves in $\bbC$. 

\item
A \emph{piecewise $C^1$-curve} is an equivalence class $[\gamma]$ of a parametrized piecewise $C^1$-curve $\gamma\colon[a, b]\to \bbC$ under the equivalence relation $\sim$. For a piecewise $C^1$-curve $C$ in $\bbC$, a parametrized piecewise $C^1$-curve $\gamma\colon[a, b]\to \bbC$ such that $C=[\gamma]$ is called a \emph{parametrization} of $C$. Note that every piecewise $C^1$-curve admits a parametrization whose domain is $[0, 1]$. 

\item 
For a piecewise $C^1$-curve $C$, the \emph{image of $C$} is the set $\gamma([a, b])$, where $\gamma\colon [a, b]\to \bbC$ is a parametrization of $C$. Note that the definition of the image of $C$ is independent of the choice of parametrization. We abusively denote by $C$ the image of $C$. 

\item
For a piecewise $C^1$-curve $C$, the \emph{initial point} (resp. the \emph{final point}) of $C$ is $\gamma(a)$ (resp. $\gamma(b)$), where $\gamma\colon [a, b]\to \bbC$ is a parametrization of $C$. Note that the definitions of the initial and final points are independent of the choice of parametrization.  

\item
A piecewise $C^1$-curve is called \emph{simple} if it admits a parametrization $\gamma\colon [a, b]\to \bbC$ that is injective on $[a, b)$. A piecewise $C^1$-curve is called \emph{closed} if its initial and final points coincide. For a simple closed piecewise $C^1$-curve $C$, we denote its interior by $\Omega(C)$.
\end{itemize}

\begin{lem}
Let $\gamma\colon [a, b]\to \bbC$ be a parametrized piecewise $C^1$-curve, let $a=t_0<t_1<\cdots<t_n=b$ be a partition such that $\gamma\colon[t_{i-1}, t_i]\to \bbC$ is a function of class $C^1$ for all $i=1, \ldots, n$, and let $f\colon \gamma([a, b])\to X$ be a bornologically continuous function. For each $i=1, \ldots, n$, the function 
\[
[t_{i-1}, t_i]\to X, \quad t\mapsto f(\gamma(t))\gamma'(t)
\]
is bornologically Riemann integrable. 
\end{lem}

\begin{proof}
Denote by $F\colon [t_{i-1}, t_i]\to X$ the function in the statement. By \cref{prp:borncontint}, it suffices to show that $F$ is bornologically continuous. Let $t\in [t_{i-1}, t_i]$. Since $f$ is bornologically continuous, there exist an open neighborhood $V\subset \gamma([a, b])$ of $\gamma(t)$ and a Banach disked set $B\subset X$ such that $f\colon V\to X_B$ is continuous. Moreover, since $\gamma\colon [t_{i-1}, t_i]\to C$ is continuous, there exists an open neighborhood $U\subset [t_{i-1}, t_i]$ of $t$ such that $\gamma(U)\subset V$. Thus $f\circ \gamma\colon U\to X_B$ is continuous, and hence $F\colon U\to X_B$ is continuous. Therefore, $F$ is bornologically continuous. 
\end{proof} 

Let $C$ be a piecewise $C^1$-curve in $\bbC$, and let $f\colon C\to X$ be a bornologically continuous function. Choose a parametrization $\gamma\colon [a, b]\to \bbC$ of $C$ and a partition $a=t_0<t_1<\cdots <t_n=b$ such that $\gamma\colon[t_{i-1}, t_i]\to \bbC$ is a function of class $C^1$ for all $i=1, \ldots, n$. We define
\[
\int_Cf=\int_Cf(z)\,dz\ceq \sum_{i=1}^n\int_{t_{i-1}}^{t_i}f(\gamma(t))\gamma'(t)\,dt. 
\]
This definition is independent of the choices of parametrization $\gamma\colon [a, b]\to \bbC$ and partition of $[a, b]$. 
We call $\int_Cf$ the \emph{bornological complex integral} of $f$ along $C$. 
If $B\subset X$ is a bounded Banach disked set such that $f(C)\subset X_B$ and $f\colon C\to X_B$ is continuous, then the bornological complex integral $\int_Cf$ agrees with the usual complex integral of $f$ regarded as an $X_B$-valued function. 

Let $C_1, \ldots,C_n\subset \bbC$ be piecewise $C^1$-curves, let $f\colon C_1\times \cdots \times C_n\to X$ be a bornologically continuous function, and let $\sigma\in \frS_n$. The function 
\begin{align*}
C_1\times \cdots \times  C_{\sigma(1)-1}\times C_{\sigma(1)+1}\times \cdots \times C_n 
\to X, \quad 
(z_1, \ldots, z_n)\mapsto \int_{C_{\sigma(1)}}f(z_1, \ldots, z_n)\, dz_{\sigma(1)}
\end{align*}
is bornologically continuous. We define the iterated bornological complex integral inductively by
\[
\int_{C_{\sigma(n)}}\cdots \int_{C_{\sigma(1)}}f(z_1, \ldots, z_n)\, dz_{\sigma(1)}\cdots dz_{\sigma(n)}\ceq
\int_{C_{\sigma(n)}}\cdots \int_{C_{\sigma(2)}}\Bigl(\int_{C_{\sigma(1)}}f(z_1, \ldots, z_n)\, dz_{\sigma(1)}\Bigr)\, dz_{\sigma(2)}\cdots dz_{\sigma(n)}. 
\]
If $B\subset X$ is a bounded Banach disked set such that $f(C_1\times \cdots \times C_n)\subset X_B$ and $f\colon C_1\times \cdots \times C_n\to X_B$ is continuous, then the above iterated bornological complex integral agrees with the usual iterated complex integral of $f$ regarded as an $X_B$-valued function. In particular, we have
\begin{equation}\label{eq;compintcom}
\int_{C_n}\cdots \int_{C_1}f(z_1, \ldots, z_n)\, dz_1\cdots dz_n=
\int_{C_{\sigma(n)}}\cdots \int_{C_{\sigma(1)}}f(z_1, \ldots, z_n)\, dz_{\sigma(1)}\cdots dz_{\sigma(n)}. 
\end{equation}

\begin{prp}[Cauchy's integral theorem]\label{prp:Cauintthm}
Let $U\subset \bbC$ be an open subset, let $C\subset U$ be a simple closed piecewise $C^1$-curve such that $\Omega(C)\subset U$, and let $f\colon U\to X$ be a bornologically holomorphic function. Then 
\[
\int_Cf=0. 
\]
\end{prp}

\begin{proof}
Choose a relatively compact open subset $V\subset \bbC$ such that $\ol{\Omega(C)}\subset V$ and $\ol{V}\subset U$. Since $\ol{V}$ is a compact subset of $U$, \cref{lem:bornholo} implies that there exists a bounded Banach disked set $B\subset X$ such that $f(V)\subset X_B$ and $f\colon V\to X_B$ is holomorphic. Hence, by Cauchy's integral theorem for functions with values in a Banach space, 
\[
\int_Cf=0 \quad \text{ in }X_B. 
\]
This proves the proposition. 
\end{proof}

\begin{prp}\label{prp:Banacontholo}
Let $U\subset \bbC$ be an open subset, let $f\colon U\to X$ be a bornologically holomorphic function, and let $B\subset X$ be a bounded Banach disked set such that $f(U)\subset X_B$. If $f\colon U\to X_B$ is continuous, then it is holomorphic. 
\end{prp}

\begin{proof}
For a simple closed piecewise $C^1$-curves $C\subset U$, \cref{prp:Cauintthm} implies that 
\[
\int_Cf=0 \quad \text{ in }X_B. 
\]
Hence, by the Morera's theorem, the function $f\colon U\to X_B$ is holomorphic. 
\end{proof}

\begin{lem}\label{lem:f'bornholo}
Let $U\subset \bbC$ be an open subset. For a bornologically holomorphic function $f\colon U\to X$, the function $f'\colon U\to X$ is bornologically holomorphic on $U$. 
\end{lem}

\begin{proof}
For $a\in U$, there exists an open neighborhood $V\subset U$ of $a$ and a bounded Banach disked set $B\subset X$ such that $f(V)\subset X_B$ and $f\colon V\to X_B$ is holomorphic on $V$. By the corresponding result for functions with values in a Banach space, the function $f'\colon V\to X_B$ is holomorphic on $V$. Hence $f'\colon U\to X$ is bornologically holomorphic on $U$. 
\end{proof}

Let $U\subset \bbC$ be an open subset, and let $f\colon U\to X$ be a bornologically holomorphic function. By \cref{lem:f'bornholo}, one can define the bornologically holomorphic functions $f^{(n)}\colon U\to X$ ($n\in \bbN$) inductively as follows: 
\[
f^{(0)}\ceq f, \qquad f^{(n)}\ceq (f^{(n-1)})' \quad (n\ge1). 
\]

\begin{prp}[Cauchy's integral formula]\label{prp:Cauintform}
Let $U\subset \bbC$ be an open subset, let $C\subset U$ be a positively oriented simple closed piecewise $C^1$-curve such that $\Omega(C)\subset U$, and let $f\colon U\to X$ be a bornologically holomorphic function. Then, for $n\in \bbN$ and $z\in\Omega(C)$, 
\[
f^{(n)}(z)=\frac{n!}{2\pi\sqrt{-1}}\int_C\frac{f(\zeta)}{(\zeta-z)^{n+1}}\, d\zeta. 
\]
\end{prp}

\begin{proof}
Choose a relatively compact open subset $V\subset \bbC$ such that $\ol{\Omega(C)}\subset V$ and $\ol{V}\subset U$. Since $\ol{V}$ is a compact subset of $U$, \cref{lem:bornholo} implies that there exists a bounded Banach disked set $B\subset X$ such that $f(\ol{V})\subset X_B$ and $f\colon V\to X_B$ is holomorphic. Hence, by Cauchy's integral formula for functions with values in a Banach space,  
\[
f^{(n)}(z)=\frac{n!}{2\pi\sqrt{-1}}\int_C\frac{f(\zeta)}{(\zeta-z)^{n+1}}\, d\zeta \quad \text{ in $X_B$}. 
\]
This proves the proposition. 
\end{proof}

\begin{prp}
Let $U\subset \bbC$ be an open subset, and let $f\colon U\to X$ be a bornologically holomorphic function. For $a\in U$ and $0<R\le \infty$ such that $D_R(a)\subset U$, 
\[
f(z)=\sum_{n=0}^\infty \frac{1}{n!}(z-a)^nf^{(n)}(a) \quad (z\in D_R(a)), 
\]
where the series on the right-hand side converges bornologically. 
\end{prp}

\begin{proof}
Let $z\in D_R(a)$. Choose $r\in \bbR_{>0}$ such that $|z-a|<r<R$. By \cref{lem:bornholo}, there exists a bounded Banach disked set $B\subset X$ such that $f(\ol{D}_r(a))\subset X_B$ and $f\colon D_r(a)\to X_B$ is holomorphic. Hence
\[
f(z)=\sum_{n=0}^\infty\frac{1}{n!}(z-a)^nf^{(n)}(a) \quad \text{ in $X_B$}. 
\]
Thus the proposition follows. 
\end{proof}

\begin{prp}
Let $U\subset \bbC$ be an open subset, let $f\colon U\to X$ be a bornologically holomorphic function, and let $\{x_n\}_{n\in \bbN}\subset X$ be a sequence. If there exist $a\in U$ and $0<R\le \infty$ such that $D_R(a)\subset U$ and 
\[
f(z)=\sum_{n=0}^\infty (z-a)^nx_n \quad (z\in D_R(a)),
\]
where the series converges bornologically, 
then $x_n=f^{(n)}(a)/n!$ for all $n\in \bbN$. 
\end{prp}

\begin{proof}
Choose $r\in \bbR_{>0}$ such that $0<r<R$. Then, by \cref{prp:Cauintform}, 
\[
f^{(n)}(a)=\frac{n!}{2\pi\sqrt{-1}}\int_{\pdd D_r(a)}\frac{f(z)}{(z-a)^{n+1}}\, dz. 
\]
By \cref{prp:onepowserabs2}, there exists a bounded Banach disked set $B\subset X$ such that $x_m\in X_B$ for all $m\in \bbN$ and 
\[
f(z)=\sum_{m=0}^\infty (z-a)^mx_m \quad (z\in \pdd D_r(a)),
\]
where the series on the right-hand side converges uniformly in $X_B$. Hence
\[
\int_{\pdd D_r(a)}\frac{f(z)}{(z-a)^{n+1}}\,dz=
\sum_{m=0}^\infty \int_{\pdd D_r(a)}(z-a)^{m-n-1}x_m\, dz=
2\pi\sqrt{-1}x_n \quad \text{ in $X_B$}. 
\]
Thus $f^{(n)}(a)=n!x_n$. 
\end{proof}

Let $C_0$ and $C_1$ are piecewise $C^1$-curves. A \emph{homotopy} from $C_0$ to $C_1$ is a continuous map $h\colon [0, 1]\times [a, b]\to \bbC$ satisfying (i) for every $s\in [0, 1]$, the map $h_s\colon [a, b]\to \bbC$, $t\mapsto h(s, t)$ is a parametrized piecewise $C^1$-curve, (ii) for every $s\in [0, 1]$, we have $h_s(a)=h_0(a)$ and $h_s(b)=h_0(b)$, (iii) $h_0$ and $h_1$ are parametrizations of $C_0$ and $C_1$, respectively. We say that $C_0$ and $C_1$ are \emph{homotopic} if there exists a homotopy from $C_0$ to $C_1$. 

\begin{lem}\label{lem:homotint}
Let $U\subset \bbC$ be an open subset, and let $f\colon U\to X$ be a bornologically holomorphic function. If piecewise $C^1$-curves $C_0$ and $C_1$ are homotopic, then 
\[
\int_{C_0}f=\int_{C_1}f. 
\]
\end{lem}

\begin{proof}
Let $h\colon [0, 1]\times[a, b]\to \bbC$ be a homotopy from $C_0$ to $C_1$. Choose a relatively compact subset $V\subset \bbC$ such that $h([0, 1]\times [a, b])\subset V$ and $\ol{V}\subset U$. Since $\ol{V}$ is compact, \cref{lem:bornholo} implies that there exists a bounded Banach disked set $B\subset X$ such that $f(\ol{V})\subset X_B$ and $f\colon V\to X_B$ is holomorphic. Hence, by the corresponding result for functions with values in a Banach space, 
\[
\int_{C_0}f=\int_{C_1}f \quad \text{ in }X_B. 
\]
This proves the lemma. 
\end{proof}

\begin{prp}\label{prp:bornLau}
Let $U\subset \bbC$ be an open subset, let $a\in \bbC$, and let $f\colon U\bs\{a\}\to X$ be a bornologically holomorphic function. For $0\le r<R\le \infty$ such that $A_{r, R}(a)\subset U$, 
\[
f(z)=\sum_{n=0}^\infty (z-a)^nx_n+\sum_{n=1}^\infty (z-a)^{-n}x_{-n} \quad (z\in A_{r, R}(a)), 
\]
where both series on the right-hand side converge bornologically, and for every $n\in \bbZ$, 
\[
x_n\ceq \frac{1}{2\pi\sqrt{-1}}\int_{\pdd D_\delta(a)}\frac{f(z)}{(z-a)^{n+1}}\, dz \quad (r<\delta<R). 
\]
Note that, by \cref{lem:homotint}, the definition of $x_n$ is independent of the choice of $\delta$. 
\end{prp}

\begin{proof}
Let $z\in A_{r, R}(a)$. Choose $s, t\in \bbR_{>0}$ such that $r<s<|z-a|<t<R$. By \cref{lem:bornholo}, there exists a bounded Banach disked set $B\subset X$ such that $f(\ol{A}_{s, t}(a))\subset X_B$ and $f\colon A_{s, t}(a)\to X_B$ is holomorphic. Hence
\[
f(z)=\sum_{n=0}^\infty (z-a)^nx_n+\sum_{n=1}^\infty (z-a)^{-n}x_{-n} \quad \text{ in } X_B, 
\]
where, for every $n\in \bbZ$, 
\[
x_n\ceq \frac{1}{2\pi\sqrt{-1}}\int_{\pdd D_\delta(a)}\frac{f(z)}{(z-a)^{n+1}}\, dz \quad (s<\delta<t). 
\]
Thus the proposition follows. 
\end{proof}

\begin{prp}\label{prp:bornLauuni}
Let $U\subset \bbC$ be an open subset, let $a\in \bbC$, let $f\colon U\bs\{a\}\to X$ be a bornologically holomorphic function, and let $\{x_n\}_{n\in \bbZ}\subset X$ be a sequence. If there exists $0\le r<R\le\infty$ such that $A_{r, R}(a)\subset U$ and 
\[
f(z)=\sum_{n=0}^\infty (z-a)^nx_n+\sum_{n=1}^\infty (z-a)^{-n}x_{-n} \quad (z\in A_{r, R}(a)), 
\]
where both series on the right-hand side converge bornologically, then for every $n\in \bbZ$, 
\[
x_n= \frac{1}{2\pi\sqrt{-1}}\int_{\pdd D_\delta(a)}\frac{f(z)}{(z-a)^{n+1}}\, dz \quad (r<\delta<R). 
\]
\end{prp}

\begin{proof}
By \cref{prp:onepowserabs2} and \cref{prp:Lauuni}, there exists a bounded Banach disked set $B\subset X$ such that $x_m\in X_B$ for all $m\in \bbN$ and 
\[
f(z)=\sum_{m=0}^\infty (z-a)^mx_m+\sum_{m=1}^\infty (z-a)^{-m}x_{-m} \quad (z\in \pdd D_\delta(a)), 
\]
where both series on the right-hand side converges uniformly in $X_B$. Hence 
\begin{align*}
\int_{\pdd D_\delta(a)}\frac{f(z)}{(z-a)^{n+1}}\, dz
&=
\sum_{m=0}^\infty \int_{\pdd D_\delta(a)}(z-a)^{m-n-1}x_m\, dz+\sum_{m=1}^\infty \int_{\pdd D_\delta(a)}(z-a)^{-m-n-1}x_{-m}\\
&=2\pi\sqrt{-1}x_n \quad \text{ in }X_B. 
\end{align*}
Thus the proposition follows. 
\end{proof}

\begin{dfn}
Let $U\subset \bbC$ be an open subset, let $a\in U$, and let $f\colon U\bs\{a\}\to X$ be a bornologically holomorphic function. Choose $R\in \bbR_{>0}$ such that $D_R(a)\subset U$, and let 
\[
f(z)=\sum_{n=0}^\infty (z-a)^nx_n+\sum_{n=1}^\infty (z-a)^{-n}x_{-n} \quad (z\in D_R^\times(a))
\]
be the Laurent expansion. The point $a$ is called a \emph{removable singularity} of $f$ if $x_{-n}=0$ for all $n\ge1$. 
\end{dfn}

\begin{prp}\label{prp:remsingext}
Let $U\subset \bbC$ be a domain, let $a\in U$, and let $f\colon U\bs\{a\}\to X$ be a bornologically holomorphic function. If $a$ is a removable singularity of $f$, then there exists a unique bornologically holomorphic function $\wt{f}\colon U\to X$ such that $\wt{f}|_{U\bs\{a\}}=f$. 
\end{prp}

\begin{proof}
Uniqueness follows from \cref{lem:idensev}. Choose $R\in \bbR_{>0}$ such that $D_R(a)\subset U$. Then the Laurent expansion of $f$ around $a$ is
\[
f(z)=\sum_{n=0}^\infty (z-a)^nx_n \quad (z\in D_R^\times(a)). 
\]
Define a function $\wt{f}\colon U\to X$ by 
\[
\wt{f}(z)\ceq
\begin{cases}
f(z) & (z\in U\bs D_R(a)) \\
\displaystyle \sum_{n=0}^\infty (z-a)^nx_n & (z\in D_R(a)). 
\end{cases}
\]
Then $\wt{f}|_{U\bs\{a\}}=f$. In particular, $\wt{f}$ is bornologically holomorphic on $U\bs\{a\}$. By \cref{prp:Taydiff}, $\wt{f}$ is also bornologically holomorphic on $D_R(a)$. Hence $\wt{f}$ is bornologically holomorphic. 
\end{proof}

\section{From factorization algebras to vertex algebras}
\label{s:FAtoVA}

In this section, we present a bornological version of Costello and Gwilliam's general method for constructing vertex algebras from prefactorization algebras. Using the bornological complex analysis developed in the previous section, we show that the discreteness condition imposed in \cite[Theorem 5.3.3]{CG} and \cite[Proposition 1.3.7]{B20a} is not necessary.

\subsection{Holomorphically translation-equivariant prefactorization algebras}
\label{ss:holotrans}

We briefly review the notion of holomorphically translation-equivariant prefactorization algebras introduced in \cite[\S5.2]{CG}, presenting it in a form suitable for our purposes. 

\smallskip

A complete separated convex bornological vector space is called a \emph{convenient vector space}. We denote by $\CVS$ the category of convenient vector spaces over $\bbC$ with bounded linear maps. 
From the discussion in \cref{ss:constBVS}, we see that $\CVS$ is a complete and cocomplete closed symmetric monoidal linear category: 
\begin{itemize}
\item 
For a family $\{X_i\}_{i\in I}$ of objects in $\CVS$, the product $\prod_{i\in I}X_i$ in $\ModC$, equipped with the product bornology, is a product in $\CVS$, and the coproduct $\bigoplus_{i\in I}X_i$ in $\ModC$, equipped with the direct sum bornology, is a coproduct in $\CVS$. 

\item 
For a morphism $f\colon X\to Y$ in $\CVS$, the linear subspace $\Ker f=\{x\in X\mid f(x)=0\}$ of $X$, equipped with the subspace bornology, is its kernel in $\CVS$, and the quotient linear space $Y/\olb{\Img f}$, equipped with the quotient bornology, is its cokernel in $\CVS$. Here $\olb{\Img f}$ denotes the bornological closure of $\Img f=f(X)$. 

\item 
The symmetric monoidal structure on $\CVS$ is given by the completed bornological tensor product $\wt{\otimes}$. 
\end{itemize}
We denote by $H^n$ (resp. $H_n$) the cohomology (resp. homology) taken in $\sfC(\CBVS)$, and by $H^n_\sep$ (resp. $H_n^\sep$) the cohomology (resp. homology) taken in $\sfC(\CVS)$. 

\smallskip
For $0<r_1, \ldots, r_n, R\le \infty$, set
\[
\Disks(r_1, \ldots, r_n; R)\ceq 
\{(z_1, \ldots, z_n)\in \bbC^n\mid \ol{D}_{r_1}(z_1)\sqcup \cdots \sqcup \ol{D}_{r_n}(z_n)\subset D_R\}, 
\]
which is an open subset of $\bbC^n$. Here we use the convention that $D_\infty\ceq\bbC$, and $\ol{D}_{r_i}(z_i)$ denotes the closed disk. 

Let $\sfM$ be a symmetric monoidal category, and let $\clF\colon \frU_\bbC\to \sfM$ be a translation-equivariant prefactorization algebra, i.e., a prefactorization algebra equivariant under the action of the additive group $\bbC$ by translations. For $0<r_1, \ldots, r_n, R\le \infty$ and $(z_1, \ldots, z_n)\in \Disks(r_1, \ldots, r_n; R)$, define a morphism $\clF^{r_1, \ldots, r_n; R}_{z_1, \ldots, z_n}$ in $\sfM$ as the composition
\[
\begin{tikzcd}[column sep=huge]
\clF^{r_1, \ldots, r_n; R}_{z_1, \ldots, z_n}\,\colon\, \displaystyle \bigotimes_{i=1}^n\clF(D_{r_i}) \arrow[r, "\otimes_{i=1}^n\sigma_{z_i, D_{r_i}}"] & [0.5cm]
\displaystyle\bigotimes_{i=1}^n\clF(D_{r_i}(z_i)) \arrow[r, "\clF^{\{D_{r_i}(z_i)\}_{i=1}^n}_{D_R}"] & [0.5cm]
\clF(D_R)
\end{tikzcd}
\]
Here $\sigma_z$ denotes the translation-equivariant structure of $\clF$. 

\begin{dfn}\label{dfn:holoPFA}
A translation-equivariant prefactorization algebra $\clF\colon \frU_\bbC\to \CVS$ is called \emph{holomorphically translation-equivariant} (resp. \emph{differentiably translation-equivariant}) if it satisfies the following condition: For $0<r_1, \ldots, r_n, R\le \infty$ and $a_i\in \clF(D_{r_i})$, the function
\[
\Disks(r_1, \ldots, r_n; R)\to \clF(D_R), \quad 
(z_1, \ldots, z_n)\mapsto \clF_{z_1, \ldots, z_n}^{r_1, \ldots, r_n; R}(a_1\otimes \cdots \otimes a_n)
\]
is bornologically holomorphic (resp. bornologically differentiable). 
\end{dfn}

We now give a systematic method for producing holomorphically translation-equivariant prefactorization algebras from ones valued in chain complexes. This method will be used in \cref{ss:FEEV}. 

\begin{dfn}\label{dfn:smPFA}
A translation-equivariant prefactorization algebra $\clF\colon \frU_\bbC\to \sfC(\CVS)$ is called \emph{differentiably translation-equivariant} if it satisfies the following condition: For $0<r_1, \ldots, r_n, R\le \infty$ and $a_i\in \clF(D_{r_i})_{m_i}$, the function
\[
\Disks(r_1, \ldots, r_n; R)\to \clF(D_R)_{m_1+\cdots +m_n}, \quad 
(z_1, \ldots, z_n)\mapsto \clF^{r_1, \ldots, r_n; R}_{z_1, \ldots, z_n}(a_1\otimes \cdots \otimes a_n)
\]
is bornologically differentiable. 
\end{dfn}

For a translation-equivariant prefactorization algebra $\clF\colon \frU_\bbC\to \sfC(\CVS)$, composing $\clF$ with the symmetric monoidal functor $H^{\sep}_\bullet\colon \sfC(\CVS)\to \CVS$, which takes homology, yields a translation-equivariant prefactorization algebra $H^\sep_\bullet\clF\colon \frU_\bbC\to \CVS$. For $0<r_1, \ldots, r_n, R\le \infty$ and $a_i\in \clF(D_{r_i})$, we have 
\begin{equation}\label{eq:HF}
(H^\sep_\bullet\clF)^{r_1, \ldots, r_n: R}_{z_1, \ldots, z_n}([a_1]\otimes \cdots \otimes [a_n])=
[\clF^{r_1, \ldots, r_n; R}_{z_1, \ldots, z_n}(a_1\otimes \cdots \otimes a_n)]. 
\end{equation}

\begin{lem}
If $\clF\colon \frU_\bbC\to \sfC(\CVS)$ is a differentiably translation-equivariant prefactorization algebra, then so is $H^\sep_\bullet\clF\colon \frU_\bbC\to \CVS$. 
\end{lem}

\begin{proof}
This follows from \cref{lem:borndiffE}, \cref{lem:borndiffboudlin}, and \eqref{eq:HF}. 
\end{proof}

\begin{lem}\label{lem:smtoholo}
Let $\clF\colon \frU_\bbC\to \sfC(\CVS)$ be a differentiably translation-equivariant prefactorization algebra satisfying the following conditions: 
There exists a family of morphisms 
\[
\delta_R\colon \clF(D_R)_\bullet\to \clF(D_R)_{\bullet+1} \quad (0<R\le \infty)
\]
of $\bbZ$-graded objects such that for $0<r_1, \ldots, r_n, R\le \infty$ and cycles $a_i\in Z(\clF(D_{r_i}))$, 
\[
\pd{}{\ol{z}_i}\clF^{r_1, \ldots, r_n; R}_{z_1, \ldots, z_n}(a_1\otimes \cdots \otimes a_n)=
\clF^{r_1, \ldots, r_n; R}_{z_1, \ldots, z_n}(a_1\otimes \cdots \otimes d\delta_{r_i}(a_i)\otimes \cdots \otimes a_n), 
\]
where $d$ denotes the differential of the complex $\clF(D_{r_i})$. 
Then $H^{\sep}_\bullet\clF\colon \frU_\bbC\to \CVS$ is a holomorphically translation-equivariant prefactorization algebra. 
\end{lem}

\begin{proof}
Let $0<r_1, \ldots, r_n, R\le \infty$, and let $[a_i]\in H^\sep_\bullet(\clF(D_{r_i}))$, where $a_i\in Z(\clF(D_{r_i}))$. By \cref{lem:boulinborncov}, \cref{lem:pddblim}, and \eqref{eq:HF}, 
\[
\pd{}{\ol{z}_i}(H^\sep_\bullet\clF)^{r_1, \ldots, r_n; R}_{z_1, \ldots, z_n}([a_1]\otimes \cdots \otimes [a_n])=
\Bigl[\pd{}{\ol{z}_i}\clF^{r_1, \ldots, r_n; R}_{z_1, \ldots, z_n}(a_1\otimes \cdots\otimes a_n)\Bigr]. 
\]
By assumption, 
\[
\pd{}{\ol{z}_i}\clF^{r_1, \ldots, r_n; R}_{z_1, \ldots, z_n}(a_1\otimes \cdots \otimes a_n)=
(-1)^{|a_1|+\cdots +|a_{i-1}|}d\clF^{r_1, \ldots, r_n; R}_{z_1, \ldots, z_n}(a_1\otimes \cdots \otimes \delta_{r_i}(a_i)\otimes \cdots \otimes a_n). 
\]
Hence 
\[
\pd{}{\ol{z}_i}(H^\sep_\bullet\clF)^{r_1, \ldots, r_n: R}_{z_1, \ldots, z_n}([a_1]\otimes \cdots \otimes [a_n])=0. 
\]
Thus the lemma follows from \cref{lem:bornCR}. 
\end{proof}

\subsection{Construction of vertex algebras from factorization algebras}
\label{ss:FAtoVA}

In this subsection we construct vertex algebras from $S^1\ltimes \bbC$-equivariant prefactorization algebras on the complex plane $\bbC$, called amenably holomorphic. 

Let $\clF\colon \frU_\bbC\to \CVS$ be an $S^1$-equivariant prefactorization algebra, where the unit circle $S^1$ acts on $\bbC$ by rotations. 
For $\Delta\in \bbZ$ and an open subset $U\subset \bbC$ that is invariant under the action of $S^1$, define 
\[
\clF(U)_\Delta\ceq \{a\in \clF(U)\mid \forall q\in S^1,\, \sigma_{q, U}(a)=q^{\Delta}a\}, 
\]
where $\sigma_q$ denotes the $S^1$-equivariant structure of $\clF$. 

The group $S^1\ltimes \bbC$ acts on $\bbC$ by $(q, z)\cdot w\ceq qw+z$. Given an $S^1\ltimes \bbC$-equivariant prefactorization algebra $\clF\colon \frU_\bbC\to \CVS$, note that $\clF$ naturally inherits both an $S^1$-equivariant structure and a translation-equivariant structure. 
The main result of this section is the following theorem. 

\begin{thm}\label{thm:FAtoVA}
Let $\clF\colon \frU_\bbC\to\CVS$ be an $S^1\ltimes \bbC$-equivariant prefactorization algebra satisfying the following conditions: 
\begin{clist}
\item 
The $\bbC$-action makes $\clF\colon \frU_\bbC\to \CVS$ a holomorphically translation-equivariant prefactorization algebra. 

\item 
For $0<R\le \infty$, we have $\clF(D_R)_\Delta=0$ ($\Delta\ll0$). 

\item
For $0<r<R\le\infty$ and $\Delta\in \bbZ$, the map $\clF^{D_r}_{D_R}\colon \clF(D_r)_\Delta\to \clF(D_R)_\Delta$ is a linear isomorphism. 
\end{clist}
Then, the linear space
\[
\bfVR(\clF)\ceq \bigoplus_{\Delta\in \bbZ}\clF(D_R)_\Delta
\]
admits the structure of a $\bbZ$-graded vertex algebra. We call such $\clF$ an \emph{amenably holomorphic prefactorization algebra}. 
\end{thm}

Bruegmann's construction of vertex algebras from prefactorization algebras proceeds through the intermediate notion of geometric vertex algebras \cite{B20a, B20b}. Here, we instead give a direct construction of vertex algebras from prefactorization algebras. Moreover, using the bornological complex analysis developed in the previous section, we can dispense with the discreteness condition on the weight spaces imposed in Bruegmann's construction \cite[Proposition 1.3.7]{B20a}. 

In what follows, we prove \cref{thm:FAtoVA}. Let $\clF\colon \frU_\bbC\to \CVS$ be an amenably holomorphic prefactorization algebra. We denote by 
\[
\sigma_{(q, z), U}\colon \clF(U)\to \clF((q, z)U) \quad ((q, z)\in S^1\ltimes \bbC,\, U\in \frU_\bbC)
\]
the $S^1\ltimes \bbC$-equivariant structure on $\clF$. Moreover, for an open subset $U\subset \bbC$, we write 
\[
\Conf_n(U)\ceq \{(z_1, \ldots, z_n)\in U^n\mid z_i\neq z_j \text{ for } i\neq j\},
\]
which is an open subset of $\bbC^n$. 

\smallskip

First, for $(z_1, \ldots, z_n)\in \Conf_n(D_R)$, by condition (iii) in \cref{thm:FAtoVA}, one can define a linear map $\mu_{z_1, \ldots, z_n}^R$ as the composition
\[
\begin{tikzcd}
\mu_{z_1, \ldots, z_n}^R\colon \bfVR(\clF)^{\otimes n} \arrow[r, "\sim"] &
\displaystyle\bigotimes_{1\le i\le n}\bfV_{\!r_i}(\clF) \arrow[r, hook] & 
\displaystyle\bigotimes_{1\le i\le n}\clF(D_{r_i}) \arrow[r] &
\displaystyle\tbotimes_{1\le i\le n} \clF(D_{r_i})\arrow[r, "\clF^{r_1, \ldots, r_n; R}_{z_1, \ldots, z_n}"] & [1cm]
\clF(D_R),
\end{tikzcd}
\]
where $r_1, \ldots, r_n\in \bbR_{>0}$ are chosen so that $\ol{D}_{r_1}(z_1)\sqcup \cdots \sqcup \ol{D}_{r_n}(z_n)\subset D_R$. Note that the definition of $\mu_{z_1, \ldots, z_n}^R$ is independent of the choices of $r_1, \ldots, r_n$. 

\begin{lem}\label{lem:qmuz}
Let $(z_1, \ldots, z_n)\in \Conf_n(D_R)$. For $q\in S^1$ and homogeneous $a_1, \ldots, a_n\in \bfVR(\clF)$, we have
\[
\sigma_{(q, 0), D_R}\,\mu_{z_1, \ldots, z_n}^R(a_1\otimes \cdots \otimes a_n)=
q^{\Delta(a_1)+\cdots+\Delta(a_n)}\mu_{qz_1, \ldots, qz_n}^R(a_1\otimes \cdots\otimes a_n). 
\]
\end{lem}

\begin{proof}
Choose $r\in \bbR_{>0}$ such that $\ol{D}_r(z_1)\sqcup \cdots \sqcup \ol{D}_r(z_n)\subset D_R$. For $q\in S^1$ and  $a_i\in \clF(D_r)$, 
\[
\sigma_{(q, 0), D_R}\clF^{r, \ldots, r; R}_{z_1, \ldots, z_n}(a_1\,\wt{\otimes}\, \cdots \,\wt{\otimes}\, a_n)=
\clF^{\{D_r(z_i)\}_{i=1}^n}_{D_R}\bigl(\wt{\otimes}_{i=1}^n\sigma_{(q, 0), D_r(z_i)}\sigma_{(1, z_i), D_r}(a_i)\bigr).
\]
Since $(q, 0)(1, z_i)=(1, qz_i)(q, 0)$ in $S^1\ltimes \bbC$, 
\[
\sigma_{(q, 0), D_R}\clF^{r, \ldots, r; R}_{z_1, \ldots, z_n}(a_1\,\wt{\otimes}\, \cdots \, \wt{\otimes}\, a_n)=
\clF^{r, \ldots, r; R}_{qz_1, \ldots, qz_n}(\sigma_{(q, 0), D_r}(a_1)\,\wt{\otimes}\, \cdots \,\wt{\otimes}\, \sigma_{(q, 0), D_r}(a_n)). 
\]
This proves the lemma. 
\end{proof}

For $a_1, \ldots, a_n\in \bfVR(\clF)$, by condition (i) in \cref{thm:FAtoVA}, the function
\[
\Conf_n(D_R)\to \clF(D_R), \quad (z_1, \ldots, z_n)\mapsto \mu_{z_1, \ldots, z_n}^R(a_1\otimes \cdots \otimes a_n)
\]
is bornologically holomorphic. 

\smallskip

We now define a vertex algebra structure on $\bfVR(\clF)$. 

\begin{itemize}
\item\emph{The vacuum}. 
Let $\vac\in \clF(D_R)$ be the image of $1\in \bbC$ under the map 
\[
\begin{tikzcd}[column sep=huge]
\bbC \arrow[r, "1_\clF"] & 
\clF(\varnothing) \arrow[r, "\clF^{\varnothing}_{D_R}"] &
\clF(D_R). 
\end{tikzcd}
\] 
Then $\vac\in \clF(D_R)_{\Delta=0}$, and hence $\vac\in \bfVR(\clF)$. 

\item\emph{The translation operator}. 
Let $a\in  \bfVR(\clF)$. Since the function
\[
D_R\to \clF(D_R), \quad z\mapsto \mu_z^R(a)
\]
is bornologically holomorphic, one can define 
\[
Ta\ceq \Bigl.\pd{}{z}\mu_z^R(a)\Bigr|_{z=0}\in \clF(D_R). 
\]
By \cref{lem:qmuz}, we have $Ta\in \clF(D_R)_{\Delta(a)+1}$. Hence, we obtain a linear map
\[
T\colon \bfVR(\clF)\to \bfVR(\clF), \quad a\mapsto Ta. 
\]

\item\emph{The state-field correspondence}.
 Let $a, b\in \bfVR(\clF)$. Since the function
\[
D_R^\times\to \clF(D_R), \quad z\mapsto \mu_{z, 0}^R(a\otimes b)
\]
is bornologically holomorphic, \cref{prp:bornLau} implies that
\[
\mu_{z, 0}^R(a\otimes b)=\sum_{n=0}^\infty z^na_{(-n-1)}b+\sum_{n=0}^{\infty}z^{-n-1}a _{(n)}b \quad (z\in D_R^\times), 
\]
where both series on the right-hand side converge bornologically. 
By \cref{lem:qmuz}, we have $a_{(n)}b\in \clF(D_R)_{\Delta(a)+\Delta(b)-n-1}$. It follows from condition (ii) in \cref{thm:FAtoVA} that $a_{(n)}b=0$ for $n\gg0$. Hence, we obtain a linear map
\[
Y\colon \bfVR(\clF)\otimes \bfVR(\clF)\to \bfVR(\clF)\dpr{z}, \quad a\otimes b\mapsto Y(a, z)b\ceq \sum_{n\in \bbZ}z^{-n-1}a_{(n)}b. 
\]
\end{itemize}

We prove that $\bfVR(\clF)$, equipped with the above data, indeed forms a $\bbZ$-graded vertex algebra. By construction, it suffices to verify that these data satisfy the conditions in \cref{dfn:VA}. Moreover, by condition (iii) in \cref{thm:FAtoVA}, we may assume that $R=\infty$. 

\begin{proof}[Proof of vacuum axiom]
Let $a\in \bfVi(\clF)$. For $z\in \bbC^\times$, we have $\mu_{z, 0}^\infty(\vac\otimes a)=a$. Hence
\[
\vac_{(n)}a=
\begin{cases}
0 & (n\neq -1) \\
a & (n=-1), 
\end{cases}
\]
which shows that $Y(\vac, z)a=a$. Next, for $z\in \bbC^\times$, we have $\mu^\infty_{z, 0}(a\otimes \vac)=\mu^\infty_z(a)$. 
Since the function $z\mapsto \mu^\infty_z(a)$ is bornologically holomorphic on $\bbC$, it follows that $a_{(n)}\vac=0$ for all $n\ge0$. Moreover, 
\[
a_{(-1)}\vac=\mu_z^\infty(a)|_{z=0}=a. 
\]
Hence $Y(a, z)\vac\in \bfVi(\clF)\dbr{z}$ and $Y(a, z)\vac|_{z=0}=a$. 
\end{proof}

\begin{lem}\label{lem:transop}
For $a, b\in \bfVi(\clF)$ and $(z, w)\in \Conf_2(\bbC)$, 
\[
\mu_{z, w}^\infty(a\otimes b)=\sum_{n\in \bbZ}(z-w)^{-n-1}\mu_w^\infty(a_{(n)}b), 
\]
where the series on the right-hand side converges bornologically in $\clF(\bbC)$. 
\end{lem}

\begin{proof}
Choose $R\in \bbR_{>0}$ such that $z-w\in D_R$. Then we have
\[
\mu_{z, w}^\infty(a\otimes b)=
\clF^{R; \infty}_w\mu_{z-w, 0}^R(a\otimes b)=
\sum_{n\in \bbZ}(z-w)^{-n-1}\mu_w^\infty(a_{(n)}b).
\]
Here we use \cref{lem:boulinborncov} for the second equality, which guarantees that the series on the right-hand side converges bornologically. 
\end{proof}

\begin{proof}[Proof of translation invariance]
For $z\in \bbC$, we have $\mu_z^\infty(\vac)=\vac$, which proves $T\vac=0$. Let $a, b\in \bfVi(\clF)$, and let $n\in \bbZ$. By \cref{lem:transop}, 
\[
\mu_{z+w, w}^\infty(a\otimes b)=\sum_{m\in \bbZ}z^{-m-1}\mu^\infty_w(a_{(m)}b) \quad (z\in \bbC^\times , w\in \bbC), 
\]
where the series on the right-hand side converges bornologically. Fix $\delta\in \bbR_{>0}$. For $w\in \bbC$, \cref{prp:bornLauuni} implies that 
\[
\mu^\infty_w(a_{(n)}b)=\frac{1}{2\pi\sqrt{-1}}\int_{\pdd D_\delta}z^n\mu^\infty_{z+w, w}(a\otimes b)\, dz. 
\]
Thus, by Cauchy's integral formula (\cref{prp:Cauintform}), 
\begin{align*}
T(a_{(n)}b)
&=
\pd{}{w}\mu^\infty_w(a_{(n)}b)\Bigm|_{w=0}=
\frac{1}{2\pi\sqrt{-1}}\int_{\pdd D_\delta}w^{-2}\mu^\infty_w(a_{(n)}b)\, dw\\
&=
\frac{1}{(2\pi\sqrt{-1})^2}\int_{\pdd D_\delta}\int_{\pdd D_\delta}z^nw^{-2}\mu^\infty_{z+w, w}(a\otimes b)\, dzdw\\
&=
\frac{1}{(2\pi\sqrt{-1})^2}\int_{\pdd D_\delta}\int_{\pdd D_\delta}z^nw^{-2}\mu^\infty_{z+w, w}(a\otimes b)\, dwdz\\
&=
\frac{1}{2\pi\sqrt{-1}}\int_{\pdd D_\delta}z^n\pd{}{w}\mu^\infty_{z+w, w}(a\otimes b)\Bigm|_{w=0}\, dz\\
&=
\frac{1}{2\pi\sqrt{-1}}\int_{\pdd D_\delta}z^n\pd{}{z}\mu^\infty_{z, 0}(a\otimes b)\, dz+\frac{1}{2\pi\sqrt{-1}}\int_{\pdd D_\delta}z^n\pd{}{w}\mu^\infty_{z, w}(a\otimes b)\Bigm|_{w=0}\, dz. 
\end{align*}
By \cref{prp:Taydiff} and \cref{prp:Laudiff}, 
\[
\pd{}{z}\mu_{z, 0}^\infty(a\otimes b)=\sum_{m\in \bbZ}(-m-1)z^{-m-2}a_{(m)}b \quad (z\in \bbC^\times), 
\]
where the series on the right-hand side converges bornologically uniformly on compact sets. Hence
\[
\frac{1}{2\pi\sqrt{-1}}\int_{\pdd D_\delta}z^n\pd{}{z}\mu^\infty_{z, 0}(a\otimes b)\, dz=-na_{(n-1)}b. 
\]
For $z\in \bbC^\times$, choose $R\in \bbR_{>0}$ such that $\ol{D}_R(z)\cap \ol{D}_R=\varnothing$. Then, for $w\in D_R$, 
\[
\mu_{z, w}^\infty(a\otimes b)=\clF^{R, R; \infty}_{z, 0}(a\,\wt{\otimes}\,\mu_w^R(b)). 
\]
Hence 
\[
\pd{}{w}\mu^\infty_{z, w}(a\otimes b)\Bigm|_{w=0}=
\mu_{z, 0}^\infty(a\otimes Tb)=
\sum_{m\in \bbZ}z^{-m-1}a_{(m)}Tb \quad (z\in \bbC^\times), 
\]
where the series on the right-hand side converges bornologically uniformly on compact sets. It follows that
\[
\frac{1}{2\pi\sqrt{-1}}\int_{\pdd D_\delta}z^n\pd{}{w}\mu^\infty_{z, w}(a\otimes b)\Bigm|_{w=0}\, dz=a_{(n)}Tb. 
\]
Therefore, we obtain $T(a_{(n)}b)=-na_{(n-1)}b+a_{(n)}Tb$. 
\end{proof}

\begin{lem}\label{lem:locality}
For $a, b, c\in \bfVi(\clF)$ and $(z, w)\in \Conf_2(\bbC^\times)$, 
\[
\mu_{z, w, 0}^\infty(a\otimes b\otimes c)=
\begin{cases}
\displaystyle\sum_{n\in \bbZ}w^{-n-1}\mu_{z, 0}^\infty(a\otimes b_{(n)}c) & \text{if }|w|<|z|, \\
\displaystyle\sum_{n\in \bbZ}z^{-n-1}\mu_{w, 0}^\infty(b\otimes a_{(n)}c) & \text{if }|z|<|w|, \\
\displaystyle\sum_{n\in \bbZ}(z-w)^{-n-1}\mu_{w, 0}^\infty(a_{(n)}b\otimes c) & \text{if }|z-w|<|w|,
\end{cases}
\]
where all series on the right-hand side converge bornologically in $\clF(\bbC)$. 
\end{lem}

\begin{proof}
If $|z-w|<|w|$, then there exist $r, s\in \bbR_{>0}$ such that $z-w\in D_r$ and $\ol{D}_r(w)\cap \ol{D}_s=\varnothing$. We have
\[
\mu_{z, w, 0}^\infty(a\otimes b\otimes c)=
\clF^{r, s; \infty}_{w, 0}(\mu_{z-w, 0}^r(a\otimes b)\otimes c)=
\sum_{n\in\bbZ}(z-w)^{-n-1}\mu_{w, 0}^\infty(a_{(n)}b\otimes c). 
\]
Here we use \cref{lem:boulinborncov} for the second equality, which guarantees that the series on the right-hand side converges bornologically. Similarly, one can prove the equations for the cases $|w|<|z|$ and $|z|<|w|$. 
\end{proof}

Define
\[
U\ceq \{(z, w)\in \bbC^\times\times \bbC^\times\mid |z-w|<|w|\}, \quad
U_0\ceq\{(z, w)\in \bbC^\times\times \bbC^\times\mid 0<|z-w|<|w|\}. 
\]
Then $\Conf_2(\bbC^\times)\cup U=(\bbC^\times)^2$ and $\Conf_2(\bbC^\times)\cap U=U_0$. 

\begin{lem}\label{lem:locality2}
Let $X$ be a convenient vector space, and let $f\colon U_0\to X$ be a bornologically holomorphic function. Suppose that for every $w\in \bbC^\times$, the Laurent expansion of
\[
D_{|w|}^\times(w)\to X, \quad z\mapsto f(z, w)
\]
at $w$ has no negative powers. Then there exists a unique bornologically holomorphic function $\wt{f}\colon U\to X$ such that $\wt{f}|_{U_0}=f$. 
\end{lem}

\begin{proof}
Uniqueness follows from \cref{lem:idensev}. Set $\Lambda\ceq \{(\delta, r, s)\mid 0<\delta<r<s<\infty\}$. For each $\lambda=(\delta, r, s)\in \Lambda$, define
\[
U_\lambda\ceq\{(z, w)\in U\mid z\in D_\delta(w), w\in A_{r, s}(0)\}. 
\]
Then $\{U_\lambda\}_{\lambda\in \Lambda}$ is an open cover of $U$. 

Fix $\lambda=(\delta, r, s)\in \Lambda$. Define a function $f_\lambda\colon U_\lambda\to X$ by
\[
f_\lambda(z, w)\ceq \frac{1}{2\pi\sqrt{-1}}\int_{\pdd D_\delta(w)}\frac{f(\zeta, w)}{\zeta-z}\, d\zeta. 
\]
We claim that $f_\lambda$ is a bornologically holomorphic function such that $f_\lambda=f$ on $U_0\cap U_\lambda$.
Set 
\[
K\ceq \{(\zeta, w)\mid \zeta\in \pdd D_\delta(w), w\in \ol{A}_{r, s}(0)\}.
\]
Then $K$ is a compact subset of $U_0$. By \cref{lem:borncont}, there exists a bounded Banach disked set $B\subset X$ such that $f(K)\subset X_B$ and $f\colon K \to X_B$ is continuous. Since $f_\lambda(U_\lambda)\subset X_B$, to prove that $f_\lambda$ is bornologically holomorphic, it suffices to show that $f_\lambda\colon U_\lambda\to X_B$ is holomorphic. For $(z, w)\in U_\lambda$, 
\[
f_\lambda(z, w)=\frac{1}{2\pi}\int_0^{2\pi}\frac{\delta e^{\sqrt{-1}t}}{w+\delta e^{\sqrt{-1}t}-z}f(w+\delta e^{\sqrt{-1}t}, w)\, dt \quad \text{ in }X_B. 
\]
For each $t\in [0, 2\pi]$, define 
\[
g_t\colon A_{r, s}(0)\to X, \quad w\mapsto f(w+\delta e^{\sqrt{-1}t}, w). 
\]
Then, by \cref{lem:compbornholo}, the function $g_t$ is bornologically holomorphic. Moreover, we have $g_t(A_{r, s}(0))\subset f(K)\subset X_B$ and $g_t\colon A_{r, s}(0)\to X_B$ is continuous. Hence, by \cref{prp:Banacontholo}, the function $g_t\colon A_{r, s}(0)\to X_B$ is holomorphic. Thus
\[
U_\lambda\to X_B\quad (z, w)\mapsto \frac{\delta e^{\sqrt{-1}t}}{w+\delta e^{\sqrt{-1}t}-z}f(w+\delta e^{\sqrt{-1}t}, w)
\]
is holomorphic. This proves that $f_\lambda$ is bornologically holomorphic. To prove that $f_\lambda=f$ on $U_0\cap U_\lambda$, let $(z, w)\in U_0\cap U_\lambda$. By assumption, the point $w$ is a removable singularity of 
\[
f_w\colon D_{|w|}^\times(w)\to X, \quad \zeta\mapsto f(\zeta, w). 
\]
Hence, by \cref{prp:remsingext}, the function $f_w$ admits a bornologically holomorphic extension $\wt{f}_w\colon D_{|w|}(w)\to X$. 
Since $z\in D_{|w|}(w)$ and $\delta<|w|$, \cref{prp:Cauintform} implies 
\[
f(z, w)=\wt{f}_w(z)=
\frac{1}{2\pi\sqrt{-1}}\int_{\pdd D_\delta(w)}\frac{\wt{f}_w(\zeta)}{\zeta-z}\, d\zeta=
\frac{1}{2\pi\sqrt{-1}}\int_{\pdd D_\delta(w)}\frac{f(\zeta, w)}{\zeta-z}\, d\zeta=
f_\lambda(z, w). 
\]

For $\lambda, \mu\in \Lambda$, since $f_\lambda=f_\mu$ on $U_0\cap U_\lambda \cap U_\mu$, \cref{lem:idensev} implies that $f_\lambda=f_\mu$ on $U_\lambda \cap U_\mu$. Hence, there exists a unique bornologically holomorphic function $\wt{f}\colon U\to X$ such that $\wt{f}|_{U_\lambda}=f_\lambda$ for all $\lambda\in \Lambda$. Since $\wt{f}_{U_0\cap U_\lambda}=f|_{U_0\cap U_\lambda}$ for all $\lambda\in \Lambda$, we have $\wt{f}|_{U_0}=f$. 
\end{proof}

\begin{proof}[Proof of locality]
Let $a, b\in \bfVi(\clF)$, and choose $N\in \bbN$ such that $a_{(n)}b=b_{(n)}a=0$ for $n\ge N$. It suffices to show that 
\[
\sum_{j=0}^N(-1)^j\binom{N}{j}a_{(m+N-j)}(b_{(n+j)}c)=
\sum_{j=0}^N(-1)^j\binom{N}{j}b_{(n+j)}(a_{(m+N-j)}c)
\]
for $m, n\in \bbZ$ and $c\in \bfVi(\clF)$. 

We claim that there exists a unique bornologically holomorphic function $F\colon (\bbC^\times)^2\to \clF(\bbC)$ such that $F(z, w)=z^mw^n(z-w)^N\mu_{z, w, 0}(a\otimes b\otimes c)$ for all $(z, w)\in \Conf_2(\bbC^\times)$. Define
\[
f\colon \Conf_2(\bbC^\times)\to \clF(\bbC), \quad (z, w)\mapsto z^mw^n(z-w)^N\mu_{z, w, 0}^\infty(a\otimes b\otimes c)
\]
and $g\ceq f|_{U_0}$. For every $w\in \bbC^\times$, the function $z\mapsto z^m$ is holomorphic on $D^\times_{|w|}(w)$, and hence its Laurent expansion at $w$ has no negative powers. Moreover, for every $w\in \bbC^\times$, \cref{lem:locality} implies that the Laurent expansion of $D^\times_{|w|}(w)\to X$, $z\mapsto (z-w)^N\mu_{z, w, 0}(a\otimes b\otimes c)$ at $w$ has no negative powers. 
Thus, by \cref{lem:locality2}, there exists a unique bornologically holomorphic function $\wt{g}\colon U\to \clF(\bbC)$ such that $\wt{g}|_{U_0}=g$. Since $\Conf_2(\bbC^\times)\cup U=(\bbC^\times)^2$ and $\Conf_2(\bbC^\times)\cap U=U_0$, there exists a unique bornologically holomorphic function $F\colon (\bbC^\times)^2\to \clF(\bbC)$ such that $F|_{\Conf_2(\bbC^\times)}=f$ and $F|_U=\wt{g}$. 

Fix $0<r<R<\infty$. By \cref{lem:locality}, 
\begin{align*}
&\frac{1}{(2\pi\sqrt{-1})^2}\int_{\pdd D_R}\int_{\pdd D_r}F(z, w)\, dwdz=\sum_{j=0}^N(-1)^j\binom{N}{j}a_{(m+N-j)}(b_{(n+j)}c), \\
&\frac{1}{(2\pi\sqrt{-1})^2}\int_{\pdd D_R}\int_{\pdd D_r}F(z, w)\, dzdw=\sum_{j=0}^N(-1)^j\binom{N}{j}b_{(n+j)}(a_{(m+N-j)}c). 
\end{align*}
Moreover, by \cref{lem:homotint} and \cref{eq;compintcom}, 
\begin{align*}
\int_{\pdd D_R}\int_{\pdd D_r}F(z, w)\, dwdz
&=\int_{\pdd D_R}\int_{\pdd D_R}F(z, w)\, dwdz\\
&=
\int_{\pdd D_R}\int_{\pdd D_R}F(z, w)\, dzdw=
\int_{\pdd D_R}\int_{\pdd D_r}F(z, w)\, dzdw. 
\end{align*}
Therefore, we obtain the desired equation. 
\end{proof}

The proof of \cref{thm:FAtoVA} is now complete.
We conclude this section by explaining the relationship between \cref{thm:FAtoVA} and the construction of Costello and Gwilliam. 

In \cite[\S5.3]{CG}, they start with an $S^1\ltimes \bbC$-equivariant prefactorization algebra taking values in the category of complexes of differentiable vector spaces. For concreteness, we consider an $S^1\ltimes \bbC$-equivariant prefactorization algebra $\clF\colon \frU_\bbC\to \sfC(\CVS)$. To bring this into our setting, one simply composes $\clF$ with the symmetric monoidal functor $H^\bullet_\sep\colon \sfC(\CVS)\to \CVS$ which takes cohomology, yielding an $S^1\ltimes \bbC$-equivariant prefactorization algebra $H_\sep^\bullet\clF\colon \frU_\bbC\to \CVS$. However, for an open subset $U\subset \bbC$ that is invariant under the action of $S^1$, there are two candidates for weight spaces: one way is $(H_\sep^\bullet\clF)(U)_\Delta$ as before, and another is $H_\sep^\bullet(\clF_\Delta(U))$, where $\clF_\Delta(U)=\{\clF_\Delta(U)_n\}_{n\in \bbZ}$ is a subcomplex of $\clF(U)$ defined by 
\[
\clF_\Delta(U)_n\ceq \{a\in \clF(U)_n\mid \forall q\in S^1,\, \sigma_{(q, 0), U}(a)=q^\Delta a\}. 
\]
To extract a vertex algebra structure from $\clF$, Costello and Gwilliam impose a technical condition on the $S^1$-action of $\clF$, called tameness, which ensures the identification $(H_\sep^\bullet\clF)(U)_\Delta=H_\sep^\bullet(\clF_\Delta(U))$. We introduce the following bornological version of tameness, adapted to our setting. 

Consider the space of distributions $\clD(S^1)\ceq C^\infty_\rmc(S^1, \bbC)'$, that is, the continuous dual of $C^\infty_\rmc(S^1, \bbC)$ equipped with the canonical LF-topology. The linear space $\clD(S^1)$ is an associative algebra under convolution: for $\varphi, \psi\in \clD(S^1)$, their convolution $\varphi*\psi\in \clD(S^1)$ is defined by
\[
\varphi*\psi\colon C^\infty_\rmc(S^1, \bbC)\to \bbC, \quad f\mapsto \varphi(\psi*f),
\]
where $\psi*f\in C^\infty_\rmc(S^1, \bbC)$ denotes the convolution with a function
\[
\psi*f\colon S^1\to \bbC, \quad q\mapsto \psi(f(q\cdot-)). 
\] 
We have an element $\delta_q\in \clD(S^1)$ for each $q\in S^1$, called the delta-distribution at $q$. This is defined by
\[
\delta_q\colon C^\infty_\rmc(S^1, \bbC)\to \bbC, \quad f\mapsto f(q). 
\]
Note that the map $\delta\colon S^1\to \clD(S^1)$, $q\mapsto \delta_q$ is an injective monoid homomorphism. We also have an element $\pi_\Delta\in \clD(S^1)$ for each $\Delta\in \bbZ$ defined by
\[
\pi_\Delta\colon C^\infty_\rmc(S^1, \bbC)\to \bbC, \quad f\mapsto \frac{1}{2\pi\sqrt{-1}}\int_{S^1}z^{-\Delta-1}f(z)\,dz. 
\]
A direct calculation shows that 
\begin{equation}\label{eq:deltapi}
\delta_q*\pi_\Delta=q^{\Delta}\pi_\Delta=\pi_\Delta*\delta_q
\end{equation}
for all $q\in S^1$ and $\Delta\in \bbZ$. 
We equip $\clD(S^1)=C^\infty_\rmc(S^1, \bbC)'$ with the strong dual topology, i.e., the topology of uniform convergence on bounded subsets. 

Let $X^\bullet$ be a complex of convenient vector spaces endowed with an $S^1$-action  
\[
\sigma\colon S^1\to \Aut_{\sfC(\CVS)}(X^\bullet), \quad q\mapsto \sigma_q. 
\]
For each $\Delta\in \bbZ$, define a subcomplex $X_\Delta^\bullet$ of $X^\bullet$ by  
\begin{equation}\label{eq:cpxDwei}
X_\Delta^n\ceq \{x\in X^n\mid \forall q\in S^1,\, \sigma_q(x)=q^\Delta x\}. 
\end{equation}

\begin{dfn}\label{dfn:tame}
Let $X^\bullet$ be a complex of convenient vector spaces endowed with an $S^1$-action. We say that the $S^1$-action on $X^\bullet$ is \emph{tame} if there exists an algebra homomorphism 
\[
\rho\colon  \clD(S^1)\to \End_{\sfC(\CVS)}(X^\bullet) 
\]
satisfying the following conditions: 
\begin{clist}
\item
For $q\in S^1$, we have $\rho(\delta_q)=\sigma_q$. 

\item 
For $\Delta\in \bbZ$, the following diagram commutes: 
\[
\begin{tikzcd}[row sep=huge, column sep=huge]
X_\Delta \arrow[r, hook] \arrow[rd, hook] &
X \arrow[d, "\rho(\pi_\Delta)"] \\
&
X
\end{tikzcd}
\]

\item 
For $n\in \bbZ$ and $x\in X^n$, the map $\clD(S^1)\to X^n$, $\varphi\mapsto \rho(\varphi)_n(x)$ is bounded. 
\end{clist}
\end{dfn}

Let $X^\bullet$ be a complex of convenient vector spaces endowed with an $S^1$-action, and suppose that the $S^1$-action on $X^\bullet$ is tame. 
Given an algebra homomorphism $\rho\colon \clD(S^1)\to \End_{\sfC(\CVS)}(X^\bullet)$ satisfying conditions (i)--(iii) in \cref{dfn:tame}, we will, by abuse of notation, denote the action of $\varphi\in \clD(S^1)$ on $X^\bullet$ by 
\[
\varphi*(-)\ceq \rho(\varphi)\colon X^\bullet\to X^\bullet. 
\]

For $\Delta\in\bbZ$, \eqref{eq:deltapi} implies $\pi_\Delta\!*X^\bullet\subset X^\bullet_\Delta$. By condition (ii) in \cref{dfn:tame}, the morphism
\[
\pi_\Delta\!*(-)\colon X^\bullet\to X^\bullet_\Delta
\]
splits the inclusion $X^\bullet_\Delta\inj X^\bullet$. 

We now make several remarks on the topology of $\clD(S^1)$. First, since $C^\infty_\rmc(S^1, \bbC)$ is a Montel space, every pointwise convergent sequence in its strong dual $\clD(S^1)$ is convergent (see \cite[\S34.4]{T67}). In particular, we have
\[
\delta_1=\sum_{\Delta\in \bbZ}\pi_\Delta
\]
in $\clD(S^1)$, as this is just the Fourier expansion of the delta-distribution. 
Moreover, since $S^1$ is compact, the topology on $C^\infty_\rmc(S^1, \bbC)=C^\infty(S^1, \bbC)$ is actually Fr\'{e}chet. Hence, in the strong dual $\clD(S^1)$, every convergent sequence is bornologically convergent (see \cite[Result 52.28]{KM} and \cite[Corollary 12.5.9]{J81}). Therefore, for each $x\in X^n$, the series $\sum_{\Delta\in\bbZ}\pi_\Delta\!*x$ converges bornologically to $x$ in $X^n$.  

\smallskip

From the above discussion, we obtain the following lemma. Recall that we denote by $H^n$ and $H^n_\sep$ the cohomology taken in $\sfC(\CBVS)$ and $\sfC(\CVS)$, respectively. 

\begin{lem}\label{lem:tameiso}
Let $X^\bullet$ be a complex of convenient vector spaces endowed with an $S^1$-action, and suppose that the $S^1$-action on $X^\bullet$ is tame. Let $n\in \bbZ$. 
\begin{enumerate}
\item 
For every $\Delta\in \bbZ$, the canonical morphism $H^n(X_\Delta^\bullet)\to H^n(X^\bullet)$ induced by the inclusion $X^\bullet_\Delta\inj X^\bullet$ is injective. 

\item
For every $\Delta\in\bbZ$, the canonical morphism $H_\sep^n(X_\Delta^\bullet)\to H_\sep^n(X^\bullet)_\Delta$ induced by the inclusion $X^\bullet_\Delta\inj X^\bullet$ is bijective. 
\end{enumerate}
\end{lem}

The following proposition is a bornological version of \cite[Theorem 5.3.3]{CG}. Note that the use of bornological complex analysis allows us to dispense with the discreteness condition (iii) in \cite[Theorem 5.3.3]{CG}.

\begin{prp}\label{prp:CGconst}
Let $\clF\colon \frU_\bbC\to \sfC(\CVS)$ be an $S^1\ltimes \bbC$-equivariant prefactorization algebra satisfying the following conditions: 
\begin{clist}
\item 
For $0<R\le \infty$, the $S^1$-action on $\clF(D_R)$ is tame. 

\item
The $\bbC$-action makes $H_\sep^\bullet\clF\colon \frU_\bbC\to \CVS$ a holomorphically translation-equivariant prefactorization algebra. 

\item 
For $0<R\le \infty$, we have $H^n_\sep(\clF(D_R))=0$ ($n\neq 0$) and $H^0_\sep(\clF_\Delta(D_R))=0$ ($\Delta\ll0$). 

\item 
For $0<r<R\le \infty$ and $\Delta\in \bbZ$, the map $H_\sep^0(\clF^{D_r}_{D_R})\colon H_\sep^0(\clF_\Delta(D_r))\to H^0_\sep(\clF_\Delta(D_R))$ is a linear isomorphism. 
\end{clist}
Then the $S^1\ltimes \bbC$-equivariant prefactorization algebra $H^\bullet_\sep\clF\colon \frU_\bbC\to \CVS$ is amenably holomorphic. Consequently, the linear space
\[
\bfVR(H_\sep^\bullet\clF)=\bigoplus_{\Delta\in \bbZ}H^0_\sep(\clF_\Delta(D_R))
\]
admits the structure of a $\bbZ$-graded vertex algebra. 
\end{prp}

\begin{proof}
This follows immediately from \cref{lem:tameiso} and \cref{thm:FAtoVA}. 
\end{proof}

\begin{rmk}\label{rmk:CGconst}
As noted right after \cite[Theorem 5.3.3]{CG}, if conditions (iii) and (iv) are replaced by their analogues in every homological degree, without assuming concentration in degree zero, the same construction yields a vertex algebra object in the symmetric monoidal category of $\bbZ$-graded vector spaces.
\end{rmk}

\subsection{The super setting: construction of vertex superalgebras}
\label{ss:FAtoVSA}

We briefly describe how the construction of vertex algebras in \cref{ss:FAtoVA} extends to the super setting. 

A \emph{convenient vector superspace} is a convenient vector space $X$ equipped with a direct sum decomposition $X=X_{\ol{0}}\oplus X_{\ol{1}}$ as a convenient vector space. 
The $\bbZ/2\bbZ$-grading is called the \emph{parity}. We denote by $\sCVS$ the category of convenient vector superspaces with parity-preserving bounded linear maps. It is naturally a complete and cocomplete closed symmetric monoidal category. The braiding of $\sCVS$ is given by
\[
X\,\wt{\otimes}\,Y\to Y\,\wt{\otimes}\,X, \quad x\otimes y\mapsto (-1)^{p(x)p(y)}y\otimes x,  
\]
where $p(\cdot)$ denotes parity. 

We identify $\Spin(2)$ with $S^1$ so that the covering map is $q\mapsto q^2$. Accordingly, $\Spin(2)$ acts on $\bbC$ by $q\cdot z\ceq q^2z$. 
Let $\clF\colon \frU_\bbC\to \sCVS$ be a $\Spin(2)$-equivariant prefactorization algebra. For $\Delta\in \frac{1}{2}\bbZ$ and an open subset $U\subset \bbC$ invariant under the $\Spin(2)$-action, define
\[
\clF(U)_{\Delta}\ceq \{a\in \clF(U)\mid \forall q\in \Spin(2),\, \sigma_{q, U}(a)=q^{2\Delta}a\}, 
\]
where $\sigma_q$ denotes the $\Spin(2)$-equivariant structure of $\clF$. Then $\clF(U)_\Delta$ is a convenient vector superspace with parity decomposition given by 
\[
\clF(U)_{\Delta, \ol{i}}\ceq \clF(U)_\Delta\cap \clF(U)_{\ol{i}} \quad (i\in\{0, 1\}). 
\]

\begin{thm}\label{thm:superFAtoVA}
Let $\clF\colon \frU_\bbC\to \sCVS$ be a $\Spin(2)\ltimes \bbC$-equivariant prefactorization algebra satisfying the following conditions: 
\begin{clist}
\item 
The $\bbC$-action makes $\clF\colon \frU_\bbC\to \sCVS$ a holomorphically translation-equivariant prefactorization algebra. 

\item 
For $0<R\le \infty$, we have $\clF(D_R)_\Delta=0$ ($\Delta\ll0$). 

\item
For $0<r<R\le \infty$ and $\Delta\in \frac{1}{2}\bbZ$, the map $\clF^{D_r}_{D_R}\colon \clF(D_r)_\Delta\to \clF(D_R)_\Delta$ is a linear isomorphism.  
\end{clist}
Then, the linear superspace
\[
\bfVR^{\super}(\clF)\ceq \bigoplus_{\Delta\in \frac{1}{2}\bbZ}\clF(D_R)_\Delta, \quad \text{ where} \quad   \bfVR^{\super}(\clF)_{\ol{i}}\ceq \bigoplus_{\Delta\in \frac{1}{2}\bbZ}\clF(D_R)_{\Delta, \ol{i}} \quad (i\in\{0, 1\})
\]
admits the structure of a $\frac{1}{2}\bbZ$-graded vertex superalgebra. We call such $\clF$ an \emph{amenably holomorphic prefactorization algebra}.  
\end{thm}

\begin{proof}
The proof of \cref{thm:FAtoVA} works with minor modifications. 
\end{proof}

\section{Factorization envelopes and enveloping vertex algebras}
\label{s:FAEVA}

In this section, we construct an amenably holomorphic prefactorization algebra via the factorization envelope, starting from a Lie conformal algebra, and prove that the associated vertex algebra is isomorphic to the enveloping vertex algebra. This construction generalizes those of the Kac--Moody factorization algebra \cite[\S5.5]{CG} and the Virasoro factorization algebra \cite{W}. Moreover, the super analogue of this construction yields new factorization algebras that correspond to vertex superalgebras, such as the Neveu--Schwarz vertex superalgebra, the $N=2$ vertex superalgebra, and the $N=4$ vertex superalgebra. 

\subsection{The compactly supported Dolbeault complex}
\label{ss:compDol}

This subsection provides some recollections on the compactly supported Dolbeault complex. We use the following notations: 
\begin{itemize}
\item 
For $n\in \bbZ_{>0}$, we denote $[n]\ceq \{1, \ldots, n\}$. 

\item 
For $n\in \bbZ_{>0}$ and $m\in \bbN$, we denote 
\[
\bbN^n_m\ceq \{\nu=(\nu_1, \ldots, \nu_n)\in \bbN^n\mid \nu_1+\cdots+\nu_n=m\}. 
\]
\end{itemize}

Let $X$ be a complex manifold. For an open subset $U\subset X$, consider the $\bbC$-linear space $\clA^{0, k}_X(U)$ of $C^\infty$-class complex $(0, k)$-differentials on $U$. Choose a local coordinate system $\clS=\{(U_\alpha, z_1^\alpha, \ldots, z_n^\alpha)\}_{\alpha\in A}$ of $U$, and denote by $\Lambda_k(U, \clS)$ the set of tuples $\lambda=(\alpha, K, m, I)$ consisting of 
\begin{itemize}
\item 
$\alpha\in A$,  

\item 
a compact set $K\subset U_\alpha$, 

\item 
$m\in \bbN$, 

\item
$I\subset [n]$ such that $\#I=k$. 
\end{itemize}
For $\lambda=(\alpha, K, m, I)\in \Lambda_k(U, \clS)$, define a semi-norm $\|\cdot\|_\lambda\colon \clA^{0, k}_X(U)\to \bbR$ by
\[
\|\xi\|_\lambda\ceq \sup\bigl\{|\pdd^\nu_xf_I(p)| \bigm|p\in K, \nu\in \bbN^{2n}_m\bigr\}.
\]
Here $x=(x_1, \ldots, x_{2n})$ denotes the real coordinates given by $z^\alpha_i=x_{2i-1}+\sqrt{-1}x_{2i}$, and $f_I$ is the component of the coordinate representation
\[
\xi|_{U_\alpha}=\sum_{\substack{I\subset [n]\\ \#I=k}}f_Id\ol{z}_I^\alpha.
\]
The linear space $\clA^{0, k}_X(U)$ equipped with the locally convex topology defined by the set of semi-norms $\{\|\cdot\|_\lambda\mid \lambda\in \Lambda_k(U, \clS)\}$ is a Fr\'{e}chet space. This topology is independent of the choice of the local coordinate system of $U$. 

\begin{lem}\label{lem:olpddcont}
Let $U\subset X$ be an open subset. 
\begin{enumerate}
\item 
The linear map $\ol{\pdd}\colon \clA^{0, k}_X(U)\to \clA^{0, k+1}_X(U)$ is continuous. 

\item 
The bilinear map $\clA^{0, k}_X(U)\times \clA^{0, l}_X(U)\to \clA^{0, k+l}_X(U)$, $(\xi, \eta)\mapsto \xi\wedge \eta$ is continuous. 
\end{enumerate}
\end{lem}

\begin{proof}
Fix a local coordinate system $\clS=\{(U_\alpha, z_1^\alpha, \ldots, z_n^\alpha)\}_{\alpha\in A}$ of $U$. 

(1) Take $\lambda=(\alpha, K, m, I)\in \Lambda_{k+1}(U, \clS)$. For $\xi\in \clA^{0, k}_X(U)$, denote by $\xi|_{U_\alpha}=\sum_Jf_Jd\ol{z}_J^\alpha$ the coordinate representation. Then 
\[
\ol{\pdd}\xi|_{U_\alpha}=
\sum_{i=1}^n\sum_{\substack{J\subset [n]\\ \#J=k}}\pd{f_J}{\ol{z}_i^\alpha}d\ol{z}_i^\alpha\wedge d\ol{z}_J^\alpha=
\sum_{i=1}^n\sum_{\substack{J\subset [n]\\ \#J=k, i\notin J}}\epsilon(i, J)\pd{f_J}{\ol{z}_i^\alpha}d\ol{z}_{\{i\}\cup J}^\alpha,
\]
where $\epsilon(i, J)\in\{\pm1\}$. Hence, 
\[
\|\ol{\pdd}\xi\|_\lambda=\sup\Bigl\{\bigl|\pdd_x^\nu\sum_{i\in I}\epsilon(i, I\bs\{i\})\pd{f_{I\bs\{i\}}}{\ol{z}_i^\alpha}(p)\bigr|\Bigm| p\in K, \nu\in \bbN^{2n}_m\Bigr\}. 
\]
For $p\in K$ and $\nu\in \bbN^{2n}_m$, we have
\[
\bigl|\pdd_x^\nu\sum_{i\in I}\epsilon(i, I\bs\{i\})\pd{f_{I\bs\{i\}}}{\ol{z}_i^\alpha}(p)\bigr|
\le
\sum_{i\in I}\bigl|\pdd_x^\nu\pd{f_{I\bs\{i\}}}{\ol{z}_i^\alpha}(p)\bigr|\le
\sum_{i\in I}\Bigl(\frac{1}{2}\bigl|\pdd^\nu_x\pd{f_{I\bs\{i\}}}{x_{2i-1}^\alpha}(p)\bigr|+\frac{1}{2}\bigl|\pdd^\nu_x\pd{f_{I\bs\{i\}}}{x_{2i}^\alpha}(p)\bigr|\Bigr).
\]
Set $\lambda_i\ceq (\alpha, K, m+1, I\bs\{i\})\in \Lambda_k(U, \clS)$. Then
\[
\|\ol{\pdd}\xi\|_\lambda\le \sum_{i\in I}\|\xi\|_{\lambda_i}. 
\]
This proves (1). 

(2) Take $\lambda=(\alpha, K, m, I)\in \Lambda_{k+l}(U, \clS)$. Since $\clA^{0, k}_X(U)$ and $\clA^{0, l}_X(U)$ are Fr\'{e}chet spaces, it suffices to show that the bilinear map $\wedge$ is separately continuous. Fix $\xi\in \clA^{0, k}_X(U)$, and denote by $\xi|_{U_\alpha}=\sum_Jf_Jd\ol{z}_J^\alpha$ the coordinate representation. For $\eta\in \clA^{0, l}_X(U)$ with the coordinate representation $\eta|_{U_\alpha}=\sum_Jg_Jd\ol{z}^\alpha_J$, 
\[
(\xi\wedge\eta)|_{U_\alpha}=
\sum_{\substack{J_1\subset [n]\\ \#J_1=k}}\sum_{\substack{J_2\subset [n]\\ \#J_2=l}}f_{J_1}g_{J_2}d\ol{z}_{J_1}^\alpha\wedge d\ol{z}_{J_2}^\alpha=
\sum_{\substack{J_1, J_2\subset [n]\\ \#J_1=k, \#J_2=l\\ J_1\cap J_2=\varnothing}}\epsilon(J_1, J_2)f_{J_1}g_{J_2}d\ol{z}_{J_1\cup J_2}^\alpha, 
\]
where $\epsilon(J_1, J_2)\in \{\pm1\}$. Hence, 
\[
\|\xi\wedge \eta\|_\lambda=
\sup\Bigl\{\bigl|\pdd_x^\nu\sum_{\substack{J\subset [n]\\ \#J=l}}\epsilon(I\bs J, J)f_{I\bs J}g_J(p)\bigr|\Bigm| p\in K, \nu\in \bbN^{2n}_m\Bigr\}. 
\]
For $p\in K$ and $\nu\in \bbN^{2n}_m$, we have
\[
\bigl|\pdd_x^\nu\sum_{\substack{J\subset [n]\\ \#J=l}}\epsilon(I\bs J, J)f_{I\bs J}g_J(p)\bigr|\le 
\sum_{\substack{J\subset [n]\\ \#J=l}}|\pdd_x^\nu(f_{I\bs J}g_J)(p)|\le 
\sum_{\substack{J\subset [n]\\ \#J=l}}\sum_{j=0}^m\sum_{\mu\in\bbN^{2n}_j}\binom{\nu}{\mu}|\pdd_x^{\nu-\mu}f_{I\bs J}(p)||\pdd_x^\mu g_J(p)|. 
\]
Here, for $\mu=(\mu_1, \ldots, \mu_{2n}),\, \nu=(\nu_1, \ldots, \nu_{2n})\in \bbN^{2n}$, we denote 
\[
\binom{\nu}{\mu}\ceq \binom{\nu_1}{\mu_1}\cdots \binom{\nu_{2n}}{\mu_{2n}}. 
\]
Set $\lambda_{j, J}\ceq (\alpha, K, j, J)\in \Lambda_l(U, \clS)$ and $C_{j, J}\ceq \sup\bigl\{\sum_{\mu\in \bbN^{2n}_j}\binom{\nu}{\mu}|\pdd_x^{\nu-\mu}f_{I\bs J}(p)|\bigm| p\in  K, \nu\in \bbN^{2n}_m\bigr\}$. 
Then
\[
\|\xi\wedge \eta\|_\lambda\le 
\sum_{j=0}^m\sum_{\substack{J\subset [n]\\\#J=l}}C_{j, J}\|\eta\|_{\lambda_{j, J}}
\]
This proves (2). 
\end{proof}

\cref{lem:olpddcont} (1) implies that $(\clA^{0, \bullet}_X(U), \ol{\pdd})$ is a cochain complex of complete Hausdorff locally convex spaces, and hence a cochain complex of convenient vector spaces. 

\begin{lem}\label{lem:restcont}
For open subsets $U, V\subset X$ such that $U\supset V$, the linear map $\clA^{0, k}_X(U)\to \clA^{0, k}_X(V)$, $\xi\mapsto \xi|_V$ is continuous. 
\end{lem}

\begin{proof}
Fix a local coordinate system $\clS=\{(U_\alpha, z_1^\alpha, \ldots, z_n^\alpha)\}_{\alpha\in A}$ of $U$. Then $\clS|_V=\{(U_\alpha\cap V, z_1^\alpha, \ldots, z_n^\alpha)\}_{\alpha\in A}$ is a local coordinate system of $V$. Note that $\Lambda_k(V, \clS|_V)\subset \Lambda_k(U, \clS)$. For $\lambda\in \Lambda_k(V, \clS|_V)$ and $\xi\in \clA^{0, k}_X(U)$, we have $\|\xi|_V\|_\lambda\le \|\xi\|_\lambda$. Thus the lemma follows. 
\end{proof}

\cref{lem:restcont} implies that the map
\[
\clA^{0, \bullet}_X\colon \frU_X^{\op}\to \sfC(\CVS), \quad U\mapsto \clA^{0, \bullet}_X(U)
\]
becomes a sheaf in the obvious way. 

\begin{lem}
For an open subset $U\subset X$ and a compact set $K\subset U$, 
\[
\clA^{0, k}_X(U)_K\ceq \{\xi\in \clA^{0, k}_X(U)\mid \supp \xi\subset K\}
\]
is a closed linear subspace of $\clA^{0, k}_X(U)$. 
\end{lem}

\begin{proof}
It suffices to show that for a sequence $\{\xi_m\}_{m\in \bbN}\subset \clA^{0, k}_X(U)_K$ which converges to $\xi\in \clA^{0, k}_X(U)$, we have $\xi\in \clA^{0, k}_X(U)_K$. Fix a local coordinate system $\clS=\{(U_\alpha, z_1^\alpha, \ldots, z_n^\alpha)\}_{\alpha\in A}$. For $p\in U$ such that $\xi_p\neq 0$, take $\alpha\in A$ so that $p\in U_\alpha$. Denote by 
\[
\xi_m|_{U_\alpha}=\sum_If_{m, I}d\ol{z}_I^\alpha, \quad 
\xi|_{U_\alpha}=\sum_If_Id\ol{z}_I^\alpha
\]
the coordinate representations. Then $\lim_{m\to \infty}f_{m, I}(p)=f_I(p)$ for each $I\subset [n]$ with $\#I=k$ since $\lim_{m\to \infty}\|\xi_m-\xi\|_\lambda=0$, where $\lambda=(\alpha, \{p\}, 0, I)$. Now, if $p\notin K$, then $\xi_{m, p}=0$ for any $m\in \bbN$, and hence $f_{m, I}(p)=0$ for any $I\subset [n]$ with $\#I=k$. This implies that $f_I(p)=0$ for any $I\subset [n]$ with $\#I=k$, or equivalently $\xi_p=0$, which is a contradiction. Thus, we have $p\in K$. 
\end{proof}

For an open subset $U\subset X$, consider the $\bbC$-linear space $\clA^{0, k}_{X, \rmc}(U)$ of compactly supported $C^\infty$-class complex $(0, k)$-differentials on $U$. Note that 
\[
\clA^{0, k}_{X, \rmc}(U)=\bigcup_{K}\clA^{0, k}_{X, \rmc}(U)_K, 
\]
where $K$ runs over all compact sets of $U$. We equip $\clA^{0, k}_{X, \rmc}(U)$ with the inductive limit topology induced by the inclusions $\clA^{0, k}_X(U)_K\inj \clA^{0, k}_{X, \rmc}(U)$. Then $\clA^{0, k}_{X, \rmc}(U)$ is an LF-space, namely, it is  a countable strict inductive limit of Fr\'{e}chet spaces (see \cite[Chapter 13]{T67} for details on LF-spaces). 

Since $\supp \ol{\pdd}\xi\subset \supp \xi$ for $\xi\in \clA^{0, k}_X(U)$, we have a linear map
\[
\ol{\pdd}\colon \clA^{0, k}_{X, \rmc}(U)\to \clA^{0, k+1}_{X, \rmc}(U).
\]
Moreover, since $\supp(\xi\wedge \eta)\subset \supp \xi\cap \supp \eta$ for $\xi\in \clA^{0, k}_X(U)$ and $\eta\in \clA^{0, l}_X(U)$, we have a bilinear map
\[
\clA^{0, k}_{X, \rmc}(U)\times \clA^{0, l}_{X, \rmc}(U)\to \clA^{0, k+l}_{X, \rmc}(U), \quad
(\xi, \eta)\mapsto \xi\wedge \eta. 
\]

\begin{lem}\label{lem:cptpddcont}
Let $U\subset X$ be an open subset. 
\begin{enumerate}
\item 
The linear map $\ol{\pdd}\colon \clA^{0, k}_{X, \rmc}(U)\to \clA^{0, k+1}_{X, \rmc}(U)$ is continuous. 

\item 
The bilinear map $\clA^{0, k}_{X, \rmc}(U)\times \clA^{0, l}_{X, \rmc}(U)\to \clA^{0, k+l}_{X, \rmc}(U)$, $(\xi, \eta)\mapsto \xi\wedge\eta$ is separately continuous. 
\end{enumerate}
\end{lem}

\begin{proof}
(1) For a compact set $K\subset U$, we have $\ol{\pdd}\clA^{0, k}_X(U)_K\subset \clA^{0, k+1}_X(U)_K$. By \cref{lem:olpddcont} (1), the linear map $\ol{\pdd}\colon \clA^{0, k}_X(U)_K\to \clA^{0, k+1}_X(U)_K$ is continuous. Hence $\ol{\pdd}\colon \clA^{0, k}_{X, \rmc}(U)\to \clA^{0, k+1}_{X, \rmc}(U)$ is continuous. 

(2) Fix $\xi\in \clA^{0, k}_{X, \rmc}(U)$. For a compact set $K\subset U$, we have $\xi\wedge \clA^{0, l}_X(U)_K\subset \clA^{0, k+l}_X(U)_K$. By \cref{lem:olpddcont} (2), the linear map $\xi\wedge(-)\colon \clA^{0, l}_X(U)_K\to \clA^{0, k+l}_X(U)_K$ is continuous. Hence the linear map $\xi\wedge(-)\colon \clA^{0, l}_{X, \rmc}(U)\to \clA^{0, k+l}_{X, \rmc}(U)$ is continuous. 
\end{proof}

\cref{lem:cptpddcont} (1) implies that $(\clA^{0, \bullet}_{X, \rmc}(U), \ol{\pdd})$ is a cochain complex of complete Hausdorff locally convex spaces, and hence a cochain complex of convenient vector spaces.  

Let $U, V\subset X$ be open subsets such that $U\subset V$. For $\xi\in \clA^{0, k}_{X, \rmc}(U)$, there exists a unique element $(\clA^{0, k}_{X, \rmc})^U_V(\xi)\in \clA^{0, k}_X(V)$ satisfying 
\[
(\clA^{0, k}_{X, \rmc})^U_V(\xi)|_U=\xi, \quad 
(\clA^{0, k}_{X, \rmc})^U_V(\xi)|_{V\bs\supp \xi}=0. 
\]
Since $\supp (\clA^{0, k}_{X, \rmc})^U_V(\xi)\subset \supp \xi$, we have $(\clA^{0, k}_{X, \rmc})^U_V(\xi)\in \clA^{0, k}_{X, \rmc}(V)$. Hence we obtain a linear map
\[
(\clA^{0, k}_{X, \rmc})^U_V\colon \clA^{0, k}_{X, \rmc}(U)\to \clA^{0, k}_{X, \rmc}(V), \quad
\xi\mapsto (\clA^{0, k}_{X, \rmc})^U_V(\xi). 
\]

\begin{lem}
For open subsets $U, V\subset X$ such that $U\subset V$, the linear map $(\clA^{0, k}_{X, \rmc})^U_V\colon \clA^{0, k}_{X, \rmc}(U)\to \clA^{0, k}_{X, \rmc}(V)$ is continuous. 
\end{lem}

\begin{proof}
For a compact set $K\subset U$, we have $(\clA^{0, k}_{X, \rmc})^U_V(\clA^{0, k}_X(U)_K)\subset \clA^{0, k}_X(V)_K$, so it suffices to show that $(\clA^{0, k}_{X, \rmc})^U_V\colon \clA_X^{0, k}(U)_K\to \clA^{0, k}_X(V)_K$ is continuous. Fix a local coordinate system $\clS=\{(V_\alpha, z_1^\alpha, \ldots, z_n^\alpha)\}_{\alpha\in A}$ of $V$. Then $\clS|_U=\{(U\cap V_\alpha, z_1^\alpha, \ldots, z_n^\alpha)\}_{\alpha\in A}$ is a local coordinate system of $U$. For $\mu=(\alpha, L, m, I)\in \Lambda(V, \clS)$ and $\xi\in \clA_X^{0, k}(U)_K$, we have $\|(\clA_{X, \rmc}^{0, k})^U_V(\xi)\|_\mu\le \|\xi\|_\lambda$, where $\lambda\ceq (\alpha, K\cap L, m, I)\in \Lambda(U, \clS|_U)$. Hence the claim follows. 
\end{proof}

For open subsets $U, V\subset X$ such that $U\subset V$,  
\[
(\clA^{0, \bullet}_{X, \rmc})^U_V\ceq \bigl\{(\clA^{0, k}_{X, \rmc})^U_V\bigr\}_{k\in \bbZ}\colon \clA^{0, \bullet}_{X, \rmc}(U)\to \clA^{0, \bullet}_{X, \rmc}(V)
\]
is a morphism in $\sfC(\CVS)$. The map
\[
\clA^{0, \bullet}_{X, \rmc}\colon \frU_X\to \sfC(\CVS), \quad U\mapsto \clA^{0, \bullet}_{X, \rmc}(U)
\]
becomes a precosheaf with the extension morphisms $(\clA^{0, \bullet}_{X, \rmc})^U_V$.

\smallskip

We now focus on the case $X=\bbC^n$, and provide some computations of compactly supported Dolbeault cohomology for later use. For short, we write $\clA^{0, \bullet}_\rmc=\clA^{0, \bullet}_{\bbC^n, \rmc}$. 

For an open subset $U\subset \bbC^n$ and $\xi\in \clA^{0, n}_\rmc(U)$, we have a continuous linear map
\[
\phi(\xi)\colon \Omega^n(U)\to \bbC, \quad \eta\mapsto \int_U\xi\wedge \eta. 
\]
Here $\Omega^n(U)$ denotes the space of holomorphic $n$-differentials on $U$, equipped with the canonical Fr\'{e}chet topology. 
The following statement is a special case of the Serre duality theorem for Stein manifolds. 

\begin{fct}[{\cite{S}, \cite[Corollary 5.4.5]{CG}}]\label{fct:Serre}
Let $U\subset \bbC^n$ be a Stein open subset. We have $H^k(\clA^{0, \bullet}_\rmc(U))=0$ whenever $k\neq n$. Moreover, 
\[
H^n(\clA^{0, \bullet}_\rmc(U))\to \Omega^n(U)', \quad [\xi]\mapsto \phi(\xi)
\]
is an isomorphism of topological vector spaces, where $H^n(\clA^{0, \bullet}_\rmc(U))$ is equipped with the quotient topology, and $\Omega^n(U)'$ with the strong dual topology of the canonical Fr\'{e}chet topology. 
\end{fct}

Consider the action of $S^1$ on $\bbC^n$ given by $q\cdot (z_1, \ldots, z_n)\ceq (qz_1, \ldots, qz_n)$. 
Let $U\subset \bbC^n$ be an open subset invariant under this action. For $k\in \bbZ$, $\varphi\in\clD(S^1)$, and $f\in C^\infty(U, \bbC)$, define 
\[
(\varphi*_kf)(z_1, \ldots, z_n)\ceq\varphi(q\mapsto q^kf(q^{-1}z_1, \ldots, q^{-1}z_n)). 
\]
Then $\varphi*_kf$ is a function of class $C^\infty$. 
Moreover, $\supp(\varphi*_kf)\subset S^1\cdot\supp f$. 
Hence we obtain a bilinear map 
\[
*_k\colon \clD(S^1)\times C^\infty_\rmc(U, \bbC)\to C^\infty_\rmc(U, \bbC), \quad (\varphi, f)\mapsto \varphi*_kf. 
\]

\begin{lem}\label{lem:starksepcont}
Let $U\subset \bbC^n$ be an open subset invariant under the $S^1$-action, and let $k\in \bbZ$. 
\begin{enumerate}
\item
The bilinear map $*_k\colon \clD(S^1)\times C^\infty_\rmc(U, \bbC)\to C^\infty_\rmc(U, \bbC)$ is separately continuous. 

\item 
For $\varphi \in \clD(S^1)$ and $f\in C^\infty_\rmc(U, \bbC)$, 
\[
\pdd_{\ol{z}_i}(\varphi*_kf)=\varphi*_{k+1}\pdd_{\ol{z}_i}f. 
\]
\end{enumerate}
\end{lem}

\begin{proof}
We first claim that,  for $\nu\in \bbN^{2n}$ and $z\in U$, 
\[
\pdd_x^\nu(\varphi*_kf)(z)=\varphi(q\mapsto q^k\pdd_x^\nu (f(q^{-1}\cdot z)),
\]
where $x=(x_1, \ldots, x_{2n})$ denotes the real coordinates given by $z_i=x_{2i-1}+\sqrt{-1}x_{2i}$. 
The function
\[
U\times S^1\to \bbC, \quad (z, q)\mapsto q^kf(q^{-1}\cdot z)
\]
is of class $C^\infty$. Hence, by the exponential law for functions of class $C^\infty$ (see \cite[Lemma 27.17]{KM}), 
\[
F\colon U\to C^\infty(S^1, \bbC), \quad (z_1, \ldots z_n)\mapsto (q\mapsto q^kf(q^{-1}\cdot z))
\]
is of class $C^\infty$. Clearly, $\varphi*_kf=\varphi\circ F$. Since $\varphi$ is continuous, we have $\pdd_x^\nu(\varphi*_kf)(z)=\varphi(\pdd_x^\nu F(z))$. For $q\in S^1$, since the evaluation map $\ev_q\colon C^\infty(S^1, \bbC)\to \bbC$ is continuous, 
\[
\ev_q(\pdd_x^\nu F(z))=\pdd_x^\nu\ev_q(F(z))=
q^k\pdd_x^\nu f(q^{-1}\cdot z). 
\]
Hence
\[
\pdd_x^\nu F(z)\colon S^1\to \bbC, \quad q\mapsto q^k\pdd_x^\nu f(q^{-1}\cdot z). 
\]
This proves the claim. 

(1) Fix $\varphi\in \clD(S^1)$. For a compact set $K\subset U$, we have $\varphi*_k(C^\infty(U, \bbC)_K)\subset C^\infty(U, \bbC)_{S^1\cdot K}$. Since $S^1$ is compact, the distribution $\varphi$ is of finite order. Hence, for $\nu\in \bbN^{2n}$, there exists $C\in \bbR_{>0}$ and $N\in \bbN$ such that 
\[
\sup_{z\in S^1\cdot K}|\pdd_x^\nu(\varphi*_kf)(z)|\le C\max_{\substack{\mu\in \bbN^{2n}\\ |\mu|\le |\nu|+N}}\sup_{z\in K}|\pdd^\mu_xf(z)| \quad (f\in C^\infty(U, \bbC)_K). 
\]
This shows that $\varphi*_k(-)\colon C^\infty(U, \bbC)_K\to C^\infty(U, \bbC)_{S^1\cdot K}$ is continuous. Thus $\varphi*_k(-)\colon C^\infty_\rmc(U, \bbC)\to C^\infty_\rmc(U, \bbC)$ is continuous. Next, fix $f\in C^\infty_\rmc(U, \bbC)$. We have $\clD(S^1)*_kf\subset C^\infty(U, \bbC)_K$, where $K\ceq S^1\cdot \supp f$. For $\nu\in \bbN^{2n}$, 
\[
B\ceq\{q\mapsto q^k\pdd_x^\nu f(q^{-1}\cdot z)\mid z\in K\}
\]
is bounded in $C^\infty(S^1, \bbC)$, and 
\[
\sup_{z\in K}|\pdd_x^\nu(\varphi*_kf)(z)|\le \sup_{g\in B}|\varphi(g)| \quad (\varphi\in \clD(S^1)). 
\]
This shows that $(-)*_kf\colon \clD(S^1)\to C^\infty(U, \bbC)_K$ is continuous. Thus $(-)*_kf\colon \clD(S^1)\to C^\infty_\rmc(U, \bbC)$ is continuous. 

(2) For $z\in U$, 
\[
\pdd_{\ol{z}_i}(\varphi*_kf)(z)=\varphi(q\mapsto q^k\pdd_{\ol{z}_i}f(q^{-1}\cdot z))=\varphi(q\mapsto q^{k+1}(\pdd_{\ol{z}_i}f)(q^{-1}\cdot z))=(\varphi*_{k+1}\pdd_{\ol{z}_i}f)(z). 
\qedhere
\]
\end{proof}

For an open subset $U\subset \bbC^n$ invariant under the $S^1$-action, define an $S^1$-action on $\clA^{0, \bullet}_\rmc(U)$ by
\[
\sigma_{q, U}(\xi)\colon z\mapsto \xi_{q^{-1}z} \quad (q\in S^1,\, \xi\in \clA^{0, k}_\rmc(U)). 
\]

\begin{lem}\label{lem:HDeltainj}
Let $U\subset \bbC^n$ be an open subset invariant under the $S^1$-action. The above $S^1$-action on $\clA^{0, \bullet}_\rmc(U)$ is tame. Consequently, for $k, \Delta\in \bbZ$, the canonical morphism $H^k(\clA_\rmc^{0, \bullet}(U)_\Delta)\to H^k(\clA_\rmc^{0, \bullet}(U))$ induced by the inclusion $\clA^{0, \bullet}_\rmc(U)_\Delta\inj\clA^{0, \bullet}_\rmc(U)$ is injective.  
\end{lem}

\begin{proof}
Define an action of $\clD(S^1)$ on $\clA^{0, \bullet}_\rmc(U)$ by 
\[
\varphi*\xi\ceq
\sum_{\substack{I\subset [n]\\\#I=k}}(\varphi*_kf_I)d\ol{z}_I \qquad (\xi=\sum_{\substack{I\subset [n]\\\#I=k}}f_Id\ol{z_I}\in \clA^{0, k}_\rmc(U)). 
\]
Then \cref{lem:starksepcont} implies that $\varphi*(-)\colon \clA^{0, \bullet}_\rmc(U)\to \clA^{0, \bullet}_\rmc(U)$ is a morphism of complexes of convenient vector spaces. It is straightforward to verify conditions (i) and (ii) in \cref{dfn:tame}. Moreover, condition (iii) in \cref{dfn:tame} follows from \cref{lem:starksepcont} (1). 
The last assertion is immediate from \cref{lem:tameiso} (1).
\end{proof}

\begin{lem}\label{lem:DRiHm}
Let $0<R_i\le \infty$ and $\Delta\in \bbZ$. 
\begin{enumerate}
\item
For $\xi\in \clA^{0, n}_\rmc(D_{R_1}\times \cdots \times D_{R_n})_\Delta$ and $\alpha_i\in \bbZ_{>0}$ such that $\alpha_1+\cdots +\alpha_n\neq \Delta$, 
\[
\phi(\xi)(z_1^{\alpha_1-1}\cdots z_n^{\alpha_n-1}dz_1\wedge\cdots \wedge dz_n)=0. 
\]

\item 
The following linear map is bijective: 
\begin{align*}
H^n(\clA^{0, \bullet}_\rmc(D_{R_1}\times \cdots \times D_{R_n})_\Delta)
&\to \bigoplus_{\substack{\Delta_1, \ldots, \Delta_n\in \bbZ_{>0}\\ \Delta_1+\cdots +\Delta_n=\Delta}}\bbC, \\
[\xi] \  &\mapsto \ \bigl\{\phi(\xi)(z_1^{\Delta_1-1}\cdots z_n^{\Delta_n-1}dz_1\wedge \cdots \wedge dz_n)\bigr\}. 
\end{align*}
Here, if $\Delta<n$, the right-hand side is understood to be $\{0\}$. 
\end{enumerate}
\end{lem}

\begin{proof}
Write $U\ceq D_{R_1}\times \cdots \times D_{R_n}$. 

(1) For $q\in S^1$, we have
\begin{align*}
q^\Delta \phi(\xi)(z_1^{\alpha_1-1}\cdots z_n^{\alpha_n-1}dz_1\wedge \cdots \wedge dz_n)
&=
\phi(\sigma_{q, U}(\xi))(z_1^{\alpha_1-1}\cdots z_n^{\alpha_n-1}dz_1\wedge \cdots \wedge dz_n)\\
&=
q^{\alpha_1+\cdots +\alpha_n}\phi(\xi)(z_1^{\alpha_1-1}\cdots z_n^{\alpha_n-1}dz_1\wedge \cdots \wedge dz_n). 
\end{align*}
Hence if $\alpha_1+\cdots +\alpha_n\neq \Delta$, then $\phi(\xi)(z_1^{\alpha_1-1}\cdots z_n^{\alpha_n-1}dz_1\wedge \cdots \wedge dz_n)=0$. 

(2) If $[\xi]\in H^n(\clA^{0, \bullet}_\rmc(U)_\Delta)$ satisfies $\phi(\xi)(z_1^{\Delta_1-1}\cdots z_n^{\Delta_n-1}dz_1\wedge \cdots \wedge dz_n)=0$ for all $\Delta_1, \ldots, \Delta_n\in \bbZ_{>0}$ with $\Delta_1+\cdots +\Delta_n=\Delta$, then (1) implies 
\[
\phi(\xi)(z_1^{\alpha_1-1}\cdots z_n^{\alpha_n-1}dz_1\wedge\cdots \wedge dz_n)=0 \quad (\alpha_1, \ldots, \alpha_n\in\bbZ_{>0}). 
\]
Since $\phi(\xi)\colon \Omega^n(U)\to \bbC$ is continuous, we have $\phi(\xi)=0$. Hence it follows from \cref{fct:Serre} and \cref{lem:HDeltainj} that $[\xi]=0$ in $H^n(\clA^{0, \bullet}_{\rmc}(U)_\Delta)$. This shows the injectivity. 

Choose a nonzero function $f\colon \bbR^n\to \bbR_{\ge0}$ of class $C^\infty$ such that $\supp f\subset (0, R_1)\times \cdots \times (0, R_n)$. Then, for $\Delta_i\in \bbZ_{>0}$ with $\Delta_1+\cdots +\Delta_n=\Delta$, we have
\[
\xi\ceq f(|z_1|, \ldots, |z_n|)\ol{z}_1^{\Delta_1-1}\cdots\ol{z}_n^{\Delta_n-1}d\ol{z}_1\wedge\cdots \wedge d\ol{z}_n\in \clA^{0, n}_\rmc(U)_\Delta. 
\]
For $\alpha_i\in \bbZ_{>0}$ with $\alpha_1+\cdots +\alpha_n=\Delta$, 
\[
\phi(\xi)(z_1^{\alpha_1-1}\cdots z_n^{\alpha_n-1}dz_1\wedge\cdots \wedge dz_n)
\begin{cases}
=0 & \text{if}\ (\alpha_1, \ldots, \alpha_n)\neq(\Delta_1, \ldots, \Delta_n)\\
\neq 0 & \text{if}\ (\alpha_1, \ldots, \alpha_n)=(\Delta_1, \ldots, \Delta_n).
\end{cases}
\]
This shows the surjectivity. 
\end{proof}

\begin{lem}\label{lem:ArRHn}
Let $0<r_i<R_i\le \infty$ and $\Delta\in \bbZ$. 
\begin{enumerate}
\item 
For $\xi\in \clA^{0, n}_\rmc(A_{r_1, R_1}\times \cdots \times A_{r_n, R_n})_\Delta$ and $\alpha_i\in \bbZ$ such that $\alpha_1+\cdots +\alpha_n\neq \Delta$, 
\[
\phi(\xi)(z_1^{\alpha_1-1}\cdots z_n^{\alpha_n-1}dz_1\wedge\cdots \wedge dz_n)=0.
\] 

\item
The following linear map is injective: 
\begin{align*}
H^n(\clA^{0, \bullet}_\rmc(A_{r_1, R_1}\times \cdots \times A_{r_n, R_n})_\Delta)
&\to \prod_{\substack{\Delta_1, \ldots, \Delta_n\in \bbZ\\ \Delta_1+\cdots +\Delta_n=\Delta}}\bbC, \\
[\xi]\ 
&\mapsto \ \bigl\{\phi(\xi)(z_1^{\Delta_1-1}\cdots z_n^{\Delta_n-1}dz_1\wedge \cdots \wedge dz_n)\bigr\}. 
\end{align*}
\end{enumerate}
\end{lem}

\begin{proof}
This can be proved in the similar way as \cref{lem:DRiHm}. 
\end{proof}

\begin{cor}\label{cor:ArRH1}
Let $0<r<R\le \infty$ and $\Delta\in \bbZ$. The following linear map is bijective: 
\[
H^1(\clA^{0, \bullet}_{\rmc}(A_{r, R})_\Delta)\to \bbC, \quad
[\xi]\mapsto \phi(\xi)(z^{\Delta-1}dz). 
\]
\end{cor}

\begin{proof}
By \cref{lem:ArRHn}, the linear map in the statement is injective. Hence it suffices to show that $H^1(\clA^{0, \bullet}_{\rmc}(A_{r, R})_\Delta)\neq 0$. Choose a nonzero function $f\colon \bbR\to \bbR_{\ge0}$ of class $C^\infty$ such that $\supp f\subset (r, R)$. Then we have $\xi\ceq f(|z|)\ol{z}^{\Delta-1}d\ol{z}\in \clA^{0, 1}_{\rmc}(A_{r, R})_\Delta$ and $\phi(\xi)(z^{\Delta-1}dz)\neq 0$. 
\end{proof}

\subsection{Higher dimensional analogue of mode Lie algebras}
\label{ss:HDCLA}

We now return to the general setting. 
Let $X$ be a complex manifold, let $L$ be a Lie conformal algebra, and let $D\colon \clA^{0, \bullet}_X\to \clA^{0, \bullet}_X$ be a morphism of sheaves with values in $\sfC(\CVS)$ satisfying  
\[
D_U(\xi\wedge \eta)=D_U\xi\wedge \eta+\xi\wedge D_U\eta \quad (\xi\in \clA^{0, k}_X(U), \eta\in \clA^{0, l}_X(U))
\]
for each open subset $U\subset X$. We equip $L$ with the fine bornology (see \cref{ss:constBVS} for the fine bornology). 

For an open subset $U\subset X$, the von Neumann bornology on $\clA^{0, k}_{X, \rmc}(U)$ is complete and separated (\cref{cor:seqcomp} and \cref{exm:Hausbornsep}). Hence, by \cref{cor:finetens}, the tensor product bornology on $\clA^{0, k}_{X, \rmc}(U)\otimes L$ is also complete and separated. 
Thus one can define a cochain complex $L^\bullet_X(U)$ of convenient vector spaces by $L^\bullet_X(U)\ceq \clA^{0, \bullet}_{X, \rmc}(U)\otimes L$. The map
\[
L_X^\bullet\colon \frU_X\to \sfC(\CVS), \quad U\mapsto L_X^\bullet(U)
\]
is a precosheaf by letting 
\[
(L_X^\bullet)^U_V\ceq (\clA^{0, \bullet}_{X, \rmc})^U_V\otimes \id_L\colon L_X^\bullet(U)\to L_X^\bullet(V)
\]
for each inclusion $U\subset V$ of open subsets of $X$. 

\begin{lem}\label{lem:mubounded}
For an open subset $U\subset X$, the linear map
\[
\mu_U\colon L_X^k(U)\otimes L_X^l(U)\to L_X^{k+l}(U), \quad 
\xi\otimes x\otimes \eta\otimes y\mapsto \sum_{n\ge 0}\frac{1}{n!}(D_U^n\xi)\wedge \eta\otimes x_{(n)}y
\]
is bounded. 
\end{lem}

\begin{proof}
It suffices to show that the multilinear map
\[
\mu\colon \clA^{0, k}_{X, \rmc}(U)\times \clA^{0, l}_{X, \rmc}(U)\times L\times L\to L_X^{k+l}(U), \quad
(\xi, \eta, x, y)\mapsto \sum_{n\ge 0}\frac{1}{n!}(D_U^n\xi)\wedge \eta\otimes x_{(n)}y
\]
is bounded, or in other words, $\mu(A_1\times A_2\times B_1\times B_2)\subset L^{k+l}_X(U)$ is bounded for any bounded sets $A_1\subset \clA^{0, k}_{X, \rmc}(U)$, $A_2\subset \clA^{0, l}_{X, \rmc}(U)$ and $B_1, B_2\subset L$. For each $n\in \bbN$, define bilinear maps $\varphi_n$ and $\psi_n$ by
\begin{align*}
&\varphi_n\colon \clA^{0, k}_{X, \rmc}(U)\times \clA^{0, l}_{X, \rmc}(U)\to \clA^{0, k+l}_{X, \rmc}(U), \quad (\xi, \eta)\mapsto D_U^n\xi\wedge \eta, \\
&\psi_n\colon L\times L\to L, \quad (x, y)\mapsto x_{(n)}y.
\end{align*}
By \cref{lem:cptpddcont} (2), the map $\varphi_n$ is bounded, since $\clA^{0, k}_{X, \rmc}(U)$ and $\clA^{0, l}_{X, \rmc}(U)$ are LF-spaces. 
The map $\psi_n$ is also bounded by definition of the fine bornology. 
Hence the multilinear map
\[
\mu_n\colon \clA^{0, k}_{X, \rmc}(U)\times \clA^{0, l}_{X, \rmc}(U)\times L\times L\to L_X^{k+l}(U), \quad 
(\xi, \eta, x, y)\mapsto \frac{1}{n!}D^n_U\xi\wedge \eta\otimes x_{(n)}y
\]
is bounded. Now, by definition of the fine bornology, there exists a finite-dimensional linear subspace $E\subset L$ such that $B_1, B_2\subset E$. Thus, taking $N\in \bbN$ so that $E_{(n)}E=\{0\}$ for all $n>N$, we obtain
\[
\mu(A_1\times A_2\times B_1\times B_2)\subset \sum_{n=0}^N\mu_n(A_1\times A_2\times B_1\times B_2). 
\]
This proves that $\mu(A_1\times A_2\times B_1\times B_2)\subset L_X^{k+l}(U)$ is bounded. 
\end{proof}

\begin{dfn}
For an open subset $U\subset X$, define a cochain complex of convenient vector spaces by 
\[
L^\bullet_{X, D}(U)\ceq L_X^\bullet(U)/\olb{\Img}(D_U\otimes \id+\id\otimes T), 
\]
where $\olb{\Img}$ denotes the image in $\sfC(\CVS)$. 
\end{dfn}

For an open subset $U\subset X$, one can define a bilinear map $[\cdot, \cdot]\colon L_{X, D}^k(U)\times L_{X, D}^l(U)\to L_{X, D}^{k+l}(U)$ by
\[
[\ol{a}, \ol{b}]\ceq \ol{\mu_U(a\otimes b)} \quad (a\in L_X^k(U), b\in L_X^l(U)). 
\]

\begin{lem}
For an open subset $U\subset X$, the bilinear map $[\cdot, \cdot]\colon L_{X, D}^k(U)\times L_{X, D}^l(U)\to L_{X, D}^{k+l}(U)$ is bounded. 
\end{lem}

\begin{proof}
This lemma is an immediate consequence of \cref{lem:mubounded} together with the definitions of bornologies on quotients and products. 
\end{proof}

The bounded bilinear map $[\cdot, \cdot]\colon L_{X, D}^k(U)\times L_{X, D}^l(U)\to L_{X, D}^{k+l}(U)$ gives rise to a morphism in $\sfC(\CVS)$
\[
[\cdot, \cdot]\colon L_{X, D}^\bullet(U)\,\wt{\otimes}\, L_{X, D}^\bullet(U)\to L_{X, D}^\bullet(U). 
\]
We find that $(L_{X, D}^\bullet(U), [\cdot, \cdot])$ is a dg Lie algebra in $\CVS$. 

For open subsets $U, V\subset X$ such that $U\subset V$, the morphism $(L_X^\bullet)^U_V\colon L_X^\bullet(U)\to L_X^\bullet(V)$ induces a morphism $(L_{X, D}^\bullet)^U_V\colon L_{X, D}^\bullet(U)\to L_{X, D}^\bullet(V)$ in $\sfC(\CVS)$. 

\begin{lem}
For open subsets $U, V\subset X$ such that $U\subset V$, the morphism $(L_{X, D}^\bullet)^U_V\colon L_{X, D}^\bullet(U)\to L_{X, D}^\bullet(V)$ is a morphism in $\dgLie{\CVS}$. 
\end{lem}

\begin{proof}
A direct calculation shows that 
\[
(L_X^{k+l})^U_V\,\mu_U(a\otimes b)=\mu_V((L_X^k)^U_V(a)\otimes (L_X^l)^U_V(b))
\]
for $a\in L_X^k(U)$ and $b\in L_X^l(U)$. This proves the lemma. 
\end{proof}

The map $L_{X, D}^\bullet\colon \frU_X\to \dgLie{\CVS}$ becomes a precosheaf with the extension morphism $(L_{X, D}^\bullet)^U_V$. 

For open subsets $U, V, W\subset X$ with $U\sqcup V\subset W$, the universality of coproducts in $\sfC(\CVS)$ induces a unique morphism $(L_{X, D}^\bullet)^{U, V}_W\colon L_{X, D}^\bullet(U)\oplus L_{X, D}^{\bullet}(V)\to L_{X, D}^\bullet(W)$ such that the following diagram commutes: 
\[
\begin{tikzcd}[row sep=huge, column sep=huge]
\LXD(U) \arrow[r, hook] \arrow[rd, "(\LXD)^U_W"'] &
\LXD(U)\oplus \LXD(V) \arrow[d, "(\LXD)^{U, V}_W"] &
\LXD(V) \arrow[l, hook'] \arrow[ld, "(\LXD)^V_W"]\\
&
\LXD(W) 
&
\end{tikzcd}
\]

\begin{lem}\label{LXDdgLie}
Let $U, V, W\subset X$ be open subsets such that $U\sqcup V\subset W$. For $a\in L^k_{X, D}(U)$ and $b\in L^l_{X, D}(V)$, 
\[
[(L_{X, D}^k)^U_W(a), (L_{X, D}^l)^V_W(b)]=0. 
\]
Hence $(L_{X, D}^\bullet)^{U, V}_W\colon L_{X, D}^\bullet(U)\oplus L_{X, D}^{\bullet}(V)\to L_{X, D}^\bullet(W)$ is a morphism in $\dgLie{\CVS}$.
\end{lem}

\begin{proof}
For $\xi\in \clA^{0, k}_{X, \rmc}(U)$, $\eta\in \clA^{0, l}_{X, \rmc}(V)$, and $x, y\in L$, we have 
\[
\mu_W((L_X^k)^U_W(\xi\otimes x)\otimes (L_X^l)^V_W(\eta\otimes y))|_{W\bs \supp \xi}=
\mu_W((L_X^k)^U_W(\xi\otimes x)\otimes (L_X^l)^V_W(\eta\otimes y))|_{W\bs \supp \eta}=0.
\]
Since $W=(W\bs\supp \xi)\cup (W\bs\supp \eta)$, it follows that $\mu_W((L_X^k)^U_W(\xi\otimes x)\otimes (L_X^l)^V_W(\eta\otimes y))=0$. This proves the lemma. 
\end{proof}


We summarize this subsection as follows: 

\begin{prp}
The precosheaf $\LXD\colon \frU_X\to \dgLie{\CVS}$ becomes a prefactorization algebra with the factorization product $(\LXD)^{U, V}_W$ and the unit $\{0\}\to \LXD(\varnothing)$. Moreover, by applying the factorization envelope, we obtain prefactorization algebras 
\[
C^\CE_\bullet(L_{X, D}^\bullet)\colon\frU_X\to \sfC(\CVS), \quad
H^\CE_\bullet(\LXD)\colon \frU_X\to \CVS. 
\]
\end{prp}

The prefactorization algebra $\LXD\colon \frU_X\to \dgLie{\CVS}$ can be viewed as a higher-dimensional analogue of the mode Lie algebra $\Lie(L)$. Indeed, as we will see in the next section, when $X=\bbC$ and $D=\pdd_z$, the values of $H^\CE_\bullet(\LCD)$ on annuli contain the universal enveloping algebra $U(\Lie(L))$. See \cref{rmk:ULieL} below. 

\subsection{Enveloping vertex algebras via factorization envelopes}
\label{ss:FEEV}

We now focus on the case where $X=\bbC$ and $D=\pdd_z$. 
As before, we denote by $H^n$ (resp. $H_n$) the cohomology (resp. homology) taken in $\sfC(\CBVS)$, and by $H^n_\sep$ (resp. $H^\sep_n$) the cohomology (resp. homology) taken in $\sfC(\CVS)$.  

\subsubsection{Remarks on cohomology in bornological vector spaces}

Before proceeding to the main results, we make several remarks on the cohomology of complexes of convex bornological vector spaces. 

\begin{lem}\label{lem:cohooplus}
Let $\{X_i\}_{i\in I}$ be a family of complexes of convex bornological vector spaces. For every $n\in\bbZ$, there is a canonical isomorphism of bornological vector spaces 
\[
H^n\bigl(\bigoplus_{i\in I}X_i\bigr)\cong \bigoplus_{i\in I}H^n(X_i).
\]
\end{lem}

\begin{proof}
Denote by $\iota_i\colon X_i\inj \bigoplus_{i\in I}X_i$ the inclusion. Then, for each $i\in I$, by definitions of subspace bornologies and quotient bornologies, 
\[
H^n(X_i)\to H^n\bigl(\bigoplus_{i\in I}X_i\bigr), \quad [x]\mapsto [\iota_i(x)]
\] 
is a bounded linear map. Hence it induces a bounded linear map 
\[
f\colon \bigoplus_{i\in I}H^n(X_i)\to H^n\bigl(\bigoplus_{i\in I}X_i\bigr). 
\]
We construct a bounded inverse of $f$. Define a linear map $g$ by 
\[
g\colon Z^n\bigl(\bigoplus_{i\in I}X_i\bigr)\to \bigoplus_{i\in I}H^n(X_i), \quad \{x_i\}_{i\in I}\mapsto \{[x_i]\}_{i\in I}. 
\]
To prove that $g$ is bounded, it suffices to show that $g(B)$ is bounded for $B=Z^n(\bigoplus_{i\in I}X_i)\cap \sum_{j\in J}\iota_j^n(B_j)$, where $J\subset I$ is a finite subset and $B_j\subset X_j^n$ is a bounded set. Denote by $\pi_i\colon Z^n(X_i)\to H^n(X_i)$ the projection and by $\ol{\iota}_i\colon H^n(X_i)\inj \bigoplus_{i\in I}H^n(X_i)$ the inclusion. Then we have $g(B)\subset \sum_{j\in J}\ol{\iota}_j\pi_j(Z^n(X_j)\cap B_j)$. The set on the right-hand side is bounded, so $g(B)$ is
bounded. Hence $g$ induces a bounded linear map $\wt{g}\colon H^n(\bigoplus_{i\in I}X_i)\to \bigoplus_{i\in I}H^n(X_i)$. It is immediate that $\wt{g}\circ f=\id$ and $f\circ \wt{g}=\id$. Therefore, $f$ is an isomorphism of bornological vector spaces. 
\end{proof}

\begin{lem}\label{lem:homotensfine}
Let $X$ be a complex of convex bornological vector spaces, and let $V$ be a linear space equipped with the fine bornology. For every $n\in \bbZ$, there is a canonical isomorphism of bornological vector spaces
\[
H^n(X\otimes V)\cong H^n(X)\otimes V,
\]
where $V$ is regarded as a complex concentrated in degree
zero.
\end{lem}

\begin{proof}
Choose a basis $\{v_i\}_{i\in I}$ of $V$. By \cref{prp:finechara} (2), there is an isomorphism $V\cong \bbC^{\oplus I}$. Since the bornological tensor product preserves coproducts in each variable, we have $X\otimes V\cong X^{\oplus I}$ and $H^n(X)\otimes V\cong H^n(X)^{\oplus I}$. Thus, by \cref{lem:cohooplus}, 
\[
H^n(X\otimes V)\cong H^n(X^{\oplus I})\cong H^n(X)^{\oplus I}\cong H^n(X)\otimes V
\]
as bornological vector spaces. 
\end{proof}

Let $X$ be a complex of convenient vector spaces. For $p\in \bbN$, the symmetric group $\frS_p$ acts on $X^{\wt{\otimes}p}$ by signed 
permutations according to the Koszul sign rule. The symmetric power $\Sym^p(X)$ is defined by 
\[
\Sym^p(X)\ceq X^{\wt{\otimes}p}/\olb{R}, \quad R\ceq \Span\{\sigma x-x\mid \sigma\in \frS_p,\, x\in X^{\wt{\otimes}p}\}. 
\]
Denote by $e_p$ the symmetrization idempotent 
\[
e_p=\frac{1}{p!}\sum_{\sigma\in \frS_p}\sigma\colon X^{\wt{\otimes}p}\to X^{\wt{\otimes}p}. 
\]
Then $e_p$ is a morphism of complexes of bornological vector spaces. A direct calculation shows that $\Ker e_p=R$. Since $e_p$ is bounded and $X^{\wt{\otimes}p}$ is separated, the linear subspace $R=\Ker e_p$ is bornologically closed in $X^{\wt{\otimes}p}$. Hence we have $\Sym^p(X)=X^{\wt{\otimes}p}/\Ker e_p$. 

\begin{lem}\label{lem:cohosym}
Let $X$ be a complex of convenient vector spaces, and let $p\in \bbN$. There is an isomorphism 
\[
X^{\wt{\otimes}p}\cong \Sym^p(X)\oplus \Ker e_p
\]
of complexes of bornological vector spaces. Consequently, for every $n\in \bbZ$, there is an isomorphism 
\[
H^n(X^{\wt{\otimes}p})\cong H^n(\Sym^p(X))\oplus H^n(\Ker e_p). 
\]
as bornological vector spaces. 
\end{lem}

\begin{proof}
Since $e_p$ is idempotent and bounded, 
\[
X^{\wt{\otimes}p}\cong \Img e_p\oplus \Ker e_p
\]
as complexes of bornological vector spaces. It suffices to show that the canonical morphism
\[
\Sym^p(X)=X^{\wt{\otimes}p}/\Ker e_p\to \Img e_p
\]
is an isomorphism of complexes of bornological vector spaces. It is bounded by definitions of subspace bornologies and quotient bornologies. Moreover, since $e_p$ is idempotent, its inverse map is given by 
\[
\begin{tikzcd}
\Img e_p \arrow[r, hook] & X^{\wt{\otimes}p} \arrow[r, ->>] & X^{\wt{\otimes}p}/\Ker e_p, 
\end{tikzcd}
\]
and is therefore bounded. Thus the claim follows.  
\end{proof}

Let $X$ be a complex of convenient vector spaces equipped with an $S^1$-action
\[
\sigma\colon S^1\to \Aut_{\sfC(\CVS)}(X), \quad q\mapsto \sigma_q. 
\]
For $p\in\bbN$, equip $X^{\wt{\otimes}p}$ with the diagonal $S^1$-action.
Note that the diagonal $S^1$-action commutes with the action of $\frS_p$. 
Hence it induces an $S^1$-action on $\Sym^p(X)$. 
For $\Delta\in \bbZ$, denote by $(X^{\wt{\otimes}p})_\Delta$ and $\Sym^p(X)_\Delta$ the corresponding $\Delta$-weight spaces, as defined in \eqref{eq:cpxDwei}. Since $\Ker e_p$ is invariant under the diagonal $S^1$-action, it inherits an $S^1$-action, whose $\Delta$-weight space is denoted by $(\Ker e_p)_\Delta$. As is clear from the proof of \cref{lem:cohosym}, the isomorphism of complexes constructed there is $S^1$-equivariant. We therefore obtain the following lemma: 

\begin{lem}\label{lem:cohosymDelta}
Let $X$ be a complex of convenient vector spaces equipped with an $S^1$-action, let $p\in \bbN$, and let $\Delta\in \bbZ$. There exists an isomorphism 
\[
(X^{\wt{\otimes}p})_\Delta \cong \Sym^p(X)_\Delta\oplus (\Ker e_p)_\Delta
\]
of complexes of bornological vector spaces. Consequently, for every $n\in \bbZ$, there is an isomorphism 
\[
H^n((X^{\wt{\otimes}p})_\Delta)\cong H^n(\Sym^p(X)_\Delta)\oplus H^n((\Ker e_p)_\Delta)
\]
as bornological vector spaces. 
\end{lem}

\subsubsection{Equivariant structure}
\label{ss:equivLCD}

For an $\bbN$-graded Lie conformal algebra $L=\bigoplus_{\Delta\in \bbN}L_\Delta$, we equip $L_{\bbC, \pdd_z}^\bullet\colon \frU_\bbC\to \dgLie{\CVS}$ with an $S^1\ltimes \bbC$-equivariant structure. Let $U\subset \bbC$ be an open subset. First, for $(q, \zeta)\in S^1\ltimes \bbC$ and $\xi\in \clA^{0, k}_{\rmc}(U)$, define $(q, \zeta)\xi\in \clA^{0, k}((q, \zeta)U)$ by
\[
(q, \zeta)\xi\colon z\mapsto \xi_{(q, \zeta)^{-1}z}. 
\]
Since $\supp (q, \zeta)\xi\subset (q, \zeta)\supp \xi$, we have a linear map
\[
(q, \zeta)\colon \clA^{0, k}_{\rmc}(U)\to \clA^{0, k}_{\rmc}((q, \zeta)U), \quad \xi\mapsto (q, \zeta)\xi. 
\]
It is straightforward to verify that this linear map is continuous. 
Hence we obtain a bounded linear map
\[
\sigma_{(q, \zeta), U}^k\colon L_\bbC^k(U)\to L^k_\bbC((q, \zeta)U), \quad 
\xi\otimes x\mapsto (q, \zeta)\xi\otimes q^{\Delta(x)-1}x. 
\]
We find that 
\[
\sigma_{(q, \zeta), U}\ceq \bigl\{\sigma_{(q, \zeta), U}^k\bigr\}_{k\in \bbZ}\colon 
L_\bbC^\bullet(U)\to L_\bbC^\bullet((q, \zeta)U)
\]
is a morphism in $\sfC(\CVS)$. Since $(q, \zeta)\pdd_z\xi=q\pdd_z(q, \zeta)\xi$ for each $\xi\in \clA^{0, k}_{\rmc}(U)$, the morphism $\sigma_{(q, \zeta), U}$ induces a morphism in $\sfC(\CVS)$
\[
\ol{\sigma}_{(q, \zeta), U}=\bigl\{\ol{\sigma}_{(q, \zeta), U}^k\bigr\}_{k\in \bbZ}\colon L_{\bbC, \pdd_z}^\bullet(U)\to L_{\bbC, \pdd_z}^\bullet((q, \zeta)U). 
\]

\begin{lem}
For an open subset $U\subset \bbC$ and $(q, \zeta)\in S^1\ltimes \bbC$, the morphism $\ol{\sigma}_{(q, \zeta), U}\colon L_{\bbC, \pdd_z}^\bullet(U)\to L_{\bbC, \pdd_z}^\bullet((q, \zeta)U)$ is a morphism in $\dgLie{\CVS}$. 
\end{lem}

\begin{proof}
A direct calculation shows that 
\[
\sigma_{(q, \zeta), U}\,\mu_U(a\otimes b)=
\mu_{(q, \zeta)U}(\sigma_{(q, \zeta), U}(a)\otimes \sigma_{(q, \zeta), U}(b))
\]
for $a\in L_\bbC^k(U)$ and $b\in L_\bbC^l(U)$. This proves the lemma. 
\end{proof}

The prefactorization algebra $L_{\bbC, \pdd_z}^\bullet\colon \frU_\bbC\to \dgLie{\CVS}$ becomes an $S^1\ltimes \bbC$-equivariant prefactorization algebra with the structure morphism $\ol{\sigma}_{(q, \zeta)}=\{\ol{\sigma}_{(q, \zeta), U}\}_{U\in \frU_\bbC}$. Hence, by taking its factorization envelope, we obtain $S^1\ltimes \bbC$-equivariant prefactorization algebras 
\[
C^\CE_\bullet\LCD\colon \frU_\bbC\to \sfC(\CVS), \qquad
H^\CE_\bullet\LCD\colon \frU_\bbC\to \CVS. 
\]
We prove that $C^\CE_\bullet\LCD$ is differentiably translation-equivariant and $H^\CE_\bullet\LCD$ is holomorphically translation equivariant. 

\begin{rmk}\label{rmk:XtimesY}
Let $X$ and $Y$ be complex manifolds. For open subsets $U\subset X$ and $V\subset Y$, as stated in \cite[Proposition 3.5.14]{CG}, there is a canonical isomorphism 
\[
\bigoplus_{k+l=n}\clA^{0, k}_X(U)\,\wt{\otimes}\,\clA^{0, l}_Y(V) \cong\clA^{0, n}_{X\times Y}(U\times V), \quad \xi\otimes \eta\mapsto \xi\boxtimes \eta
\]
of bornological vector spaces. Here $\boxtimes$ denotes the external product of differential forms. Since the tensor product $\wt{\otimes}$ preserves colimits in each variable, this induces an isomorphism 
\[
\bigoplus_{k+l=n}\clA^{0, k}_{X, \rmc}(U)\,\wt{\otimes}\,\clA^{0, l}_{Y, \rmc}(V)\cong \clA^{0, n}_{X\times Y, \rmc}(U\times V)
\]
of bornological vector spaces. Hence we obtain an isomorphism of complexes of bornological vector spaces 
\[ 
\clA^{0, \bullet}_{X, \rmc}(U)\,\wt{\otimes}\,\clA^{0, \bullet}_{Y, \rmc}(V)\cong \clA^{0, \bullet}_{X\times Y, \rmc}(U\times V). 
\]
\end{rmk}

Let $U\subset \bbC$ be an open subset. By \cref{rmk:XtimesY}, we have an isomorphism of complexes of bornological vector spaces
\begin{equation}\label{eq:LCiso}
L_\bbC^{\bullet+1}(U)^{\wt{\otimes}p}=(\clA^{0, \bullet+1}_\rmc(U)\otimes L)^{\wt{\otimes} p}\cong 
\clA^{0, \bullet+p}_\rmc(U^p)\otimes L^{\otimes p}. 
\end{equation}

\begin{lem}\label{lem:otimesquotcom}
Let $X$ be a complex of convenient vector spaces, and let
$M\subset X$ be a subcomplex such that $M^n$ is bornologically closed in $X^n$ for every $n\in\bbZ$. For $p\in \bbZ_{>0}$, there is a canonical isomorphism of complexes of convenient vector spaces
\[
(X/M)^{\wt{\otimes}p}\cong X^{\wt{\otimes}p}/\olb{N_p}, \quad N_p\ceq \sum_{i=1}^p\Img(X^{\totimes(i-1)}\totimes M \totimes X^{\totimes (p-i)}\to X^{\totimes p}). 
\]
\end{lem}

\begin{proof}
Denote by $\pi\colon X\to X/M$ the canonical projection. For each $p\in \bbZ_{>0}$, set 
\[
A_p\ceq \bigoplus_{i=1}^pX^{\totimes(i-1)}\totimes M \totimes X^{\totimes (p-i)}, 
\]
and denote by $\iota_p\colon A_p\to X^{\totimes p}$ the morphism in $\sfC(\CVS)$ whose restriction to the $i$-th summand is 
\[
\id_X^{\totimes (i-1)}\totimes \iota\totimes \id_X^{\totimes (p-i)}\colon X^{\totimes(i-1)}\totimes M \totimes X^{\totimes (p-i)}\to X^{\totimes p},
\]
where $\iota\colon M\inj X$ is the inclusion. Since $\Img \iota_p=N_p$, it suffices to show that 
\[
\begin{tikzcd}[column sep=huge]
A_p \arrow[r, "\iota_p"] & X^{\totimes p} \arrow[r, "\pi^{\totimes p}"] & (X/M)^{\totimes p}
\end{tikzcd}
\]
is a cokernel in $\sfC(\CVS)$. 
We prove this claim by induction on $p$. 
The case $p=1$ is immediate, since $M$ is degreewise bornologically closed. 
Consider the case $p>1$, and suppose that the claim holds for $p-1$. Let $f\colon X^{\totimes p}\to Y$ be a morphism in $\sfC(\CVS)$ such that $f\circ \iota_p=0$. In particular, $f\circ (\iota_{p-1}\otimes \id_X)=0$. 
Since the completed bornological tensor product preserves cokernels in each variable, there exists a unique morphism $g\colon (X/M)^{\totimes (p-1)}\totimes X\to Y$ in $\sfC(\CVS)$ such that $g\circ (\pi^{\totimes (p-1)}\totimes \id_X)=f$. We have 
\[
g\circ (\id \totimes \iota)\circ (\pi^{\totimes(p-1)}\totimes \id)=g\circ (\pi^{\totimes(p-1)}\totimes \id)\circ (\id \totimes \iota)=f\circ (\id \totimes \iota)=0, 
\]
and hence $g\circ (\id \totimes \iota)=0$. Thus $g$ induces a unique morphism $h\colon (X/M)^{\totimes p}\to Y$ in $\sfC(\CVS)$ such that $h\circ (\id\totimes \pi)=g$. We find that $h$ is a unique morphism satisfying $h\circ \pi^{\totimes p}=f$. 
\end{proof}

Let $U\subset \bbC$ be an open subset. Applying the preceding lemma with $X=L_\bbC^{\bullet+1}(U)$ and $M=\olb{\Img}(\pdd_z\otimes \id+\id\otimes T)$, we obtain a canonical surjective bounded morphism 
\[
L^{\bullet+1}_\bbC(U)^{\totimes p}\srj \LCDs(U)^{\totimes p}. 
\]
Consequently, there is a surjective bounded morphism
\[
\clA^{0, \bullet+p}_\rmc(U^p)\otimes L^{\otimes p}\cong L_\bbC^{\bullet+1}(U)^{\wt{\otimes}p}\srj
\LCDs(U)^{\totimes p}\srj 
\Sym^p(\LCDs(U)). 
\]

\begin{prp}\label{prp:CEdiff}
Let $L$ be an $\bbN$-graded Lie conformal algebra. The $S^1\ltimes \bbC$-equivariant prefactorization algebra $C^\CE_\bullet\LCD\colon \frU_\bbC\to \sfC(\CVS)$ is differentiably translation-equivariant. 
\end{prp}

\begin{proof}
Let $0<r_1, \ldots, r_n, R\le \infty$, and let $a_i\in C^\CE_\bullet\LCD(D_{r_i})$. Since $C^\CE_\bullet\LCD$ is the direct sum of symmetric powers, we may assume that $a_i\in \Sym^{p_i}(\LCDs(D_{r_i}))$. Set $p\ceq p_1+\cdots+p_n$. Denote by
\[
s_i\colon \clA_\rmc^{0, \bullet+p_i}(D_{r_i}^{p_i})\otimes L^{\otimes p_i}\srj \Sym^{p_i}(\LCDs(D_{r_i})), \quad
s\colon \clA^{0, \bullet+p}_\rmc(D_R^p)\otimes L^{\otimes p}\srj \Sym^p(\LCDs(D_R))
\]
the above surjective bounded morphisms. By surjectivity and linearity, we may assume that $a_i=s_i(\xi_i\otimes x_i)$, where $\xi_i\in \clA^{0, \bullet+p_i}_\rmc(D_{r_i}^{p_i})$ and $x_i\in L^{\otimes p_i}$. Define 
\[
F\colon \Disks(r_1, \ldots, r_n; R)\to \clA^{0, \bullet+p}_\rmc(D_R^p), \quad 
(z_1,\ldots, z_n)\mapsto \boxtimes_{i=1}^n(1, z_i)\xi_i. 
\]
Then, for $(z_1, \ldots, z_n)\in \Disks(r_1, \ldots, r_n; R)$, 
\[
(C^\CE_\bullet\LCD)^{r_1, \ldots, r_n; R}_{z_1, \ldots, z_n}(a_1\otimes \cdots \otimes a_n)=s(F(z_1, \ldots, z_n)\otimes x_1\otimes \cdots \otimes x_n). 
\]
Hence, by \cref{lem:borndiffboudlin}, it suffices to show that $F$ is bornologically differentiable. 

Let $z^0=(z^0_1, \ldots, z^0_n)\in \Disks(r_1, \ldots, r_n; R)$. Choose a precompact open neighborhood $U$ of $z^0$ such that $\ol{U}\subset \Disks(r_1, \ldots, r_n; R)$. Set 
\[
K\ceq\bigl\{(\zeta^1+z_1, \ldots, \zeta^n+z_n)\mid \zeta^i\in \supp \xi_i, (z_1, \ldots, z_n)\in \ol{U}\bigr\}. 
\]
Here, writing $\zeta^i=(\zeta^i_1, \ldots, \zeta^i_{p_i})\in \supp \xi_i$, we denote $\zeta^i+z_i\ceq (\zeta^i_1+z_i, \ldots, \zeta^i_{p_i}+z_i)$. Then $K$ is compact, since it is the image of compact set $\supp \xi_1\times \cdots \times \supp \xi_n\times \ol{U}$ under the continuous map
\[
(\zeta^1, \ldots, \zeta^n, z_1, \ldots, z_n)\mapsto (\zeta^1+z_1, \ldots, \zeta^n+z_n). 
\]
Moreover, we have $K\subset D_R^p$, since $\ol{U}\subset \Disks(r_1, \ldots, r_n; R)$. For $(z_1, \ldots, z_n)\in U$, 
\[
\supp F(z_1, \ldots, z_n)\subset \supp((1, z_1)\xi_1)\times \cdots \times \supp((1, z_n)\xi_n)\subset K.
\]
Hence $F(U)\subset \clA^{0, \bullet}(D_R^p)_K$. Clearly, the map $F\colon U\to \clA^{0, \bullet}(D_R^p)_K$ is of class $C^1$. Thus, by \cref{lem:Frebordiff}, it is bornologically differentiable. Since the inclusion $\clA^{0, \bullet}(D_R^p)_K\inj\clA^{0, \bullet}_\rmc(D_R^p)$ is bounded, \cref{lem:borndiffboudlin} implies that $F\colon U\to \clA^{0, \bullet}_\rmc(D_R^p)$ is bornologically differentiable. This proves the claim. 
\end{proof}

\begin{prp}\label{prp:CEholo}
Let $L$ be an $\bbN$-graded Lie conformal algebra. The $S^1\ltimes \bbC$-equivariant prefactorization algebra $H^\CE_\bullet\LCD\colon \frU_\bbC\to \CVS$ is holomorphically translation-equivariant. 
\end{prp}

\begin{proof}
It suffices to verify the condition in \cref{lem:smtoholo}. For $0<R\le \infty$, define a morphism $\delta_R\colon \LCD(D_R)\to L_{\bbC, \pdd_z}^{\bullet-1}(D_R)$ of graded objects of convenient vector spaces by
\[
\delta_R\colon L_{\bbC, \pdd_z}^1(D_R)\to L_{\bbC, \pdd_z}^0(D_R), \quad \ol{f(z)d\ol{z}\otimes x}\mapsto \ol{f(z)\otimes x}, 
\]
and extend it to $\delta_R\colon C_\bullet^\CE\LCD(D_R)\to C_{\bullet+1}^\CE\LCD(D_R)$ by the Leibniz rule. We write $d^\CE=\ol{\pdd}+d_{[\cdot, \cdot]}$ for the differential of $C_\bullet^\CE\LCD(D_R)$. Then we have
\[
\ol{\pdd}\circ \delta_R+\delta_R\circ \ol{\pdd}=\pdd_{\ol{z}}, \quad 
d_{[\cdot, \cdot]}\circ \delta_R+\delta_R\circ d_{[\cdot, \cdot]}=0. 
\]
Hence $d^\CE\circ \delta_R+\delta_R\circ d^\CE=\pdd_{\ol{z}}$. Thus, for $0<r_1, \ldots, r_n, R\le \infty$ and cycles $a_i\in Z(C^\CE_\bullet\LCD(D_{r_i}))$, 
\begin{align*}
\pd{}{\ol{z}_i}(C^\CE_\bullet\LCD)^{r_1, \ldots, r_n; R}_{z_1, \ldots, z_n}(a_1\otimes \cdots \otimes a_n)
&=
(C^\CE_\bullet\LCD)^{r_1, \ldots, r_n; R}_{z_1, \ldots, z_n}(a_1\otimes \cdots \otimes \pdd_{\ol{z}}a_i\otimes \cdots \otimes a_n)\\
&=
(C^\CE_\bullet\LCD)^{r_1, \ldots, r_n; R}_{z_1, \ldots, z_n}(a_1\otimes \cdots \otimes d^{\CE}\delta_{r_i}a_i\otimes \cdots \otimes a_n). 
\qedhere
\end{align*}
\end{proof}

Finally, we prove that the $S^1$-action on $C^\CE_\bullet\LCD$ is tame. 

\begin{lem}\label{lem:CEtame}
Let $U\subset \bbC$ be an open subset invariant under the $S^1$-action. Then the $S^1$-action on $C^\CE_\bullet\LCD(U)$ is tame. 
\end{lem}

\begin{proof}
Define an action of $\clD(S^1)$ on $L^{\bullet+1}_\bbC(U)^{\totimes p}$ by 
\[
\varphi*(\xi\otimes x)\ceq
\sum_{\substack{I\subset [p]\\ \#I=k}}(\varphi*_{k+\Delta-p}f_I)d\ol{z}_I\otimes x \quad 
\bigl(\xi=\sum_{\substack{I\subset [p]\\ \#I=k}}f_Id\ol{z}_I\in \clA^{0, k}_\rmc(U^p),\,x\in (L^{\otimes p})_\Delta\bigr),
\]
where we used the isomorphism \eqref{eq:LCiso}. Then \cref{lem:starksepcont} implies that $\varphi*(-)\colon L_\bbC^{\bullet+1}(U)^{\totimes p}\to L_\bbC^{\bullet+1}(U)^{\totimes p}$ is a morphism in $\sfC(\CVS)$. By \cref{lem:otimesquotcom}, this induces an action of $\clD(S^1)$ on $\LCDs(U)^{\totimes p}$. Since the resulting action commutes with the action of $\frS_p$, it induces an action of $\clD(S^1)$ on $\Sym^p(\LCDs(U))$. Taking direct sum over $p$, we obtain an action of $\clD(S^1)$ on $\Sym(\LCDs(U))$. One can check that this action commutes with $d^\CE$. Thus it defines an action of $\clD(S^1)$ on $C^\CE_\bullet\LCD(U)$. 
It is straightforward to verify conditions (i) and (ii) in \cref{dfn:tame}. Moreover, condition (iii) in \cref{dfn:tame} follows from \cref{lem:starksepcont} (1). 
\end{proof}

\subsubsection{The case without central extension}

We are now ready to state and prove one of the main theorems of this paper. 

\begin{thm}\label{thm:main}
Let $L=\bigoplus_{\Delta\in \bbN}L_\Delta$ be an $\bbN$-graded Lie conformal algebra satisfying the following condition: 
\begin{itemize}
\item [$(*)$]
There exists a graded linear subspace $E\subset L$ such that 
$\bbC[T]\otimes E\to L$, $p(T)\otimes x\mapsto p(T)x$
is a linear isomorphism. 
\end{itemize}
Then the $S^1\ltimes \bbC$-equivariant prefactorization algebra $H^\CE_\bullet\LCD\colon \frU_\bbC\to \CVS$ is amenably holomorphic. Moreover, the associated vertex algebra $\bfVR(H^\CE_\bullet\LCD)$ is isomorphic to the enveloping vertex algebra $V(L)$ as an $\bbN$-graded vertex algebra: 
\[
\bfVR(H^\CE_\bullet\LCD) \cong V(L). 
\]
\end{thm}

Fix a graded linear subspace $E\subset L$ satisfying the condition in \cref{thm:main}. 
For an open subset $U\subset \bbC$, 
\[
\clA^{0, \bullet}_{\bbC, \rmc}(U)\otimes E\to \LCD(U), \quad
\xi\otimes x\mapsto \ol{\xi\otimes x}
\]
is an isomorphism of complexes of bornological vector spaces.  
Hence, by \cref{rmk:XtimesY}, we have an isomorphism of complexes of bornological vector spaces 
\begin{equation}\label{eq:LCDpiso}
L_{\bbC, \pdd_z}^{\bullet+1}(U)^{\wt{\otimes}p}\cong 
(\clA^{0, \bullet+1}_{\rmc}(U)\otimes E)^{\wt{\otimes}p} \cong
\clA^{0, \bullet+p}_\rmc(U^p)\otimes E^{\otimes p}. 
\end{equation}
Since the tensor product bornology of fine bornologies is again the fine bornology, for every $n\in \bbZ$, \cref{lem:homotensfine} yields an isomorphism of bornological vector spaces 
\[
H_n(L_{\bbC, \pdd_z}^{\bullet+1}(U)^{\wt{\otimes}p})\cong H^{-n+p}(\clA^{0, \bullet}_\rmc(U^p))\otimes E^{\otimes p}. 
\]
In particular, by \cref{fct:Serre}, we have $H_n(L_{\bbC, \pdd_z}^{\bullet+1}(U)^{\wt{\otimes}p})=0$ whenever $n\neq 0$. Moreover, by \cref{fct:Serre} and \cref{prp:tenssep}, the convex bornological vector space $H_0(L_{\bbC, \pdd_z}^{\bullet+1}(U)^{\wt{\otimes}p})$ is separated. 

\begin{lem}\label{lem:HSymLCD}
Let $p\in \bbN$, and let $U\subset \bbC$ be an open subset. 
\begin{enumerate}
\item
$H_n(\Sym^p(\LCDs(U)))=0$ for every $n\neq 0$. 

\item
$H_0(\Sym^p(\LCDs(U)))$ is separated. 
\end{enumerate}
\end{lem}

\begin{proof}
Let $e_p\colon \LCDs(U)^{\wt{\otimes}p}\to \LCDs(U)^{\wt{\otimes}p}$ be the symmetrization idempotent. Then, by \cref{lem:cohosym}, for every $n\in \bbZ$, 
\[
H_n(\LCDs(U)^{\wt{\otimes}p})\cong H_n(\Sym^p(\LCDs(U)))\oplus H_n(\Ker e_p)
\]
as bornological vector spaces. Assertion (1) follows from the vanishing of $H_n(\LCDs(U)^{\wt{\otimes}p})$ for $n\neq 0$. Moreover, assertion (2) follows from the separatedness of $H_0(\LCDs(U)^{\wt{\otimes}p})$ and \cref{prp:sepoplus}. 
\end{proof}

Let $p\in \bbN$, let $\Delta\in \bbZ$, and let $R\in \bbR_{>0}$. The isomorphism \eqref{eq:LCDpiso} induces an isomorphism of complexes of bornological vector spaces 
\[
(L_{\bbC, \pdd_z}^{\bullet+1}(D_R)^{\wt{\otimes}p})_\Delta\cong
\bigoplus_{\Delta_0+\cdots +\Delta_p=\Delta+p}\clA^{0, \bullet+p}_\rmc(D_R^p)_{\Delta_0}\otimes E_{\Delta_1}\otimes \cdots \otimes E_{\Delta_p}. 
\]
Hence, by \cref{lem:cohooplus} and \cref{lem:homotensfine}, there is an isomorphism of bornological vector spaces 
\begin{equation}\label{eq:HLCDpisowei}
H_n((L_{\bbC, \pdd_z}^{\bullet+1}(D_R)^{\wt{\otimes}p})_\Delta)\cong \bigoplus_{\Delta_0+\cdots +\Delta_p=\Delta+p}H^{-n+p}(\clA^{0, \bullet}_\rmc(D_R^p)_{\Delta_0})\otimes E_{\Delta_1}\otimes \cdots \otimes E_{\Delta_p}. 
\end{equation}
By \cref{fct:Serre} and \cref{lem:HDeltainj}, we have $H_n((\LCDs(D_R)^{\wt{\otimes}p})_\Delta)=0$ whenever $n\neq 0$. 
Moreover, since $E$ is $\bbN$-graded, \cref{lem:DRiHm} (2) implies that $H_0((L_{\bbC, \pdd_z}^{\bullet+1}(D_R)^{\wt{\otimes}p})_\Delta)=0$ if $\Delta<0$. Furthermore, since the linear isomorphism in \cref{lem:DRiHm} (2) is bounded, the convex bornological vector space $H^p(\clA^{0, \bullet+p}_\rmc(D_R^p)_{\Delta_0})$ is separated, and hence $H_0((L_{\bbC, \pdd_z}^{\bullet+1}(D_R)^{\wt{\otimes}p})_\Delta)$ is separated. 

\begin{lem}\label{lem:HDeltaSymLCD}
Let $p\in \bbN$, let $\Delta\in \bbZ$, and let $0<R\le \infty$. 
\begin{enumerate}
\item 
$H_n(\Sym^p(\LCDs(D_R))_\Delta)=0$ for every $n\neq 0$. 

\item 
$H_0(\Sym^p(\LCDs(D_R))_\Delta)=0$ if $\Delta<0$. 

\item 
$H_0(\Sym^p(\LCDs(D_R))_\Delta)$ is separated. 
\end{enumerate}
\end{lem}

\begin{proof}
Let $e_p\colon \LCDs(D_R)^{\wt{\otimes}p}\to \LCDs(D_R)^{\wt{\otimes}p}$ be the symmetrization idempotent. Then, by \cref{lem:cohosymDelta}, 
\[
H_n((\LCDs(D_R)^{\wt{\otimes}p})_\Delta)\cong H_n(\Sym^p(\LCDs(D_R))_\Delta)\oplus H_n((\Ker e_p)_\Delta)
\]
as bornological vector spaces. Assertion (1) follows from the vanishing of $H_n((\LCDs(D_R)^{\wt{\otimes}p})_\Delta)$ for $n\neq 0$, and assertion (2) follows from the vanishing of $H_0((\LCDs(D_R)^{\wt{\otimes}p})_\Delta)$ for $\Delta<0$. Moreover, assertion (3) follows from the separatedness of $H_0((\LCDs(D_R)^{\wt{\otimes}p})_\Delta)$ and \cref{prp:sepoplus}. 
\end{proof}

We can now prove that the $S^1\ltimes \bbC$-equivariant prefactorization algebra $H^\CE_\bullet\LCD\colon \frU_\bbC\to \CVS$ is amenably holomorphic. By \cref{prp:CEholo} and \cref{lem:CEtame}, it suffices to verify conditions (iii) and (iv) in \cref{prp:CGconst}.  

For arbitrary open subset $U\subset \bbC$, we show that $H^\sep_n(C^\CE_\bullet\LCD(U))=0$ whenever $n\neq 0$. Define an increasing filtration on $C^\CE_\bullet\LCD(U)$ by
\[
F_pC^\CE_\bullet\LCD(U)\ceq \Sym^{\le p}(\LCDs(U)). 
\]
We then consider the spectral sequence $E_{p, q}^r$ associated with this filtration: 
\[
E_{p, q}^1\cong H_{p+q}(\Sym^p(\LCDs(U)))\Longrightarrow H_n(C^\CE_\bullet\LCD(U)). 
\]
By \cref{lem:HSymLCD} (1), we have $E_{p, q}^1=0$ whenever $p+q\neq 0$. Hence the spectral sequence $E_{p, q}^r$ degenerates at the first page. It follows that, for $n\neq 0$, 
\[
F_pH_n(C^\CE_\bullet \LCD(U))/F_{p-1}H_n(C^\CE_\bullet\LCD(U))\cong 
H_n(\Sym^p(\LCDs(U)))=0 \quad (p\in \bbN), 
\]
which implies $H_n(C^\CE_\bullet\LCD(U))=0$. Since $H_n^\sep$ is the separation of $H_n$, we obtain $H^\sep_n(C^\CE_\bullet\LCD(U))=0$ for $n\neq0$. 

Next, for $0<R\le \infty$ and $\Delta\in \bbZ$, define an increasing filtration on $C^\CE_\bullet\LCD(D_R)_\Delta$ by 
\[
F_pC^\CE_\bullet\LCD(D_R)_\Delta\ceq \Sym^{\le p}(\LCDs(D_R))_\Delta. 
\]
Consider the spectral sequence $(E_R^\Delta)_{p, q}^r$ associated with this filtration: 
\begin{equation}\label{eq:ERDspesec}
(E_R^\Delta)_{p, q}^1\cong H_{p+q}(\Sym^p(\LCDs(D_R))_\Delta)\Longrightarrow H_n(C^\CE_\bullet\LCD(D_R)_\Delta). 
\end{equation}
By \cref{lem:HDeltaSymLCD} (1), we have $(E_R^\Delta)_{p, q}^1=0$ whenever $p+q\neq 0$. Hence the spectral sequence $(E_R^\Delta)_{p, q}^r$ degenerates at the first page. Thus, for each $p\in \bbN$, we have a linear isomorphism
\begin{equation}\label{eq:Fp/Fp-1}
F_pH_0(C^\CE_\bullet\LCD(D_R)_\Delta)/F_{p-1}H_0(C^\CE_\bullet\LCD(D_R)_\Delta)\sto
H_0(\Sym^p(\LCDs(D_R))_\Delta). 
\end{equation}
By \cref{lem:HDeltaSymLCD} (2), if $\Delta<0$, then $H_0(\Sym^p(\LCDs(D_R))_\Delta)=0$, and hence $H_0(C^\CE_\bullet\LCD(D_R)_\Delta)=0$. Thus, we obtain $H^\sep_0(C^\CE_\bullet\LCD(D_R)_\Delta)=0$ whenever $\Delta<0$. 

For each $p\in\bbN$, since the linear isomorphism \eqref{eq:Fp/Fp-1} is bounded, \cref{lem:HDeltaSymLCD} (3) implies that 
\[
F_pH_0(C^\CE_\bullet\LCD(D_R)_\Delta)/F_{p-1}H_0(C^\CE_\bullet\LCD(D_R)_\Delta)
\]
is separated. 
If $V\subset F_pH_0(C^\CE_\bullet\LCD(D_R)_\Delta)$ is a bounded linear subspace, then the image of $V$ under the canonical projection $F_pH_0\to F_pH_0/F_{p-1}H_0$ is a bounded linear subspace. Since $F_pH_0/F_{p-1}H_0$ is separated, this image is zero. Thus $V\subset F_{p-1}H_0(C^\CE_\bullet\LCD(D_R)_\Delta)$. 
By induction on $p$, it follows that $F_pH_0(C^\CE_\bullet\LCD(D_R)_\Delta)$ is separated for every $p\in \bbN$. 

The Chevalley--Eilenberg complex is the
direct sum of its symmetric-degree components. By definitions of the direct-sum bornology and the quotient bornology, every bounded subset of $H_0(C^\CE_\bullet\LCD(D_R)_\Delta)$ is contained in $F_pH_0(C^\CE_\bullet\LCD(D_R)_\Delta)$ for some $p\in \bbN$. Thus $H_0(C^\CE_\bullet\LCD(D_R)_\Delta)$ is separated. 

It remains to show that, for $0<r<R\le \infty$ and $\Delta\in \bbZ$, the morphism 
\[
H_0((C^\CE_\bullet\LCD)^{D_r}_{D_R})\colon H_0(C^\CE_\bullet\LCD(D_r)_\Delta)\to H_0(C^\CE_\bullet\LCD(D_R)_\Delta)
\]
is a linear isomorphism. For every $p\in \bbN$, by \cref{eq:HLCDpisowei} and \cref{lem:DRiHm} (2), the extension morphism
\[
H_0((\LCDs(D_r)^{\wt{\otimes} p})_\Delta)\to H_0((\LCDs(D_R)^{\wt{\otimes} p})_\Delta)
\]
is a linear isomorphism. Since this linear isomorphism commutes with the symmetrization idempotents,
it respects the direct-sum decompositions of
\cref{lem:cohosymDelta}. It is therefore the direct sum of the induced maps on the symmetric-power summands and the kernel summands. Consequently, the extension morphism 
\[
H_0(\Sym^p(\LCDs(D_r))_\Delta)\to H_0(\Sym^p(\LCDs(D_R))_\Delta)
\]
is a linear isomorphism. Thus, by \eqref{eq:Fp/Fp-1}, the map $H_0((C^\CE_\bullet\LCD)^{D_r}_{D_R})$ induces a linear isomorphism between the associated graded spaces, which shows the claim. 

\smallskip

The proof that $H^\CE_\bullet\LCD\colon \frU_\bbC\to \CVS$ is amenably holomorphic is now complete. 
Therefore, we obtain an $\bbN$-graded vertex algebra $\bfVR(H^\CE_\bullet\LCD)$. 

To construct an isomorphism $\bfVR(H^\CE_\bullet\LCD)\cong V(L)$, we examine the values of $H^\CE_\bullet\LCD$ on annuli. For an open subset $I\subset \bbR$, set
\[
A_I\ceq \{q\exp(t)\mid q\in S^1, t\in I\}, 
\]
which is an open subset of $\bbC$. The assignment
\[
\clU_L\colon \frU_\bbR\to \ModC, \quad I\mapsto \clU_L(I)\ceq\bigoplus_{\Delta\in \bbZ}(H^\CE_\bullet\LCD)(A_I)_\Delta
\] 
inherits a natural prefactorization algebra structure, defined as follows: 
\begin{itemize}
\item 
The extension morphism is given by
\[
(\clU_L)^I_J\ceq (H^\CE_\bullet\LCD)^{A_I}_{A_J}\colon \clU_L(I)\to \clU_L(J)
\]
for open subsets $I, J\subset \bbR$ with $I\subset J$. 

\item 
The factorization product is given by 
\[
(\clU_L)^{I, J}_K\ceq (H^\CE_\bullet\LCD)^{A_I, A_J}_{A_K}\colon \clU_L(I)\otimes \clU_L(J)\to \clU_L(K)
\]
for open subsets $I, J, K\subset \bbR$ with $I\sqcup J\subset K$. 

\item 
The unit is given by that of $H^\CE_\bullet\LCD$. 
\end{itemize}

\smallskip
Let $I\subset \bbR$ be an open interval. For $n\in \bbZ$, by \cref{cor:ArRH1}, there exists $\xi_n\in \clA_{\bbC, \rmc}^{0, 1}(A_I)_{-n}$ such that $\phi(\xi_n)(z^{-n-1}dz)=2\pi\sqrt{-1}$. For $x\in L$, we have 
\[
\ol{\xi_n\otimes x}\in L_{\bbC, \pdd_z}^1(A_I)_{\Delta(x)-n-1}\subset C^\CE_0\LCD(A_I)_{\Delta(x)-n-1}, 
\]
and hence 
\[
[\ol{\xi_n\otimes x}]\in (H^\CE_\bullet\LCD)(A_I)_{\Delta(x)-n-1}\subset \clU_L(I). 
\]
Since $[\pdd_z\xi_n]=n[\xi_{n-1}]$ in $H^1(\clA^{0, \bullet}_{\bbC, \rmc}(A_I)_{-n+1})$, one can define a linear map $\Phi_I$ by
\[
\Phi_I\colon \Lie(L)\to \clU_L(I), \quad x_{[n]}\mapsto [\ol{\xi_n\otimes x}]. 
\]
The definition of $\Phi_I$ is independent of the choice of $\xi_n$. 

For $f\in C^\infty_\rmc(\bbR, \bbC)$, define
\[
f^\times\colon \bbC^\times \to \bbC, \quad z\mapsto f(\log|z|). 
\]
Clearly, $f^\times\in C^\infty_\rmc(\bbC^\times , \bbC)$. Moreover, for an open interval $I\subset \bbR$, if $\supp f\subset I$, then $\supp f^\times\subset A_I$. 
For $f\in C^\infty_\rmc(\bbR, \bbC)$ and $n\in \bbZ$, define 
\[
\xi^f_n\ceq \frac{1}{2}f^\times(z)|z|^{-2}z^{n+1}d\ol{z}\in \clA^{0, 1}_{\bbC, \rmc}(\bbC^\times)_{-n}.
\]
If $\int_\bbR fdt=1$, then $\phi(\xi^f_n)(z^{-n-1}dz)=2\pi\sqrt{-1}$. 

\begin{lem}\label{lem:PhiILiehom}
Let $I, I_{-1}, I_0, I_1\subset \bbR$ be open intervals such that $I_{-1}\sqcup I_0\sqcup I_1\subset I$ and $I_{-1}<I_0<I_1$. Then, for $x, y\in L$ and $m, n\in \bbZ$, we have
\[
(\clU_L)^{I_0, I_1}_I\bigl(\Phi_{I_0}(x_{[m]})\otimes \Phi_{I_1}(y_{[n]})\bigr)-(\clU_L)^{I_{-1}, I_0}_I\bigl(\Phi_{I_{-1}}(y_{[n]})\otimes \Phi_{I_0}(x_{[m]})\bigr)=\Phi_I([x_{[m]}, y_{[n]}]). 
\]
\end{lem}

\begin{proof}
For each $i\in\{-1, 0, 1\}$, choose $f_i\in C^\infty_\rmc(\bbR, \bbC)$ such that $\int_\bbR f_{i}\,dt=1$ and $\supp f_i\subset I_i$. Then we have
\[
(\clU_L)^{I_0, I_1}_I\bigl(\Phi_{I_0}(x_{[m]})\otimes \Phi_{I_1}(y_{[n]})\bigr)=
(\clU_L)^{I_0, I_1}_I\bigl([\ol{\xi^{f_0}_m\otimes x}]\otimes[\ol{\xi^{f_1}_n\otimes y}]\bigr)=
[(\ol{\xi^{f_0}_m\otimes x})(\ol{\xi^{f_1}_n\otimes y})]. 
\]
Similarly, we have
\[
(\clU_L)^{I_{-1}, I_0}_I\bigl(\Phi_{I_{-1}}(y_{[n]})\otimes \Phi_{I_0}(x_{[m]})\bigr)=
[(\ol{\xi^{f_{-1}}_n\otimes y})(\ol{\xi^{f_0}_m\otimes x})]=
[(\ol{\xi^{f_0}_m\otimes x})(\ol{\xi^{f_{-1}}_n\otimes y})]. 
\]
Define a function $g\colon \bbR\to \bbC$ by 
\[
g(t)\ceq \int_{-\infty}^tf_{-1}(\tau)\,d\tau-\int_{-\infty}^tf_1(\tau)\,d\tau.
\]
Then $\supp g\subset I$ and 
\[
\ol{\pdd}g^\times(z)=\frac{1}{2}f^\times_{-1}(z)|z|^{-2}zd\ol{z}-\frac{1}{2}f^\times_1(z)|z|^{-2}zd\ol{z}. 
\]
We obtain $(\ol{\xi^{f_0}_m\otimes x})(\ol{g^\times(z)z^n\otimes y})\in C^\CE_\bullet\LCD(A_I)_{\Delta(x)+\Delta(y)-m-n-2}$ such that 
\begin{align*}
&d_1^\CE\bigl((\ol{\xi^{f_0}_m\otimes x})(\ol{g^\times(z)z^n\otimes y})\bigr)\\
&=
(\ol{\xi^{f_0}_m\otimes x})(\ol{\ol{\pdd}g^\times(z)z^n\otimes y})+[\ol{\xi^{f_0}_m\otimes x},\, \ol{g^\times(z)z^n\otimes y}]\\
&=
(\ol{\xi^{f_0}_m\otimes x})(\ol{\xi^{f_{-1}}_n\otimes y})-(\ol{\xi^{f_0}_m\otimes x})(\ol{\xi^{f_1}_n\otimes y})
+\sum_{j\ge 0}\frac{1}{j!}\ol{(\pdd_z^j\xi^{f_0}_m)g^\times(z)z^n\otimes x_{(j)}y}. 
\end{align*}
It follows that $(\pdd_z^j\xi^{f_0}_m)g^\times(z)z^n\in \clA_{\bbC, \rmc}^{0, 1}(A_I)_{-m-n+j}$ and 
\[
\phi((\pdd_z^j\xi^{f_0}_m)g^\times(z)z^n)(z^{-m-n+j-1}dz)=2\pi\sqrt{-1}j!\binom{m}{j}.
\] 
Therefore, 
\begin{align*}
\Phi_I([x_{[m]}, y_{[n]}])
&=
\sum_{j\ge 0}\binom{m}{j}\Phi_I((x_{(j)}y)_{[m+n-j]})=
\sum_{j\ge 0}\frac{1}{j!}[\ol{(\pdd_z^j\xi^{f_0}_m)g^\times(z)z^n\otimes x_{(j)}y}]\\
&=
[(\ol{\xi^{f_0}_m\otimes x})(\ol{\xi^{f_1}_n\otimes y})]-[(\ol{\xi^{f_0}_m\otimes x})(\ol{\xi^{f_{-1}}_n\otimes y})]. 
\end{align*}
This proves the lemma. 
\end{proof}

\begin{rmk}\label{rmk:ULieL}
For an open interval $I\subset \bbR$, by \cref{lem:PhiILiehom}, one can define a linear map by
\[
\wt{\Phi}_I\colon U(\Lie(L))\to \clU_L(I), \quad a_1\cdots a_p\mapsto (\clU_L)^{I_1, \ldots, I_p}_I(\Phi_{I_1}(a_1)\otimes \cdots \otimes\Phi_{I_p}(a_p)), 
\]
where $I_1, \ldots, I_p\subset \bbR$ are open intervals satisfying $I_1\sqcup \cdots \sqcup I_p\subset I$ and $I_1<\cdots <I_p$. The definition of $\wt{\Phi}_I$ is independent of the choices of $I_1, \ldots, I_p$. We claim that $\wt{\Phi}_I$ is injective. Since $\wt{\Phi}_I(U(\Lie(L))_\Delta)\subset H_0^\sep(C^\CE_\bullet\LCD(A_I)_\Delta)$ for each $\Delta\in \bbZ$, it suffices to show that $\wt{\Phi}_I\colon U(\Lie(L))_\Delta\to H_0^\sep(C^\CE_\bullet\LCD(A_I)_\Delta)$ is injective. 
To prove this, we first equip $U(\Lie(L))$ with the PBW filtration, and $U(\Lie(L))_\Delta$ with the induced filtration. Then 
\[
F_pU(\Lie(L))_\Delta/F_{p-1}U(\Lie(L))_\Delta\cong \Sym^p(\Lie(L))_\Delta. 
\]
Next, define an increasing filtration on $C^\CE_\bullet\LCD(A_I)_\Delta$ by
\[
F_pC^\CE_\bullet\LCD(A_I)_\Delta\ceq 
\Sym^{\le p}(\LCDs(A_I))_\Delta. 
\]
Consider the spectral sequence $(E_I^\Delta)_{p, q}^r$ associated with this filtration: 
\[
(E_I^\Delta)_{p, q}^1\cong H_{p+q}(\Sym^p(\LCDs(A_I))_\Delta)\Longrightarrow
H_n(C^\CE_\bullet\LCD(A_I)_\Delta). 
\]
As in \cref{lem:HDeltaSymLCD} (1), we have $H_n(\Sym^p(\LCDs(A_I))_\Delta)=0$ whenever $n\neq0$. 
Hence the spectral sequence $(E_I^\Delta)_{p, q}^r$ degenerates at the first page. For each $p\in \bbN$, we have a bounded linear isomorphism 
\[
F_pH_0(C^\CE_\bullet\LCD(A_I)_\Delta)/F_{p-1}H_0(C^\CE_\bullet\LCD(A_I)_\Delta)\sto
H_0(\Sym^p(\LCDs(A_I))_\Delta). 
\]
As in \cref{lem:HDeltaSymLCD} (3), the convex bornological vector space $H_0(\Sym^p(\LCDs(A_I))_\Delta)$ is separated. This implies that $H_0(C^\CE_\bullet\LCD(A_I)_\Delta)$ is also separated. 
Note that $\wt{\Phi}_I\colon U(\Lie(L))_\Delta\to H_0(C^\CE_\bullet\LCD(A_I)_\Delta)$ preserves filtrations. 
We find that $\wt{\Phi}_I$ induces an injective linear map between the associated graded spaces, which proves the claim. 
\end{rmk}

Let $0<R\le \infty$. Choose $r\in \bbR_{>0}$ and an open interval $I\subset \bbR$ such that $D_r\sqcup A_I\subset D_R$. The factorization product 
\[
(H_\bullet^\CE\LCD)^{A_I, D_r}_{D_R}\colon 
(H_\bullet^\CE\LCD)(A_I)\otimes (H_\bullet^\CE\LCD)(D_r)\to (H_\bullet^\CE\LCD)(D_R)
\]
induces a linear map
\[
\clU_L(I)\otimes \bfV_{\!r}(H^\CE_\bullet\LCD)\to \bfVR(H^\CE_\bullet\LCD). 
\]
Using the linear map $\Phi_I\colon \Lie(L)\to \clU_L(I)$, we obtain a linear map
\[
\begin{tikzcd}
\alpha_R\,\colon\, \Lie(L) \otimes \bfVR(H^\CE_\bullet\LCD) \arrow[r] &
\clU_L(I) \otimes \bfV_{\!r}(H^\CE_\bullet\LCD) \arrow[r] &
\bfVR(H^\CE_\bullet\LCD),
\end{tikzcd}
\]
which is independent of the choices of $r$ and $I$. 
By \cref{lem:PhiILiehom}, the linear map $\alpha_R$ defines a $U(\Lie(L))$-module structure on $\bfVR(H^\CE_\bullet\LCD)$. 

\begin{lem}\label{lem:killnge0}
For $x\in L$ and $n\in \bbN$, we have $\alpha_R(x_{[n]}, \vac)=0$. 
\end{lem}

\begin{proof}
By the unit axiom of prefactorization algebras,  
\[
\alpha_R(x_{[n]}, \vac)=H_0^\sep\bigl((C^\CE_\bullet\LCD)^{A_I}_{D_R}\bigr)(\Phi_I(x_{[n]})), 
\]
where $I\subset \bbR$ is an open interval such that $A_I\subset D_R$. Choose $f\in C^\infty_\rmc(\bbR, \bbC)$ satisfying $\int_\bbR fdt=1$ and $\supp f\subset I$. Then $\Phi_I(x_{[n]})=[\ol{\xi^f_n\otimes x}]$ in $H_0^\sep(C^\CE_\bullet\LCD(A_I)_{\Delta(x)-n-1})$. Define a function $g\colon \bbR\to \bbC$ by
\[
g(t)\ceq -\int^{\infty}_tf(\tau)\,d\tau. 
\]
We have $\supp g^\times(z)z^n\subset D_R$ and $\ol{\pdd} g^\times(z)z^n=(\clA^{0, 1}_{\bbC, \rmc})^{A_I}_{D_R}(\xi^f_n)$.
Hence $(C^\CE_\bullet\LCD)^{A_I}_{D_R}(\xi^f_n\otimes x)$ is exact in $C^\CE_\bullet\LCD(D_R)_{\Delta(x)-n-1}$. This proves the lemma. 
\end{proof}

By \cref{lem:killnge0}, there exists a unique morphism $\theta_R\colon V(L)\to \bfVR(H^\CE_\bullet\LCD)$ of $U(\Lie(L))$-modules such that $\theta_R(\vac)=\vac$. Clearly, we have $\theta_R(V(L)_\Delta)\subset H_0(C^\CE_\bullet\LCD(D_R)_\Delta)$ for every $\Delta\in \bbZ$. 

\begin{lem}\label{lem:thetaRiso}
The map $\theta_R\colon V(L)\to \bfVR(H^\CE_\bullet\LCD)$ is an isomorphism of $U(\Lie(L))$-modules. 
\end{lem}

\begin{proof}
It suffices to show that $\theta_R\colon V(L)_\Delta\to H_0(C^\CE_\bullet\LCD(D_R)_\Delta)$ is a linear isomorphism for every $\Delta\in \bbZ$. Note that $\Lie_-(L)\ceq \Span_{\bbC}\{x_{[n]}\mid x\in L, n<0\}$ is a Lie subalgebra of $\Lie(L)$. Consider the PBW filtration on $U(\Lie_-(L))$, and equip $U(\Lie_-(L))_\Delta$ with the induced filtration. Define 
\[
F_pV(L)_\Delta\ceq F_pU(\Lie_-(L))_\Delta\vac.  
\]
Then $\{F_pV(L)_\Delta\}_{p\in \bbZ}$ forms an increasing filtration on $V(L)_\Delta$. By the PBW theorem, 
\[
F_pV(L)_\Delta/F_{p-1}V(L)_\Delta\cong \Sym^p(\Lie_-(L))_\Delta. 
\]
Recall the increasing filtration on $H_0((C^\CE_\bullet\LCD)_\Delta(D_R))$ induced by the spectral sequence \eqref{eq:ERDspesec}. By \eqref{eq:Fp/Fp-1}, 
\[
F_pH_0(C^\CE_\bullet\LCD(D_R)_\Delta)/F_{p-1}H_0(C^\CE_\bullet\LCD(D_R)_\Delta)\cong 
H_0(\Sym^p(\LCDs(D_R))_\Delta).
\]
The morphism $\theta_R$ preserves these filtrations. 
We claim that $\theta_R$ induces a linear isomorphism between the associated graded spaces, which proves the lemma. Let $p\in \bbN$. Equip $t^{-1}\bbC[t^{-1}]$ with the grading determined by $\deg(t^{-n-1})\ceq n$. Since $L\cong \bbC[T]\otimes E$, there is a graded linear isomorphism 
\[
t^{-1}\bbC[t^{-1}]\otimes E \sto \Lie_-(L), \quad 
t^{-n-1}\otimes x\mapsto x_{[-n-1]}. 
\]
Consequently, 
\[
(\Lie_-(L)^{\otimes p})_\Delta\cong \bigoplus_{\Delta_0+\cdots +\Delta_p=\Delta}(t^{-1}\bbC[t^{-1}])^{\otimes p}_{\Delta_0}\otimes E_{\Delta_1}\otimes \cdots \otimes E_{\Delta_p}. 
\]
Here, for each $n\in \bbZ$, 
\[
(t^{-1}\bbC[t^{-1}])^{\otimes p}_n=
\bigoplus_{\substack{n_1, \ldots, n_p\in \bbN\\ n_1+\cdots +n_p=n}}\bbC t^{-n_1-1}\otimes \cdots \otimes \bbC t^{-n_p-1}=
\bigoplus_{\substack{n_1, \ldots, n_p\in \bbZ_{>0}\\ n_1+\cdots +n_p=n+p}}\bbC t^{-n_1}\otimes \cdots \otimes \bbC t^{-n_p}. 
\]
By \cref{lem:DRiHm} (2) and \eqref{eq:HLCDpisowei}, we obtain a linear isomorphism
\[
(\Lie_-(L)^{\otimes p})_\Delta\cong H_0((\LCDs(D_R)^{\wt{\otimes}p})_\Delta). 
\]
This induces a linear isomorphism 
\[
\Sym^p(\Lie_-(L))_\Delta\cong H_0(\Sym^p(\LCDs(D_R))_\Delta), 
\]
which proves the claim. 
\end{proof}

\begin{lem}\label{lem:alphaRn}
Let $x\in L$, and let $n\in \bbZ$. For every $b\in \bfVR(H^\CE_\bullet\LCD)$, 
\[
\alpha_R(x_{[n]}\otimes b)=\alpha_R(x_{[-1]}\otimes\vac)_{(n)}b. 
\]
\end{lem}

\begin{proof}
Set $a\ceq \alpha_R(x_{[-1]}, \vac)$. It suffices to show that 
\[
\mu_{\zeta, 0}^R(a\otimes b)=\sum_{n=0}^\infty \zeta^n\alpha_R(x_{[-n-1]}, b)+\sum_{n=0}^\infty \zeta^{-n-1}\alpha_R(x_{[n]}, b) \quad (\zeta\in D_R^{\times}). 
\]
Let $\zeta\in D_R^{\times}$. Choose a bounded open interval $I\subset \bbR$ such that $\zeta\in A_I\subset D_R$. For each $n\in \bbZ$, take $\xi_n\in \clA^{0, 1}_\rmc(A_I)_{-n}$ such that $\phi(\xi_n)(z^{-n-1}dz)=2\pi\sqrt{-1}$. Define a linear map $m_I$ as the composition 
\[
\begin{tikzcd}
m_I\colon H^1(\clA^{0, \bullet}_\rmc(A_I))\otimes \bfVR(H^{\CE}_\bullet\LCD) \arrow[r, "\sim"] & H^1(\clA^{0, \bullet}_\rmc(A_I))\otimes \bfV_{\!r}(H^{\CE}_\bullet\LCD)\arrow[r] & (H^{\CE}_\bullet\LCD)(D_R)
\end{tikzcd}
\]
where $r\in \bbR_{>0}$ is chosen so that $D_r\sqcup A_I\subset D_R$ and the second morphism is 
\[
H^1(\clA^{0, \bullet}_\rmc(A_I))\otimes \bfV_{\!r}(H^{\CE}_\bullet\LCD)\to (H^{\CE}_\bullet\LCD)(D_R), \quad
[\xi]\otimes v\mapsto (H^{\CE}_\bullet\LCD)_{D_R}^{A_I, D_r}([\ol{\xi\otimes x}]\otimes v). 
\]
The definition of $m_I$ is independent of the choice of $r$. For $n\in \bbZ$, we have 
\[
\alpha_R(x_{[n]}\otimes b)=m_I([\xi_n]\otimes b). 
\]
Denote by 
\[
s\colon H^1(\clA^{0, \bullet}_{\rmc}(A_I))\to \Omega^1(A_I)', \quad [\xi]\mapsto \phi(\xi)
\]
the Serre duality isomorphism. Define 
\[
u\colon A_I\to H^1(\clA^{0, \bullet}_\rmc(A_I)), \quad w\mapsto s^{-1}(\varphi_w), 
\]
where $\varphi_w\colon \Omega^1(A_I)\to \bbC$, $f(z)dz\mapsto 2\pi\sqrt{-1}f(w)$. We claim that $u$ is bornologically holomorphic. Let $w_0\in A_I$ be arbitrary. Choose: 
\begin{itemize}
\item 
$\delta\in \bbR_{>0}$ such that $D_{\delta}(w_0)\subset A_I$,  

\item 
an open interval $J\subset \bbR$ such that $A_J\subset D_{\delta/3}$, 

\item 
$f\in C_\rmc^\infty(\bbR, \bbC)$ such that $\supp f\subset J$ and $\int_\bbR f=1$.
\end{itemize}
Then, for $w\in D_{\delta/3}(w_0)$, we have $A_J(w)\subset D_{2\delta/3}(w_0)$ and 
\[
u(w)=\bigl[(\clA^{0, 1}_\rmc)^{A_J(w)}_{A_I}(1, w)\xi^f_{-1}\bigr].
\]
For $w\in D_{\delta/3}(w_0)$, define 
\[
\xi(w)\ceq (\clA^{0, 1}_\rmc)^{A_J(w)}_{A_I}(1, w)\xi^f_{-1}=F_w(z)d\ol{z}, \quad F_w(z)\ceq \frac{1}{2}f^\times(z-w)|z-w|^{-2}. 
\]
We have $\supp \xi(w)\subset K\ceq \ol{D}_{2\delta/3}(w_0)$, and hence $\xi(D_{\delta/3}(w_0))\subset \clA^{0, 1}(A_I)_K$. Clearly, the function $\xi\colon D_{\delta/3}(w_0)\to \clA^{0, 1}(A_I)_K$ is of class $C^\infty$. By \cref{lem:Frebordiff}, it is bornologically differentiable. Since the inclusion $\clA^{0, 1}(A_I)_K\inj \clA^{0, 1}_{\rmc}(A_I)$ is bounded with respect to the von Neumann bornologies, \cref{lem:borndiffboudlin} implies that $\xi\colon D_{\delta/3}(w_0)\to \clA^{0, 1}_\rmc(A_I)$ is bornologically differentiable. Thus, \cref{lem:borndiffboudlin} also implies that $u\colon D_{\delta/3}(w_0)\to H^1(\clA^{0, \bullet}_\rmc(A_I))$ is bornologically differentiable. We have
\[
\pd{\xi}{\ol{w}}(w)=-\pd{F_w}{\ol{z}}(z)d\ol{z}=-\ol{\pdd}F_w, 
\]
and hence
\[
\pd{u}{\ol{w}}(w)=\Bigl[\pd{\xi}{\ol{w}}(w)\Bigr]=[-\ol{\pdd}F_w]=0.
\]
Therefore, by \cref{lem:bornCR}, the function $u\colon D_{\delta/3}(w_0)\to H^1(\clA^{0, \bullet}_\rmc(A_I))$ is bornologically holomorphic. This proves the claim. 
Moreover, by taking $w_0=\zeta$, we obtain 
\[
m_I(u(\zeta)\otimes b)=
(H^{\CE}_\bullet\LCD)_{D_R}^{A_I, D_r}(\bigl[\ol{(\clA^{0, 1}_\rmc)^{A_J(\zeta)}_{A_I}(1, \zeta)\xi^f_{-1}\otimes x}\bigr]\otimes b)=
\mu_{\zeta, 0}^R(a\otimes b). 
\]
By \cref{prp:bornLau}, 
\[
u(w)=\sum_{n=0}^\infty w^nu_{-n-1}+\sum_{n=0}^\infty w^{-n-1}u_n \quad (w\in A_I),
\]
where both series on the right-hand side converge bornologically in $H^1(\clA^{0, \bullet}_\rmc(A_I))$. Since $s$ is bounded, 
\[
\varphi_w=\sum_{n=0}^\infty w^ns(u_{-n-1})+\sum_{n=0}^\infty w^{-n-1}s(u_n) \quad (w\in A_I),
\]
where both series on the right-hand side converge bornologically in $\Omega^1(A_I)'$. By \cref{exm:LCVSborncov}, the above series converges with respect to the strong dual topology. For $w\in A_I$ and $g(z)dz\in \Omega^1(A_I)$, we have
\[
\varphi_w(g(z)dz)=\sum_{n=0}^\infty w^ns(u_{-n-1})(g(z)dz)+\sum_{n=0}^\infty w^{-n-1}s(u_n)(g(z)dz), 
\]
while the usual Laurent expansion of $g$ gives
\[
\varphi_w(g(z)dz)=2\pi\sqrt{-1}g(w)=\sum_{n=0}^\infty w^n\phi(\xi_{-n-1})(g(z)dz)+\sum_{n=0}^\infty w^{-n-1}\phi(\xi_n)(g(z)dz).
\]
Thus, for each $n\in \bbZ$, we have $s(u_n)=\phi(\xi_n)$, and hence $u_n=[\xi_n]$. Since $m_I(-\otimes b)$ is bounded, 
\begin{align*}
\mu_{\zeta, 0}^R(a\otimes b)=m_I(u(\zeta)\otimes b)
&=
\sum_{n=0}^\infty \zeta^nm_I([\xi_{-n-1}]\otimes b)+\sum_{n=0}^\infty \zeta^{-n-1}m_I([\xi_n]\otimes b)\\
&=
\sum_{n=0}^\infty \zeta^n\alpha_R(x_{[-n-1]}, b)+\sum_{n=0}^\infty \zeta^{-n-1}\alpha_R(x_{[n]}, b). 
\qedhere
\end{align*}
\end{proof}

For $x\in L$ and $v\in V(L)$, it follows from \cref{lem:alphaRn} that 
\[
\theta_R(Y(x_{[-1]}\vac, z)v)=Y(\theta_R(x_{[-1]}\vac), z)\theta_R(v). 
\]
Hence $\theta_R\colon V(L)\to \bfVR(H^\CE_\bullet\LCD)$ is a morphism of vertex algebras. The proof of \cref{thm:main} is now complete.

\subsubsection{The case with central extension}

Let $L=\bigoplus_{\Delta\in \bbN}L_\Delta$ be an $\bbN$-graded Lie conformal algebra, and let $\omega_\lambda\colon L\otimes L\to \bbC[\lambda]$ be a $2$-cocycle of $L$. Suppose that $\omega_\lambda$ is compatible with the grading on $L$, in the sense that  
\[
\omega_{(n)}(L_\Delta, L_{\Delta'})=0
\]
whenever $\Delta+\Delta'-n-1\neq0$. We call such a 2-cocycle graded. 
For each open subset $U\subset \bbC$, one can define a morphism $\omega_U\colon \LCD(U)\totimes \LCD(U)\to \bbC c[-1]$ in $\sfC(\CVS)$ by
\[
\omega_U(\ol{\xi\otimes x}, \ol{\eta\otimes y})\ceq 
\frac{1}{2\pi\sqrt{-1}}\sum_{n\ge0}\frac{1}{n!}\Bigl(\int_U(\pdd_z^n\xi)\wedge \eta\wedge dz\Bigr)\, \omega_{(n)}(x, y)c.
\]
Then $\omega\ceq \{\omega_U\}_{U\in \frU_\bbC}$ is a $(-1)$-shifted $2$-cocycle of $\LCD$. Hence, by taking its twisted factorization envelope, we obtain $S^1\ltimes \bbC$-equivariant prefactorization algebras 
\[
C^\CE_{\bullet, \omega}\LCD\colon \frU_\bbC\to \sfC(\CVS), \qquad
H^\CE_{\bullet, \omega}\LCD\colon \frU_\bbC\to \CVS. 
\]

\begin{thm}\label{thm:main2}
Let $L=\bigoplus_{\Delta\in \bbN}L_\Delta$ be an $\bbN$-graded Lie conformal algebra satisfying the condition $(*)$ in \cref{thm:main}, and let $\omega_\lambda\colon L\otimes L\to \bbC[\lambda]$ be a graded $2$-cocycle of $L$. Then the $S^1\ltimes \bbC$-equivariant prefactorization algebra $H^\CE_{\bullet\,\omega}\LCD\colon \frU_\bbC\to \CVS$ is amenably holomorphic. Moreover, the associated vertex algebra $\bfVR(H^\CE_{\bullet,\omega}\LCD)$ is isomorphic to the enveloping vertex algebra $V(L_{\omega_\lambda})$ as an $\bbN$-graded vertex algebra: 
\[
\bfVR(H^\CE_{\bullet,\omega}\LCD)\cong V(L_{\omega_\lambda}). 
\]
\end{thm}

\begin{proof}
The proof of \cref{thm:main} works with minor modifications. We give only a sketch. 

Let $U\subset \bbC$ be an open subset. The underlying graded object of $C^\CE_{\bullet, \omega}\LCD(U)$ is $\Sym(\LCDs(U)\oplus \bbC c)$. Since the completed bornological tensor product preserves colimits in each variable, there is a canonical isomorphism of complexes of convenient vector spaces
\[
\Sym(\LCDs(U)\oplus \bbC c)\cong \Sym(\LCDs(U))\totimes \Sym(\bbC c)\cong \Sym(\LCDs(U))\otimes \bbC[c], 
\]
where $\bbC[c]$ is equipped with the fine bornology. Hence, for $n\in \bbZ$, \cref{lem:homotensfine} implies an isomorphism of bornological vector spaces 
\[
H_n(\Sym(\LCDs(U)\oplus\bbC c))\cong H_n(\Sym(\LCDs(U)))\otimes \bbC[c].
\]
As in the proof of \cref{thm:main}, using the spectral sequence associated with the filtration by symmetric degree, one can verify conditions (iii) and (iv) of \cref{prp:CGconst}. Moreover, the analogues of \cref{prp:CEholo} and \cref{lem:CEtame} hold for the twisted factorization envelope. Thus $H^\CE_{\bullet,\omega}\LCD\colon \frU_\bbC\to \CVS$ is amenably holomorphic. 

The assignment 
\[
\clU_{L, \omega_\lambda}\colon \frU_\bbR\to \ModC, \quad I\mapsto \bigoplus_{\Delta\in \bbZ}(H^\CE_{\bullet, \omega}\LCD)(A_I)_\Delta
\]
inherits a natural prefactorization algebra structure. For an open interval $I\subset \bbR$, define a linear map $\Phi_I$ by 
\[
\Phi_I\colon \Lie(L_{\omega_\lambda})\to \clU_{L, \omega_\lambda}(I), \quad 
x_{[n]}\mapsto [\ol{\xi_n\otimes x}], \quad 
C_{[n]}\mapsto \delta_{n, -1}[c]. 
\]
An argument analogous to the proof of \cref{lem:PhiILiehom}, using the defining formula for $\omega$, shows that $\Phi_I$ satisfies the corresponding commutator relation. Hence we obtain a $U(\Lie (L_{\omega_\lambda}))$-module structure on $\bfVR(H^\CE_{\bullet,\omega}\LCD)$. 
As in \cref{lem:killnge0}, the actions of nonnegative modes annihilate the vacuum.  Thus there exists a unique morphism $\theta_R\colon V(L_{\omega_\lambda})\to \bfVR(H^{\CE}_{\bullet, \omega}\LCD)$ of $U(\Lie(L_{\omega_\lambda}))$-modules such that $\theta_R(\vac)=\vac$. 
For each $\Delta\in \bbZ$, the morphism $\theta_R$ preserves the filtrations on $V(L_{\omega_\lambda})_\Delta$ and $H_0(C^\CE_{\bullet, \omega}\LCD(D_R)_\Delta)$, where $V(L_{\omega_\lambda})_\Delta$ is equipped with the filtration induced by the PBW filtration on $U(\Lie_-(L_{\omega_\lambda}))$ and $H_0(C^\CE_{\bullet, \omega}\LCD(D_R)_\Delta)$ is equipped with the filtration induced by symmetric degree. The associated graded of the former is
\[
\Sym(\Lie_-(L_{\omega_\lambda}))_\Delta
\cong \Sym(\Lie_-(L))_\Delta\otimes\bbC[c],
\]
while the spectral sequence identifies the associated graded of the latter with
\[
H_0(\Sym(\LCDs(D_R)\oplus \bbC c)_\Delta)\cong H_0(\Sym(\LCDs(D_R))_\Delta)\otimes\bbC[c].
\]
As in the proof of \cref{lem:thetaRiso}, the morphism $\theta_R$ induces a linear isomorphism between the associated graded spaces, and is therefore a linear isomorphism. Finally, the analogue of \cref{lem:alphaRn} holds, which shows that $\theta_R$ preserves the vertex operators. 
\end{proof}

\begin{exm}\label{exm:VirKacfct}
\ 
\begin{enumerate}
\item
Recall that the Virasoro conformal algebra $\CA$ is the one-dimensional central extension of the Lie conformal algebra $\bbC[T]L$ determined by the 2-cocycle \eqref{eq:Virconf2cocy}. Note that $\bbC[T]L$ is $\bbN$-graded by declaring $\Delta(L)\ceq 2$. Moreover, the graded linear subspace $\bbC L\subset \bbC[T]L$ satisfies the condition $(*)$ of \cref{thm:main}. Hence \cref{thm:main2} can be applied to $\bbC[T]L$ together with the 2-cocycle \eqref{eq:Virconf2cocy}. Thus we obtain an amenably holomorphic prefactorization algebra whose associated vertex algebra is isomorphic to the enveloping vertex algebra of $\CA$. 

\item
Recall that, for a Lie algebra $\frg$ and a symmetric invariant form $\kappa\colon \frg\otimes \frg\to \bbC$, the current Lie conformal algebra $\Cur_\kappa(\frg)$ is the one-dimensional central extension of the Lie conformal algebra $\bbC[T]\otimes \frg$ determined by the 2-cocycle \eqref{eq:curconf2cocy}. Note that $\bbC[T]\otimes \frg$ is $\bbN$-graded by declaring $\Delta(a)\ceq 1$ for every $a\in \frg$. Moreover, the graded linear subspace $\frg\subset \bbC[T]\otimes \frg$ satisfies the condition $(*)$ of \cref{thm:main}. Hence \cref{thm:main2} can be applied to $\bbC[T]\otimes\frg$ together with the 2-cocycle \eqref{eq:curconf2cocy}. Thus we obtain an amenably holomorphic prefactorization algebra whose associated vertex algebra is isomorphic to the enveloping vertex algebra of $\Cur_\kappa(\frg)$. 
\end{enumerate}
\end{exm}

By \cref{exm:VirKacfct}, we can view \cref{thm:main2} as a generalization of the constructions of the Kac--Moody factorization algebra \cite[\S5.5]{CG} and of the Virasoro factorization algebra \cite{W}. 

\subsubsection{The super case}

By modifying the construction of \cref{ss:HDCLA}, for a complex manifold $X$, a Lie conformal superalgebra $L$, and a morphism $D\colon \clA^{0, \bullet}_X\to \clA^{0, \bullet}_X$ of sheaves on $X$ with values in $\sfC(\CVS)$ satisfying 
\[
D_U(\xi\wedge \eta)=D_U\xi\wedge\eta+\xi\wedge D_U\eta \quad (\xi\in \clA^{0, k}_X(U), \eta\in \clA^{0, l}_X(U)), 
\]
one can construct a prefactorization algebra $\LXD\colon \frU_X\to \dgLie{\sCVS}$. 
As in \cref{ss:equivLCD}, in the case where $X=\bbC$, $D=\pdd_z$, and $L$ is $\frac{1}{2}\bbN$-graded, one can equip 
\[
\LCD\colon \frU_\bbC\to \dgLie{\sCVS}
\]
with a $\Spin(2)\ltimes \bbC$-equivariant structure. 

\begin{thm}\label{thm:main3}
Let $L=\bigoplus_{\Delta\in \frac{1}{2}\bbN}L_\Delta$ be a $\frac{1}{2}\bbN$-graded Lie conformal superalgebra satisfying the following condition: There exists a graded linear subsuperspace $E\subset L$ such that $\bbC[T]\otimes E\to L$, $p(T)\otimes x\mapsto p(T)x$ is a linear isomorphism. 
\begin{enumerate}
\item 
The $\Spin(2)\ltimes \bbC$-equivariant prefactorization algebra $H^\CE_\bullet\LCD\colon \frU_\bbC\to \sCVS$ is amenably holomorphic, and its associated vertex superalgebra $\bfVR^\super(H^\CE_\bullet\LCD)$ is isomorphic to the enveloping vertex superalgebra $V(L)$ as a $\frac{1}{2}\bbN$-graded vertex superalgebra. 

\item 
For a graded $2$-cocycle $\omega_\lambda\colon L\otimes L\to \bbC[\lambda]$ of $L$, the $\Spin(2)\ltimes \bbC$-equivariant prefactorization algebra $H^\CE_{\bullet,\omega}\LCD\colon \frU_\bbC\to \sCVS$ is amenably holomorphic, and its associated vertex superalgebra $\bfVR^\super(H^\CE_{\bullet,\omega}\LCD)$ is isomorphic to the enveloping vertex superalgebra $V(L_{\omega_\lambda})$ as a $\frac{1}{2}\bbN$-graded vertex superalgebra. 
\end{enumerate}
\end{thm}

\begin{proof}
The proofs of \cref{thm:main} and \cref{thm:main2} work with minor modifications. 
\end{proof}

\begin{exm}\label{exm:N124}
\ 
\begin{enumerate}
\item
Recall that the Neveu--Schwarz Lie conformal superalgebra $L_{N=1}$ is the one-dimensional central extension of the Lie conformal superalgebra $\bbC[T]L\oplus \bbC[T]G$ determined by the 2-cocycle \eqref{eq:N1cocy}. Note that $\bbC[T]L\oplus \bbC[T]G$ is $\frac{1}{2}\bbN$-graded by declaring $\Delta(L)\ceq2$ and $\Delta(G)\ceq 3/2$. Moreover, the graded linear subsuperspace $\bbC L\oplus \bbC G\subset \bbC[T]L\oplus \bbC[T]G$ satisfies the condition of \cref{thm:main3}. Hence \cref{thm:main3} (2) can be applied to $\bbC[T]L\oplus \bbC[T]G$ together with the 2-cocycle \eqref{eq:N1cocy}. Thus we obtain an amenably holomorphic prefactorization algebra whose associated vertex superalgebra is the enveloping vertex superalgebra of $L_{N=1}$. 

\item 
Recall that the $N=2$ Lie conformal superalgebra $L_{N=2}$ is the one-dimensional central extension of the Lie conformal superalgebra 
\[
M=\bbC[T]L\oplus \bbC[T]J\oplus \bbC[T]G^+\oplus \bbC[T]G^- 
\]
determined by the 2-cocycle \eqref{eq:N2cocy}. Note that $M$ is $\frac{1}{2}\bbN$-graded by declaring $\Delta(L)\ceq2$, $\Delta(J)\ceq 1$, and $\Delta(G^\pm)\ceq 3/2$. Moreover, the graded linear subsuperspace $\bbC L\oplus \bbC J\oplus \bbC G^+\oplus \bbC G^-\subset M$ satisfies the condition of \cref{thm:main3}. Hence \cref{thm:main3} (2) can be applied to $M$ together with the 2-cocycle \eqref{eq:N2cocy}. Thus we obtain an amenably holomorphic prefactorization algebra whose associated vertex superalgebra is the enveloping vertex superalgebra of $L_{N=2}$. 

\item 
Recall that the $N=4$ Lie conformal algebra $L_{N=4}$ is the one-dimensional central extension of the Lie conformal superalgebra 
\[
M=\bbC[T]L\oplus \bbC[T]J^0\oplus\bbC[T]J^+\oplus \bbC[T]J^-\oplus \bbC[T]G^+\oplus \bbC[T]G^-\oplus \bbC[T]\ol{G}^+\oplus \bbC[T]\ol{G}^-
\]
determined by the 2-cocycle \eqref{eq:N4cocy}. Note that $M$ is $\frac{1}{2}\bbN$-graded by declaring $\Delta(L)\ceq2$, $\Delta(J^0)=\Delta(J^\pm)\ceq 1$, and $\Delta(G^\pm)=\Delta(\ol{G}^\pm)\ceq 3/2$. Moreover, the graded linear subsuperspace 
\[
\bbC L\oplus \bbC J^0\oplus \bbC J^+\oplus \bbC J^-\oplus\bbC G^+\oplus \bbC G^-\oplus\bbC \ol{G}^+\oplus \bbC\ol{G}^-\subset M 
\]
satisfies the condition of \cref{thm:main3}. Hence \cref{thm:main3} (2) can be applied to $M$ together with the 2-cocycle \eqref{eq:N4cocy}. Thus we obtain an amenably holomorphic prefactorization algebra whose associated vertex superalgebra is the enveloping vertex superalgebra of $L_{N=4}$. 
\end{enumerate}
\end{exm}

As seen in \cref{exm:N124}, we obtain new prefactorization algebras corresponding to vertex superalgebras.

\end{document}